\documentclass[a4paper,12pt,centertags,oneside]{amsart}
\usepackage{amsmath,amstext,amsthm,a4,amssymb,amscd}
\usepackage[mathscr]{eucal}
\usepackage{mathrsfs}
\usepackage{epsf}
\usepackage{typearea}
\usepackage{charter}

\usepackage{graphicx}
\usepackage[all]{xy}

\newcommand{\R}{\mathbb{R}}
\newcommand{\C}{\mathbb{C}}
\newcommand{\Z}{\mathbb{Z}}
\newcommand{\N}{\mathbb{N}}

\newcommand{\smooth}{{\mathscr{C}^\infty}}

\newcommand{\e}{\varepsilon}

\newcommand{\im}{\mathrm{Im}}

\newcommand{\End}{\mathrm{End}}

\newcommand{\n}{\nabla}
\newcommand{\Sp}{\mathrm{Sp}}

\newcommand{\bd}{{\bf{bd}}}

\newcommand{\LL}{\mathscr{L}}
\newcommand{\hh}{\mathscr{H}^\bullet}
\newcommand{\hha}{\mathscr{H}^\bullet_\mathrm{abs}}
\newcommand{\hhr}{\mathscr{H}^\bullet_\mathrm{rel}}

\newcommand{\DV}{D^V}
\newcommand{\DsV}{D^{V,2}}
\newcommand{\DVbd}{D^V_{\bf{bd}}}
\newcommand{\DsVbd}{D^{V,2}_{\bf{bd}}}
\newcommand{\DVT}{D^V_T}
\newcommand{\DsVT}{D^{V,2}_T}
\newcommand{\DVTbd}{D^V_{T,{\bf{bd}}}}

\newcommand{\wDVT}{\widetilde{D}^V_T}
\newcommand{\wDVTbd}{\widetilde{D}^V_{T,{\bf{bd}}}}

\newcommand{\DY}{D^Y}
\newcommand{\DsY}{D^{Y,2}}
\newcommand{\DIYR}{D^{IY_R}}

\newcommand{\DIYRT}{D^{IY_R}_T}
\newcommand{\DsIYRT}{D^{IY_R,2}_T}

\newcommand{\wDIYRT}{\widetilde{D}^{IY_R}_T}

\newcommand{\DRT}{D^{Z_R}_T}
\newcommand{\DsRT}{D^{Z_R,2}_T}
\newcommand{\DjRT}{D^{Z_{j,R}}_T}
\newcommand{\DsjRT}{D^{Z_{j,R},2}_T}
\newcommand{\DthreeRT}{D^{IY_R}_T}

\newcommand{\wDjRT}{\widetilde{D}^{Z_{j,R}}_T}

\newcommand{\Djinf}{D^{Z_{j,\infty}}}

\newcommand{\Doneinf}{D^{Z_{1,\infty}}}
\newcommand{\Dtwoinf}{D^{Z_{2,\infty}}}

\newcommand{\DRzero}{D^{Z_R}}
\newcommand{\DsRzero}{D^{Z_R,2}}

\newcommand{\diffRT}{d^{Z_R}_T}
\newcommand{\stdiffRT}{d^{Z_R,*}_T}
\newcommand{\diffjRT}{d^{Z_{j,R}}_T}
\newcommand{\stdiffjRT}{d^{Z_{j,R},*}_T}

\newcommand{\Res}{\mathscr{R}_{d^F}}
\newcommand{\stRes}{\mathscr{R}_{d^{F,*}}}

\makeatletter
\newcommand*{\rom}[1]{\expandafter\@slowromancap\romannumeral #1@}
\makeatother

\DeclareMathOperator{\Id}{Id}
\DeclareMathOperator{\tr}{Tr}
\DeclareMathOperator{\Ker}{Ker}

\newtheorem{prop}{Proposition}[section]
\newtheorem{thm}[prop]{Theorem}
\newtheorem{lemme}[prop]{Lemma}
\newtheorem{cor}[prop]{Corollary}

\theoremstyle{definition}
\newtheorem{defn}[prop]{Definition}

\theoremstyle{remark}
\newtheorem{rem}[prop]{Remark}

\numberwithin{equation}{section}
\title[analytic torsion forms]{Adiabatic limit, Witten deformation \\
and analytic torsion forms}
\date{\today}

\author{Martin PUCHOL}
\address{D{\'e}partement de Math{\'e}matiques,
B{\^a}timent IMO, Bureau 2G3,
Facult{\'e} des Sciences d'Orsay,
Universit{\'e} Paris-Saclay,
F-91405 Orsay Cedex, France}
\email{martin.puchol@math.cnrs.fr}

\author{Yeping ZHANG}
\address{School of Mathematics,
Korea Institute for Advanced Study,
Hoegiro 85, Dongdaemungu,
Seoul 02455, Korea}
\email{ypzhang@kias.re.kr}

\author{Jialin ZHU}
\address{Mathematical Science Research Center,
Chongqing University of Technology,
No. 69 Hongguang Road,
Chongqing 400054, China}
\email{jialinzhu@cqut.edu.cn}

\begin{document}

\begin{abstract}
We consider a smooth fibration equipped with a flat complex vector bundle
and a hypersurface cutting the fibration into two pieces.
Our main result is a gluing formula
relating the Bismut-Lott analytic torsion form of the whole fibration
to that of each piece.
This result confirms a conjecture
proposed in a conference in G{\"o}ttingen in 2003.
Our approach combines
an adiabatic limit along the normal direction of the hypersurface
and a Witten type deformation on the flat vector bundle.
\end{abstract}

\maketitle

\tableofcontents

\section{Introduction}
\label{sect-intro}

We consider a unitarily flat complex vector bundle $(F,\n^F)$ over a compact smooth manifold $X$ without boundary
whose cohomology with coefficients in $F$ vanishes,
i.e., $H^\bullet(X,F)=0$.
Franz \cite{fr}, Reidemeister \cite{rei} and de Rham \cite{dr}
constructed a topological invariant associated with $(F,\n^F)$,
known as Reidemeister-Franz topological torsion (RF-torsion).
RF-torsion is the first algebraic-topological invariant
which can distinguish the homeomorphism types of certain homotopy-equivalent topological spaces \cite{fr,rei}.
RF-torsion could be extended to the case $H^\bullet(X,F)\neq 0$ \cite{dr,mil,wh}.
Both the original construction of RF-torsion
and its extensions are
based on a complex of simplicial chains in $X$ with values in $F$.

By replacing the complex of simplicial chains by the de Rham complex,
Ray and Singer \cite{rs}
obtained an analytic version of RF-torsion,
known as Ray-Singer analytic torsion (RS-torsion).
In the same paper,
Ray and Singer conjectured that
RF-torsion and RS-torsion are equivalent.

Ray-Singer conjecture was proved independently by
Cheeger \cite{c-cm} and M{\"u}ller \cite{m-cm}.
Their result is now known as Cheeger-M{\"u}ller theorem.
Bismut, Zhang and M{\"u}ller simultaneously considered its extension.
M{\"u}ller \cite{m-cm-2} extended Cheeger-M{\"u}ller theorem to the unimodular case,
i.e., the induced metric on the determinant line bundle $\det F$ is flat.
Bismut and Zhang \cite{bz} extended Cheeger-M{\"u}ller theorem to arbitrary flat vector bundle.
There are also various extensions to equivariant cases
\cite{bz2,lr,lu}.

Wagoner \cite{wa} conjectured that
RF-torsion and RS-torsion can be extended to invariants of a fiber bundle,
i.e., a fibration $\pi: M\rightarrow S$
together with a flat complex vector bundle $(F,\n^F)$ over $M$.
Bismut and Lott \cite{bl} confirmed the analytic part of Wagoner's conjecture
by constructing analytic torsion forms (BL-torsion),
which are even differential forms on $S$.
Inspired by the work of Bismut and Lott,
Igusa \cite{ig} constructed higher topological torsions,
known as Igusa-Klein torsion (IK-torsion).
Goette, Igusa and Williams \cite{g-i-w,g-i} used IK-torsion to detect the exotic smooth structure of fiber bundles.
Dwyer, Weiss and Williams \cite{dww} constructed another version of higher topological torsion (DWW-torsion).
Then a natural and important problem is to understand the relation among these higher torsion invariants.

Bismut and Goette \cite{bg} established a higher version of Cheeger-M{\"u}ller/Bismut-Zhang theorem
under the assumption that
there exist a fiberwise Morse function $f: M \rightarrow \R$ and a fiberwise Riemannian metric
such that the fiberwise gradient of $f$ is Morse-Smale \cite{sm}.
Goette \cite{g-a1,g-a2} extended the results in \cite{bg} to fiberwise Morse functions
whose gradient vector fields are not necessarily Morse-Smale.
Bismut and Goette \cite{bg} also extended BL-torsion to the equivariant case.
And there are related works \cite{bu,bg2}.
We refer to the survey by Goette \cite{g} for an overview on higher torsions.

Igusa \cite{ig2} axiomatized higher torsion invariants.
His axiomatization consists of two axioms:
additivity axiom and transfer axiom.
Igusa proved that IK-torsion satisfies his axioms.
Moreover,
any higher torsion invariant satisfying Igusa's axioms
is a linear combination of
IK-torsion and the higher Miller-Morita-Mumford class \cite{mum,mor,miller}.
Badzioch, Dorabiala, Klein and Williams \cite{bdkw} showed that
DWW-torsion satisfies Igusa's axioms.
Ma \cite{ma} studied the behavior of BL-torsion under the composition of submersions.
The result of Ma implies that
BL-torsion satisfies the transfer axiom.
The additivity of BL-torsion was proposed as an open problem
in a conference on higher torsion invariants in 2003
\footnote{Smooth Fibre Bundles and Higher Torsion Invariants,
http://www.uni-math.gwdg.de/wm03/,
G{\"o}ttingen, 2003.}.

Igusa's theory begins with higher torsion invariants for fibrations with closed fibers.
In the additivity axiom \cite[(3.1)]{ig2},
Igusa used fiberwise double to avoid considering fibrations with boundaries.
Assuming that the torsion invariant in question is also defined for fibrations with boundaries,
Igusa \cite[\textsection 5]{ig2} stated a gluing axiom
equivalent to the additivity axiom.
More precisely,
given a hypersurface cutting the fibration into two pieces,
the gluing axiom basically says that
the torsion invariant of the total fibration should be the sum of that of each piece.

The gluing formula for BL-torsion was first precisely formulated by Zhu \cite{imrn-zhu}.
Zhu constructed analytic torsion forms for fibrations with boundaries
and formulated a precise gluing formula for BL-torsion.
This gluing formula, once proved,
will lead to the conclusion that BL-torsion satisfies the gluing axiom.

Now we briefly recall previous works on the gluing formula for RS-torsion and BL-torsion.
The gluing formula for RS-torsion associated with unitarily flat vector bundles
was proved by L{\"u}ck \cite{lu}.
The proof is based on Cheege-M{\"u}ller theorem and the work of Lott and Rothenberg \cite{lr}.
Vishik \cite{vi} gave an alternative proof
without using Cheege-M{\"u}ller theorem or the work of Lott and Rothenberg.
The gluing formula for RS-torsion
was proved by Br{\"u}ning and Ma \cite{bm} in full generality.
The proof is based on the work of Bismut and Zhang \cite{bz2},
which is the equivariant version of \cite{bz},
and the work of Br{\"u}ning and Ma \cite{bm2}.
In our earlier paper \cite{pzz} (announced  in \cite{pzz-cras}),
we gave another proof
by means of adiabatic limit along the normal direction of the hypersurface,
which is also one of the key tools in the present paper.
There are also related works \cite{ha,les,mjmw}.
Zhu \cite{imrn-zhu} proved the gluing formula for BL-torsion
under the same assumption as in \cite{bg}.
Zhu \cite{israel-zhu} also proved the gluing formula for BL-torsion under the assumption that
the fiberwise cohomology of the hypersurface vanishes.
This vanishing condition yields a uniform spectral gap of the fiberwise Hodge de Rham operator
as the metric on the normal direction tends to infinity,
which considerably simplifies the analysis involved.

The method used in \cite{pzz} cannot be directly generalized to the family case.
In other words,
it does not lead to a proof of the gluing formula for BL-torsion in full generality.
The main reason is the lack of a good interpretation of the limit of the analytic torsion forms
when the metric on the normal direction of the hypersurface tends to infinity.

The purpose of this paper is to prove a gluing formula for BL-torsion in full generality,
i.e., to solve the problem proposed in the conference on higher torsion invariants mentioned above.
The technical core of this paper consists of two analytic tools:
the adiabatic limit \cite{dw,m,clm,pw} along the normal direction of hypersurface,
which is exactly the same as in our earlier paper \cite{pzz},
and a Witten type deformation \cite{wi,hs4,bz,bz2,book-zh} on the flat vector bundle.
By introducing the Witten type deformation,
we overcome the difficulties mentioned in the previous paragraph.
We will give a more detailed explanation by the end of this introduction.

Now we briefly recall previous works on the two analytic tools used in this paper.
The adiabatic limit of $\eta$-invariant first appeared
in the work of Bismut and Freed \cite{bf} and in the work of Bismut and Cheeger \cite{bc}.
The adiabatic limit used in our paper first appeared in the work of Douglas and Wojciechowski \cite{dw}
and was further developed in \cite{m,clm,pw}.
We refer to the introduction of \cite{pzz}
for more details on previous works on the adiabatic limit.
The Witten deformation was introduced by Witten \cite{wi} in the language of physics.
In a series of works \cite{hs1,hs2,hs3,hs4},
Helffer and Sj{\"o}strand showed that
the Witten instanton complex,
which arises from Witten deformation,
is isomorphic to the Thom-Smale complex.
Bismut and Zhang \cite[\textsection 8]{bz} extended the result of Helffer and Sj{\"o}strand to arbitrary flat vector bundles.
Later they gave a simple proof in \cite[\textsection 6]{bz2} (cf. \cite[\textsection 6]{book-zh}),
where they did not use the work of Helffer and Sj{\"o}strand.

Let us now give more details about the matter of this paper.
\\

\noindent\textbf{Bismut-Lott's Riemann-Roch-Grothendieck type formula and analytic torsion forms.}
Let $M$ be a smooth manifold.
Let $(F,\n^F)$ be a flat complex vector bundle over $M$ with flat connection $\n^F$,
i.e., $\big(\n^F\big)^2 = 0$.
Let $h^F$ be a Hermitian metric on $F$.
Let $\overline{F}^*$ be the bundle of antilinear functionals on $F$.
We will view $h^F$ as a map from $F$ to $\overline{F}^*$.
Following \cite[(4.1)]{bz} and \cite[(1.31)]{bl}, 
set 
\begin{equation}
\label{intro-def-omegaF}
\omega(F,h^F) = \big(h^F\big)^{-1} \n^F h^F
\in \Omega^1(M,\mathrm{End}(F)) \;.
\end{equation}
Let $f$ be an odd polynomial,
i.e., $f(-x)=-f(x)$.
We fix a square root of $i$,
which we denote by $i^{1/2}$.
In what follows,
the choice of square root will be irrelevant.
Following \cite[(1.34)]{bl}, 
set
\begin{equation}
f(\n^F,h^F) = (2\pi i)^{1/2} \tr\Big[f\Big(\frac{1}{2}(2\pi i)^{-1/2}\omega(F,h^F)\Big)\Big]
\in \Omega^\mathrm{odd}(M) \;.
\end{equation}
Bismut and Lott \cite[\textsection 1]{bl} showed that $f(\n^F,h^F)$ is closed
and its de Rham cohomolgy class
\begin{equation}
f(\n^F) := \big[f(\n^F,h^F)\big]\in H^\mathrm{odd}(M)
\end{equation}
is independent of $h^F$.
For a $\Z$-graded flat complex vector bundle
$\big(F^\bullet = \bigoplus_k F^k,\n^{F^\bullet} = \bigoplus_k \n^{F^k}\big)$
and a Hermitian metric $h^{F^\bullet} = \bigoplus_k h^{F^k}$ on $F^\bullet$,
we denote
\begin{align}
\begin{split}
\label{intro-fsum}
f(\n^{F^\bullet},h^{F^\bullet}) & = \sum_k (-1)^k f(\n^{F^k},h^{F^k})
\in \Omega^\mathrm{odd}(M) \;,\\
f(\n^{F^\bullet}) & = \sum_k (-1)^k f(\n^{F^k})
\in H^\mathrm{odd}(M) \;.
\end{split}
\end{align}
If $f$ is an odd formal power series,
the constructions still make sense.
In the sequel,
we take
\begin{equation}
\label{intro-f}
f(x) = xe^{x^2} \;.
\end{equation}

Now let $\pi: M \rightarrow S$ be a fibration with compact fiber $Z$.
Let $o(TZ)$ be the orientation line of the fiberwise tangent bundle $TZ$.
Let $e(TZ)\in H^{\dim Z}(M,o(TZ))$ be the Euler class of $TZ$
(cf. \cite[(3.17)]{bz}).
Let $H^\bullet(Z,F)$ be the fiberwise de Rham cohomology of $Z$ with coefficients in $F$.
Then $H^\bullet(Z,F)$ is a $\Z$-graded complex vector bundle over $S$
equipped with a canonical flat connection $\n^{H^\bullet(Z,F)}$ (see \cite[Def. 2.4]{bl}).
Bismut and Lott \cite[Thm. 3.17]{bl} established the following Riemann-Roch-Grothendieck type formula
\begin{equation}
\label{intro-RRG}
f\big(\n^{H^\bullet(Z,F)}\big) = \int_Z e(TZ) f(\n^F) \in H^\mathrm{odd}(S) \;.
\end{equation}

Bismut and Lott \cite{bl} refined equation \eqref{intro-RRG}.
We consider a connection of the fibration,
i.e., a splitting
\begin{equation}
\label{intro-THM}
TM = T^HM \oplus TZ\;,
\end{equation}
a metric $g^{TZ}$ on $TZ$
and a Hermitian metric $h^F$ on $F$.
Let $\n^{TZ}$ be the Bismut connection associated with $T^HM$ and $g^{TZ}$ \cite[Def. 1.6]{b}.
Let $e(TZ,\n^{TZ})\in\Omega^{\dim Z}\big(M,o(TZ)\big)$ be the Euler form (cf. \cite[(3.17)]{bz}).
Let $h^{H^\bullet(Z,F)}$ be the $L^2$-metric on $H^\bullet(Z,F)$ induced by the Hodge theory.
Let $Q^S$ be the vector space of real even differential forms on $S$.
Bismut and Lott \cite[Def. 3.22]{bl} constructed a differential form
$\mathscr{T}\in Q^S$
depending on $\big(T^HM,g^{TZ},h^F\big)$
and showed that
\begin{equation}
\label{intro-dT}
d\mathscr{T} =
\int_Z e(TZ,\n^{TZ}) f(\n^F,h^F) -
f\big(\n^{H^\bullet(Z,F)},h^{H^\bullet(Z,F)}\big) \;.
\end{equation}
The differential form $\mathscr{T}$ is called the analytic torsion form of Bismut-Lott.
Now we explain the setup of our gluing formula for analytic torsion forms of Bismut-Lott.
\\

\noindent\textbf{Gluing formula.}
Let $N\subseteq M$ be a hypersurface transversal to $Z$.
We suppose that $\pi\big|_N: N \rightarrow S$ is surjective.
Then $\pi\big|_N$ is a fibration over $S$ with fiber $Y := N \cap Z$.
We suppose that $N$ cuts $M$ into two pieces,
which we denote by $M'_1$ and $M'_2$.
We identify a tubular neighborhood of $N$ with
\begin{equation}
IN:=[-1,1]\times N
\end{equation}
such that
\begin{equation}
IN \cap M'_1 = [-1,0]\times N \;, \hspace{5mm}
IN \cap M'_2 = [0,1]\times N \;.
\end{equation}
Set $\pi_3 = \pi\big|_{IN}: IN \rightarrow S$.
Then $\pi_3$ is a fibration over $S$ with fiber
\begin{equation}
IY:=[-1,1]\times Y \;.
\end{equation}
For $j=1,2$, set $M_j = M_j'\cup IN$.
Let $\pi_j: M_j \rightarrow S$ be the restriction of $\pi$.
Then $\pi_j$ is a fibration over $S$ with fiber
$Z_j:= M_j\cap Z$.
For convenience,
we denote
\begin{equation}
\pi_0 = \pi \;,\hspace{5mm} M_0 = M, \;\hspace{5mm} Z_0 = Z \;,\hspace{5mm}
M_3 = IN \;,\hspace{5mm} Z_3=IY \;.
\end{equation}
Then, for $j=0,1,2,3$,
we have a fibration $\pi_j: M_j\rightarrow S$
with fiber $Z_j$.

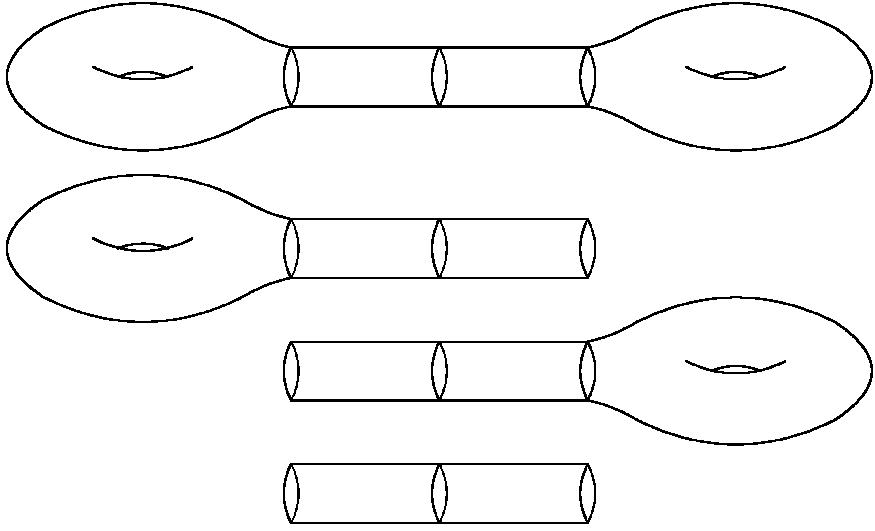
\begin{figure}[ht]
\setlength{\unitlength}{0.65cm}
\centering
\begin{picture}(18,11.5)
\newsavebox{\zleft}
\savebox{\zleft}(3,3)[bl]{
\qbezier(1,1)(3,0)(5,1) 
\qbezier(1,3)(3,4)(5,3) 
\qbezier(1,1)(-0.5,2)(1,3) 
\qbezier(2,2.2)(3,1.7)(4,2.2) 
\qbezier(2.5,2)(3,2.2)(3.5,2) 
\qbezier(5,1)(5.5,1.3)(6,1.4) 
\qbezier(5,3)(5.5,2.7)(6,2.6) 
}
\newsavebox{\zright}
\savebox{\zright}(3,3)[bl]{
\qbezier(13,1)(15,0)(17,1) 
\qbezier(13,3)(15,4)(17,3) 
\qbezier(17,1)(18.5,2)(17,3) 
\qbezier(14,2.2)(15,1.7)(16,2.2) 
\qbezier(14.5,2)(15,2.2)(15.5,2) 
\qbezier(13,1)(12.5,1.3)(12,1.4) 
\qbezier(13,3)(12.5,2.7)(12,2.6) 
}
\newsavebox{\cylinder}
\savebox{\cylinder}(3,3)[bl]{
\qbezier(6,1.4)(9,1.4)(12,1.4) 
\qbezier(6,2.6)(9,2.6)(12,2.6) 
\qbezier(6,1.4)(5.7,2)(6,2.6) 
\qbezier[30](6,1.4)(6.3,2)(6,2.6) 
\qbezier(9,1.4)(8.7,2)(9,2.6) 
\qbezier[30](9,1.4)(9.3,2)(9,2.6) 
\qbezier(12,1.4)(11.7,2)(12,2.6) 
\qbezier[30](12,1.4)(12.3,2)(12,2.6) 
}
\put(0,7.5){\usebox{\zleft}}
\put(0,7.5){\usebox{\zright}}
\put(0,7.5){\usebox{\cylinder}}
\put(0,4){\usebox{\zleft}}
\put(0,4){\usebox{\cylinder}}
\put(0,1.5){\usebox{\zright}}
\put(0,1.5){\usebox{\cylinder}}
\put(0,-1){\usebox{\cylinder}}
\end{picture}
\caption{from the top to bottom: $Z_0=Z$, $Z_1$, $Z_2$ and $Z_3=IY$}
\label{fig-intro}
\end{figure}

Let $(u,y)\in [-1,1]\times Y$ be coordinates on $IY$.
We suppose that the splitting \eqref{intro-THM} on $IN$
is the pullback of a splitting
\begin{equation}
\label{intro-TN}
TN = T^HN \oplus TY \;.
\end{equation}
In particular,
we have
\begin{equation}
\label{intro-THM-THN}
T^HM\big|_{IN} = p^*\big(T^HN\big|_N\big) \;,
\end{equation}
where $p: IN \rightarrow N$ is the canonical projection.
Let $g^{TY}$ be the metric on $TY$
induced by the canonical embedding $Y \hookrightarrow Z$.
We suppose that the metric $g^{TZ}$ is product on $IN$, i.e.,
\begin{equation}
g^{TZ}\big|_{\{u\}\times Y} = du^2 + g^{TY} \;.
\end{equation}
We trivialize $F\big|_{IN}$ using the parallel transport
along the curve $[-1,1]\ni u \mapsto (u,y)$ with respect to $\n^F$.
Since $\n^F$ is flat,
we have
\begin{equation}
\label{eq-intro-nF}
(F,\n^F)\big|_{IN} = p^*\big(F\big|_N,\n^F\big|_N\big) \;,
\end{equation}
where $p: IN \rightarrow N$ is the canonical projection.
We assume that
\begin{equation}
\label{eq-intro-hF}
h^F\big|_{IN} = p^* \big( h^F\big|_N \big) \;.
\end{equation}

For $j=0,1,2,3$,
let $d^{Z_j}$ be the fiberwise de Rham operator on $Z_j$ with values in $F$.
Let $d^{Z_j,*}$ be the formal adjoint of $d^{Z_j}$ with respect to the $L^2$-product (see \eqref{eq-def-L2-product}).
The Hodge de Rham operator is defined as
\begin{equation}
D^{Z_j} = d^{Z_j} + d^{Z_j,*} \;.
\end{equation}

We identify the normal bundle $\mathfrak{n}$ of $\partial Z_j$
with the orthogonal complement of $T(\partial Z_j) \subseteq TZ_j\big|_{\partial Z_j}$.
We denote by $e_{\mathfrak{n}}$ the inward pointing unit normal vector field on $\partial Z_j$.
Let $e^{\mathfrak{n}}$ be the dual vector field.
We denote by $i_\cdot$ (resp. $\wedge \cdot$) the interior (resp. exterior) multiplication.
Following \cite[(1.11),(1.12)]{bm} and \cite[(1.4),(1.5))]{pzz}, 
we denote
\begin{align}
\label{eq-def-abs-bd}
\begin{split}
\Omega^\bullet_\mathrm{abs}(Z_j,F) & = \Big\{ \omega\in \Omega^\bullet(Z_j,F) \;:\;
i_{e_\mathfrak{n}}\omega\big|_{\partial Z_j}=0 \Big\} \;,\\
\Omega^\bullet_{\mathrm{abs},D^2}(Z_j,F) & = \Big\{ \omega\in \Omega^\bullet(Z_j,F) \;:\;
i_{e_\mathfrak{n}}\omega\big|_{\partial Z_j} = i_{e_\mathfrak{n}}(d^{Z_j}\omega)\big|_{\partial Z_j} = 0 \Big\} \;.
\end{split}
\end{align}

The self-adjoint extensions of $D^{Z_j}$ and $D^{Z_j,2}$ with domains
\begin{equation}
\label{eq2-def-abs-bd}
\mathrm{Dom}\big(D^{Z_j}\big) = \Omega^\bullet_\mathrm{abs}(Z_j,F) \;,\hspace{5mm}
\mathrm{Dom}\big(D^{Z_j,2}\big) = \Omega^\bullet_{\mathrm{abs},D^2}(Z_j,F) \;,
\end{equation}
will also be denoted by $D^{Z_j}$ and $D^{Z_j,2}$.
In the sequel,
the boundary condition as above will be called absolute boundary condition.
By the Hodge theorem (cf. \cite[Thm. 1.1]{bm}),
we have an isomorphism
\begin{equation}
\label{eq-Hodge-DF}
H^\bullet(Z_j,F) \simeq \Ker\big(D^{Z_j,2}\big) \subseteq \Omega^\bullet(Z_j,F) \;.
\end{equation}
Let $h^{H^\bullet(Z_j,F)}$
be the Hermitian metric on $H^\bullet(Z_j,F)$
induced by the $L^2$-metric on $\Omega^\bullet(Z_j,F)$
via the identification \eqref{eq-Hodge-DF}.

We have a Mayer-Vietoris exact sequence of flat complex vector bundles over $S$,
\begin{equation}
\label{eq-mv-top}
\cdots\rightarrow
H^k(Z,F) \rightarrow
H^k(Z_1,F) \oplus H^k(Z_2,F) \rightarrow
H^k(IY,F) \rightarrow\cdots \;.
\end{equation}
Let $\mathscr{T}_\mathscr{H} \in Q^S$ be the torsion form
(\cite[Def. 2.20]{bl}, cf. \textsection \ref{subsec-tf})
associated with the exact sequence \eqref{eq-mv-top}
equipped with Hermitian metrics $\big(h^{H^\bullet(Z_j,F)}\big)_{j=0,1,2,3}$.
By \cite[Def. 2.22]{bl},
we have
\begin{equation}
d \mathscr{T}_\mathscr{H} = \sum_{j=0}^3
(-1)^{j(j-3)/2} f\big(\n^{H^\bullet(Z_j,F)},h^{H^\bullet(Z_j,F)}\big) \;.
\end{equation}

We put the absolute boundary condition on the boundary of $Z_j$
(see \eqref{eq-def-abs-bd} and \eqref{eq2-def-abs-bd}).
The analytic torsion form for fibration with boundary
equipped with absolute boundary condition
was constructed by Zhu \cite[Def. 2.18]{imrn-zhu}.
For $j=0,1,2,3$,
let $\mathscr{T}_j \in Q^S$
be the analytic torsion form associated with
\begin{equation}
\big(\pi_j,T^HM\big|_{M_j},g^{TZ}\big|_{M_j},F\big|_{M_j},\n^F\big|_{M_j},h^F\big|_{M_j}\big) \;.
\end{equation}
We denote
\begin{equation}
[\partial Z_j:Y] = \left\{
\begin{array}{rl}
0 & \text{if } j = 0;\\
1 & \text{if } j = 1,2;\\
2 & \text{if } j = 3.
\end{array} \right.
\end{equation}
In other words,
$\partial Z_j$ consists of $[\partial Z_j:Y]$ copies of $Y$.
Let $\n^{TY}$ be the Bismut connection on $TY$ associated with $T^HN$ and $g^{TY}$
\cite[Def. 1.6]{b}.
By \cite[Thm. 2.19]{imrn-zhu},
we have
\begin{align}
\label{eq-d-tf}
\begin{split}
d \mathscr{T}_j
& = \int_{Z_j} e(TZ,\n^{TZ}) f(\n^F,h^F)
+ \frac{[\partial Z_j:Y]}{2} \int_Y e(TY,\n^{TY}) f(\n^F,h^F) \\
& \hspace{5mm} - f\big(\n^{H^\bullet(Z_j,F)},h^{H^\bullet(Z_j,F)}\big) \;.
\end{split}
\end{align}

Let $Q^{S,0}\subseteq Q^S$ be the vector subspace of exact real even differential forms on $S$.

The main result in this paper is the following theorem.

\begin{thm}
\label{thm-gluing}
The following equation holds,
\begin{equation}
\label{eq-thm-gluing}
\mathscr{T}
-\mathscr{T}_1
-\mathscr{T}_2
+\mathscr{T}_3
+\mathscr{T}_\mathscr{H}
\in Q^{S,0} \;.
\end{equation}
\end{thm}

For any closed oriented submanifold $\mathfrak{O}\subseteq S$,
the following map
\begin{equation}
\label{eq-int-QS}
\int_\mathfrak{O}: \; Q^S \rightarrow \R
\end{equation}
may be viewed as a linear function on $Q^S/Q^{S,0}$.
By the Stokes' formula and the de Rham theorem (cf. \cite[Thm. 1.1 (d)]{bm}),
these linear functions separate the elements of $Q^S/Q^{S,0}$.
As a consequence,
to prove Theorem \ref{thm-gluing},
it is sufficient to show that
the integration of the left hand side of \eqref{eq-thm-gluing}
on each $\mathfrak{O}$ vanishes.
Hence,
without loss of generality,
we may and we will assume that
{\bf $S$ is a compact manifold without boundary} in the whole paper.

We note that in Theorem \ref{thm-gluing},
we only use the absolute boundary condition,
whereas the relative boundary condition appears
in the gluing formula for the RS-tosrion in \cite{bm, pzz}
and in the Zhu's formulation of the gluing formula for the BL-torsion \cite{imrn-zhu}.
In fact,
Theorem \ref{thm-gluing} implies Zhu's formula.
In order to keep this paper to a reasonable length,
this will be proved in a subsequent paper,
in which we will also discuss more precisely the link
between BL-torsion and IK-torsion resulting from the work of Igusa \cite{ig2} combined with \cite{ma} and Theorem \ref{thm-gluing}.

Now we briefly describe the strategy of our proof.
\\

\noindent\textbf{A two-parameter deformation and anomaly formulas.}
For $j=1,2$,
set $M_j'' = M_j \backslash IN$.
For $R\geqslant 1$,
set
\begin{equation}
\label{intro-INR-MR}
IN_R = [-R,R]\times N \;,\hspace{5mm}
M_R = M_1'' \cup_N IN_R \cup_N M_2'' \;,
\end{equation}
where we identify $\partial M_j'' = N$ with $\{(-1)^jR\}\times N\subseteq IN_R$ for $j=1,2$.
Then $M_R$ is a closed manifold.
In particular, $M_R\big|_{R=1} = M$.
We construct a smooth fibration
\begin{equation}
\label{intro-piR}
\pi_R : M_R \rightarrow S
\end{equation}
as follows:
$\pi_R\big|_{M_j''} = \pi\big|_{M_j''}$ for $j=1,2$
and $\pi_R\big|_{IN_R}$ being the composition
of the canonical projection $IN_R\rightarrow N$ and $\pi\big|_N: N \rightarrow S$.

For $R\geqslant 1$,
let $\phi_R: [-1,1] \rightarrow [-R,R]$ be a smooth bijective map such that
\begin{align}
\begin{split}
& \phi_R'(u)\geqslant 1 \;,\hspace{5mm}
\phi(-u) = -\phi(u) \hspace{2mm}\text{for } u\in[-1,1] \;,\\
& \phi_R(u) = u - R + 1 \hspace{2mm}\text{for } u\in[-1,-1/2] \;.
\end{split}
\end{align}
We construct a diffeomorphism $\varphi_R: M \rightarrow M_R$ as follows:
\begin{equation}
\label{intro-varphiR}
\varphi_R\big|_{M_1''\cup M_2''} = \mathrm{Id} \;,\hspace{5mm}
\varphi_R\big|_{IN}: (u,y) \mapsto (\phi_R(u),y) \;.
\end{equation}
Then the following diagram commutes
\begin{equation}
\xymatrix{
M \ar[r]^{\varphi_R} \ar[d]_{\pi} & M_R \ar[dl]^{\pi_R} \\
S \;. & }
\end{equation}

Let $Z_R$ be the fiber of $\pi_R$.
We construct a metric $g^{TZ_R}$ on $TZ_R$ as follows:
\begin{equation}
\label{intro-gTZR}
g^{TZ_R}\big|_{M_1'' \cup M_2''} =
g^{TZ}\big|_{M_1'' \cup M_2''} \;,\hspace{5mm}
g^{TZ_R}\big|_{IN_R} = du^2 + g^{TY} \;.
\end{equation}
Set $g^{TZ}_R = \varphi_R^*\big(g^{TZ_R}\big)$.
It is obvious that
$\big(\pi:M\rightarrow S,g^{TZ}_R\big)$ and $\big(\pi_R:M_R\rightarrow S,g^{TZ_R}\big)$ are isometric.
We will work on one or another depending on the context.

Let $f_\infty:[-1,1]\rightarrow\R$ be a self-indexed Morse function such that
\begin{align}
\label{intro-finf}
\begin{split}
& \big\{ u\in[-1,1] \;:\;f'_\infty(u)=0 \big\}
= \big\{-1,0,1\big\} \;,\\
& f_\infty(-1)=f_\infty(1)=0\;,\hspace{5mm}
f_\infty(0)=1 \;.
\end{split}
\end{align}
We can construct a family smooth function $\big(f_T: [-1,1]\rightarrow\R\big)_{T\geqslant 0}$ such that
\begin{equation}
f_T(u) = 0 \;,\hspace{5mm}
\text{for } |u\pm 1| \leqslant e^{-T^2} \;;\hspace{5mm}
f'_T(u) - f'_\infty(u) = \mathscr{O}\big(e^{-T^2}\big) \;.
\end{equation}
We will view $f_T$ as a smooth function on $M_R$ as follows:
\begin{equation}
\label{intro-fT-MR}
f_T\big|_{M\backslash IN} = 0 \;,\hspace{5mm}
f_T(u,y)=f_T(u/R) \hspace{2mm}\text{for } (u,y)\in IN_R \;.
\end{equation}
Then $\varphi_R^*(f_T)$ is a smooth function on $M$.
Set
\begin{equation}
h^F_{R,T} = \exp\big(-2T\varphi_R^*(f_T)\big)h^F \;.
\end{equation}

Replacing $\big(g^{TZ},h^F\big)$ by $\big(g^{TZ}_R,h^F_{R,T}\big)$
and proceeding in the same way as before,
we get analytic torsion forms
$\mathscr{T}_{j,R,T}\in Q^S$ ($j=0,1,2,3$)
and torsion form $\mathscr{T}_{\mathscr{H},R,T} \in Q^S$.
By anomaly formulas
\cite[Thm. 3.24]{bl},
\cite[Thm. 1.5]{israel-zhu},
the class
\begin{equation}
\label{eq-intro-lhs-gluing}
\big[\mathscr{T}_{R,T}
-\mathscr{T}_{1,R,T}
-\mathscr{T}_{2,R,T}
+\mathscr{T}_{3,R,T}
+\mathscr{T}_{\mathscr{H},R,T}\big] \in Q^S/Q^{S,0}
\end{equation}
is independent of $R,T$.
As a consequence,
to prove Theorem \ref{thm-gluing},
it is sufficient to show that
\eqref{eq-intro-lhs-gluing} tends to zero as $R,T\rightarrow+\infty$.
\\

\noindent\textbf{Spectral gap and Witten type theorem.}
For simplicity,
the pushforward $\varphi_{R,*}(F,\n^F,h^F)$
will also be denoted by $(F,\n^F,h^F)$.
We construct a family of Hermitian metrics on $F$ over $Z_R$ as follows:
\begin{equation}
\label{intro-hFT}
h^F_T = e^{-2Tf_T} h^F \;.
\end{equation}
Then we have $h^F_T = \varphi_{R,*}\big(h^F_{R,T}\big)$.
Replacing $\big(g^{TZ_j},h^F\big)$ by $\big(g^{TZ_j}_R,h^F_{R,T}\big)$
in the construction of the Hodge de Rham operator $D^{Z_j}$
and identifying $(Z_j,g^{TZ_j}_R)$ with $(Z_{j,R},g^{TZ_{j,R}})$ via the isometry $\varphi_R\big|_Z$,
we obtain $\wDjRT$ acting on $\Omega^\bullet_\mathrm{abs}(Z_{j,R},F)$.
The operator $\wDjRT$ is self-adjoint
with respect to the $L^2$-metric induced by $g^{TZ_{j,R}}$ and $h^F_T$.
For convenience,
we consider the conjugated operator
$\DjRT = e^{-Tf_T} \wDjRT e^{Tf_T}$,
which is self-adjoint with respect to
the $L^2$-metric induced by $g^{TZ_{j,R}}$ and $h^F$.

We fix a constant $\kappa\in]0,1/3[$.
The following result is crucial (see Theorem \ref{thm-central-spgap}):
there exists $\alpha>0$ such that
for $T = R^\kappa \gg 1$,
we have
\begin{equation}
\label{eq-intro-s-gap}
\Sp\big(R\DjRT\big) \subseteq
]-\infty,-\alpha \sqrt{T}] \cup
[-1,1] \cup
[\alpha \sqrt{T},+\infty[ \;,
\end{equation}
where $\Sp(\cdot)$ is the spectrum.
We call the eigenvalues of $R\DjRT$ lying in $[-1,1]$ (resp. out of $[-1,1]$)
small eigenvalues (resp. large eigenvalues).
Let $\mathscr{E}_{j,R,T}^{[-1,1]}\subseteq\Omega^\bullet(Z_{j,R},F)$ be
the eigenspace of $R\DjRT$ associated with small eigenvalues.
Set
\begin{equation}
\diffjRT = e^{-Tf_T} d^{Z_{j,R}} e^{Tf_T} \;.
\end{equation}
Since $\diffjRT$ commutes with $\DsjRT$,
we get a finite dimensional complex
\begin{equation}
\label{eq-intro-E}
\big(\mathscr{E}_{j,R,T}^{[-1,1]},\diffjRT\big) \;.
\end{equation}
We will show that $\dim \mathscr{E}_{j,R,T}^{[-1,1]}$ is independent of $R$ for $R \gg 1$
and explicitly construct a complex $(C^{\bullet,\bullet}_j,\partial)$ and show that
the complex \eqref{eq-intro-E} is `asymptotic' to $(C^{\bullet,\bullet}_j,\partial)$
as  $T = R^\kappa \rightarrow + \infty$ (see Theorem \ref{thm-central-SRT}).
For instance, taking $j=0$, we have
\begin{align}
\label{eq-intro-C}
\begin{split}
& C^{k,\bullet}_0 = 0 \hspace{5mm}\text{for } k\neq 0,1 \;,\\
& C^{0,\bullet}_0 = H^\bullet(Z_1,F) \oplus H^\bullet(Z_2,F) \;,\hspace{5mm}
C^{1,\bullet}_0 = H^\bullet(Y,F) = H^\bullet(IY,F)
\end{split}
\end{align}
with
$\partial: H^\bullet(Z_1,F) \oplus H^\bullet(Z_2,F) \rightarrow H^\bullet(IY,F)$
being the same map as in \eqref{eq-mv-top}.
This result may be viewed as a variation of the Witten deformation.
\\

\noindent\textbf{Finite propagation speed.}
By the finite propagation speed
for solutions of hyperbolic equations
(cf. \cite[\textsection 2.6, Thm. 6.1]{tay},
\cite[Appendix D.2]{mm}),
the contribution of large eigenvalues to \eqref{eq-intro-lhs-gluing}
tends to $0$ as $T = R^\kappa \rightarrow + \infty$.
On the other hand,
we can explicitly estimate the contribution of small eigenvalues
by applying our Witten type theorem (Theorem \ref{thm-central-SRT}).
These estimates will lead to the conclusion that
\eqref{eq-intro-lhs-gluing} tends to zero
as $T = R^\kappa\rightarrow + \infty$.
\\

If we take $T=0$ and $R\rightarrow + \infty$,
the situation on each fiber is exactly
what was studied in our earlier paper \cite{pzz}.
We owe readers an explanation for introducing the second parameter $T$.
Now we try to answer the following questions.
\begin{itemize}
\item[-] Why we cannot prove the gluing formula for analytic torsion forms
by simply taking $T=0$ and $R\rightarrow + \infty$ ?
\item[-] How does the second parameter $T$ improve the situation ?
\end{itemize}
Both in \cite{pzz} and in this paper,
the contribution of large eigenvalues can be controlled by means of the finite propagation speed method.
The difficulties come from the small eigenvalues.
In \cite{pzz},
the small eigenvalues are handled in a rather brutal way:
we estimate the contribution of each eigenvalue and take the sum of them.
Such a proof highly relies on the expression of the analytic torsion in terms of the zeta-function associated with the eigenvalues,
which does not hold for analytic torsion forms.
An alternative way is to build a model encoding the asymptotic limit of the small eigenvalues.
However,
with $T=0$ and $R\rightarrow + \infty$,
we find infinitely many small eigenvalues.
It seems hopeless to find a reasonable model.
This problem is solved by taking $T = R^\kappa\rightarrow + \infty$.
With the new parameter $T$ introduced,
there remain finitely many small eigenvalues (see \eqref{eq-intro-E}).
Moreover,
for $T = R^\kappa$ large enough,
the dimension of the eigenspace associated with small eigenvalues is a constant.
And a model $(C^{\bullet,\bullet}_j,\partial)$ is built accordingly (see \eqref{eq-intro-C}).

Now we explain the model in more detail.
Recall that the eigenspace associated with small eigenvalues is denoted by $\mathscr{E}_{j,R,T}^{[-1,1]}$
(see \eqref{eq-intro-E}).
Since we work with a fibration over $S$,
both $C^{\bullet,\bullet}_j$ and $\mathscr{E}_{j,R,T}^{[-1,1]}$ are vector bundles over $S$.
The vague word `model' should be interpreted as follows:
we construct a bijection (parameterized by $R,T$) between vector bundles $C^{\bullet,\bullet}_j \rightarrow \mathscr{E}_{j,R,T}^{[-1,1]}$,
which we denote by $\mathscr{S}_{j,R,T}$ in this paper (see Theorem \ref{thm-central-SRT}).
We denote
\begin{equation}F^1\mathscr{E}_{j,R,T}^{[-1,1]}
= \mathscr{S}_{j,R,T}(C^{1,\bullet}_j)
\subseteq \mathscr{E}_{j,R,T}^{[-1,1]} \;.
\end{equation}
Then we have induced bijections
\begin{equation}
\label{eq-intro-C2filtE}
C^{0,\bullet}_j \rightarrow \mathscr{E}_{j,R,T}^{[-1,1]}/F^1\mathscr{E}_{j,R,T}^{[-1,1]} \;,\hspace{5mm}
C^{1,\bullet}_j \rightarrow F^1\mathscr{E}_{j,R,T}^{[-1,1]} \;.
\end{equation}
There is a canonical way to equip $C^{\bullet,\bullet}_j$ and $\mathscr{E}_{j,R,T}^{[-1,1]}$ with superconnections (parameterized by $R,T$).
As $T = R^\kappa \rightarrow \infty$,
the maps in \eqref{eq-intro-C2filtE} tend to be compatible with the superconnections in certain sense.
Similar phenomena appeared in various works on the analytic torsion forms
(cf. \cite[\textsection 10, \textsection 11]{bg}, \cite[Thm. 2.9]{ma-ajm} and  \cite[Thm. 4.4]{ma}).

The situation becomes more clear once we pass to the cohomology.
We will construct a bijection (see \eqref{eq2-def-SHRT}, \eqref{eq-cor-SRTH-bij}, \eqref{eq-def-SRTj} and \eqref{eq-def-SRT3})
\begin{equation}
\label{eq-intro-modelH}
\big[\mathscr{S}_{j,R,T}^H\big]_T: H^\bullet\big(C^{\bullet,\bullet}_j,\partial\big) \rightarrow H^\bullet\big(\mathscr{E}_{j,R,T}^{[-1,1]},\diffjRT \big) \simeq H^\bullet(Z_j,F) \;,
\end{equation}
where the last isomorphism is induced by the Hodge theory.
Here $\big[\mathscr{S}_{j,R,T}^H\big]_T$ is not directly induced by $\mathscr{S}_{j,R,T}$.
Necessary modification is required since $\mathscr{S}_{j,R,T}$ is not a map between complexes (see \eqref{eq4-thm-central-SRT}).
Both $H^\bullet\big(C^{\bullet,\bullet}_j,\partial\big)$ and $H^\bullet(Z_j,F)$ are flat vector bundles over $S$.
But $\big[\mathscr{S}_{j,R,T}^H\big]_T$ is not necessarily a map between flat vector bundles.
To properly interpret the flatness of $\big[\mathscr{S}_{j,R,T}^H\big]_T$,
we consider the short exact sequence induced by \eqref{eq-intro-modelH},
\begin{equation}
0 \rightarrow H^1\big(C^{\bullet,\bullet}_j,\partial\big)
\rightarrow H^\bullet(Z_j,F)
\rightarrow H^0\big(C^{\bullet,\bullet}_j,\partial\big) \rightarrow 0 \;.
\end{equation}
This is indeed an exact sequence of flat vector bundles.
\\

This paper is organized as follows.

In \textsection \ref{sect-appendix},
we establish several technical results
concerning the finite dimensional Hodge theory and torsion forms.
We also recall the construction of analytic torsion forms.

In \textsection \ref{sect-model},
we build up a finite dimensional model of the problem addressed in this paper.

In \textsection \ref{sect-construction},
we state several intermediate results
and show that these results lead to Theorem \ref{thm-gluing}.
The proof of these results are delayed to
\textsection \ref{sect-al-wd}, \ref{sect-tf}, \ref{sect-mv}.

In \textsection \ref{sect-dim-one-wd},
we study a one-dimensional Witten type deformation.

In \textsection \ref{sect-al-wd},
we establish the crucial spectral gap \eqref{eq-intro-s-gap}
and study the asymptotics of the complex
$\big(\mathscr{E}_{j,R,T}^{[-1,1]},\diffjRT\big)$.

In  \textsection \ref{sect-tf},
we study the asymptotics of the analytic torsion forms $\mathscr{T}_{j,R,T}$
as $T = R^\kappa \rightarrow +\infty$.

In \textsection \ref{sect-mv},
we study the asymptotics of the torsion form $\mathscr{T}_{\mathscr{H},R,T}$
as $T = R^\kappa\rightarrow +\infty$.
\\

\noindent\textbf{Notations.}
Hereby we summarize some frequently used notations and conventions.

For a manifold $X$ and a flat complex vector bundle $(F,\n^F)$ over $X$,
we denote
\begin{equation}
\Omega^\bullet(X,F) = \smooth(X,\Lambda^\bullet(T^*X) \otimes F) \;,
\end{equation}
the vector space of smooth differential forms on $X$ with values in $F$.
The de Rham operator on $\Omega^\bullet(X,F)$ is defined as follows:
\begin{equation}
\label{eq-derham-op}
d^X : \omega \otimes s \mapsto d\omega \otimes s + (-1)^{\deg \omega}\omega\wedge\n^F s \;,\hspace{5mm}
\text{for } \omega\in\Omega^\bullet(X) \;,\; s\in\smooth(X,F) \;.
\end{equation}
Then $\big(\Omega^\bullet(X,F),d^X\big)$ is
the de Rham complex of smooth differential forms on $X$ with values in $F$.
Its cohomology is denoted by $H^\bullet(X,F)$.

For a submanifold $U\subseteq X$
and $\omega\in\Omega^\bullet(X,F)$,
we denote by
$\omega\big|_U \in  \smooth(U,\Lambda^\bullet(T^*X) \otimes F)$
its restriction on $U$.
Let $j: U\rightarrow X$ be the canonical embedding.
For $\omega\in\Omega^\bullet(X,F)$ closed,
we denote $[\omega]\big|_U = j^*[\omega] \in H^\bullet(U,F)$.
We remark that in general $\omega\big|_U \notin [\omega]\big|_U$,
unless $\dim U = \dim X$.

If $TX$ is equipped with a Riemannian metric $g^{TX}$,
and $F$ is equipped with a Hermitian metric $h^F$,
we denote by
$\big\lVert\cdot\big\rVert_X$ (resp. $\big\langle\cdot,\cdot\big\rangle_X$)
the $L^2$-norm (resp. $L^2$-product)
on $\Omega^\bullet(X,F)$.
More precisely,
for $\omega,\mu\in\Omega^\bullet(X,F)$,
we have
\begin{equation}
\label{eq-def-L2-product}
\big\langle\omega,\mu\big\rangle_X =
\int_X \big\langle\omega_x,\mu_x\big\rangle_{\Lambda^\bullet(T_x^*X)\otimes F_x}dv(x) \;,
\end{equation}
where
$\big\langle\cdot,\cdot\big\rangle_{\Lambda^\bullet(T_x^*X)\otimes F_x}$
is the scalar product on $\Lambda^\bullet(T_x^*X)\otimes F_x$ induced by $g^{TX}_x$ and $h^F_x$,
and $dv$ is the Riemannian volume form on $(X,g^{TX})$.
For a submanifold $U\subseteq X$,
we denote by $\big\lVert\cdot\big\rVert_U$ (resp. $\big\langle\cdot,\cdot\big\rangle_U$)
the $L^2$-norm (resp. $L^2$-product) on $\smooth(U,\Lambda^\bullet(T^*X) \otimes F)$
with respect to the induced Riemannian metric on $TU$.
For simplicity,
for $\omega,\mu\in\Omega^\bullet(X,F)$,
we will abuse the notations as follows,
\begin{equation}
\big\lVert\omega\big\rVert_U = \big\lVert\omega\big|_U\big\rVert_U\;,\hspace{5mm}
\big\langle\omega,\mu\big\rangle_U = \big\langle\omega\big|_U,\mu\big|_U\big\rangle_U \;.
\end{equation}

For any set $X$,
we denote by $\Id_X: X \rightarrow X$ the identity map.

For a self-adjoint operator $A$,
we denote by $\Sp(A)$ its spectrum.
\\

\noindent\textbf{Acknowledgments.}
We are grateful to Professor Xiaonan Ma for having raised the question which is solved in this paper.
Y. Z. thanks his advisor Professor Jean-Michel Bismut
for helpful discussions.
Y. Z. thanks Professor Weiping Zhang and Professor Huitao Feng
for their warm reception in Chern Institute of Mathematics.
J. Z. thanks Shanghai Center for Mathematical Sciences
for its working environment and support during the completion of this work.

Y. Z.  was supported by Japan Society for the Promotion of Science (JSPS)
KAKENHI grant No. JP17F17804
and Korea Institute for Advanced Study individual grant No. MG077401.
J. Z. was supported by National Natural Science Foundation of China (NNSFC)
grants No. 11601089 and No. 11571183.

\section{Preliminaries}
\label{sect-appendix}

This section is organized as follows.
In \textsection \ref{subsect-hodge},
we state the finite dimensional Hodge theory and
establish several estimates concerning the spectral decomposition of the Hodge Laplacian.
In \textsection \ref{subsec-tf},
we recall the definition of torsion forms
and establish several estimates concerning the comparison of torsion forms.
In \textsection \ref{subsec-atf},
we recall the definition of analytic torsion forms \cite{bl,imrn-zhu}.

\subsection{Finite dimensional Hodge theory and some estimates}
\label{subsect-hodge}

Let
\begin{equation}
(W^\bullet, \partial) : 0 \rightarrow W^0 \rightarrow \cdots \rightarrow W^n \rightarrow 0
\end{equation}
be a chain complex of finite dimensional complex vector spaces.
Let $H^\bullet(W^\bullet, \partial)$ be the cohomology of $(W^\bullet, \partial)$.
Let $h^{W^\bullet} = \bigoplus_{k=0}^n h^{W^k}$ be a Hermitian metric on $W^\bullet$.
Let $\partial^*$ be the adjoint of $\partial$.
Set
\begin{equation}
D = \partial + \partial^* \;,
\end{equation}
which is self-adjoint with respect to $h^{W^\bullet}$.

Now we state the finite dimensional Hodge theorem without proof.

\begin{thm}
The following orthogonal decomposition holds,
\begin{equation}
\label{eq-decomposition-hodge-dim-fini}
W^\bullet = \Ker D \oplus \im \partial \oplus \im \partial^* \;.
\end{equation}
We have
\begin{equation}
\label{eq-inj-hodge-dim-fini}
\Ker D = \Ker D^2 = \Ker \partial \cap \Ker \partial^* \subseteq W^\bullet \;.
\end{equation}
Moreover, the induced map
\begin{align}
\label{eq-hodge-dim-fini}
\begin{split}
\Ker D^2 & \rightarrow H^\bullet( W^\bullet, \partial ) \\
w & \mapsto [w]
\end{split}
\end{align}
is an isomorphism.
\end{thm}

Let
\begin{equation}
\label{eq1-W-decomposition}
W^\bullet = \bigoplus_{\lambda \geqslant 0} W^\bullet_\lambda
\end{equation}
be the spectral decomposition with respect to $D^2$,
i.e., $D^2\big|_{W^\bullet_\lambda} = \lambda\Id$.
We denote
\begin{equation}
\label{eq2-W-decomposition}
{W^\bullet_\lambda}' = W^\bullet_\lambda \cap \Ker \partial \;, \hspace{5mm}
{W^\bullet_\lambda}'' = W^\bullet_\lambda \cap \Ker \partial^* \;.
\end{equation}
The following orthogonal decomposition holds for $\lambda>0$,
\begin{equation}
\label{eq3-W-decomposition}
W^\bullet_\lambda = {W^\bullet_\lambda}' \oplus {W^\bullet_\lambda}'' \;.
\end{equation}
Let $\big\lVert\cdot\big\rVert$ be the norm on $W^\bullet$ induced by $h^{W^\bullet}$.
For $w'\in {W^\bullet_\lambda}'$ and $w''\in {W^\bullet_\lambda}''$,
we have
\begin{equation}
\label{eq4-W-decomposition}
\big\lVert \partial^* w' \big\rVert^2 = \lambda \big\lVert w'\big\rVert^2 \;,\hspace{5mm}
\big\lVert \partial w'' \big\rVert^2 = \lambda \big\lVert w''\big\rVert^2 \;.
\end{equation}

For $\Lambda\subseteq\R$,
let
\begin{equation}
P^\Lambda : W^\bullet \rightarrow \bigoplus_{\lambda \in \Lambda} W^\bullet_\lambda
\end{equation}
be the orthogonal projection.

We state a naive estimate without proof.

\begin{prop}
\label{prop-proj-estimate}
Let $\alpha,\beta\geqslant 0$ and $w\in W^\bullet$.
If $\big\lVert D w \big\rVert^2 \leqslant \alpha\beta$,
then $\big\lVert w - P^{[0,\beta]}w\big\rVert^2 \leqslant \alpha$.
\end{prop}

Now we establish a more sophisticated estimate.

\begin{prop}
\label{prop-proj-estimate-better}
Let $\alpha,\beta,\gamma\geqslant 0$ and $w,v\in W^\bullet$.
If
\begin{equation}
\label{eq1-prop-proj-estimate-better}
\big\lVert \partial w \big\rVert^2 \leqslant \alpha\gamma \;,\hspace{5mm}
\big\lVert \partial^* v \big\rVert^2 \leqslant \alpha\gamma \;,\hspace{5mm}
\big\lVert w-v \big\rVert^2 \leqslant \beta  \;,
\end{equation}
then
\begin{equation}
\label{eq2-prop-proj-estimate-better}
\big\lVert w - P^{[0,\gamma]}w\big\rVert^2 \leqslant 3\alpha+2\beta \;,\hspace{5mm}
\big\lVert v - P^{[0,\gamma]}v\big\rVert^2 \leqslant 3\alpha+2\beta \;.
\end{equation}
\end{prop}
\begin{proof}
Let
\begin{equation}
\label{eq1-pf-prop-proj-estimate-better}
w = \sum_\lambda w_\lambda \;,\hspace{5mm}
v = \sum_\lambda v_\lambda
\end{equation}
be the decompositions with respect to \eqref{eq1-W-decomposition},
i.e., $w_\lambda,v_\lambda\in W^\bullet_\lambda$.
For $\lambda>0$, let
\begin{equation}
\label{eq2-pf-prop-proj-estimate-better}
w_\lambda = w'_\lambda + w''_\lambda \;,\hspace{5mm}
v_\lambda = v'_\lambda + v''_\lambda
\end{equation}
be the decompositions with respect to \eqref{eq3-W-decomposition},
i.e., $w'_\lambda,v'_\lambda\in{W^\bullet_\lambda}'$
and $w''_\lambda,v''_\lambda\in{W^\bullet_\lambda}''$.

By \eqref{eq4-W-decomposition},
\eqref{eq1-pf-prop-proj-estimate-better} and
\eqref{eq2-pf-prop-proj-estimate-better},
we have
\begin{equation}
\label{eq3-pf-prop-proj-estimate-better}
\big\lVert \partial w \big\rVert^2 = \sum_{\lambda>0} \lambda \big\lVert w''_\lambda \big\rVert^2 \;,\hspace{5mm}
\big\lVert \partial^* v \big\rVert^2 = \sum_{\lambda>0} \lambda \big\lVert v'_\lambda \big\rVert^2 \;.
\end{equation}
By \eqref{eq1-prop-proj-estimate-better}
and \eqref{eq3-pf-prop-proj-estimate-better},
we have
\begin{equation}
\label{eq4-pf-prop-proj-estimate-better}
\Big\lVert \sum_{\lambda>\gamma}w''_\lambda \Big\rVert^2 =
\sum_{\lambda>\gamma} \big\lVert w''_\lambda \big\rVert^2
\leqslant \alpha  \;,\hspace{5mm}
\Big\lVert \sum_{\lambda>\gamma}v'_\lambda \Big\rVert^2 =
\sum_{\lambda>\gamma} \big\lVert v'_\lambda \big\rVert^2
\leqslant \alpha \;.
\end{equation}
On the other hand,
by the third inequality in \eqref{eq1-prop-proj-estimate-better},
\eqref{eq1-pf-prop-proj-estimate-better}
and \eqref{eq2-pf-prop-proj-estimate-better},
we have
\begin{equation}
\label{eq5-pf-prop-proj-estimate-better}
\Big\lVert \sum_{\lambda>\gamma}w'_\lambda - \sum_{\lambda>\gamma}v'_\lambda \Big\rVert^2
\leqslant \beta \;.
\end{equation}
By the first identity in \eqref{eq2-pf-prop-proj-estimate-better},
\eqref{eq4-pf-prop-proj-estimate-better}
and \eqref{eq5-pf-prop-proj-estimate-better},
we have
\begin{align}
\begin{split}
\Big\lVert \sum_{\lambda>\gamma}w_\lambda \Big\rVert^2
& = \Big\lVert \sum_{\lambda>\gamma}w''_\lambda \Big\rVert^2 +
\Big\lVert \sum_{\lambda>\gamma}w'_\lambda \Big\rVert^2 \\
& \leqslant \Big\lVert \sum_{\lambda>\gamma}w''_\lambda \Big\rVert^2 +
2 \bigg( \Big\lVert \sum_{\lambda>\gamma}v'_\lambda \Big\rVert^2 +
\Big\lVert \sum_{\lambda>\gamma}w'_\lambda - \sum_{\lambda>\gamma}v'_\lambda \Big\rVert^2 \bigg)
\leqslant 3\alpha+2\beta \;,
\end{split}
\end{align}
which implies the first inequality in \eqref{eq2-prop-proj-estimate-better}.
The second inequality in \eqref{eq2-prop-proj-estimate-better}
can be proved in the same way.
This completes the proof of Proposition \ref{prop-proj-estimate-better}.
\end{proof}

For $w\in W^\bullet$,
we define $\big\lVert w \big\rVert_1^2 = \big\lVert w \big\rVert^2 + \big\lVert Dw \big\rVert^2$.

\begin{cor}
\label{cor-proj-estimate-better}
Propositions \ref{prop-proj-estimate}, \ref{prop-proj-estimate-better}
hold with $\big\lVert\cdot\big\rVert$ replaced by $\big\lVert\cdot\big\rVert_1$.
\end{cor}
\begin{proof}
All the properties concerning $\big\lVert\cdot\big\rVert$
hold for $\big\lVert\cdot\big\rVert_1$.
In particular,
the adjoint of $\partial$
with respect to $\big\lVert\cdot\big\rVert_1$
is still $\partial^*$.
\end{proof}

\subsection{Torsion forms and some estimates}
\label{subsec-tf}

Let $S$ be a compact manifold without boundary.

Let
\begin{equation}
(W^\bullet,\partial) : 0 \rightarrow W^0 \rightarrow \cdots \rightarrow W^n \rightarrow 0
\end{equation}
be a chain complex of complex vector bundles over $S$,
i.e., $\partial : W^\bullet \rightarrow W^{\bullet+1}$ is a linear map between complex vector bundles
satisfying
\begin{equation}
\label{eq-partial-sq}
\partial^2=0 \;.
\end{equation}
We extend the action of $\partial$ to $\Omega^\bullet(S,W^\bullet)$ as follows:
for $\tau\in\Omega^k(S)$ and $w\in\smooth(S,W^\bullet)$,
\begin{equation}
\partial \big(\tau \otimes w\big) =
(-1)^k \tau \otimes \partial w \;.
\end{equation}

Let $\n^{W^\bullet} = \bigoplus_{k=0}^n \n^{W^k}$ be a connection on $W^\bullet$.
We extend the action of $\n^{W^\bullet}$ to $\Omega^\bullet(S,W^\bullet)$
in the same way as in \eqref{eq-derham-op}.
We assume that $\n^{W^\bullet}$ is a flat connection.
Equivalently,
we assume that
\begin{equation}
\label{eq-nW-sq}
\big(\n^{W^\bullet}\big)^2 = 0 \;.
\end{equation}

Now we assume that $(W^\bullet,\n^{W^\bullet},\partial)$
is a chain complex of flat complex vector bundles.
Equivalently,
we assume that
\begin{equation}
\label{eq-partial-nW}
\partial \n^{W^\bullet} + \n^{W^\bullet} \partial = 0 \;.
\end{equation}

By \eqref{eq-partial-nW},
$\partial$ is covariantly constant with respect to the connection $\n^{W^\bullet}$.
Thus there is a $\Z$-graded complex vector bundle $H^\bullet$ over $S$
whose fiber over $s\in S$ is the cohomology of $\big(W^\bullet_s,\partial\big|_{W^\bullet_s}\big)$
(see \cite[p. 307]{bl}).
Let $\n^{H^\bullet}$ be the connection on $H^\bullet$ induced by $\n^{W^\bullet}$ in the sense of \cite[Def. 2.4]{bl}.
By \cite[Prop. 2.5]{bl},
$(H^\bullet,\n^{H^\bullet})$ is a $\Z$-graded flat complex vector bundle.

Recall that $f(z)=ze^{z^2}$.
Let $f\big(\n^{W^\bullet}\big), f\big(\n^{H^\bullet}\big) \in H^\mathrm{odd}(S)$
be as in \eqref{intro-fsum}.
By \cite[Thm. 2.19]{bl},
we have
\begin{equation}
f\big(\n^{W^\bullet}\big) = f\big(\n^{H^\bullet}\big) \;.
\end{equation}

Set
\begin{equation}
\label{eq-Aprimprim}
A'' = \partial + \n^{W^\bullet} \;.
\end{equation}
By \eqref{eq-partial-sq},
\eqref{eq-nW-sq},
\eqref{eq-partial-nW}
and \eqref{eq-Aprimprim},
we have
\begin{equation}
\big( A'' \big)^2 = 0 \;,
\end{equation}
i.e., $A''$ is a flat superconnection in the sense of \cite[\textsection 1]{bl}.

Let $h^{W^\bullet} = \bigoplus_{k=0}^n h^{W^k}$ be a Hermitian metric on $W^\bullet$.
Let $\omega^{W^\bullet} \in \Omega^1\big(S,\End(W^\bullet)\big)$
be as in \eqref{intro-def-omegaF} with $(\n^F,h^F)$ replaced by $(\n^{W^\bullet},h^{W^\bullet})$,
i.e.,
\begin{equation}
\label{eq-def-omegaW}
\omega^{W^\bullet} =
\big(h^{W^\bullet}\big)^{-1} \n^{W^\bullet} h^{W^\bullet} \;.
\end{equation}
Let $\partial^*$ be the adjoint of $\partial$.
Let $A'$ be the adjoint superconnection of $A''$ in the sense of \cite[\textsection 1]{bl}.
By \cite[\textsection 2(b)]{bl},
we have
\begin{equation}
A' = \partial^* + \n^{W^\bullet} + \omega^{W^\bullet} \;.
\end{equation}
Set
\begin{equation}
X
= \frac{1}{2} (A' - A'')
= \frac{1}{2}(\partial^*-\partial) + \frac{1}{2} \omega^{W^\bullet}
\in \Omega^\bullet\big(S,\End(W^\bullet)\big) \;.
\end{equation}

Let $N^{W^\bullet}$ be the number operator on $W^\bullet$,
i.e., $N^{W^\bullet}\big|_{W^k} = k\Id$.
For $t>0$,
set $h^{W^\bullet}_t = t^{N^{W^\bullet}}h^{W^\bullet}$.
Let $X_t$ be the operator $X$ associated with $h^{W^\bullet}_t$.
We have
\begin{equation}
\label{eq-def-Xt}
X_t = \frac{1}{2}(t \partial^*-\partial) + \frac{1}{2} \omega^{W^\bullet} \;.
\end{equation}
We define $\varphi: \Omega^\mathrm{even}(S) \rightarrow \Omega^\mathrm{even}(S)$ as follows,
\begin{equation}
\label{eq-def-varphi}
\varphi \omega = (2\pi i)^{-k }\omega \;,\hspace{5mm} \text{for }\omega\in\Omega^{2k}(S) \;.
\end{equation}
We remark that $f'(z) = (1+2z^2) e^{z^2}$.
Set
\begin{equation}
\label{eq-def-fhat}
f^\wedge(A'',h^{W^\bullet}_t) =
\varphi \tr \left[(-1)^{N^{W^\bullet}}\frac{N^{W^\bullet}}{2} f'(X_t) \right]
\in \Omega^\mathrm{even}(S) \;.
\end{equation}
Set
\begin{equation}
\mathfrak{X}_t = t^{N^{W^\bullet}/2} X_t t^{-N^{W^\bullet}/2} =
\frac{\sqrt{t}}{2}(\partial^*-\partial) + \frac{1}{2} \omega^{W^\bullet} \;.
\end{equation}
We have an alternative definition,
\begin{equation}
\label{eq2-def-fhat}
f^\wedge(A'',h^{W^\bullet}_t) =
\varphi \tr \left[(-1)^{N^{W^\bullet}}\frac{N^{W^\bullet}}{2} f'(\mathfrak{X}_t) \right] \;.
\end{equation}

We denote
\begin{equation}
\chi'(W^\bullet) = \sum_k (-1)^kk \,\mathrm{rk}\big(W^k\big) \;,\hspace{5mm}
\chi'(H^\bullet) = \sum_k (-1)^kk \,\mathrm{rk}\big(H^k\big) \;.
\end{equation}

The following definition is due to Bismut and Lott \cite[Def. 2.20]{bl}.

\begin{defn}
\label{def-tf}
The torsion form associated with
$(\n^{W^\bullet},\partial,h^{W^\bullet})$ is defined by
\begin{align}
\label{eq-def-torsion-form}
\begin{split}
\mathscr{T}\big(\n^{W^\bullet},\partial,h^{W^\bullet}\big)
& =  - \int_0^{+\infty} \bigg[ f^\wedge(A'',h^{W^\bullet}_t) - \frac{1}{2}\chi'(H^\bullet) \\
& \hspace{25mm} - \frac{1}{2}\big(\chi'(W^\bullet) - \chi'(H^\bullet)\big) f'\Big(\frac{i\sqrt{t}}{2}\Big) \bigg] \frac{dt}{t} \;.
\end{split}
\end{align}
By \cite[Thm. 2.13, Prop. 2.18]{bl},
the integrand in \eqref{eq-def-torsion-form} is integrable.
\end{defn}

Let $h^{H^\bullet}$ be the Hermitian metric on $H^\bullet$ induced by $h^{W^\bullet}$
via the identification $H^\bullet \simeq \Ker\big((\partial+\partial^*)^2\big) \hookrightarrow W^\bullet$
defined by \eqref{eq-hodge-dim-fini}.
Let $f\big(\n^{W^\bullet},h^{W^\bullet}\big), f\big(\n^{H^\bullet},h^{H^\bullet}\big) \in \Omega^\mathrm{odd}(S)$
be as in \eqref{intro-fsum}.
By \cite[Thm. 2.22]{bl},
we have
\begin{equation}
d \mathscr{T}\big(\n^{W^\bullet},\partial,h^{W^\bullet}\big)
= f\big(\n^{W^\bullet},h^{W^\bullet}\big) - f\big(\n^{H^\bullet},h^{H^\bullet}\big) \;.
\end{equation}

Let
$(\widetilde{W}^\bullet = \bigoplus_{k=0}^n \widetilde{W}^k,\n^{\widetilde{W}^\bullet},\widetilde{\partial})$
be another chain complex of flat complex vector bundles over $S$.
Let $\widetilde{H}^\bullet$ be its cohomology.
We assume that for $k=0,\cdots,n$,
\begin{equation}
\label{eq-assume-rk}
\mathrm{rk}\big(W^k\big) = \mathrm{rk}\big(\widetilde{W}^k\big) \;,\hspace{5mm}
\mathrm{rk}\big(H^k\big) = \mathrm{rk}\big(\widetilde{H}^k\big) \;.
\end{equation}
Let $h^{\widetilde{W}^\bullet} = \bigoplus_{k=0}^n h^{\widetilde{W}^k}$
be a Hermitian metric on $\widetilde{W}^\bullet$.

Let $g^{TS}$ be a Riemannian metric on $TS$.
Let $\big|\cdot\big|$ be the norm on $TS$ induced by $g^{TS}$.
For $\omega\in\Omega^\bullet(S)$,
we denote
\begin{equation}
\big| \omega \big| =
\sup_{k\in\N,\; x\in S,\; v_1,\cdots,v_k\in T_xS,\; |v_1|,\cdots,|v_k|\leqslant 1}
\big| \omega(v_1,\cdots,v_k) \big| \;.
\end{equation}
For an operator $A$ on $W^\bullet$,
we denote by $\big\lVert A \big\rVert$
its operator norm with respect to $h^{W^\bullet}$.
For $A\in\Omega^\bullet(S,\mathrm{End}(W^\bullet))$,
we denote
\begin{equation}
\big\lVert A \big\rVert =
\sup_{k\in\N,\; x\in S,\; v_1,\cdots,v_k\in T_xS,\; |v_1|,\cdots,|v_k|\leqslant 1}
\big\lVert A(v_1,\cdots,v_k) \big\rVert \;.
\end{equation}

Let $0<\lambda_\mathrm{min}\leqslant\lambda_\mathrm{max}$ such that
\begin{equation}
\label{eq-inclu-sp}
\Sp\big((\partial^*+\partial)^2\big) \subseteq
\{0\} \cup [\lambda^2_\mathrm{min},\lambda^2_\mathrm{max}] \;.
\end{equation}
Let $l>0$ such that
\begin{equation}
\label{eq-estim-omega}
\big\lVert \omega^{W^\bullet} \big\rVert \leqslant l \;.
\end{equation}

\begin{prop}
\label{prop-comparison-tf}
There exists a function
$C: \N \times \N \times \R_+ \times \R_+ \rightarrow \R_+$
such that for any $(W^\bullet,\n^{W^\bullet},\partial,h^{W^\bullet})$,
$(\widetilde{W}^\bullet,\n^{\widetilde{W}^\bullet},\widetilde{\partial},h^{\widetilde{W}^\bullet})$,
$\lambda_\mathrm{min}$, $\lambda_\mathrm{max}$ and $l$ as above,
if there exist an isomorphism of graded complex vector bundles
$\alpha: W^\bullet\rightarrow \widetilde{W}^\bullet$
and $0< \delta < 13^{-1}\lambda_\mathrm{min}\lambda_\mathrm{max}^{-1}$
satisfying
\begin{equation}
\label{eq1-prop-comparison-tf}
\big\lVert \alpha^*\widetilde{\partial} - \partial \big\rVert
\leqslant \lambda_\mathrm{min}\delta \;,\hspace{5mm}
- \delta h^{W^\bullet} \leqslant
\alpha^*h^{\widetilde{W}^\bullet} - h^{W^\bullet}
\leqslant \delta h^{W^\bullet} \;,\hspace{5mm}
\big\lVert \alpha^*\omega^{\widetilde{W}^\bullet} - \omega^{W^\bullet} \big\rVert
\leqslant \delta \;,
\end{equation}
then
\begin{equation}
\label{eq-prop-comparison-tf}
\Big| \mathscr{T}\big(\n^{W^\bullet},\partial,h^{W^\bullet}\big)
- \mathscr{T}\big(\n^{\widetilde{W}^\bullet},\widetilde{\partial},h^{\widetilde{W}^\bullet}\big) \Big|
\leqslant C\big(\dim S, \mathrm{rk}(W^\bullet),l,\lambda_\mathrm{max}/\lambda_\mathrm{min}\big) \delta^{1/2} \;.
\end{equation}
\end{prop}
\begin{proof}
Replacing $\partial$ by $\lambda_\mathrm{min}^{-1}\partial$ and
replacing $\widetilde{\partial}$ by $\lambda_\mathrm{min}^{-1}\widetilde{\partial}$,
we may assume that $\lambda_\mathrm{min} = 1$.
Then \eqref{eq-inclu-sp} and the first inequality in \eqref{eq1-prop-comparison-tf} become
\begin{equation}
\label{eqx-pf-prop-comparison-tf}
\Sp\big((\partial^*+\partial)^2\big) \subseteq
\{0\} \cup [1,\lambda^2_\mathrm{max}] \;, \hspace{5mm}
\big\lVert \alpha^*\widetilde{\partial} - \partial \big\rVert
\leqslant \delta \;.
\end{equation}

By \eqref{eqx-pf-prop-comparison-tf},
we have
\begin{equation}
\label{eqb-pf-prop-comparison-tf}
\big\lVert \alpha^*\widetilde{\partial} \big\rVert \leqslant
\delta + \big\lVert \partial \big\rVert \leqslant
\delta + \lambda_\mathrm{max} \;.
\end{equation}
Since $\big\lVert A \big\rVert = \big\lVert A^* \big\rVert$ for any operator $A$ on $W^\bullet$,
we have
\begin{equation}
\big\lVert \big(\alpha^*\widetilde{\partial}\big)^* - \partial^* \big\rVert
= \big\lVert \alpha^*\widetilde{\partial} - \partial \big\rVert \;.
\end{equation}
Let $\widetilde{\partial}^*$ be the adjoint of $\widetilde{\partial}$ with respect to $h^{\widetilde{W}^\bullet}$.
Note that $\alpha^*\widetilde{\partial}^*$ is the adjoint of $\alpha^*\widetilde{\partial}$ with respect to $\alpha^*h^{\widetilde{W}^\bullet}$
and $\big(\alpha^*\widetilde{\partial}\big)^*$ is the adjoint of $\alpha^*\widetilde{\partial}$ with respect to $h^{W^\bullet}$,
by the second inequality in \eqref{eq1-prop-comparison-tf},
we have
\begin{equation}
\label{eqc-pf-prop-comparison-tf}
\big\lVert \alpha^*\widetilde{\partial}^* - \big(\alpha^*\widetilde{\partial}\big)^* \big\rVert
\leqslant 3\delta \big\lVert \alpha^*\widetilde{\partial} \big\rVert \;.
\end{equation}
By \eqref{eqx-pf-prop-comparison-tf}-\eqref{eqc-pf-prop-comparison-tf}
and the assumption $0 < \delta < 13^{-1}\lambda^{-1}_\mathrm{max}$,
we have
\begin{align}
\label{eqd-pf-prop-comparison-tf}
\begin{split}
\big\lVert \alpha^*\widetilde{\partial}^* - \partial^* \big\rVert
& \leqslant \big\lVert \alpha^*\widetilde{\partial}^* - \big(\alpha^*\widetilde{\partial}\big)^* \big\rVert
+ \big\lVert \big(\alpha^*\widetilde{\partial}\big)^* - \partial^* \big\rVert \\
& \leqslant 3\delta(\delta + \lambda_\mathrm{max}) + \delta
\leqslant 5 \delta \lambda_\mathrm{max} \;.
\end{split}
\end{align}
By \eqref{eqx-pf-prop-comparison-tf}, \eqref{eqd-pf-prop-comparison-tf}
and the assumption $0 < \delta < 13^{-1}\lambda^{-1}_\mathrm{max}$,
we have
\begin{align}
\label{eqe-pf-prop-comparison-tf}
\begin{split}
\big\lVert \partial^*+\partial-\alpha^*(\widetilde{\partial}^*+\widetilde{\partial}) \big\rVert
& \leqslant 6 \delta \lambda_\mathrm{max} \leqslant \frac{6}{13} \;, \\
\big\lVert \partial^*-\partial-\alpha^*(\widetilde{\partial}^*-\widetilde{\partial}) \big\rVert
& \leqslant 6 \delta \lambda_\mathrm{max} \leqslant \frac{6}{13} \;.
\end{split}
\end{align}
By \eqref{eqx-pf-prop-comparison-tf} and \eqref{eqe-pf-prop-comparison-tf},
we have
\begin{equation}
\label{eqf-pf-prop-comparison-tf}
\Sp\big((\widetilde{\partial}^*+\widetilde{\partial})^2\big) \subseteq
\Big[0,\frac{6^2}{13^2}\Big] \cup
\Big[\frac{7^2}{13^2},3\lambda^2_\mathrm{max}\Big] \;.
\end{equation}
Moreover,
the dimension of the eigenspace of $(\widetilde{\partial}^*+\widetilde{\partial})^2$
associated with eigenvalues in $\big[0,\frac{6^2}{13^2}\big]$
equals the dimension of $\Ker\big((\partial^*+\partial)^2\big)$.
On the other hand,
by \eqref{eq-hodge-dim-fini} and the second identity in \eqref{eq-assume-rk},
we have
\begin{equation}
\label{eqg-pf-prop-comparison-tf}
\dim \Ker\big((\partial^*+\partial)^2\big) =
\mathrm{rk} H^\bullet = \mathrm{rk} \widetilde{H}^\bullet =
\dim \Ker\big((\widetilde{\partial}^*+\widetilde{\partial})^2\big) \;.
\end{equation}
As a consequence,
the only possible eigenvalue of $(\widetilde{\partial}^*+\widetilde{\partial})^2$
in $\big[0,\frac{6^2}{13^2}\big]$ is zero, i.e.,
\begin{equation}
\label{eqh-pf-prop-comparison-tf}
\Sp\big((\widetilde{\partial}^*+\widetilde{\partial})^2\big) \subseteq
\big\{0\big\} \cup
\Big[\frac{7^2}{13^2},3\lambda^2_\mathrm{max}\Big] \;.
\end{equation}

In the sequel,
we will use $C_1,C_2,\cdots$ to denote constants depending on
$\dim S$, $\mathrm{rk}W^\bullet$, $l$ and $\lambda_\mathrm{max}/\lambda_\mathrm{min}$.

Let $\omega^{\widetilde{W}^\bullet}$ be as in \eqref{eq-def-omegaW}
with $(W^\bullet,\n^{W^\bullet},h^{W^\bullet})$ replaced by $(\widetilde{W}^\bullet,\n^{\widetilde{W}^\bullet},h^{\widetilde{W}^\bullet})$.
For $t>0$, we denote
\begin{equation}
\mathfrak{X}_t =
\frac{\sqrt{t}}{2}(\partial^*-\partial)
+ \frac{1}{2} \omega^{W^\bullet} \;,\hspace{5mm}
\widetilde{\mathfrak{X}}_t = \alpha^*\Big(
\frac{\sqrt{t}}{2}(\widetilde{\partial^*}-\widetilde{\partial})
+ \frac{1}{2} \omega^{\widetilde{W}^\bullet} \Big) \;.
\end{equation}
Set $U = \big\{\lambda\in\C\;:\; -1<\mathrm{Re}(z)<1\big\}$.
By \eqref{eq-estim-omega},
the third inequality in \eqref{eq1-prop-comparison-tf},
\eqref{eqx-pf-prop-comparison-tf},
\eqref{eqe-pf-prop-comparison-tf}
and \eqref{eqh-pf-prop-comparison-tf},
for $\lambda\in\partial U$ and $t>0$,
we have
\begin{equation}
\label{eq31-pf-prop-comparison-tf}
\big\lVert (\lambda - \mathfrak{X}_t )^{-1} \big\rVert \leqslant C_1 \;,\hspace{5mm}
\big\lVert (\lambda - \widetilde{\mathfrak{X}}_t )^{-1} \big\rVert \leqslant C_1 \;,\hspace{5mm}
\big\lVert \mathfrak{X}_t - \widetilde{\mathfrak{X}}_t \big\rVert \leqslant C_1(1+\sqrt{t})\delta \;.
\end{equation}
Let $f^\wedge\big(\widetilde{A}'',h^{\widetilde{W}^\bullet}_t\big)$ be as in \eqref{eq-def-fhat}
with $(\n^{W^\bullet},\partial,h^{W^\bullet})$
replaced by $(\n^{\widetilde{W}^\bullet},\widetilde{\partial},h^{\widetilde{W}^\bullet})$.
For $t>0$,
by \eqref{eq2-def-fhat},
we have
\begin{align}
\label{eq32-pf-prop-comparison-tf}
\begin{split}
& f^\wedge\big(A'',h^{W^\bullet}_t\big) -
f^\wedge\big(\widetilde{A}'',h^{\widetilde{W}^\bullet}_t\big) \\
& = \frac{1}{2\pi i} \int_{\partial U}
\varphi \tr\bigg[ (-1)^{N^{W^\bullet}}\frac{N^{W^\bullet}}{2}
\Big( (\lambda - \mathfrak{X}_t )^{-1} - (\lambda - \widetilde{\mathfrak{X}}_t )^{-1} \Big) \bigg]
f'(\lambda) d\lambda \\
& = \frac{1}{2\pi i} \int_{\partial U}
\varphi \tr\bigg[ (-1)^{N^{W^\bullet}}\frac{N^{W^\bullet}}{2}
(\lambda - \mathfrak{X}_t )^{-1}
(\mathfrak{X}_t-\widetilde{\mathfrak{X}}_t)
(\lambda - \widetilde{\mathfrak{X}}_t )^{-1} \bigg]
f'(\lambda) d\lambda \;.
\end{split}
\end{align}
By
\eqref{eq31-pf-prop-comparison-tf} and
\eqref{eq32-pf-prop-comparison-tf},
we have
\begin{equation}
\label{eq3-pf-prop-comparison-tf}
\Big| f^\wedge\big(A'',h^{W^\bullet}_t\big) -
f^\wedge\big(\widetilde{A}'',h^{\widetilde{W}^\bullet}_t\big) \Big|
\leqslant  C_2 (1+\sqrt{t}) \delta \;.
\end{equation}

Proceeding the same way as in the proofs of \cite[Prop. 2.18, Thm. 2.13]{bl}
and applying \eqref{eq-estim-omega},
the third inequality in \eqref{eq1-prop-comparison-tf},
\eqref{eqx-pf-prop-comparison-tf}
and \eqref{eqh-pf-prop-comparison-tf},
we obtain the following estimates:
for $0<t<1$,
\begin{equation}
\label{eq1-pf-prop-comparison-tf}
\Big|  f^\wedge\big(A'',h^{W^\bullet}_t\big)
- \frac{1}{2} \chi'(W^\bullet) \Big|
\leqslant C_3 t \;,\hspace{5mm}
\Big| f^\wedge\big(\widetilde{A}'',h^{\widetilde{W}^\bullet}_t\big)
- \frac{1}{2} \chi'(W^\bullet) \Big|
\leqslant C_3 t \;;
\end{equation}
for $t>1$,
\begin{equation}
\label{eq2-pf-prop-comparison-tf}
\Big|  f^\wedge\big(A'',h^{W^\bullet}_t\big)
- \frac{1}{2} \chi'(H^\bullet) \Big|
\leqslant \frac{C_4}{\sqrt{t}} \;,\hspace{5mm}
\Big| f^\wedge\big(\widetilde{A}'',h^{\widetilde{W}^\bullet}_t\big)
- \frac{1}{2} \chi'(H^\bullet) \Big|
\leqslant \frac{C_4}{\sqrt{t}} \;.
\end{equation}

By \eqref{eq-def-torsion-form}, \eqref{eq-assume-rk}
and \eqref{eq3-pf-prop-comparison-tf}-\eqref{eq2-pf-prop-comparison-tf},
we have
\begin{align}
\begin{split}
& \Big| \mathscr{T}\big(\n^{W^\bullet},\partial,h^{W^\bullet}\big)
- \mathscr{T}\big(\n^{\widetilde{W}^\bullet},\widetilde{\partial},h^{\widetilde{W}^\bullet}\big) \Big| \\
& \leqslant \; 2 \int_0^{\delta} C_3 t \frac{dt}{t}
+ \int_{\delta}^{\delta^{-1}} C_2 (1+\sqrt{t}) \delta \frac{dt}{t}
+ 2 \int_{\delta^{-1}}^{+\infty} \frac{C_4}{\sqrt{t}} \frac{dt}{t}
\leqslant\; C_5 \delta^{1/2} \;.
\end{split}
\end{align}
This completes the proof of Proposition \ref{prop-comparison-tf}.
\end{proof}

\begin{rem}
\label{lem-notlfat}
Let $\mu\in\Omega^1\big(S,\mathrm{End}(\widetilde{W}^\bullet)\big)$.
We assume that $\mu$ preserves the degree, i.e.,
$\mu\Big(\smooth\big(S,W^k\big)\Big)\subseteq\Omega^1\big(S,W^k\big)$ for $k=0,\cdots,n$.
Set
\begin{equation}
f^\wedge(\widetilde{\partial},h^{\widetilde{W}^\bullet}_t,\mu) =
\varphi \tr \left[ (-1)^{N^{\widetilde{W}^\bullet}}\frac{N^{\widetilde{W}^\bullet}}{2}
f'\Big( \frac{\sqrt{t}}{2}(\widetilde{\partial^*}-\widetilde{\partial}) + \frac{1}{2}\mu \Big) \right] \;.
\end{equation}
Let $\mathscr{T}\big(\widetilde{\partial},h^{\widetilde{W}^\bullet},\mu\big)$
be as in \eqref{eq-def-torsion-form} with $f^\wedge(A'',h^{W^\bullet}_t)$
replaced by $f^\wedge(\widetilde{\partial},h^{\widetilde{W}^\bullet}_t,\mu)$.
Then Proposition \ref{prop-comparison-tf} holds with
$\omega^{\widetilde{W}^\bullet}$
replaced by $\mu$
and $\mathscr{T}\big(\n^{\widetilde{W}^\bullet},\widetilde{\partial},h^{\widetilde{W}^\bullet}\big)$
replaced by $\mathscr{T}\big(\widetilde{\partial},h^{\widetilde{W}^\bullet},\mu\big)$.
\end{rem}

Let $(F,\n^F)$ be a flat complex vector bundle over $S$.
Let $h^F_1$ and $h^F_2$ be Hermitian metrics on $F$.
Let $\omega^F_1$ (resp. $\omega^F_2$) be as in \eqref{intro-def-omegaF}
with $h^F$ replaced by $h^F_1$ (resp. $h^F_2$).
We consider the chain complex $F \xrightarrow{\mathrm{Id}} F$,
where the first $F$ is equipped with Hermitian metric $h^F_1$
and the second $F$ is equipped with Hermitian metric $h^F_2$.
Let $\mathscr{T}\big(\n^F,h^F_1,h^F_2\big)$ be its torsion form.

Let $l>0$ such that
\begin{equation}
\big\lVert \omega^F_1 \big\rVert \leqslant l \;.
\end{equation}

\begin{cor}
\label{cor-comparison-tf}
There exists a function
$C: \N \times \N \times \R_+ \rightarrow \R_+$
such that for any $(F,\n^F,h^F_1,h^F_2)$ and $l$ as above,
if there exists $\delta\in(0,13^{-1})$
satisfying
\begin{equation}
\label{eq1-cor-comparison-tf}
- \delta h^F_1 \leqslant h^F_2 - h^F_1 \leqslant \delta h^F_1 \;,\hspace{5mm}
\big\lVert \omega^F_2 - \omega^F_1 \big\rVert \leqslant \delta \;,
\end{equation}
then
\begin{equation}
\label{eq-cor-comparison-tf}
\Big| \mathscr{T}\big(\n^F,h^F_1,h^F_2\big) \Big|
\leqslant C\big(\dim S, \mathrm{rk}(F),l\big) \delta^{1/2} \;.
\end{equation}
\end{cor}
\begin{proof}
Note that $\mathscr{T}\big(\n^F,h^F_1,h^F_1\big)=0$,
the inequality \eqref{eq-cor-comparison-tf} is a direct consequence of Proposition \ref{prop-comparison-tf}.
\end{proof}

\subsection{Analytic torsion forms}
\label{subsec-atf}

Let $\pi: M \rightarrow S$ be a smooth fibration with compact fiber $Z$.
Let $N = \partial M$.
We assume that $\pi\big|_N: N \rightarrow S$ is a smooth fibration with fiber $Y$.
Then we have $Y = \partial Z$.

We identify a tubular neighborhood of $N \subseteq M$ with $[-1,0]\times N$ such that
$N$ is identified with $\{0\}\times N$
and the following diagram commutes,
\begin{equation}
\xymatrix{
[-1,0] \times N \ar@{^{(}->}[r] \ar[d]_{\mathrm{pr}_2} & M \ar[d]^{\pi} \\
N \ar[r]^{\pi|_N} & S \;,}
\end{equation}
where $\mathrm{pr}_2: [-1,0]\times N \rightarrow N$ is the projection to the second factor.

Let $T^HM \subseteq TM$ be a horizontal sub bundle of $TM$, i.e.,
\begin{equation}
\label{eq-decomposition-TM}
TM = T^HM \oplus TZ \;.
\end{equation}
Then we have
\begin{equation}
\label{eq-id-wedge}
\Lambda^\bullet(T^*M)
= \Lambda^\bullet(T^{H,*}M) \otimes \Lambda^\bullet(T^*Z)
\simeq \pi^*\big(\Lambda^\bullet(T^*S)\big) \otimes \Lambda^\bullet(T^*Z) \;.
\end{equation}
We assume that $T^HM$ is product on $[-1,0]\times N$, i.e.,
\begin{equation}
T^HM\big|_N \subseteq TN \;,\hspace{5mm}
T^HM\big|_{[-1,0]\times N} = \mathrm{pr}_2^*\big(T^HM\big|_N\big) \;.
\end{equation}
We remark that $T^HN := T^HM\big|_N \subseteq TN$ is a horizontal sub bundle of $TN$, i.e.,
\begin{equation}
TN = T^HN \oplus TY \;.
\end{equation}

Let $g^{TZ}$ be a Riemannian metric on $TZ$.
Let $g^{TY}$ be the Riemannian metric on $TY$ induced by $g^{TZ}$ via the embedding $N = \partial M \hookrightarrow M$.
Let $(u,y)\in[-1,0]\times N$ be coordinates.
We assume that $g^{TZ}$ is product on $[-1,0]\times N$, i.e.,
\begin{equation}
g^{TZ}_{(u,y)} = du^2 + g^{TY}_y \;.
\end{equation}

Let $(F,\n^F)$ be a flat complex vector bundle over $M$.
We trivialize $F\big|_{[-1,0]\times N}$ along the curve $[-1,0] \ni u \mapsto (u,y)$
using the parallel transport with respect to $\n^F$.
We have
\begin{equation}
\label{eq-tr-F-tub}
(F,\n^F)\big|_{[-1,0]\times N} = \mathrm{pr}_2^*(F\big|_N,\n^F\big|_N) \;.
\end{equation}

Let $h^F$ be a Hermitian metric on $F$.
We assume that $h^F$ is product on $[-1,0]\times N$, i.e.,
under the identification \eqref{eq-tr-F-tub},
we have
\begin{equation}
h^F\big|_{[-1,0]\times N} = \mathrm{pr}_2^* \big(h^F\big|_N\big) \;.
\end{equation}

Set $\mathscr{F} = \Omega^\bullet(Z,F)$,
which is a $\Z$-graded complex vector bundle of infinite dimension over $S$.
By \eqref{eq-id-wedge},
we have the formal identity $\Omega^\bullet(M,F) = \Omega^\bullet(S,\mathscr{F})$.

For $U\in TS$,
let $U^H \in T^HM$ be its horizontal lift, i.e., $\pi_*U^H = U$.
For $U\in\smooth(S,TS)$,
let $L_{U^H}$ be the Lie differentiation operator acting on $\Omega^\bullet(M,F)$.
For $U\in\smooth(S,TS)$ and $s\in\Omega^\bullet(S,\mathscr{F}) = \Omega^\bullet(M,F)$,
we define
\begin{equation}
\label{eq-def-n-F}
\n^\mathscr{F}_U s = L_{U^H} s \;.
\end{equation}
Then $\n^\mathscr{F}$ is a connection on $\mathscr{F}$ preserving the grading.

Let $P^{TZ}: TM \rightarrow TZ$ be the projection with respect to \eqref{eq-decomposition-TM}.
For $U,V\in\smooth(S,TS)$,
set
\begin{equation}
\label{eq-def-mathcalT}
\mathcal{T}(U,V) = - P^{TZ}[U^H,V^H] \in \smooth(M,TZ) \;.
\end{equation}
Then $\mathcal{T}\in\smooth\big(M,\pi^*\big(\Lambda^2(T^*S)\big)\otimes TZ\big)$.
Let $i_\mathcal{T}\in\smooth\big(M,\pi^*\big(\Lambda^2(T^*S)\big)\otimes\mathrm{End}\big(\Lambda^\bullet(T^*Z)\big)\big)$
be the interior multiplication by $\mathcal{T}$ in the vertical direction.

The flat connection $\n^F$ (resp. $\n^F\big|_Z$)
naturally extends to an exterior differentiation operator
on $\Omega^\bullet(M,F)$ (resp. $\Omega^\bullet(Z,F) = \mathscr{F}$),
which we denote by $d^M$ (resp. $d^Z$).
In the sense of \cite[\textsection 2(a)]{bl},
the operator $d^M$ is a superconnection of total degree $1$ on $\mathscr{F}$.
By \cite[Prop. 3.4]{bl},
we have
\begin{equation}
d^M = d^Z + \n^\mathscr{F} + i_\mathcal{T} \;.
\end{equation}

Let $\mathcal{T}^*\in\smooth\big(M,\pi^*\big(\Lambda^2(T^*S)\big)\otimes T^*Z\big)$
be the dual of $\mathcal{T}$ with respect to $g^{TZ}$.

Let $h^\mathscr{F}$ be the $L^2$-metric on $\mathscr{F}$
with respect to $g^{TZ}$ and $h^F$.
Let $d^{M,*}, d^{Z,*},\n^{\mathscr{F},*}$
be the formal adjoints of  $d^M,d^Z,\n^\mathscr{F}$ with respect to $h^\mathscr{F}$
in the sense of \cite[Def. 1.6]{bl}.
By \cite[Prop. 3.7]{bl},
we have
\begin{equation}
d^{M,*} = d^{Z,*} + \n^\mathscr{F,*} - \mathcal{T}^*\!\wedge \;.
\end{equation}

Let $N^{TZ}$ be the number operator on $\Lambda^\bullet(T^*Z)$,
i.e., $N^{TZ}\big|_{\Lambda^p(T^*Z)} = p \Id$.
Then $N^{TZ}$ acts on $\mathscr{F}$ in the obvious way.
For $t>0$,
let $d^{M,*}_t$ be the formal adjoints of  $d^M$ with respect to
$h^\mathscr{F}_t := t^{N^{TZ}}h^\mathscr{F}$.
We have
\begin{equation}
d^{M,*}_t = t d^{Z,*} + \n^\mathscr{F,*} - \frac{1}{t} \mathcal{T}^*\!\wedge \;.
\end{equation}
Set
\begin{align}
\label{eq-def-Dt}
\begin{split}
\mathscr{D}_t & = t^{N^{TZ}/2} \big( d^{M,*}_t - d^M \big) t^{-N^{TZ}/2} \\
& = \frac{\sqrt{t}}{2} \big(d^{Z,*}-d^Z\big)
+ \frac{1}{2}\big(\n^{\mathscr{F},*}-\n^\mathscr{F}\big)
- \frac{1}{2\sqrt{t}}\big(\mathcal{T}^*\!\wedge+i_\mathcal{T}\big) \;.
\end{split}
\end{align}
We denote
\begin{equation}
\label{eq-def-omegaF}
\omega^\mathscr{F} = \n^{\mathscr{F},*} - \n^\mathscr{F} \in \Omega^1(S,\mathrm{End}(\mathscr{F})) \;.
\end{equation}
For $X\in TZ$,
we denote by $X^*\in T^*Z$ its dual with respect to $g^{TZ}$.
For $X\in TZ$,
we denote
\begin{equation}
\label{eq-def-hatc}
\hat{c}(X) = X^*\!\wedge + i_X \in \mathrm{End}(\Lambda^\bullet(T^*Z)) \;.
\end{equation}
By \eqref{eq-def-Dt}-\eqref{eq-def-hatc},
we have
\begin{equation}
\label{eq-Dt}
\mathscr{D}_t =
\frac{\sqrt{t}}{2} \big(d^{Z,*}-d^Z\big) + \frac{1}{2}\omega^\mathscr{F} - \frac{1}{2\sqrt{t}}\hat{c}(\mathcal{T}) \;.
\end{equation}
In particular,
\begin{equation}
\label{eq-Dt2}
\mathscr{D}_t^2 =
- \frac{t}{4} \big(d^{Z,*}d^Z+d^Zd^{Z,*}\big) + \text{nilpotent operator} \;,
\end{equation}
where $d^{Z,*}d^Z+d^Zd^{Z,*}$ is the fiberwise Hodge Laplacian.

By \eqref{eq-Dt2},
the operator $\mathscr{D}_t^2$ is fiberwise essentially self-adjoint
with respect to the absolute boundary condition (see \eqref{eq-def-abs-bd}).
Its self-adjoint extension with respect to the absolute boundary condition will still be denoted by $\mathscr{D}_t^2$.
Let $\mathrm{End}_\mathrm{tr}(\mathscr{F}) \subseteq \mathrm{End}(\mathscr{F})$
be the sub vector bundle of trace class operators.
Recall that $f'(z) = (1+2z^2)e^{z^2}$.
By \eqref{eq-Dt2},
we have $f'\big(\mathscr{D}_t^2\big)\in\Omega^\bullet\big(\mathrm{End}_\mathrm{tr}(\mathscr{F})\big)$.

Let $\tr: \mathrm{End}_\mathrm{tr}(\mathscr{F}) \rightarrow \C$ be the trace map,
which extends to $\tr: \mathrm{End}_\mathrm{tr}(\mathscr{F})\otimes\Lambda^\bullet(T^*S) \rightarrow \Lambda^\bullet(T^*S)$.
Let $\varphi$ be as in \eqref{eq-def-varphi}.

Let $H^\bullet(Z,F)$ be the fiberwise singular cohomology of $Z$ with coefficients in $F$.
Then $H^\bullet(Z,F)$ is a $\Z$-graded complex vector bundle over $S$.
We denote
\begin{equation}
\label{eq-def-chiprim}
\chi'(Z,F) = \sum_{p=0}^{\dim Z} (-1)^pp\, \mathrm{rk}\big(H^p(Z,F)\big) \;.
\end{equation}

Now we recall the definition of analytic torsion forms \cite[Def. 2.18]{imrn-zhu}, \cite[Def. 3.22]{bl}.

\begin{defn}
\label{def-atf}
The analytic torsion form associated with $(T^HM,g^{TZ},h^F)$ is defined by
\begin{align}
\label{eq-def-atf}
\begin{split}
\mathscr{T}(T^HM,g^{TZ},h^F) & = - \int_0^{+\infty}
\bigg\{ \varphi \tr\Big[(-1)^{N^{TZ}}\frac{N^{TZ}}{2}f'\big(\mathscr{D}_t\big)\Big]
- \frac{\chi'(Z,F)}{2} \\
& \hspace{15mm} - \Big(\frac{\dim Z \mathrm{rk}(F)\chi(Z)}{4} - \frac{\chi'(Z,F)}{2}\Big)f'\Big(\frac{i\sqrt{t}}{2}\Big) \bigg\}
\frac{dt}{t} \;.
\end{split}
\end{align}
The convergence of the integral in \eqref{eq-def-atf} follows from
the family local index theorem \cite[Thm 3.21]{bl} \cite[Thm 2.17]{imrn-zhu}.
And $d \mathscr{T}(T^HM,g^{TZ},h^F)$ is given by \eqref{eq-d-tf} with $j=0$.
\end{defn}

Recall that $Q^S$ is the vector space of real even differential forms on $S$
and $Q^{S,0} \subseteq Q^S$ is the sub vector space of exact forms.
The analytic torsion form $\mathscr{T}(T^HM,g^{TZ},h^F)$ is viewed as an element in $Q^S/Q^{S,0}$.

\section{Finite dimensional model}
\label{sect-model}

The construction in this section
may be viewed as a model of the problem addressed in this paper,
in which the fibration has  zero-dimensional fibers.
This section is organized as follows.
In \textsection \ref{subsect-simplical-complex},
we construct a short exact sequence of chain complexes from a pair of linear maps.
In \textsection \ref{subsec-tf-model},
we extend the constructions in \textsection \ref{subsect-simplical-complex}
to flat complex vector bundles.

\subsection{Chain complexes from a pair of linear maps}
\label{subsect-simplical-complex}

Let $W_1$, $W_2$ and $V$ be finite dimensional complex vector spaces.
Let $\tau_1: W_1\rightarrow V$
and $\tau_2: W_2\rightarrow V$ be linear maps.
We define a chain complex
$\big(C^\bullet(\tau_1,\tau_2),\partial\big)$
as follows,
\begin{align}
\label{eq1-def-complex}
\begin{split}
0 \rightarrow
C^0(\tau_1,\tau_2) := W_1\oplus W_2 & \xrightarrow{\partial}
C^1(\tau_1,\tau_2) := V \rightarrow 0  \\
(w_1,w_2) & \mapsto \tau_2(w_2) - \tau_1(w_1) \;.
\end{split}
\end{align}

We denote
\begin{equation}
\label{eq-def-V1-V2}
V_1 = \mathrm{Im}(\tau_1) \subseteq V \;,\hspace{5mm}
V_2 = \mathrm{Im}(\tau_2) \subseteq V \;.
\end{equation}
We define a chain complex
$\big(C^\bullet_\mathrm{r}(\tau_1,\tau_2),\partial\big)$
as follows,
\begin{align}
\label{eq-def-complex-r}
\begin{split}
0 \rightarrow
C^0_\mathrm{r}(\tau_1,\tau_2) := V_1\oplus V_2 & \xrightarrow{\partial}
C^1_\mathrm{r}(\tau_1,\tau_2) := V \rightarrow 0 \\
(v_1,v_2) & \mapsto v_2 - v_1 \;.
\end{split}
\end{align}

We denote
\begin{equation}
\label{eq-def-K}
K_1 = \mathrm{Ker}(\tau_1) \subseteq W_1 \;,\hspace{5mm}
K_2 = \mathrm{Ker}(\tau_2) \subseteq W_2 \;.
\end{equation}
We have a short exact sequence of chain complexes,
\begin{equation}
\label{eq-ses-CCr}
\xymatrix{
0 \ar[r] &
0 \ar[rr]  &&
C^1(\tau_1,\tau_2) \ar[rr]^{\mathrm{Id}} &&
C^1_\mathrm{r}(\tau_1,\tau_2)  \ar[r] &
0 \\
0 \ar[r] &
K_1 \oplus K_2 \ar[rr]  \ar[u] &&
C^0(\tau_1,\tau_2) \ar[rr]^{\tau_1\oplus\tau_2} \ar[u]_{\partial} &&
C^0_\mathrm{r}(\tau_1,\tau_2)  \ar[r] \ar[u]_{\partial} &
0  \;,}
\end{equation}
where $K_1\oplus K_2 \rightarrow C^0(\tau_1,\tau_2) = W_1 \oplus W_2$
is the direct sum of the embeddings in \eqref{eq-def-K}.

Set
\begin{equation}
\label{eq-def-Cj}
C^\bullet_0 = C^\bullet(\tau_1,\tau_2) \;,\;
C^\bullet_1 = C^\bullet(\tau_1,\Id_V) \;,\;
C^\bullet_2 = C^\bullet(\Id_V,\tau_2) \;,\;
C^\bullet_3 = C^\bullet(\Id_V,\Id_V)\;.
\end{equation}
We define
\begin{equation}
\alpha_1: C^\bullet_0 \rightarrow C^\bullet_1 \;,\hspace{5mm}
\alpha_2: C^\bullet_0 \rightarrow C^\bullet_2 \;,\hspace{5mm}
\beta_1: C^\bullet_1 \rightarrow C^\bullet_3 \;,\hspace{5mm}
\beta_2: C^\bullet_2 \rightarrow C^\bullet_3
\end{equation}
as follows,
\begin{align}
\begin{split}
& \alpha_1\big|_{C^0_0} = \Id_{W_1} \oplus \tau_2 \;,\hspace{5mm}
\alpha_2\big|_{C^0_0} = \tau_1 \oplus \Id_{W_2} \;,\hspace{5mm}
\alpha_1\big|_{C^1_0} = \alpha_2\big|_{C^1_0} = \Id_V \;,\\
& \beta_1\big|_{C^0_1} = \tau_1 \oplus \Id_V \;,\hspace{5mm}
\beta_2\big|_{C^0_2} = \Id_V \oplus \tau_2 \;,\hspace{5mm}
\beta_1\big|_{C^1_1} = \beta_2\big|_{C^1_2} = \Id_V \;.
\end{split}
\end{align}
We have a short exact sequence of chain complexes,
\begin{equation}
\label{eq-ses-CCC}
\xymatrix{
0 \ar[r] &
C^\bullet_0 \ar[r]^{\hspace{-5mm}\alpha_1 \oplus \alpha_2} &
C^\bullet_1 \oplus C^\bullet_2 \ar[r]^{\hspace{5mm}\beta_2-\beta_1} &
C^\bullet_3 \ar[r] &
0 \;.}
\end{equation}
For $j=0,1,2,3$,
let $H^k\big(C^\bullet_j,\partial\big)$
be the $k$-th cohomology group of $\big(C^\bullet_j,\partial\big)$,
i.e.,
\begin{equation}
H^k\big(C^\bullet_j,\partial\big) =
\frac{\mathrm{Ker}\big(\partial: C^k_j\rightarrow C^{k+1}_j\big)}
{\mathrm{Im}\big(\partial: C^{k-1}_j\rightarrow C^k_j\big)} \;.
\end{equation}
From \eqref{eq-ses-CCC},
we get a long exact sequence of cohomology groups,
\begin{equation}
\label{eq-mv-sequence-I123}
\xymatrix{
\cdots \ar[r] &
H^k\big(C^\bullet_0,\partial\big) \ar[r] &
H^k\big(C^\bullet_1 \oplus C^\bullet_2,\partial\big)  \ar[r] &
H^k\big(C^\bullet_3,\partial\big)  \ar[r] &
\cdots \;.}
\end{equation}
We denote
\begin{equation}
\label{eq-def-W12}
W_{12} = \Big\{(w_1,w_2)\in W_1 \oplus W_2 \;:\;\tau_1(w_1)=\tau_2(w_2)\Big\} \;.
\end{equation}
A direct calculation yields
\begin{align}
\label{eq-HC}
\begin{split}
& H^0\big( C^\bullet_0,\partial \big) = W_{12} \;,\hspace{5mm}
H^0\big( C^\bullet_1 \oplus C^\bullet_2,\partial \big) = W_1 \oplus W_2 \;,\hspace{5mm}
H^0\big( C^\bullet_3,\partial \big) = V \;,\\
& H^1\big( C^\bullet_0,\partial \big) = V/\big(V_1 + V_2\big) \;,\hspace{5mm}
H^1\big( C^\bullet_1 \oplus C^\bullet_2,\partial \big) =
H^1\big( C^\bullet_3,\partial \big) = 0 \;.
\end{split}
\end{align}
Thus the long exact sequence \eqref{eq-mv-sequence-I123} is
\begin{equation}
\xymatrix{
0 \ar[r] &
W_{12} \ar@{^{(}->}[r]  &
W_1 \oplus W_2 \ar[r]^{\hspace{5mm}\tau_2-\tau_1} &
V  \ar@{->>}[r] &
V/\big(V_1 + V_2\big) \ar[r] &
0  \;,}
\end{equation}
where $W_{12} \hookrightarrow W_1 \oplus W_2$
is the direct sum of the obvious embeddings
$W_{12}\hookrightarrow W_1$ and $W_{12}\hookrightarrow W_2$.

\subsection{A flat family of complexes}
\label{subsec-tf-model}

Now let $W_1$, $W_2$ and $V$ be flat complex vector bundles over a smooth manifold $S$.
Let $\tau_1: W_1\rightarrow V$
and $\tau_2: W_2\rightarrow V$
be morphisms between flat complex vector bundles.
Then the chain complexes
$\big(C^\bullet_j,\partial\big)$ ($j=0,1,2,3$)
considered in \textsection \ref{subsect-simplical-complex}
become chain complexes of flat complex vector bundles over $S$.

Let $h^{W_1}$, $h^{W_2}$ and $h^V$
be Hermitian metrics on $W_1$, $W_2$ and $V$.
For $j=0,1,2,3$,
we construct a Hermitian metric $h^{C^\bullet_j}= h^{C^0_j} \oplus h^{C^1_j}$ on $C^\bullet_j$ as follows,
\begin{align}
\begin{split}
& h^{C^0} =
h^{W_1} \oplus h^{W_2} \;;\hspace{5mm}
h^{C^0_3} =
\frac{1}{2} h^V \oplus \frac{1}{2} h^V \;;\hspace{5mm}
h^{C^0_j} =
h^{W_j} \oplus \frac{1}{2} h^V
\hspace{5mm} \text{for } j=1,2 \;;\\
& h^{C^1_j} = h^V
\hspace{5mm}\text{for } j=0,1,2,3 \;.
\end{split}
\end{align}
Recall that $Q^{S,0} \subseteq Q^S \subseteq \Omega^\bullet(S)$
were defined in the paragraph containing \eqref{intro-dT}.
Let $\mathscr{T}_j \in Q^S$
be the torsion form (cf. \textsection \ref{subsec-tf})
associated with $\big(C^\bullet_j,\partial,h^{C^\bullet_j}\big)$.

The exact sequence \eqref{eq-mv-sequence-I123}
becomes an exact sequence of flat complex vector bundles.
Let $\mathscr{T}_\mathscr{H} \in Q^S$
be the torsion form (cf. \textsection \ref{subsec-tf})
associated with the exact sequence \eqref{eq-mv-sequence-I123}
equipped with Hermitian metrics
induced by $h^{C^\bullet_j}$ via \eqref{eq-hodge-dim-fini}.

The following theorem is a consequence of \cite[Thm. 7.37]{g-a2},
which may be viewed as an analogue of a result of Ma on the Bott-Chern forms \cite[Thm 1.2]{ma-ajm},
and may also be viewed as a finite dimensional version of \cite[Thm 0.1]{ma}.

\begin{thm}
\label{thm-model-tf-gluing}
The following equation holds,
\begin{equation}
\mathscr{T} - \mathscr{T}_1 - \mathscr{T}_2 + \mathscr{T}_3 + \mathscr{T}_\mathscr{H}
\in Q^{S,0}  \;.
\end{equation}
\end{thm}

\section{Gluing formula for analytic torsion forms}
\label{sect-construction}

This section is the heart of this paper.
The central idea is
to deform the metrics on $TZ$ and $F$ such that
the gluing formula considered in Theorem \ref{thm-gluing}
degenerates to the gluing formula given in Theorem \ref{thm-model-tf-gluing}.
This section is organized as follows.
In \textsection \ref{subsec-RT-deformation},
we introduce a two-parameter deformation of the objects constructed in the introduction.
In \textsection \ref{subsec-intermediate},
we prove Theorem \ref{thm-gluing}.
The proof is based on several intermediate results.
Their proofs are delayed to
\textsection \ref{sect-al-wd}, \ref{sect-tf}, \ref{sect-mv}.

\subsection{A two-parameter deformation}
\label{subsec-RT-deformation}

Recall that $\pi_R: M_R \rightarrow S$ was constructed
in the paragraph containing \eqref{intro-piR}.
By the second identity in \eqref{intro-INR-MR},
we may view $M_1''$, $M_2''$ and $IN_R$ as subsets of $M_R$.
Set
\begin{equation}
M_{1,R} = M_1'' \cup IN_R \;,\hspace{5mm}
M_{2,R} = M_2'' \cup IN_R \;,\hspace{5mm}
M_{3,R} = IN_R \;.
\end{equation}
For convenience,
we denote $M_{0,R} = M_R$.
For $j=0,1,2,3$,
set
\begin{equation}
\pi_{j,R} = \pi_R\big|_{M_{j,R}}: M_{j,R} \rightarrow S \;.
\end{equation}
Let $Z_{j,R}$ be the fiber of $\pi_{j,R}$.
We denote $Z_R = Z_{0,R}$.

Recall that the diffeomorphism $\varphi_R: M \rightarrow M_R$
was constructed in \eqref{intro-varphiR}.
Recall that $T^HM \subseteq TM$ was constructed
in the paragraph containing \eqref{intro-THM}.
Set
\begin{equation}
\label{eq0-THMR}
T^HM_R = \varphi_{R,*}\big(T^HM\big) \subseteq TM_R \;.
\end{equation}
Then we have
\begin{equation}
\label{eq-THMR}
TM_R = T^HM_R \oplus TZ_R \;.
\end{equation}
Recall that $T^HN \subseteq TN$ was constructed in the paragraph containing \eqref{intro-TN}.
By \eqref{intro-THM-THN}, \eqref{intro-varphiR} and \eqref{eq0-THMR},
we have
\begin{equation}
T^HM_R \big|_{IN_R} = \mathrm{pr}_2^*\big(T^HN\big) \;,
\end{equation}
where $\mathrm{pr}_2: IN_R = [-R,R]\times N \rightarrow N$
is the projection to the second factor.
For $j=0,1,2,3$,
set
\begin{equation}
T^HM_{j,R} = T^HM_R\big|_{M_{j,R}} \subseteq TM_{j,R} \;.
\end{equation}

Recall that the metric $g^{TZ_R}$ on $TZ_R$ was constructed in \eqref{intro-gTZR}.
For $j=0,1,2,3$,
set
\begin{equation}
g^{TZ_{j,R}} = g^{TZ_R}\big|_{M_{j,R}} \;.
\end{equation}

Recall that
the flat complex vector bundle $(F,\n^F)$ over $Z_R$
and the Hermitian metric $h^F$ on $F$
were constructed in the paragraph containing \eqref{intro-hFT}.
We remark that
\eqref{eq-intro-nF} and \eqref{eq-intro-hF}
hold with $IN$ replaced by $IN_R$.

Let $f_\infty: [-1,1] \rightarrow \R$ be as in \eqref{intro-finf}.
We further assume that
\begin{align}
\label{eq-def-finf}
\begin{split}
& f_\infty(s) = f_\infty(-s) \;,\hspace{5mm}
\big| f'_\infty(s) \big| \leqslant 2 \;,\hspace{5mm}
\text{for } |s| \leqslant 1 \;;\\
& f_\infty(s) = 1 - s^2/2 \;,\hspace{5mm}
\text{for } |s|\leqslant \frac{1}{4} \;;\\
& f_\infty(s) =
(s-b)^2/2 \;,\hspace{5mm}
\text{for } b = \pm 1 ,\; |s-b|\leqslant \frac{1}{4} \;.
\end{split}
\end{align}
Let $\chi: \R \rightarrow \R$ be a smooth function such that
\begin{equation}
\label{eq-def-chi}
0 \leqslant \chi \leqslant 1 \;,\hspace{5mm}
\chi\big|_{]-\infty,1/4]} = 0 \;,\hspace{5mm}
\chi\big|_{[1/2,+\infty[} = 1 \;,\hspace{5mm}
0 \leqslant \chi' \leqslant 8 \;.
\end{equation}
As $f_\infty'$ is an odd function,
for $T\geqslant 0$,
there exists a unique smooth function $f_T: [-1,1] \rightarrow \R$ satisfying
\begin{equation}
\label{eq-def-fT}
f_T(-1) = f_T(1) = 0 \;,\hspace{5mm}
f'_T(s) = f'_\infty(s) \chi\big(e^{T^2}(1-|s|)\big) \;.
\end{equation}
By \eqref{eq-def-finf}-\eqref{eq-def-fT},
the following uniform estimates hold,
\begin{equation}
\label{eq-compare-f-fT}
f_T(s) = f_\infty(s) + \mathscr{O}\big(e^{-T^2}\big) \;,\hspace{5mm}
f_T'(s) = f'_\infty(s) + \mathscr{O}\big(e^{-T^2}\big) \;,\hspace{5mm}
f_T''(s) = \mathscr{O}\big(1\big) \;.
\end{equation}
Moreover,
we have $\mathrm{supp}(f_T) \subseteq \big[-1+e^{-T^2}/4,1-e^{-T^2}/4\big]$.
We will view $f_T$ as a smooth function on $M_R$ in the sense of \eqref{intro-fT-MR}.
Set
\begin{equation}
h^F_T = e^{-2Tf_T}h^F \;.
\end{equation}

For $j=0,1,2,3$,
let
\begin{equation}
\label{eq-def-TjRT}
\mathscr{T}_{j,R,T} \in Q^S
\end{equation}
be the analytic torsion form (cf. Definition \ref{def-atf}) associated with
\begin{equation}
\big(\pi_{j,R},T^HM_{j,R},g^{TZ_{j,R}},F\big|_{M_{j,R}},\nabla^F\big|_{M_{j,R}},h^F_T\big|_{M_{j,R}}\big) \;.
\end{equation}

For $j=0,1,2,3$,
let $d^{Z_{j,R}}$ be the de Rham operator on $\Omega^\bullet(Z_{j,R},F)$.
Let $\big\lVert\cdot\big\rVert_{Z_{j,R}}$ be the $L^2$-metric on $\Omega^\bullet(Z_{j,R},F)$
with respect to $g^{TZ_{j,R}}$ and $h^F$.
Let $d^{Z_{j,R},*}$ be the formal adjoint of $d^{Z_{j,R}}$
with respect to $\big\lVert\cdot\big\rVert_{Z_{j,R}}$.
Set
\begin{align}
\label{eq-def-DRT-conj}
\begin{split}
& \diffjRT = e^{-Tf_T}d^{Z_{j,R}}e^{Tf_T} \;,\hspace{5mm}
\stdiffjRT = e^{Tf_T}d^{Z_{j,R},*}e^{-Tf_T} \;,\\
& \DjRT = \diffjRT + \stdiffjRT \;.
\end{split}
\end{align}
We remark that $e^{Tf_T} \DjRT e^{-Tf_T}$
is the Hodge de Rham operator with respect to $g^{TZ_{j,R}}$ and $h^F_T$.
The self-adjoint extension of $\DjRT$
with domain $\mathrm{Dom}\big(\DjRT\big) = \Omega^\bullet_\mathrm{abs}(Z_{j,R},F)$
(cf. \cite[(1.4)]{pzz})
will still be denoted by $\DjRT$.
By the Hodge theorem (cf. \cite[Thm 3.1]{bm2} \cite[Thm. 1.1]{pzz}), the following map is bijective,
\begin{align}
\label{eq-hodge-DRT}
\begin{split}
\Ker\big(\DjRT\big) & \rightarrow H^\bullet(Z_{j,R},F) \\
\omega & \mapsto \big[e^{Tf_T}\omega\big] \;.
\end{split}
\end{align}

For $j=0,1,2,3$,
let $h^{H^\bullet(Z_j,F)}_{R,T}$ be the Hermitian metric on
$H^\bullet(Z_j,F) = H^\bullet(Z_{j,R},F)$
induced by $\big\lVert\cdot\big\rVert_{Z_{j,R}}$
via the identification \eqref{eq-hodge-DRT}.
Let
\begin{equation}
\mathscr{T}_{\mathscr{H},R,T}\in Q^S
\end{equation}
be the torsion form
(\cite[\textsection 2]{bl}, cf. \textsection \ref{subsec-tf})
associated with the exact sequence \eqref{eq-mv-top}
equipped with Hermitian metrics $\big(h^{H^\bullet(Z_j,F)}_{R,T}\big)_{j=0,1,2,3}$.

\subsection{Several intermediate results}
\label{subsec-intermediate}

We fix a constant $0<\kappa<1/3$.

\begin{thm}
\label{thm-central-spgap}
There exists $\alpha>0$ such that
for $j=0,1,2,3$ and $T = R^\kappa \gg 1$,
we have
\begin{equation}
\label{eq-thm-central-spgap}
\Sp \Big(R\DjRT\Big) \subseteq
\big]-\infty,-\alpha\sqrt{T}\big]
\cup \big[-1,1\big] \cup
\big[\alpha\sqrt{T},+\infty\big[ \;.
\end{equation}
\end{thm}

Let $\mathscr{E}^{[-1,1]}_{j,R,T}\subseteq\Omega^\bullet(Z_{j,R},F)$
be eigenspace of $R\DjRT$ associated with eigenvalues in $[-1,1]$.
Since $\diffjRT$ commutes with $\DsjRT$,
we have a finite dimensional complex
\begin{equation}
\label{eq-def-ERT}
\big(\mathscr{E}^{[-1,1]}_{j,R,T}\,,\,\diffjRT\big) \;.
\end{equation}

Recall that
the chain complexes of flat complex vector bundles $\big(C^\bullet_j,\partial\big)$ with $j=0,1,2,3$
were constructed in \textsection \ref{subsec-tf-model}.
Their construction depends on the morphisms
$\tau_1: W_1 \rightarrow V$ and
$\tau_2: W_2 \rightarrow V$.
In the sequel, we take
\begin{equation}
\label{eq-complex-evaluation}
V = H^\bullet(Y,F) \;,\hspace{5mm}
W_j = H^\bullet(Z_j,F) \;,\hspace{5mm}
\tau_j\big([\alpha]\big) = [\alpha]\big|_Y \hspace{5mm}
\text{ for } \; j=1,2 \;.
\end{equation}
Then $W_1$, $W_2$ and $V$ are $\Z$-graded.
We will use the notations
$W^\bullet_1$, $W^\bullet_2$, $V^\bullet$
and $\big(C^{\bullet,\bullet}_j,\partial\big)$
to emphasis the grading, i.e.,
$C^{0,k}_0 = W^k_1 \oplus W^k_2$, $C^{1,k}_0 = V^k$, etc.

Now we construct a Hermitian metric on $C^{\bullet,\bullet}_j$.
Let $\DY$ be the Hodge de Rham operator on $\Omega^\bullet(Y,F)$
with respect to $g^{TY}$ and $h^F\big|_N$.
We denote
$\hh(Y,F) = \Ker\big(\DY\big)$.
For $j=1,2$,
let
\begin{equation}
\label{eq-hh}
\hha(Z_{j,\infty},F) \subseteq
\Big\{ \omega\in\Omega^\bullet(Z_{j,\infty},F) \;:\;
d^{Z_{j,\infty}}\omega = d^{Z_{j,\infty},*}\omega = 0 \Big\}
\times \hh(Y,F)
\end{equation}
be as in \cite[(2.52)]{pzz}.
By \cite[Prop. 3.16, Thm. 3.19]{pzz},
the map
\begin{align}
\label{eq-hha2coh}
\begin{split}
\hha(Z_{j,\infty},F) & \rightarrow H^\bullet(Z_{j,\infty},F) = W^\bullet_j \\
(\omega,\hat{\omega}) & \mapsto [\omega]
\end{split}
\end{align}
is bijective.
By \cite[(2.39), (2.52)]{pzz},
the following diagram commutes,
\begin{equation}
\label{eq-commute-hha12}
\xymatrix{
\hha(Z_{j,\infty},F)
\ar[r]
\ar[d]_{(\omega,\hat{\omega})\mapsto\hat{\omega}}
& W^\bullet_j
\ar[d]^{[\omega]\mapsto[\omega]|_Y} \\
\hh(Y,F) \ar[r]^{\hspace{5mm}\hat{\omega}\mapsto[\hat{\omega}]}
& V^\bullet \;. }
\end{equation}
Let $\Djinf$ be the Hodge de Rham operator on $\Omega^\bullet(Z_{j,\infty},F)$
with respect to $g^{TZ_{j,\infty}}$ and $h^F$.
By \cite[(2.40)]{pzz},
we have
\begin{equation}
\label{eq-WKK}
W^\bullet_j =
K^\bullet_j \oplus
K^{\bullet,\perp}_j
\end{equation}
with
\begin{align}
\label{eq-def-WKK}
\begin{split}
K^\bullet_j &
= \Big\{ [\omega] \;:\;
(\omega,\hat{\omega})\in\hha(Z_{j,\infty},F)\;,\;
\hat{\omega} = 0 \Big\} \;,\\
K^{\bullet,\perp}_j &
= \Big\{ [\omega] \;:\;
(\omega,\hat{\omega})\in\hha(Z_{j,\infty},F)\;,\\
& \hspace{25mm} \omega \text{ generalized eigensection of } \Djinf  \text{ associated with } 0 \Big\} \;.
\end{split}
\end{align}
\begin{rem}
As a convention,
a generalized eigenvalue (resp. eigensection)
is always associated with the absolutely continuous spectrum.
In other words,
a generalized eigenvalue (resp. eigensection)
is not an eigenvalue (resp. eigensection).
\end{rem}

By \eqref{eq-commute-hha12},
the definition of $K^\bullet_j$ in \eqref{eq-def-WKK}
is compatible with \eqref{eq-def-K}.
We construct a Hermitian metric
$h^{K^\bullet_j}$ on $K^\bullet_j$
as follows:
for $(\omega,\hat{\omega})\in\hha(Z_{j,\infty},F)$
with $\hat{\omega}=0$,
\begin{equation}
\label{eq-def-hK}
h^{K^\bullet_j}\big([\omega],[\omega]\big)
= \big\lVert\omega\big\rVert^2_{Z_{j,\infty}} \;.
\end{equation}
By \cite[(2.53)]{pzz},
we have $\big\lVert\omega\big\rVert^2_{Z_{j,\infty}}<+\infty$.
Hence $h^{K^\bullet_j}$ is well-defined.
We construct a Hermitian metric
$h^{K^{\bullet,\perp}_j}$ on $K^{\bullet,\perp}_j$
as follows:
for $(\omega,\hat{\omega})\in\hha(Z_{j,\infty},F)$
with $\omega$ a generalized eigensection of $\Djinf$,
\begin{equation}
h^{K^{\bullet,\perp}_j}\big([\omega],[\omega]\big)
= \big\lVert\hat{\omega}\big\rVert^2_Y \;.
\end{equation}
Set
\begin{equation}
\label{eq-def-hW}
h^{W^\bullet_j}_{R,T} =
h^{K^\bullet_j} \oplus
\frac{\sqrt{\pi}}{2} R T^{-1/2} h^{K^{\bullet,\perp}_j} \;.
\end{equation}
Let $h^{V^\bullet}$ be the Hermitian metric on $V^\bullet$
induced by $\big\lVert\cdot\big\rVert_Y$
via the identification $V^\bullet = \hh(Y,F)$
induced by the Hodge theory.
Set
\begin{equation}
\label{eq-def-hVRT}
h^{V^\bullet}_{R,T} =
\sqrt{\pi}  R T^{-1/2} h^{V^\bullet} \;.
\end{equation}
We construct a Hermitian metric
$h^{C^{\bullet,\bullet}_j}_{R,T}$ on $C^{\bullet,\bullet}_j$
as follows,
\begin{align}
\label{eq-def-hC}
\begin{split}
& h^{C^{0,\bullet}_0}_{R,T} =
h^{W^\bullet_1}_{R,T} \oplus h^{W^\bullet_2}_{R,T} \;,\\
& h^{C^{0,\bullet}_j}_{R,T} =
h^{W^\bullet_j}_{R,T} \oplus \frac{1}{2} h^{V^\bullet}_{R,T}
\hspace{5mm}\text{for } j=1,2 \;,\\
& h^{C^{0,\bullet}_3}_{R,T} =
\frac{1}{2} h^{V^\bullet}_{R,T} \oplus \frac{1}{2} h^{V^\bullet}_{R,T} \;,\\
& h^{C^{1,\bullet}_j}_{R,T} = h^{V^\bullet}_{R,T}
\hspace{5mm}\text{for } j=0,1,2,3 \;.
\end{split}
\end{align}

For a positive function $G(R,T)$ on $R,T$
and a two-parameter family of operators $A_{R,T}\in\mathrm{End}\big(C^{\bullet,\bullet}_j\big)$,
we denote $A_{R,T} = \mathscr{O}_{R,T}\big(G(R,T)\big)$
if there exists $C>0$ such that
the operator norm of $A_{R,T}$ with respect to $h^{C^{\bullet,\bullet}_j}_{R,T}$
is bounded by $C G(R,T)$.

\begin{thm}
\label{thm-central-SRT}
There exist linear maps
\begin{equation}
\mathscr{S}_{j,R,T}: \; C^{\bullet,\bullet}_j \rightarrow \mathscr{E}^{[-1,1]}_{j,R,T}
\end{equation}
with $j=0,1,2,3$ such that
\begin{itemize}
\item[-] the map $\mathscr{S}_{j,R,T}$ preserves the grading, i.e.,
\begin{equation}
\mathscr{S}_{j,R,T}\big(C^{p,q}_j\big)
\subseteq \mathscr{E}^{[-1,1]}_{j,R,T} \cap \Omega^{p+q}(Z_{j,R},F) \;;
\end{equation}
\item[-] for $T = R^\kappa \gg 1$,
the map $\mathscr{S}_{j,R,T}$ is bijective
(as a consequence,
$\dim \mathscr{E}^{[-1,1]}_{j,R,T}$ is independent of $T = R^\kappa$);
\item[-] for $T = R^\kappa \gg 1$ and $\sigma\in C^{\bullet,\bullet}_j$,
we have
\begin{align}
\begin{split}
\Big\lVert \mathscr{S}_{j,R,T}(\sigma) \Big\rVert^2_{Z_{j,R}}
 = h^{C^{\bullet,\bullet}_j}_{R,T}(\sigma,\sigma)
\Big( 1 + \mathscr{O}\big(R^{-1/2+\kappa/4}\big) \Big) \;;
\end{split}
\end{align}
\item[-] for $T = R^\kappa \gg 1$,
we have
\begin{equation}
\label{eq4-thm-central-SRT}
\mathscr{S}_{j,R,T}^{-1}\circ \diffjRT \circ \mathscr{S}_{j,R,T} =
\pi^{-1/2} R^{-1} T^{1/2} e^{-T}
\Big( \partial + \mathscr{O}_{R,T}\big(R^{-1/2+\kappa/4}\big) \Big) \;.
\end{equation}
\end{itemize}
\end{thm}

For ease of notations, we denote
\begin{equation}
\label{eq-def-partial-T}
\partial_T = \pi^{-1/2} T^{1/2} e^{-T} \partial
:\; C^{0,\bullet}_j \rightarrow C^{1,\bullet}_j \;.
\end{equation}
Let $\widehat{\mathscr{T}}^k_{j,R,T}\in Q^S$
be the torsion form (cf. \textsection \ref{subsec-tf}) associated with $\big(C^{\bullet,k}_j,R^{-1}\partial_T,h^{C^{\bullet,k}_j}_{R,T}\big)$.
We view $\big(C^{\bullet,\bullet}_j,R^{-1}\partial_T\big)$ as a complex,
whose component of degree $k$ is given by $\bigoplus_{p+q=k}C^{p,q}_j$.
Let $\widehat{\mathscr{T}}_{j,R,T}\in Q^S$
be the torsion form associated with $\big(C^{\bullet,\bullet}_j,R^{-1}\partial_T,h^{C^{\bullet,\bullet}_j}_{R,T}\big)$.
The following identity is a consequence of \cite[Thm. 7.37]{g-a2},
which may be viewed as an analogue of a result of Ma on the Bott-Chern forms \cite[Thm 1.2]{ma-ajm},
and may also be viewed a finite dimensional version of \cite[Thm 0.1]{ma},
\begin{equation}
\label{eq-def-TR}
\widehat{\mathscr{T}}_{j,R,T}
= \sum_{k=0}^{\dim Z} (-1)^k \widehat{\mathscr{T}}^k_{j,R,T} \;.
\end{equation}

For $G(R,T)$ a positive function on $R,T\geqslant 1$
and $\big(\tau_{R,T}\big)_{R,T\geqslant 1}$ a family of differential forms on $S$ with values in a Hermitian vector bundle $\big(E,\big\lVert\cdot\big\rVert_E\big)$,
we write
\begin{equation}
\label{eq-OE}
\tau_{R,T} = \mathscr{O}_E\big(G(R,T)\big) \;,
\end{equation}
if there exists $C>0$
such that the $\mathscr{C}^0$-norm of $\tau_{R,T}$ is dominated by $C G(R,T)$ for $R,T\geqslant 1$.
We remark that $\mathscr{O}_E(\cdot)$ is independent of the norm $\big\lVert\cdot\big\rVert_E$.
If $E$ is a trivial line bundle,
we abbreviate \eqref{eq-OE} as $\tau_{R,T} = \mathscr{O}\big(G(R,T)\big)$.

We equip $Q^S$ with the $\mathscr{C}^0$-norm.
Then,
by the de Rham theorem (cf. \cite[Thm. 1.1 (d)]{bm}),
$Q^{S,0} \subseteq Q^S$ is closed.
We equip $Q^S/Q^{S,0}$ with quotient norm.
For a family of elements in $Q^S/Q^{S,0}$ parameterized by $R,T\geqslant 1$,
we use the same notation as in \eqref{eq-OE} with the $\mathscr{C}^0$-norm replaced by the quotient norm.

\begin{thm}
\label{thm-central-tf}
For $T = R^\kappa \gg 1$,
the following identity holds in $Q^S/Q^{S,0}$,
\begin{equation}
\label{eq-thm-central-tf}
\sum_{j=0}^3 (-1)^{j(j-3)/2} \mathscr{T}_{j,R,T} =
\sum_{j=0}^3 (-1)^{j(j-3)/2} \widehat{\mathscr{T}}_{j,R,T}
+ \mathscr{O}\big(R^{-\kappa/4}\big) \;.
\end{equation}
\end{thm}

We have a Mayer-Vietoris exact sequence of flat complex vector bundles over $S$,
\begin{align}
\label{eq-mv-model-k}
\begin{split}
0 \rightarrow
H^0\big( C^{\bullet,k}_0,R^{-1}\partial_T \big) & \rightarrow
H^0\big( C^{\bullet,k}_1 \oplus C^{\bullet,k}_2,R^{-1}\partial_T \big) \\
& \rightarrow
H^0\big( C^{\bullet,k}_3,R^{-1}\partial_T \big) \rightarrow
H^1\big( C^{\bullet,k}_0,R^{-1}\partial_T \big) \rightarrow 0 \;,
\end{split}
\end{align}
which is induced by \eqref{eq-ses-CCC} with $\partial$ replaced by $R^{-1}\partial_T$.
We equip the cohomology groups in \eqref{eq-mv-model-k}
with Hermitian metrics
induced by $h^{C^{\bullet,\bullet}_j}_{R,T}$
via \eqref{eq-hodge-dim-fini}.
Let $\widehat{\mathscr{T}}_{\mathscr{H},R,T}^k\in Q^S$ be the torsion form
(cf. \textsection \ref{subsec-tf})
associated with the exact sequence \eqref{eq-mv-model-k}.
Set
\begin{equation}
\label{eq-def-THR}
\widehat{\mathscr{T}}_{\mathscr{H},R,T} =
\sum_k (-1)^k \widehat{\mathscr{T}}_{\mathscr{H},R,T}^k \in Q^S\;.
\end{equation}

\begin{thm}
\label{thm-central-mv}
For $T = R^\kappa \gg 1$,
the following identity holds in $Q^S/Q^{S,0}$,
\begin{equation}
\mathscr{T}_{\mathscr{H},R,T} =
\widehat{\mathscr{T}}_{\mathscr{H},R,T} + \mathscr{O}\big(R^{-1/4+\kappa/8}\big) \;.
\end{equation}
\end{thm}

\begin{proof}[Proof of Theorem \ref{thm-gluing}]
By Theorem \ref{thm-model-tf-gluing},
\eqref{eq-def-TR} and \eqref{eq-def-THR},
the following identity holds in $Q^S/Q^{S,0}$,
\begin{equation}
\label{eq2-pf-thm-central}
\sum_{j=0}^3 (-1)^{j(j-3)/2} \widehat{\mathscr{T}}_{j,R,T}
+ \widehat{\mathscr{T}}_{\mathscr{H},R,T}
= 0 \;.
\end{equation}
By Theorems \ref{thm-central-tf}, \ref{thm-central-mv}
and \eqref{eq2-pf-thm-central},
the following identity holds in $Q^S/Q^{S,0}$
as $T = R^\kappa \rightarrow +\infty$,
\begin{equation}
\label{eq-thm-central}
\sum_{j=0}^3 (-1)^{j(j-3)/2} \mathscr{T}_{j,R,T} + \mathscr{T}_{\mathscr{H},R,T}
= \mathscr{O}\big(R^{-\kappa/4}\big) \;.
\end{equation}
On the other hand,
using the anomaly formula \cite[Thm. 3.24]{bl} \cite[Thm. 1.5]{israel-zhu} in the same way as in \cite[\textsection 1.7]{israel-zhu},
we can show that the left hand side of \eqref{eq-thm-central} is independent of $R$ and $T$.
Hence,
for any $R\geqslant 1$ and $T\geqslant 0$,
we have
\begin{equation}
\label{eq2-thm-central}
\sum_{j=0}^3 (-1)^{j(j-3)/2} \mathscr{T}_{j,R,T} + \mathscr{T}_{\mathscr{H},R,T} \in Q^{S,0} \;.
\end{equation}
Taking $R=1$ and $T=0$ in \eqref{eq2-thm-central},
we obtain \eqref{eq-thm-gluing}.
This completes the proof of Theorem \ref{thm-gluing}.
\end{proof}

\section{One-dimensional Witten type deformation}
\label{sect-dim-one-wd}

The construction in this section
may be viewed as a model of the problem addressed in this paper,
in which the fibration has one-dimensional fibers.
This one-dimensional model and the zero-dimensional model constructed in \textsection \ref{sect-model}
are linked by a Witten deformation, i.e., to take $T\rightarrow +\infty$.
This section is organized as follows.
In \textsection \ref{subsec-hodge-interval},
we construct a sheaf $\mathscr{V}$ on $[-1,1]$
and establish a Hodge theorem for $\mathscr{V}$.
In \textsection \ref{subsect-wd-interval},
we consider a Witten type deformation of the Hodge Laplacian in \textsection \ref{subsec-hodge-interval}.
In \textsection \ref{subsec-wd-cylinder},
we consider a Witten type deformation of the Hodge Laplacian on a cylinder.

\subsection{Hodge theory for an interval}
\label{subsec-hodge-interval}

We denote $I = [-1,1]$.
Let $u\in I$ be the coordinate.
Let $V$ be a finite dimensional complex vector space.
Let $V_1,V_2 \subseteq V$ be vector subspaces.
We construct a sheaf $\mathscr{V}$ on $I$ as follows:
for any open subset $U\subseteq I$,
\begin{align}
\begin{split}
\mathscr{V}(U) =
\Big\{ & \text{ locally constant function }
\alpha: U \rightarrow V \;:\; \\
& \hspace{20mm} \alpha(-1) \in V_1 \; \text{ if } -1\in U \;,\;
\alpha(1) \in V_2 \; \text{ if } 1\in U \Big\} \;.
\end{split}
\end{align}
We construct sheaves $\big(\mathscr{R}^k\big)_{k=0,1}$ on $I$ as follows:
for any open subset $U\subseteq I$,
\begin{align}
\begin{split}
& \mathscr{R}^0(U) = \Big\{s\in\smooth(U,V)\;:\;
s(-1) \in V_1 \; \text{ if } -1\in U \;,\;
s(1) \in V_2 \; \text{ if } 1\in U \Big\} \;,\\
& \mathscr{R}^1(U) = \Omega^1(U,V) \;.
\end{split}
\end{align}
Let $i: \mathscr{V}\rightarrow\mathscr{R}^0$ is the obvious injection.
Let $d: \mathscr{R}^0\rightarrow\mathscr{R}^1$ be the de Rham operator.
Then
\begin{equation}
\xymatrix{
\mathscr{V} \ar[r]^i & \mathscr{R}^0 \ar[r]^d & \mathscr{R}^1
}
\end{equation}
is a resolution of $\mathscr{V}$ by fine sheaves.
Let $H^\bullet(I,\mathscr{V})$ be the sheaf theoretic cohomology of $I$ with coefficients in $\mathscr{V}$.
We have
\begin{equation}
H^\bullet(I,\mathscr{V}) = H^\bullet(\mathscr{R}^\bullet(I),d) \;.
\end{equation}
Then a direct calculation yields
\begin{equation}
\label{eq-coh-sheaf-model}
H^0(I,\mathscr{V}) = V_1 \cap V_2 \;,\hspace{5mm}
H^1(I,\mathscr{V}) = V/(V_1+V_2) \;.
\end{equation}

Let $h^V$ be a Hermitian metric on $V$.
We denote $V[du] = V \oplus  V du = V \otimes \Lambda^\bullet(T^*I)$.
Let $\big\lVert\cdot\big\rVert_{V[du]}$ be the norm on $V[du]$ induced by
$h^V$ and the metric on $\Lambda^\bullet(T^*I)$ such that $\big|du\big|=1$.
We introduce the following Clifford actions on $V[d u]$,
\begin{equation}
\label{eq-def-cchat}
c = d u \wedge - \; i_{\frac{\partial}{\partial u}} \;,\hspace{5mm}
\hat{c} = d u \wedge + \; i_{\frac{\partial}{\partial u}} \;.
\end{equation}
Then $c$ (resp. $\hat{c}$) is skew-adjoint (resp. self-adjoint)
with respect to $\big\lVert\cdot\big\rVert_{V[du]}$.
Moreover,
\begin{equation}
\label{eq-c-hatc-algebra}
c^2 = -1 \;,\hspace{5mm} {\hat{c}}^2 = 1 \;,\hspace{5mm} c\hat{c} + \hat{c}c = 0 \;.
\end{equation}

For $u\in I$ and $\omega\in\Omega^\bullet(I,V)=\smooth(I,V[du])$,
we denote by $\omega_u\in V[d u]$ the value of $\omega$ at $u$.
Let $\big\lVert\cdot\big\rVert_{[-1,1]}$ be the $L^2$-norm on $\Omega^\bullet(I,V)$,
i.e.,
\begin{equation}
\big\lVert \omega \big\rVert_{[-1,1]}^2 =
\int_{-1}^1 \big\lVert\omega_u\big\rVert_{V[du]}^2 du \;.
\end{equation}
Let $d^V$ be the de Rham operator on $\Omega^\bullet(I,V)$.
Let $d^{V,*}$ be its formal adjoint.
Set
\begin{equation}
\DV = d^V + d^{V,*} \;.
\end{equation}
Then,
by \eqref{eq-def-cchat},
we have
\begin{equation}
\DV = c \frac{\partial}{\partial u} \;,\hspace{5mm}
\DsV = - \frac{\partial^2}{\partial u^2} \;.
\end{equation}

For $j=1,2$,
let $V_j^\perp\subseteq V$ be the orthogonal complement of $V_j\subseteq V$ with respect to $h^V$.
Set
\begin{equation}
\Omega^\bullet_\bd(I,V) =
\Big\{ \omega\in\Omega^\bullet(I,V) \;:\;
\omega_{-1}\in V_1 \oplus V_1^\perp d u \;,\; \omega_{1}\in V_2 \oplus V_2^\perp d u \Big\} \;.
\end{equation}
Let $\DVbd$ be the self-adjoint extension of $\DV$ with
domain $\mathrm{Dom}\big(\DVbd\big) = \Omega^\bullet_\bd(I,V)$.
We will also consider $\DsVbd$ with domain
\begin{equation}
\mathrm{Dom}\big(\DsVbd\big) =
\Big\{ \omega\in\Omega^\bullet_\bd(I,V) \;:\; \DV\omega\in\Omega^\bullet_\bd(I,V) \Big\}\;.
\end{equation}
We have
\begin{equation}
\label{eq-ker-dirac-model}
\Ker\big(\DsVbd\big)\Big|_{\Omega^0_\bd(I,V)}
= V_1 \cap V_2 \;,\hspace{5mm}
\Ker\big(\DsVbd\big)\Big|_{\Omega^1_\bd(I,V)}
= \big(V_1^\perp \cap V_2^\perp\big) d u \;,
\end{equation}
where the right hand sides are viewed as constant functions on $I$ with values in $V[d u]$.
From \eqref{eq-coh-sheaf-model} and
\eqref{eq-ker-dirac-model},
we get a natural isomorphism
\begin{equation}
\label{eq-hodge-model}
\Ker\big(\DsVbd\big) \simeq H^\bullet(I,\mathscr{V}) \;.
\end{equation}

\subsection{Witten type deformation on an interval}
\label{subsect-wd-interval}

For $T \geqslant 0$,
set
\begin{equation}
\label{eq-def-dVT}
d^V_{T} = e^{-Tf_T} d^V e^{Tf_T} \;,\hspace{5mm}
d^{V,*}_{T} = e^{Tf_T} d^{V,*} e^{-Tf_T} \;,\hspace{5mm}
\DVT = d^V_T + d^{V,*}_T \;,
\end{equation}
where $f_T$ was defined by \eqref{eq-def-fT}.
The operator $\DVT$ is formally self-adjoint with respect to $\big\lVert\cdot\big\rVert_{[-1,1]}$.
We have
\begin{equation}
\label{eq-DV-Tf}
\DVT = \DV + Tf'_T\hat{c} =
c \frac{\partial}{\partial u} + Tf'_T\hat{c} \;,\hspace{5mm}
\DsVT =  -\frac{\partial^2}{\partial u^2}+ Tf''_Tc\hat{c} + T^2|f'_T|^2 \;.
\end{equation}
Let $\DVTbd$ be the self-adjoint extension of $\DVT$  with domain
$\mathrm{Dom}\big(\DVTbd\big) = \Omega^\bullet_\bd(I,V)$.

\begin{thm}
\label{thm-witten-spec}
There exist $\beta>\alpha>0$ such that for $T \gg 1$,
we have
\begin{equation}
\label{eq-1-thm-witten-spec}
\Sp\big(\DVTbd\big) \subseteq
\big]-\infty,-\alpha\sqrt{T}\big] \cup
\big[-\beta\sqrt{T}e^{-T},\beta\sqrt{T}e^{-T}\big] \cup
\big[\alpha\sqrt{T},+\infty\,\big[ \;.
\end{equation}
\end{thm}
\begin{proof}
Recall that $f_\infty$ was defined by \eqref{eq-def-finf}.
For $T\geqslant 0$,
set
\begin{equation}
\label{eq0-pf-thm-witten-spec}
\wDVT = \DV + Tf'_\infty\hat{c} \;.
\end{equation}
Let $\wDVTbd$ be the self-adjoint extension of $\wDVT$
with domain $\mathrm{Dom}\big(\wDVTbd\big) = \Omega^\bullet_\bd(I,V)$.
By \eqref{eq-compare-f-fT},
\eqref{eq-DV-Tf}
and \eqref{eq0-pf-thm-witten-spec},
the operator norm of $\DVT - \wDVT$ is bounded by $\mathscr{O}\big(Te^{-T^2}\big)$.
Hence it is sufficient to show that
there exist $\beta>\alpha>0$ such that
\begin{equation}
\label{eq1-pf-thm-witten-spec}
\Sp\Big(\wDVTbd\Big) \subseteq
\big]-\infty,-2\alpha\sqrt{T}\big] \cup
\big[-\beta\sqrt{T}e^{-T}/2,\beta\sqrt{T}e^{-T}/2\big] \cup
\big[2\alpha\sqrt{T},+\infty\,\big[ \;.
\end{equation}

Recall that $\chi$ was defined by \eqref{eq-def-chi}.
For $T\geqslant 0$,
we construct smooth functions $\phi_{1,T}, \phi_{2,T}, \phi_{3,T}: I \rightarrow \R$ as follows,
\begin{align}
\begin{split}
\phi_{1,T}(u) = \phi_{2,T}(-u) & = \big(1-\chi(4u+4)\big)\exp\big(-T(u+1)^2/2\big) \;,\\
\phi_{3,T}(u) & = \big(1-\chi(4|u|)\big)\exp\big(-Tu^2/2\big) \;,\hspace{5mm}\text{for } u\in I \;.
\end{split}
\end{align}
Let $\big(C^\bullet_\mathrm{r},\partial\big)$ be the complex in \eqref{eq-def-complex-r} associated with $V_1,V_2\subseteq V$.
For $T\geqslant 0$,
we construct a linear map $J_T: C^\bullet_\mathrm{r} \rightarrow \Omega^\bullet(I,V)$ as follows,
\begin{align}
\label{eq2-pf-thm-witten-spec}
\begin{split}
\text{for } (v_1,v_2)\in V_1 \oplus V_2 & = C^0_\mathrm{r} \;,\hspace{2.5mm}
J_T(v_1,v_2) = \phi_{1,T} v_1 + \phi_{2,T} v_2 \in \smooth(I,V) \,;\\
\text{for } v\in V & = C^1_\mathrm{r} \;,\hspace{2.5mm}
J_T(v) = \phi_{3,T} du \otimes v \in \Omega^1(I,V) \;.
\end{split}
\end{align}
Proceeding in the same way as in \cite[\textsection 6]{bz2}
with $J_T$ in \cite[Def. 6.5]{bz2} replaced by the $J_T$ constructed in \eqref{eq2-pf-thm-witten-spec},
we obtain \eqref{eq1-pf-thm-witten-spec}.
This completes the proof of Theorem \ref{thm-witten-spec}.
\end{proof}

For $\Lambda\subseteq \R$,
we denote by $E^\Lambda_T$ the eigenspace of $\DVTbd$ associated with eigenvalues in $\Lambda$.

\begin{thm}
\label{thm-witten-estimates}
For $T \gg 1$,
we have
$\dim E^{[-1,1]}_T = \dim C^\bullet_\mathrm{r}$.
\end{thm}
\begin{proof}
Let $\wDVTbd$ be as in the proof of Theorem \ref{thm-witten-spec}.
Let $\widetilde{E}^{[-1,1]}_T$ be the eigenspace of $\wDVTbd$
associated with eigenvalues in $[-1,1]$.
Proceeding in the same way as in \cite[\textsection 6]{bz2}
with $J_T$ in \cite[Def. 6.5]{bz2} replaced by the $J_T$ constructed in \eqref{eq2-pf-thm-witten-spec},
we obtain $\dim \widetilde{E}^{[-1,1]}_T = \dim C^\bullet_\mathrm{r}$.
On the other hand,
by the proof of Theorem \ref{thm-witten-spec},
we have $\dim E^{[-1,1]}_T = \dim \widetilde{E}^{[-1,1]}_T$.
This completes the proof of Theorem \ref{thm-witten-estimates}.
\end{proof}

For $j=1,2$,
let
\begin{equation}
P_j : V[d u] \rightarrow V_j \oplus V_j^\perp du
\end{equation}
be orthogonal projections with respect to $\big\lVert\cdot\big\rVert_{V[du]}$.
We denote
$P_j^\perp = \mathrm{Id} - P_j$.
Let
\begin{equation}
\label{eq-def-PLambdaT}
P^\Lambda_T : \Omega^\bullet(I,V) \rightarrow E^\Lambda_T
\end{equation}
be the orthogonal projection with respect to $\big\lVert\cdot\big\rVert_{[-1,1]}$.
For $\omega\in\Omega^\bullet(I,V)$, we denote
\begin{equation}
\big\lVert\omega\big\rVert_{V[du],\mathrm{max}}
= \max\Big\{\big\lVert\omega_u\big\rVert_{V[du]}\;:\;u\in[-1,1]\Big\} \;.
\end{equation}

\begin{prop}
\label{prop-Df-eigen-approx}
For $T \gg 1$,
$\e>0$,
$0<\epsilon<\sqrt{T}$,
$-\sqrt{T}<\lambda<\sqrt{T}$
and $\omega\in\Omega^\bullet(I,V)$
satisfying
\begin{equation}
\label{eq1-prop-Df-eigen-approx}
\DVT\omega = \lambda\omega \;,\hspace{5mm}
\big\lVert P_1^\perp\omega_{-1} \big\rVert_{V[du]} + \big\lVert P_2^\perp\omega_1 \big\rVert_{V[du]}
\leqslant \epsilon\e \big\lVert\omega\big\rVert_{V[du],\mathrm{max}}  \;,
\end{equation}
we have
\begin{equation}
\label{eq2-prop-Df-eigen-approx}
\Big\lVert \omega - P^{[\lambda-\epsilon,\lambda+\epsilon]}_T\omega \Big\rVert_{[-1,1]}
= \mathscr{O}\big(T^{3/2}\big) \e \big\lVert \omega \big\rVert_{[-1,1]}  \;.
\end{equation}
\end{prop}
\begin{proof}
Set $\chi_T(u) = \chi(Tu-T+1)$,
where $\chi$ was constructed in \eqref{eq-def-chi}.
We construct $\omega'\in\Omega^\bullet_\bd(I,V)$ as follows,
\begin{align}
\label{eq1-pf-prop-Df-eigen-approx}
\begin{split}
\omega'_u = \left\{
\begin{array}{rl}
\omega_u - \chi_T(-u)P_1^\perp\omega_{-1}  & \text{if } u < 0 \;,\vspace{1mm}\\
\omega_u - \chi_T(u)P_2^\perp\omega_1 & \text{if } u \geqslant 0 \;.
\end{array} \right.
\end{split}
\end{align}
By the inequality in \eqref{eq1-prop-Df-eigen-approx}, \eqref{eq1-pf-prop-Df-eigen-approx}
and the assumption $0<\epsilon<\sqrt{T}$,
we have
\begin{equation}
\label{eq2-pf-prop-Df-eigen-approx}
\big\lVert \omega - \omega' \big\rVert_{[-1,1]} =
\mathscr{O}\big(T^{-1/2}\big)
\epsilon\e \big\lVert\omega\big\rVert_{V[du],\mathrm{max}}
= \mathscr{O}\big(1\big)
\e \big\lVert\omega\big\rVert_{V[du],\mathrm{max}} \;.
\end{equation}

A direct calculation yields
\begin{align}
\label{eq3-pf-prop-Df-eigen-approx}
\begin{split}
& \Big(\big(\DVTbd-\lambda\big)\omega'\Big)_u \\
& = \left\{
\begin{array}{rl}
\chi_T'(-u) c P_1^\perp\omega_{-1}
+ \big(\lambda-Tf'_T\hat{c}\big) \chi_T(-u) P_1^\perp\omega_{-1}
& \text{if } u < 0\;,\vspace{1mm}\\
-\chi_T'(u) c P_2^\perp\omega_1
+ \big(\lambda-Tf'_T\hat{c}\big) \chi_T(u) P_2^\perp\omega_1
& \text{if } u \geqslant 0\;.
\end{array} \right.
\end{split}
\end{align}
By the inequality in \eqref{eq1-prop-Df-eigen-approx},
\eqref{eq3-pf-prop-Df-eigen-approx}
and the construction of $\chi_T$,
we get
\begin{equation}
\label{eq4-pf-prop-Df-eigen-approx}
\Big\lVert \big(\DVTbd-\lambda\big)\omega'\Big\rVert_{[-1,1]} =
\mathscr{O}\big(\sqrt{T}\big) \epsilon\e \big\lVert\omega\big\rVert_{V[du],\mathrm{max}} \;.
\end{equation}
By Proposition \ref{prop-proj-estimate} and \eqref{eq4-pf-prop-Df-eigen-approx},
we have
\begin{equation}
\label{eq5-pf-prop-Df-eigen-approx}
\Big\lVert \omega' - P^{[\lambda-\epsilon,\lambda+\epsilon]}_T\omega' \Big\rVert_{[-1,1]} =
\mathscr{O}\big(\sqrt{T}\big) \e \big\lVert\omega\big\rVert_{V[du],\mathrm{max}} \;.
\end{equation}

By \eqref{eq-c-hatc-algebra}, \eqref{eq-DV-Tf} and the identity in \eqref{eq1-prop-Df-eigen-approx},
we have
\begin{equation}
\label{eq6a-pf-prop-Df-eigen-approx}
\frac{\partial}{\partial u}\omega =
\big( Tf'_Tc\hat{c} - \lambda c \big) \omega \;.
\end{equation}
Let $\big\lVert\cdot\big\rVert_{H^1,[-1,1]}$
be the $H^1$-norm on $\Omega^\bullet(I,V)$.
By Sobolev inequality and \eqref{eq6a-pf-prop-Df-eigen-approx},
\begin{equation}
\label{eq6-pf-prop-Df-eigen-approx}
\big\lVert\omega\big\rVert_{V[du],\mathrm{max}} =
\mathscr{O}\big(1\big) \big\lVert\omega\big\rVert_{H^1,[-1,1]} =
\mathscr{O}\big(T\big) \big\lVert\omega\big\rVert_{[-1,1]} \;.
\end{equation}

From \eqref{eq2-pf-prop-Df-eigen-approx},
\eqref{eq5-pf-prop-Df-eigen-approx}
and \eqref{eq6-pf-prop-Df-eigen-approx},
we obtain \eqref{eq2-prop-Df-eigen-approx}.
This completes the proof of Proposition \ref{prop-Df-eigen-approx}.
\end{proof}

\subsection{Witten type deformation on a cylinder}
\label{subsec-wd-cylinder}

Let $(Y,g^{TY})$ be a closed Riemannian manifold.
For $R>0$, we denote $I_R=[-R,R]$ and $IY_R= I_R\times Y$.
Let $(u,y)\in[-R,R]\times Y$ be the coordinates.
We will also use the coordinates
$(s,y) = (u/R,y)\in [-1,1]\times Y$.
We equip $T(IY_R)$ with the Riemannian metric
$d u^2 + g^{TY}$.

Let $F$ be a flat complex vector bundle over $Y$.
Let $h^F$ be a Hermitian metric on $F$.
The pull-back of $F$ (resp. $h^F$)
via the canonical projection $IY_R \rightarrow Y$
will still be denoted by $F$ (resp. $h^F$).
Let $\DY$ be the Hodge de Rham operator on $\Omega^\bullet(Y,F)$.
Under the identification
\begin{align}
\begin{split}
\Omega^\bullet(I_R,\Omega^\bullet(Y,F))
& \rightarrow \Omega^\bullet(IY_R,F) \\
\sigma + du \otimes \tau
& \mapsto \sigma + du \wedge \tau \;,\hspace{5mm}
\text{for } \sigma,\tau\in\smooth(I_R,\Omega^\bullet(Y,F)) \;,
\end{split}
\end{align}
the Hodge de Rham operator on $\Omega^\bullet(IY_R,F)$ is given by
\begin{equation}
\DIYR  = \hat{c}c \DY + c \; \frac{\partial}{\partial u} =
R^{-1} \Big(  \hat{c}c R\DY + c \; \frac{\partial}{\partial s} \Big) \;,
\end{equation}
where the term $\hat{c}c = i_{\frac{\partial}{\partial u}} du\wedge - du\wedge i_{\frac{\partial}{\partial u}}$
comes from the fact that $\DY$ anti-commutes with $du\wedge$.

Recall that the function $f_T:I\rightarrow\R$ was constructed in \eqref{eq-def-fT}.
We will view $f_T$ as a function on $IY_R$, i.e.,
$f_T(s,y) = f_T(s)$,
$f_T(u,y) = f_T(u/R)$.
We denote
\begin{equation}
\label{eq-def-fTder}
f_T' = \frac{\partial}{\partial s} f_T = R \frac{\partial}{\partial u} f_T \;.
\end{equation}

For $T\geqslant 0$,
let $\wDIYRT$ be the Hodge de Rham operator
with respect to
$d u^2 + g^{TY}$ and
$h^F_T := e^{-2Tf_T}h^F$.
Set $\DIYRT = e^{-Tf_T} \wDIYRT e^{Tf_T}$.
We have
\begin{align}
\label{eq-def-DRT-cylinder}
\begin{split}
& R \DIYRT
= \hat{c}c R\DY  + c \; \frac{\partial}{\partial s} + Tf'_T\hat{c} \;,\\
& R^2 \DsIYRT
= R^2 \DsY - \frac{\partial^2}{\partial s^2} + Tf''_Tc\hat{c} + T^2|f'_T|^2 \;.
\end{split}
\end{align}

Let
\begin{equation}
\label{eq-Emu}
\Omega^\bullet(Y,F) = \bigoplus_\mu \mathscr{E}^\mu(Y,F)
\end{equation}
be the spectral decomposition with respect to $\DY$,
i.e., $\DY\big|_{\mathscr{E}^\mu(Y,F)} = \mu\Id$.
We denote $\hh(Y,F) = \Ker\big(\DY\big) = \mathscr{E}^0(Y,F)$.
We have the formal decomposition
\begin{equation}
\Omega^\bullet(IY_R, F) =
\bigoplus_\mu \Omega^\bullet\big(I_R,\mathscr{E}^\mu(Y,F)\big) \;.
\end{equation}
Let $D^{\mathscr{E}^\mu(Y,F)}_T$ be the operator $\DVT$ in \textsection \ref{subsect-wd-interval}
with $V$ replaced by $\mathscr{E}^\mu(Y,F)$.
We have
\begin{equation}
\label{eq-dirac-cylinder-spec-expansion}
R\DIYRT\Big|_{\Omega^\bullet(I_R,\mathscr{E}^\mu(Y,F))} =
R\mu \,\hat{c}c + D^{\mathscr{E}^\mu(Y,F)}_T =
R\mu \,\hat{c}c + c\frac{\partial}{\partial s} + Tf'_T\hat{c} \;.
\end{equation}
As a consequence, we have
\begin{equation}
\label{eq-square-witten-cylinder-decomp}
R^2\DsIYRT\Big|_{\Omega^\bullet(I_R,\mathscr{E}^\mu(Y,F))} =
R^2\mu^2 -\frac{\partial^2}{\partial s^2} + T^2|f'_T|^2 + Tf''_Tc\hat{c} \;.
\end{equation}
In particular, we have
\begin{equation}
\label{eq-dirac-cylinder-spec-expansion-hh}
R\DIYRT\Big|_{\Omega^\bullet(I_R,\hh(Y,F))} =
D^{\hh(Y,F)}_T =
c\frac{\partial}{\partial s} + Tf'_T\hat{c} \;.
\end{equation}

For $\alpha>0$,
we define $C_\alpha: \R_+ \times \R_+ \times [-1,1] \rightarrow \R$ as follows,
\begin{equation}
C_\alpha(a,b,s) =
\big(1-e^{-4\alpha}\big)^{-1}
\Big( \big(a-be^{-2\alpha}\big) e^{\alpha(-s-1)}
+ \big(b-ae^{-2\alpha}\big) e^{\alpha(s-1)} \Big) \;.
\end{equation}
Then the following identities hold,
\begin{equation}
C_\alpha(a,b,-1) = a \;,\hspace{5mm}
C_\alpha(a,b,1)  = b \;,\hspace{5mm}
\left( \frac{\partial^2}{\partial s^2} - \alpha^2 \right) C_\alpha(a,b,s) = 0 \;.
\end{equation}

\begin{lemme}
\label{lem-cylinder-nz-estimation}
There exists $\alpha>0$
such that for
$T = R^\kappa \gg 1$,
$\mu\in\Sp\big(\DY\big)\backslash\{0\}$,
$\omega\in \Omega^\bullet\big(I_R,\mathscr{E}^\mu(Y,F)\big)$
and $-\sqrt{R}\leqslant\lambda\leqslant\sqrt{R}$ satisfying
\begin{equation}
\label{eq1-lem-cylinder-nz-estimation}
\Big(R\mu \,\hat{c}c + D^{\mathscr{E}^\mu(Y,F)}_T\Big)\omega  = \lambda \omega \;,
\end{equation}
we have
\begin{equation}
\label{eq2-lem-cylinder-nz-estimation}
\big\lVert \omega_s \big\rVert^2_Y \leqslant
 C_{\alpha R}\left(\big\lVert\omega_{-1}\big\rVert^2_Y,
\big\lVert\omega_1\big\rVert^2_Y, s \right) \;,\hspace{5mm}
\text{for } s\in[-1,1] \;.
\end{equation}
\end{lemme}
\begin{proof}
We assert that
for $g\in\smooth([-1,1],\R_+)$ satisfying $\left( \frac{\partial^2}{\partial s^2} - \alpha^2 R^2 \right)g \geqslant 0$,
we have $g(s) \leqslant C_{\alpha R}\big(g(-1),g(1),s\big)$ for $s\in[-1,1]$.
To prove the assertion, we take $s_0\in[-1,1]$ such that
\begin{equation}
h := g - C_{\alpha R}\big(g(-1),g(1),\cdot\big)\in\smooth([-1,1],\R)
\end{equation}
reaches its maximum value at $s_0$.
If $h(s_0)>0$, then $s_0 \neq \pm 1$ and $h''(s_0)>0$, which is a contradiction.

Now it remains to show that
\begin{equation}
\label{eqc-pf-lem-cylinder-nz-estimation}
\left( \frac{\partial^2}{\partial s^2} - \alpha^2 R^2 \right)
\big\lVert \omega_s \big\rVert^2_Y \geqslant 0 \;.
\end{equation}
By \eqref{eq-dirac-cylinder-spec-expansion},
\eqref{eq-square-witten-cylinder-decomp}
and \eqref{eq1-lem-cylinder-nz-estimation},
we have
\begin{equation}
\label{eq1-pf-lem-cylinder-nz-estimation}
\frac{\partial^2}{\partial s^2} \omega_s =
\Big( R^2\mu^2 + Tf''_T(s)c\hat{c} + T^2|f'_T|^2(s) - \lambda^2 \Big)\omega_s \;.
\end{equation}
Set $\alpha = \min\Big\{ |\mu| : \mu\in\Sp\big(\DY\big)\backslash\{0\} \Big\}$.
For $R,T,\mu,\lambda$ satisfying the hypothesis of Lemma \ref{lem-cylinder-nz-estimation},
we have
\begin{equation}
\label{eq2-pf-lem-cylinder-nz-estimation}
R^2\mu^2 + Tf''_T(s)c\hat{c} + T^2|f'_T|^2(s) - \lambda^2 \geqslant \alpha^2 R^2/2 \;.
\end{equation}
From \eqref{eq1-pf-lem-cylinder-nz-estimation},
\eqref{eq2-pf-lem-cylinder-nz-estimation}
and the obvious identity
\begin{equation}
\frac{\partial^2}{\partial s^2} \big\lVert \omega_s \big\rVert^2_Y =
\Big\langle \frac{\partial^2}{\partial s^2}\omega_s,\omega_s \Big\rangle_Y +
\Big\langle \omega_s,\frac{\partial^2}{\partial s^2}\omega_s \Big\rangle_Y +
2 \Big\lVert \frac{\partial}{\partial s}\omega_s \Big\rVert^2_Y \;,
\end{equation}
we obtain \eqref{eqc-pf-lem-cylinder-nz-estimation}.
This completes the proof of Lemma \ref{lem-cylinder-nz-estimation}.
\end{proof}

\begin{lemme}
\label{lem-cylinder-zm-estimation}
For $\omega\in \Omega^\bullet\big(I,\hh(Y,F)\big)$ and $\lambda\in\R$ satisfying
\begin{equation}
\label{eq-zm-omega-eigen-mu}
D^{\hh(Y,F)}_T \omega = \lambda \omega \;,
\end{equation}
we have
\begin{equation}
\label{eq-lem-cylinder-zm-estimation}
\frac{\partial}{\partial s} \big\langle c\,\omega_s,\omega_s \big\rangle_Y = 0 \;,\hspace{5mm}
\text{for } s\in[-1,1] \;.
\end{equation}
\end{lemme}

\begin{proof}
By \eqref{eq-dirac-cylinder-spec-expansion-hh}
and \eqref{eq-zm-omega-eigen-mu},
we have
\begin{equation}
\label{eq1-pf-lem-cylinder-zm-estimation}
\frac{\partial}{\partial s} \omega_s =
\big( Tf'_T(s)c\hat{c} -\lambda c \big) \omega_s \;.
\end{equation}
Note that $c$ is skew-adjoint and that $\hat{c}$ is self-adjoint,
equation \eqref{eq-lem-cylinder-zm-estimation}
follows from \eqref{eq-c-hatc-algebra}
and \eqref{eq1-pf-lem-cylinder-zm-estimation}.
This completes the proof of Lemma \ref{lem-cylinder-zm-estimation}.
\end{proof}

\section{Adiabatic limit and Witten type deformation}
\label{sect-al-wd}

The purpose of this section is to prove Theorems \ref{thm-central-spgap} and \ref{thm-central-SRT}.

First we summarize the proof of Theorem \ref{thm-central-spgap} given in this section.
Let $\lambda$ be a reasonably small eigenvalue of $R\DRT$.
Let $\omega$ be an eigensection associated with $\lambda$.
First,
in Lemma \ref{prop-zm-minoration},
we show that $\omega^\mathrm{zm}$ (the zero-mode of $\omega$, see \eqref{eq-def-zmnz}) is the principal contributor to the norm of $\omega$. 
Second,
in Lemma \ref{prop-zm-sca-estimate},
we show that $\omega^\mathrm{zm}$ almost lies in the domain of $D^{\hh(Y,F)}_{T,\bd}$,
which is the operator in \eqref{eq-DV-Tf} with $V = \hh(Y,F) := \Ker(D^Y)$.
Combining the results above,
we show that $\lambda$ is very close to an eigenvalue of $D^{\hh(Y,F)}_{T,\bd}$.
On the other hand,
by Theorem \ref{thm-witten-spec},
$D^{\hh(Y,F)}_{T,\bd}$ satisfies the desired spectral gap.
Hence so does $R\DRT$.

Concerning the proof of Theorem \ref{thm-central-SRT}, 
the idea is to explicitly construct $G_{R,T}^+: C^{0,\bullet}_0 \rightarrow \Omega^\bullet(Z_R,F)$ and $I_{R,T}^+: C^{1,\bullet}_0 \rightarrow \Omega^\bullet(Z_R,F)$
(see \eqref{eq1-def-FRTplus} and \eqref{eq1-def-IRTplus}).
The map $\mathscr{S}_{R,T}: C^{\bullet,\bullet}_0 \rightarrow \Omega^\bullet(Z_R,F)$
is then defined by composing $G_{R,T}^+ \oplus I_{R,T}^+$
with the orthogonal projection to the eigenspace of $\DRT$ associated with small eigenvalues.
In the proof of the injectivity of $\mathscr{S}_{R,T}$,
the most subtle part is the injectivity of $\mathscr{S}_{R,T}\big|_{C^{0,\bullet}_0}$.
This is obtained by
constructing an auxiliary map $F_{R,T}^+: C^{0,\bullet}_0 \rightarrow \Omega^\bullet(Z_R,F)$ and
applying Proposition \ref{prop-proj-estimate-better} with $w = F_{R,T}^+$ and $v = G_{R,T}^+$
(see Proposition \ref{prop-FRTplus}).
The proof of the surjectivity of $\mathscr{S}_{R,T}$ highly relies on the results mentioned in the last paragraph:
we reduce the problem to $D^{\hh(Y,F)}_{T,\bd}$
whose spectrum is studied in \textsection \ref{subsect-wd-interval}.

For various reasons,
the kernel of $\DRT$ needs to be studied separately.
Here the strategy is exactly the same as in the last paragraph.
We explicitly construct $F_{R,T}, G_{R,T}: H^0(C^{\bullet,\bullet}_0,\partial) \rightarrow \Omega^\bullet(Z_R,F)$
and $I_{R,T}, J_{R,T}: H^1(C^{\bullet,\bullet}_0,\partial) \rightarrow \Omega^\bullet(Z_R,F)$
(see \eqref{eq1-def-FRT} and \eqref{eq1-def-IRT}).
We construct a bijection $\mathscr{S}_{R,T}^H: H^\bullet(C^{\bullet,\bullet}_0,\partial) \rightarrow \Ker\big(\DRT\big)$
by composing $F_{R,T} \oplus I_{R,T}$ with the orthogonal projection to the kernel of $\DRT$.
To show the injectivity of $\mathscr{S}_{R,T}^H$,
we apply Proposition \ref{prop-proj-estimate-better} with $w = F_{R,T} \oplus I_{R,T}$ and $v = G_{R,T} \oplus J_{R,T}$.

This section is organized as follows.
In \textsection  \ref{subsec-approx-kernel},
we estimate the kernel of the Witten Laplacian.
In \textsection  \ref{subsec-approx-small-eigen},
we estimate the eigenspace of the Witten Laplacian associated with small eigenvalues.
Theorem \ref{thm-central-spgap} will be proved in this subsection.
In \textsection \ref{subsec-approx-de-rham},
we estimate the action of the de Rham operator on
the eigenspace associated with small eigenvalues.
In \textsection \ref{subsec-approx-metrics},
we estimate the $L^2$-metric on the eigenspace associated with small eigenvalues.
Theorem \ref{thm-central-SRT} will be proved in this subsection.

\subsection{Kernel of $\DRT$}
\label{subsec-approx-kernel}
Recall that $\DRT$ was defined in \eqref{eq-def-DRT-conj}.
For convenience,
we denote $\DRzero = \DRT\big|_{T=0}$.
By elliptic estimate (see the proof of \cite[Prop. 3.4]{pzz}),
we may define the $H^1$-norm on $\Omega^\bullet(Z_R,F)$ as follows:
for $\omega\in\Omega^\bullet(Z_R,F)$,
\begin{equation}
\label{eq-def-H1-norm}
\big\lVert \omega \big\rVert^2_{H^1,Z_R} =
\big\lVert \omega \big\rVert^2_{Z_R} + \big\lVert \DRzero\omega \big\rVert^2_{Z_R} \;.
\end{equation}

We fix $\kappa\in]0,1/3[$ as in \eqref{eq-intro-s-gap}.

\begin{prop}
\label{prop-sobolev}
For $T = R^\kappa \gg 1$ and $\omega\in\Omega^\bullet(Z_R,F)$,
we have
\begin{equation}
\label{eq-prop-sobolev}
\big\lVert\omega\big\rVert^2_{H^1,Z_R} \leqslant
2 \big\lVert\omega\big\rVert^2_{Z_R} + \big\lVert\DRT\omega\big\rVert^2_{Z_R} \;.
\end{equation}
\end{prop}
\begin{proof}
By \eqref{eq-compare-f-fT}, \eqref{eq-def-DRT-cylinder}
and the assumption $T = R^\kappa$,
we have
\begin{equation}
\label{eq2-pf-prop-sobolev}
\DsRT + \Id \geqslant \DsRzero \;.
\end{equation}
From \eqref{eq-def-H1-norm} and \eqref{eq2-pf-prop-sobolev},
we obtain \eqref{eq-prop-sobolev}.
This completes the proof of Proposition \ref{prop-sobolev}.
\end{proof}

We will always use the canonical isometric embeddings
\begin{equation}
\label{eq-can-iso}
IY_R \subseteq Z_{j,R} \subseteq Z_{j,\infty} \;,\hspace{5mm}
Z_{j,R} \subseteq Z_R \;,\hspace{5mm} \text{for } j=1,2 \;.
\end{equation}
Recall that the vector subspaces
$\hh(Y,F)\subseteq\Omega^\bullet(Y,F)$ and
$\mathscr{E}^\mu(Y,F)\subseteq\Omega^\bullet(Y,F)$
were defined in the paragraph containing \eqref{eq-Emu}.
For $\omega\in \Omega^\bullet(Z_{j,R},F)$ with $j=0,1,2,3$,
we have the orthogonal decomposition
\begin{equation}
\label{eq-def-zmnz}
\omega\big|_{IY_R} = \omega^\mathrm{zm} + \omega^\mathrm{nz} \;,
\end{equation}
with
\begin{equation}
\omega^\mathrm{zm} \in \Omega^\bullet\big(I_R,\hh(Y,F)\big) \;,\hspace{5mm}
\omega^\mathrm{nz} \in \bigoplus_{\mu\neq 0} \Omega^\bullet\big(I_R,\mathscr{E}^\mu(Y,F)\big) \;.
\end{equation}
We call $\omega^\mathrm{zm}$ (resp. $\omega^\mathrm{nz}$)
the zero-mode (resp. non-zero-mode) of $\omega$.

Recall that $\hha(Z_{j,\infty},F) \subseteq \hh(Z_{j,\infty},F)$ was defined in \eqref{eq-hh}.
Let
\begin{equation}
\label{eq1-def-resol}
\Res\;,\;\stRes:\; \hha(Z_{1,\infty},F) \rightarrow \Omega^\bullet([0,+\infty)\times Y,F)
\end{equation}
be as in \cite[(2.44)]{pzz} with $X_\infty$ replaced by $Z_{1,\infty}$
and $\hh(X_\infty,F)$ replaced by $\hha(Z_{1,\infty},F)$.
Here we recall their construction.
We identify $Z_{1,\infty}$ with $Z_{1,0}\cup[0,+\infty[\times Y$.
By \cite[Prop. 2.5]{pzz} and \cite[(2.10)]{pzz},
for $(\omega,\hat{\omega})\in\hha(Z_{1,\infty},F)$,
we have
\begin{equation}
\label{eq0-def-Res}
\omega\big|_{[0,+\infty[\times Y}
= \hat{\omega} + \omega^\mathrm{nz}
= \hat{\omega} +
\sum_{\mu\neq 0,\; \mu\in\Sp(D^Y)} e^{-|\mu|u} \big( \tau_{\mu,1} - du\wedge\tau_{\mu,2} \big) \;,
\end{equation}
where $\hat{\omega}\in\hh(Y,F)[du]$ is viewed as a constant section in
\begin{equation}
\smooth\big([0,+\infty[,\hh(Y,F)[du]\big)
= \Omega^\bullet\big([0,+\infty[,\hh(Y,F)\big)
\subseteq \Omega^\bullet\big([0,+\infty[\times Y,F\big) \;,
\end{equation}
and $\tau_{\mu,1},\tau_{\mu,2}\in\Omega^\bullet(Y,F)$ satisfy
\begin{equation}
\label{eq1-def-Res}
d^Y \tau_{\mu,1} = d^{Y,*} \tau_{\mu,2} = 0 \;,\hspace{5mm}
d^{Y,*} \tau_{\mu,1} = |\mu| \tau_{\mu,2} \;,\hspace{5mm}
d^Y \tau_{\mu,2} = |\mu| \tau_{\mu,1} \;.
\end{equation}
We define
\begin{align}
\label{eq-def-Res}
\begin{split}
\Res(\omega,\hat{\omega}) & = \sum_{\mu\neq 0,\; \mu\in\Sp(D^Y)} \frac{1}{|\mu|} e^{-|\mu|u} \tau_{\mu,2} \;,\\
\stRes(\omega,\hat{\omega}) & = \sum_{\mu\neq 0,\; \mu\in\Sp(D^Y)} \frac{1}{|\mu|} e^{-|\mu|u} du\wedge\tau_{\mu,1} \;.
\end{split}
\end{align}

Let
\begin{equation}
\label{eq2-def-resol}
\Omega^\bullet([0,+\infty)\times Y,F) \rightarrow \Omega^\bullet(IY_R,F)
\end{equation}
be induced by the isometric identifications $IY_R = [-R,R] \times Y \simeq [0,2R]\times Y \hookrightarrow [0,+\infty)\times Y$.
Composing \eqref{eq1-def-resol} and \eqref{eq2-def-resol},
we get
\begin{equation}
\Res\;,\;\stRes:\; \hh(Z_{1,\infty},F) \rightarrow \Omega^\bullet(IY_R,F) \;.
\end{equation}
We construct
\begin{equation}
\Res\;,\;\stRes:\; \hh(Z_{2,\infty},F) \rightarrow \Omega^\bullet(IY_R,F)
\end{equation}
in the same way.
For $j=1,2$ and $(\omega,\hat{\omega})\in\hh(Z_{j,\infty},F)$,
we have
\begin{align}
\label{eq-resol-RdFRdFstar}
\begin{split}
& d^{Z_R} \Res(\omega,\hat{\omega}) =
d^{Z_R,*}\stRes(\omega,\hat{\omega}) = \omega^\mathrm{nz} \;,\\
& d^{Z_R,*} \Res(\omega,\hat{\omega}) =
d^{Z_R} \stRes(\omega,\hat{\omega}) = 0 \;,\\
& i_{\frac{\partial}{\partial u}} \Res(\omega,\hat{\omega}) =
du\wedge \stRes(\omega,\hat{\omega}) = 0 \;.
\end{split}
\end{align}

Set
\begin{equation}
\label{eq-def-hha12}
\hha(Z_{12,\infty},F) =
\Big\{(\omega_1,\omega_2,\hat{\omega}) \;:\:
(\omega_j,\hat{\omega})\in \hha(Z_{j,\infty},F) \hspace{5mm}\text{for }j=1,2
\Big\} \;.
\end{equation}

Recall that the smooth function $\chi: \R \rightarrow \R$ was defined in \eqref{eq-def-chi}.
Set
\begin{equation}
\label{eq-def-chi12}
\chi_1(s) = 1-\chi\big(4(s+1)\big) \;,\hspace{5mm}
\chi_2(s) = 1-\chi\big(4(1-s)\big) \;.
\end{equation}
We will view $\chi_j$  ($j=1,2$) as functions on $IY_R$,
i.e.,
\begin{equation}
\label{eq-chi-IYR}
\chi_j(s,y) = \chi_j(s) \;,\hspace{5mm}
\chi_j(u,y) = \chi_j(u/R) \;.
\end{equation}
Following \cite[(3.26)]{pzz},
we define
\begin{equation}
\label{eq0-def-FRT}
F_{R,T}, G_{R,T}: \hha(Z_{12,\infty},F) \rightarrow \Omega^\bullet(Z_R,F)
\end{equation}
as follows:
for $(\omega_1,\omega_2,\hat{\omega})\in\hha(Z_{12,\infty},F)$,
\begin{align}
\label{eq1-def-FRT}
\begin{split}
& F_{R,T}(\omega_1,\omega_2,\hat{\omega}) \big|_{Z_{j,0}} =
G_{R,T}(\omega_1,\omega_2,\hat{\omega}) \big|_{Z_{j,0}} = \omega_j \;,\hspace{5mm} \text{for } j=1,2 \;,\\
& F_{R,T}(\omega_1,\omega_2,\hat{\omega}) \big|_{IY_R}
= e^{-Tf_T}\hat{\omega}
+ e^{-Tf_T}d^{Z_R}\Big( \chi_1\Res(\omega_1,\hat{\omega})
+ \chi_2\Res(\omega_2,\hat{\omega}) \Big) \;,\\
& G_{R,T}(\omega_1,\omega_2,\hat{\omega}) \big|_{IY_R}
= e^{-Tf_T}\hat{\omega}
+ e^{Tf_T}d^{Z_R,*}\Big( \chi_1\stRes(\omega_1,\hat{\omega})
+ \chi_2\stRes(\omega_2,\hat{\omega}) \Big) \;,
\end{split}
\end{align}
where we use the identifications in \eqref{eq-can-iso}.
By \eqref{eq-resol-RdFRdFstar},
$F_{R,T}$ and $G_{R,T}$ are well-defined.
Similarly to \cite[(3.28)]{pzz},
by \eqref{eq-def-DRT-conj} and the identities $D^Y \hat{\omega}_j = i_{\frac{\partial}{\partial u}} \hat{\omega}_j = 0$ for $j=1,2$,
we have
\begin{equation}
\label{eq3-def-FRT}
\diffRT F_{R,T}(\omega_1,\omega_2,\hat{\omega}) =
\stdiffRT G_{R,T}(\omega_1,\omega_2,\hat{\omega}) = 0 \;.
\end{equation}

Let
$P_{R,T} : \Omega^\bullet(Z_R,F) \rightarrow \Ker\big(\DRT\big)$
be the orthogonal projection with respect the $L^2$-metric induced by $g^{TZ_R}$ and $h^F$.

\begin{prop}
\label{prop-FRT}
For $T = R^\kappa \gg 1$ and $(\omega_1,\omega_2,\hat{\omega})\in\hha(Z_{12,\infty},F)$,
we have
\begin{equation}
\label{eq-prop-FRT}
\Big\lVert \big( \Id - P_{R,T} \big)F_{R,T}(\omega_1,\omega_2,\hat{\omega}) \Big\rVert^2_{H^1,Z_R}
= \mathscr{O}\big(R^{-2+\kappa}\big) \Big(\big\lVert\omega_1\big\rVert^2_{Z_{1,0}} + \big\lVert\omega_2\big\rVert^2_{Z_{2,0}} \Big) \;.
\end{equation}
\end{prop}
\begin{proof}
The proof follows closely the proof of \cite[Prop. 3.5]{pzz}.
It consists of several steps.

\noindent \textit{Step 1}.
We calculate $(F_{R,T}-G_{R,T})(\omega_1,\omega_2,\hat{\omega})$
and $\DRT (F_{R,T}-G_{R,T})(\omega_1,\omega_2,\hat{\omega})$.

Recall that $IY_R = [-R,R]\times Y$.
By \eqref{eq-def-chi12} and \eqref{eq-chi-IYR},
we have
\begin{equation}
\label{eq11-pf-prop-FRT}
\chi_2\big|_{[-R,0]\times Y} = 0 \;.
\end{equation}
By \eqref{eq-resol-RdFRdFstar}, \eqref{eq1-def-FRT} and \eqref{eq11-pf-prop-FRT},
we have
\begin{align}
\label{eq12-pf-prop-FRT}
\begin{split}
(F_{R,T}-G_{R,T})(\omega_1,\omega_2,\hat{\omega})\Big|_{[-R,0] \times Y}
& =  \chi_1\Big(e^{-Tf_T} - e^{Tf_T}\Big)\omega_1^\mathrm{nz} \\
& \hspace{-20mm} + \frac{\partial\chi_1}{\partial u} \Big( e^{-Tf_T} du\wedge \Res(\omega_1,\hat{\omega})
+ e^{Tf_T} i_{\frac{\partial}{\partial u}} \stRes(\omega_1,\hat{\omega}) \Big) \;.
\end{split}
\end{align}
By \eqref{eq0-def-Res}, \eqref{eq1-def-Res}, \eqref{eq-def-Res} and \eqref{eq-resol-RdFRdFstar},
we have
\begin{align}
\label{eq13-pf-prop-FRT}
\begin{split}
& d^{Z_R} \Big(du\wedge\Res(\omega_1,\hat{\omega})\Big) =
-du\wedge\omega^\mathrm{nz}_1 \;,\hspace{5mm}
d^{Z_R,*}\Big(i_{\frac{\partial}{\partial u}} \stRes(\omega_1,\hat{\omega})\Big) = -i_{\frac{\partial}{\partial u}} \omega^\mathrm{nz}_1 \;,\\
& d^{Z_R,*} \Big(du\wedge\Res(\omega_1,\hat{\omega})\Big) =
-i_{\frac{\partial}{\partial u}}\omega^\mathrm{nz}_1 \;,\hspace{5mm}
d^{Z_R} \Big(i_{\frac{\partial}{\partial u}} \stRes(\omega_1,\hat{\omega})\Big) =
-du\wedge \omega^\mathrm{nz}_1 \;.
\end{split}
\end{align}
By \eqref{eq-def-DRT-conj},
the third identity in \eqref{eq-resol-RdFRdFstar}
and \eqref{eq13-pf-prop-FRT},
we have
\begin{align}
\label{eq14-pf-prop-FRT}
\begin{split}
& \DRT \Big(du\wedge\Res(\omega_1,\hat{\omega})\Big) =
- du\wedge\omega^\mathrm{nz}_1
- i_{\frac{\partial}{\partial u}} \omega^\mathrm{nz}_1
+ T\frac{\partial f_T}{\partial u} \Res(\omega_1,\hat{\omega}) \;,\\
& \DRT \Big(i_{\frac{\partial}{\partial u}} \stRes(\omega_1,\hat{\omega})\Big) =
- du\wedge\omega^\mathrm{nz}_1
- i_{\frac{\partial}{\partial u}} \omega^\mathrm{nz}_1
+ T\frac{\partial f_T}{\partial u} \stRes(\omega_1,\hat{\omega}) \;.
\end{split}
\end{align}
By \eqref{eq12-pf-prop-FRT}
and \eqref{eq14-pf-prop-FRT},
we have
\begin{align}
\label{eq16-pf-prop-FRT}
\begin{split}
& \DRT (F_{R,T}-G_{R,T})(\omega_1,\omega_2,\hat{\omega})\Big|_{[-R,0] \times Y} \\
& = 2T \frac{\partial f_T}{\partial u} \chi_1 \Big(e^{-Tf_T} i_{\frac{\partial}{\partial u}} - e^{Tf_T}du\wedge\Big)\omega^\mathrm{nz}_1
 -2\frac{\partial\chi_1}{\partial u}\Big(e^{-Tf_T} i_{\frac{\partial}{\partial u}} + e^{Tf_T}du\wedge \Big) \omega_1^\mathrm{nz} \\
& \hspace{5mm} + 2T \frac{\partial f_T}{\partial u}\frac{\partial\chi_1}{\partial u} \Big( e^{-Tf_T} \Res(\omega_1,\hat{\omega})
+ e^{Tf_T} \stRes(\omega_1,\hat{\omega}) \Big) \\
& \hspace{5mm} + \frac{\partial^2\chi_1}{\partial u^2} \Big( - e^{-Tf_T} \Res(\omega_1,\hat{\omega})
+  e^{Tf_T} \stRes(\omega_1,\hat{\omega}) \Big) \;.
\end{split}
\end{align}

\noindent \textit{Step 2}.
We estimate $\Big\lVert (F_{R,T}-G_{R,T})(\omega_1,\omega_2,\hat{\omega}) \Big\rVert_{Z_R}$
and $\Big\lVert \DRT (F_{R,T}-G_{R,T})(\omega_1,\omega_2,\hat{\omega}) \Big\rVert_{Z_R}$.

For
\begin{equation}
\label{eq21-pf-prop-FRT}
\tau\in
\Big\{ \; \omega^\mathrm{nz}_1 \;,\;
\Res(\omega_1,\hat{\omega}) \;,\;
\stRes(\omega_1,\hat{\omega}) \; \Big\}
\end{equation}
and $-R\leqslant u\leqslant 0$,
by \eqref{eq0-def-Res} and \eqref{eq-def-Res},
we have
\begin{equation}
\label{eq22-pf-prop-FRT}
\big\lVert \tau \big\rVert_{\{u\}\times Y}
= \mathscr{O}\big(e^{-a(R+u)}\big) \big\lVert\omega^\mathrm{nz}_1\big\rVert_{\partial Z_{1,0}}
= \mathscr{O}\big(e^{-a(R+u)}\big) \big\lVert\omega_1\big\rVert_{\partial Z_{1,0}} \;,
\end{equation}
where $a>0$ is a universal constant.
Since $\omega_1\in\Ker\big(\Doneinf\big)$,
by the Trace theorem for Sobolev spaces,
we have
\begin{equation}
\label{eq23-pf-prop-FRT}
\big\lVert\omega_1\big\rVert_{\partial Z_{1,0}} =
\mathscr{O}\big(1\big) \big\lVert\omega_1\big\rVert_{Z_{1,0}} \;.
\end{equation}
By \eqref{eq22-pf-prop-FRT} and \eqref{eq23-pf-prop-FRT},
for $-R\leqslant u\leqslant 0$,
we have
\begin{equation}
\label{eq24-pf-prop-FRT}
\big\lVert \tau \big\rVert_{\{u\}\times Y}
= \mathscr{O}\big(e^{-a(R+u)}\big) \big\lVert\omega_1\big\rVert_{Z_{1,0}} \;.
\end{equation}

By \eqref{eq-def-fT}, \eqref{eq-def-chi12} and \eqref{eq-chi-IYR},
for $-R\leqslant u\leqslant 0$,
we have
\begin{align}
\label{eq25-pf-prop-FRT}
\begin{split}
& \big|f_T\big| = \mathscr{O}\big(R^{-2}\big) (R+u)^2 \;,\hspace{5mm}
\left|\frac{\partial f_T}{\partial u}\right| = \mathscr{O}\big(R^{-2}\big) (R+u) \;,\\
& \big|\chi_1\big| \leqslant 1  \;,\hspace{5mm}
\left|\frac{\partial\chi_1}{\partial u}\right| = \mathscr{O}\big(R^{-1}\big) \;,\hspace{5mm}
\left|\frac{\partial^2\chi_1}{\partial u^2}\right| = \mathscr{O}\big(R^{-2}\big) \;.
\end{split}
\end{align}

By \eqref{eq12-pf-prop-FRT},
\eqref{eq16-pf-prop-FRT},
\eqref{eq24-pf-prop-FRT},
\eqref{eq25-pf-prop-FRT}
and the assumption $T = R^\kappa$,
we have
\begin{align}
\label{eq26-pf-prop-FRT}
\begin{split}
\Big\lVert (F_{R,T}-G_{R,T})(\omega_1,\omega_2,\hat{\omega})\Big\rVert^2_{[-R,0] \times Y} +
\Big\lVert \DRT (F_{R,T}-G_{R,T})(\omega_1,\omega_2,\hat{\omega})\Big\rVert^2_{[-R,0] \times Y} & \\
= \mathscr{O}\big(R^{-2+\kappa}\big)\big\lVert\omega_1\big\rVert^2_{Z_{1,0}} & \;.
\end{split}
\end{align}
The same argument also shows that
\eqref{eq26-pf-prop-FRT} holds with
$[-R,0] \times Y$ replaced by $[0,R] \times Y$ and
$\big\lVert\omega_1\big\rVert^2_{Z_{1,0}}$ replaced by $\big\lVert\omega_2\big\rVert^2_{Z_{2,0}}$.
On the other hand,
by \eqref{eq1-def-FRT},
we have
\begin{equation}
\label{eq27-pf-prop-FRT}
(F_{R,T}-G_{R,T})(\omega_1,\omega_2,\hat{\omega}) \Big|_{Z_{1,0}\cup Z_{2,0}} = 0 \;.
\end{equation}
By \eqref{eq26-pf-prop-FRT}
and \eqref{eq27-pf-prop-FRT},
we have
\begin{align}
\label{eq28-pf-prop-FRT}
\begin{split}
\Big\lVert (F_{R,T}-G_{R,T})(\omega_1,\omega_2,\hat{\omega})\Big\rVert^2_{Z_R}
+ \Big\lVert \DRT (F_{R,T}-G_{R,T})(\omega_1,\omega_2,\hat{\omega})\Big\rVert^2_{Z_R} & \\
= \mathscr{O} \big(R^{-2+\kappa}\big) \Big(\big\lVert\omega_1\big\rVert^2_{Z_{1,0}} + \big\lVert\omega_2\big\rVert^2_{Z_{2,0}} \Big) & \;.
\end{split}
\end{align}

The estimates in \textsection \ref{subsect-hodge} hold
with $\big(W^\bullet, \partial,\big\lVert\cdot\big\rVert\big)$
replaced by $\big(\Omega^\bullet(Z_R,F),\diffRT,\big\lVert\cdot\big\rVert_{Z_R}\big)$.
Applying
\eqref{eq3-def-FRT},
\eqref{eq28-pf-prop-FRT}
and Corollary \ref{cor-proj-estimate-better}
with $\gamma = 0$,
$w = F_{R,T}(\omega_1,\omega_2,\hat{\omega})$ and
$v = G_{R,T}(\omega_1,\omega_2,\hat{\omega})$,
we get
\begin{align}
\label{eq7-pf-prop-FRT}
\begin{split}
\Big\lVert \big( \Id - P_{R,T} \big)F_{R,T}(\omega_1,\omega_2,\hat{\omega}) \Big\rVert^2_{Z_R}
+ \Big\lVert \DRT \big( \Id - P_{R,T} \big)F_{R,T}(\omega_1,\omega_2,\hat{\omega}) \Big\rVert^2_{Z_R} & \\
= \mathscr{O}\big(R^{-2+\kappa}\big) \Big(\big\lVert\omega_1\big\rVert^2_{Z_{1,0}} + \big\lVert\omega_2\big\rVert^2_{Z_{2,0}} \Big) & \;.
\end{split}
\end{align}
From Proposition \ref{prop-sobolev} and \eqref{eq7-pf-prop-FRT},
we obtain \eqref{eq-prop-FRT}.
This completes the proof of Proposition \ref{prop-FRT}.
\end{proof}

For $j=1,2$,
let
\begin{equation}
\label{eq-def-LL}
\LL^\bullet_{j,\mathrm{abs}} \subseteq \hh(Y,F) \;,\hspace{5mm}
\LL^\bullet_{j,\mathrm{rel}} \subseteq \hh(Y,F)du \;,\hspace{5mm}
\LL^\bullet_j = \LL^\bullet_{j,\mathrm{abs}} \oplus \LL^\bullet_{j,\mathrm{rel}}
\end{equation}
be as in \cite[(2.47),(2.49)]{pzz} with $X_\infty$ replaced by $Z_{j,\infty}$. 
More precisely, 
under the identification $\hh(Y,F) = H^\bullet(Y,F)$,
we have 
\begin{equation}
\label{eq1-LL}
\LL^\bullet_{j,\mathrm{abs}} = \mathrm{Im}\big( H^\bullet(Z_j,F) \rightarrow H^\bullet(Y,F) \big) \;,
\end{equation}
where the map is induced by $Y = \partial Z_j \hookrightarrow Z_j$, 
and
\begin{equation}
\label{eq2-LL}
\LL^{\bullet+1}_{j,\mathrm{rel}} = du \LL^{\bullet,\perp}_{j,\mathrm{abs}} \;,
\end{equation}
where $\LL^{\bullet,\perp}_{j,\mathrm{abs}} \subseteq \hh(Y,F)$ is the orthogonal complement of $\LL^\bullet_{j,\mathrm{abs}}$ with respect to $\big\lVert\cdot\big\rVert_Y$.
Hence we have
\begin{equation}
\LL^{\bullet+1}_{1,\mathrm{rel}}\cap\LL^{\bullet+1}_{2,\mathrm{rel}} =
du \big( \LL^{\bullet,\perp}_{1,\mathrm{abs}}\cap\LL^{\bullet,\perp}_{2,\mathrm{abs}} \big) \;.
\end{equation}
Let $\hhr(Z_{j,\infty},F)$ be as in \cite[(2.52)]{pzz} with $X_\infty$ replaced by $Z_{j,\infty}$.
For $\hat{\omega}\in \LL^{\bullet,\perp}_{1,\mathrm{abs}}\cap\LL^{\bullet,\perp}_{2,\mathrm{abs}}$,
let
\begin{equation}
\label{eq01-def-IRT}
(\omega_1,du\wedge\hat{\omega})\in\mathscr{H}^{\bullet+1}_\mathrm{rel}(Z_{1,\infty},F) \;,\hspace{5mm}
(\omega_2,du\wedge\hat{\omega})\in\mathscr{H}^{\bullet+1}_\mathrm{rel}(Z_{2,\infty},F)
\end{equation}
be the unique element
such that $\omega_1$ (resp. $\omega_2$) is
a generalized eigensection of $\Doneinf$ (resp. $\Dtwoinf$).
The existence and uniqueness are guaranteed by \cite[(2.40)]{pzz}.

Similarly to \eqref{eq-def-chi12},
set
\begin{equation}
\label{eq-def-chi3}
\chi_3(s) = 1 - \chi\big(4|s|\big) \;.
\end{equation}
We will view $\chi_3$ as a function on $IY_R$
in the same way as $\chi_1,\chi_2$ in \eqref{eq-chi-IYR}.
We define
\begin{equation}
\label{eq0-def-IRT}
I_{R,T}, J_{R,T}: \LL^{\bullet,\perp}_{1,\mathrm{abs}}\cap\LL^{\bullet,\perp}_{2,\mathrm{abs}}
\rightarrow \Omega^{\bullet+1}(Z_R,F)
\end{equation}
as follows:
for $\hat{\omega}\in \LL^{\bullet,\perp}_{1,\mathrm{abs}}\cap\LL^{\bullet,\perp}_{2,\mathrm{abs}}$,
\begin{align}
\label{eq1-def-IRT}
\begin{split}
& I_{R,T}(\hat{\omega}) \big|_{Z_{j,0}} = 0 \;,\hspace{5mm}
J_{R,T}(\hat{\omega}) \big|_{Z_{j,0}} = e^{-T}\omega_j \;,\hspace{5mm} \text{for } j=1,2 \;,\\
& I_{R,T}(\hat{\omega}) \big|_{IY_R} = \chi_3 e^{Tf_T-T}du\wedge\hat{\omega} \;,\\
& J_{R,T}(\hat{\omega}) \big|_{IY_R} = e^{Tf_T-T}du\wedge\hat{\omega} \\
& + e^{Tf_T-T}d^{Z_R,*}\Big( \chi_1\stRes(\omega_1,du\wedge\hat{\omega})
+ \chi_2\stRes(\omega_2,du\wedge\hat{\omega}) \Big) \;.
\end{split}
\end{align}
By \eqref{eq-def-fT}, \eqref{eq-def-DRT-conj} and \eqref{eq-resol-RdFRdFstar},
$I_{R,T}$ and $J_{R,T}$ are well-defined.
Moreover,
we have
\begin{equation}
\label{eq3-def-IRT}
\diffRT I_{R,T}(\hat{\omega}) =
\stdiffRT J_{R,T}(\hat{\omega}) = 0 \;.
\end{equation}

\begin{prop}
\label{prop-IRT}
There exists $a>0$ such that
for $T = R^\kappa \gg 1$ and $\hat{\omega}\in\LL^{\bullet,\perp}_{1,\mathrm{abs}}\cap\LL^{\bullet,\perp}_{2,\mathrm{abs}}$,
we have
\begin{equation}
\label{eq-prop-IRT}
\Big\lVert \big( \Id - P_{R,T} \big)I_{R,T}(\hat{\omega}) \Big\rVert^2_{H^1,Z_R}
= \mathscr{O}\big(e^{-aT}\big) \big\lVert\hat{\omega}\big\rVert^2_Y \;.
\end{equation}
\end{prop}
\begin{proof}
We proceed in the same way as in the proof of Proposition \ref{prop-FRT}.
The map $I_{R,T}$ (resp. $J_{R,T}$) plays the role of  $F_{R,T}$ (resp. $G_{R,T}$).
\end{proof}

By
\eqref{eq-def-W12},
\eqref{eq-complex-evaluation},
\eqref{eq-commute-hha12}
and \eqref{eq-def-hha12},
we have
\begin{equation}
\label{eq-iso-hha12}
W^\bullet_{12} \simeq \hha(Z_{12,\infty},F) \;.
\end{equation}
By
\eqref{eq-def-V1-V2},
\eqref{eq-complex-evaluation},
\eqref{eq-commute-hha12} 
and \eqref{eq1-LL}, 
we have
\begin{equation}
\label{eq-LLVj}
V^\bullet_j \simeq \LL^\bullet_{j,\mathrm{abs}} \;,\hspace{5mm} \text{for } j=1,2 \;.
\end{equation}
As a consequence,
we have
\begin{equation}
\label{eq-iso-llr12}
V^\bullet/(V^\bullet_1+V^\bullet_2) \simeq
\LL^{\bullet,\perp}_{1,\mathrm{abs}}\cap\LL^{\bullet,\perp}_{2,\mathrm{abs}} \;.
\end{equation}
Recall that the complex $(C^{\bullet,\bullet}_0,\partial)$ was defined by
\eqref{eq1-def-complex} and \eqref{eq-complex-evaluation}.
By \eqref{eq-HC},
\eqref{eq-iso-hha12} and
\eqref{eq-iso-llr12},
we have
\begin{equation}
\label{eq-id-HCHmathscr}
H^0(C^{\bullet,\bullet}_0,\partial)
\simeq \hha(Z_{12,\infty},F) \;,\hspace{5mm}
H^1(C^{\bullet,\bullet}_0,\partial)
\simeq \LL^{\bullet,\perp}_{1,\mathrm{abs}}\cap\LL^{\bullet,\perp}_{2,\mathrm{abs}} \;.
\end{equation}
We define a map
\begin{align}
\label{eq2-def-SHRT}
\begin{split}
& \mathscr{S}^H_{R,T}: H^\bullet(C^{\bullet,\bullet}_0,\partial) \rightarrow \Ker\big(\DRT\big) \;,\\
& \mathscr{S}^H_{R,T}\Big|_{H^0(C^{\bullet,\bullet}_0,\partial)} = P_{R,T}F_{R,T} \;,\hspace{5mm}
\mathscr{S}^H_{R,T}\Big|_{H^1(C^{\bullet,\bullet}_0,\partial)} = P_{R,T}I_{R,T} \;.
\end{split}
\end{align}

\begin{thm}
\label{thm-SRTH-bij}
For $T = R^\kappa \gg 1$,
the map $\mathscr{S}^H_{R,T}$ is bijective.
\end{thm}
\begin{proof}
By \eqref{eq0-def-FRT}, \eqref{eq1-def-FRT}, \eqref{eq0-def-IRT}, \eqref{eq1-def-IRT}
and the fact that $\chi_1\chi_3 = \chi_2\chi_3 = 0$,
for $(\omega_1,\omega_2,\hat{\omega})\in\hha(Z_{12,\infty},F)$
and $\hat{\tau}\in \LL^{\bullet,\perp}_{1,\mathrm{abs}}\cap\LL^{\bullet,\perp}_{2,\mathrm{abs}}$,
we have
\begin{align}
\label{eq1-pf-thm-SRTH-bij}
\begin{split}
& \Big\Vert F_{R,T}(\omega_1,\omega_2,\hat{\omega}) \Big\rVert^2_{Z_R}
\geqslant \big\lVert \omega_1 \big\rVert^2_{Z_{1,0}} + \big\lVert \omega_2 \big\rVert^2_{Z_{2,0}} \;,\hspace{5mm}
\Big\Vert I_{R,T}(\hat{\tau}) \Big\rVert^2_{Z_R}
\geqslant \big\lVert \hat{\tau} \big\rVert^2_Y \;,\\
& \Big\langle F_{R,T}(\omega_1,\omega_2,\hat{\omega}),I_{R,T}(\hat{\tau}) \Big\rangle_{Z_R} = 0 \;.
\end{split}
\end{align}
By Propositions \ref{prop-FRT}, \ref{prop-IRT}
and \eqref{eq1-pf-thm-SRTH-bij},
we have
\begin{equation}
\label{eq2-pf-thm-SRTH-bij}
\Big\lVert P_{R,T}F_{R,T}(\omega_1,\omega_2,\hat{\omega}) +
P_{R,T}I_{R,T}(\hat{\tau}) \Big\rVert^2_{Z_R}
\geqslant \frac{1}{2} \Big(
\big\lVert \omega_1 \big\rVert^2_{Z_{1,0}}
+ \big\lVert \omega_2 \big\rVert^2_{Z_{2,0}}
+ \big\lVert \hat{\tau} \big\rVert^2_Y \Big) \;.
\end{equation}
By \eqref{eq-id-HCHmathscr}-\eqref{eq2-def-SHRT}
and \eqref{eq2-pf-thm-SRTH-bij},
the map $\mathscr{S}^H_{R,T}$ is injective.
On the other hand,
by the exactness of \eqref{eq-mv-top},
the construction of $(C^{\bullet,\bullet}_0,\partial)$
and the Hodge theorem,
we have
\begin{equation}
\dim H^\bullet(C^{\bullet,\bullet}_0,\partial) =
\dim H^\bullet(Z,F) =
\dim \Ker\big(\DRT\big) \;.
\end{equation}
Hence the map $\mathscr{S}^H_{R,T}$ is bijective.
This completes the proof of Theorem \ref{thm-SRTH-bij}.
\end{proof}

For $\omega\in\Omega^\bullet(Z_R,F)$ satisfying $\diffRT\omega = 0$,
i.e., $d^{Z_R} \big( e^{Tf_T} \omega \big) = 0$,
we denote
\begin{equation}
\label{eq-def-T-class}
\big[\omega\big]_T = \big[e^{Tf_T}\omega\big]
\in H^\bullet(Z_R,F) = H^\bullet(Z,F) \;.
\end{equation}

\begin{cor}
\label{cor-SRTH-bij}
For $T = R^\kappa \gg 1$,
the map
\begin{align}
\label{eq-cor-SRTH-bij}
\begin{split}
\big[\mathscr{S}^H_{R,T}\big]_T : \; H^\bullet(C^{\bullet,\bullet}_0,\partial) & \rightarrow H^\bullet(Z,F) \\
\sigma & \mapsto \big[\mathscr{S}^H_{R,T}(\sigma)\big]_T
\end{split}
\end{align}
is bijective.
\end{cor}
\begin{proof}
This is a direct consequence of the Hodge theorem and Theorem \ref{thm-SRTH-bij}.
\end{proof}

\begin{rem}
\label{rem-bij-compatible}
By Corollary \ref{cor-SRTH-bij} and \eqref{eq-id-HCHmathscr},
we have a bijection
\begin{equation}
\label{eq1-rem-bij-compatible}
\hha(Z_{12,\infty},F) \oplus \Big( \LL^{\bullet,\perp}_{1,\mathrm{abs}}\cap\LL^{\bullet,\perp}_{2,\mathrm{abs}} \Big)
\xrightarrow{\sim} H^\bullet(Z,F) \;.
\end{equation}
Set
\begin{equation}
\label{eq2-rem-bij-compatible}
\hh(Z_{12,\infty},F) =
\Big\{(\omega_1,\omega_2,\hat{\omega}) \;:\:
(\omega_j,\hat{\omega})\in \hh(Z_{j,\infty},F) \hspace{5mm}\text{for }j=1,2 \Big\} \;.
\end{equation}
By \cite[Thm. 3.7]{pzz} and the Hodge theorem,
we have a bijection
\begin{equation}
\label{eq3-rem-bij-compatible}
\hh(Z_{12,\infty},F) \xrightarrow{\sim} H^\bullet(Z,F) \;.
\end{equation}
The vector spaces in \eqref{eq1-rem-bij-compatible} and \eqref{eq3-rem-bij-compatible} are linked by the short exact sequence
\begin{equation}
0 \rightarrow
\hha(Z_{12,\infty},F) \rightarrow
\hh(Z_{12,\infty},F) \rightarrow
\LL^{\bullet,\perp}_{1,\mathrm{abs}}\cap\LL^{\bullet,\perp}_{2,\mathrm{abs}}
\rightarrow 0 \;,
\end{equation}
which follows from \eqref{eq-def-hha12}, \eqref{eq2-rem-bij-compatible} and \cite[(2.49)-(2.53)]{pzz}.
\end{rem}

\subsection{Eigenspace of $\DRT$ associated with small eigenvalues}
\label{subsec-approx-small-eigen}

For $\omega\in\Omega^\bullet(IY_R,F) = \smooth([-R,R],\hh(Y,F)[du])$,
we denote
\begin{equation}
\big\lVert\omega\big\rVert_{Y,\mathrm{max}} =
\max\Big\{\big\lVert\omega_u\big\rVert_Y\;:\;u\in[-R,R]\Big\} \;.
\end{equation}

\begin{lemme}
\label{prop-zm-minoration}
For $T = R^\kappa \gg 1$ and $\omega\in\Omega^\bullet(Z_R,F)$ an eigensection of $R\DRT$
associated with eigenvalue $\lambda\in\big[-\sqrt{R},\sqrt{R}\big]\backslash\{0\}$,
we have
\begin{equation}
\big\lVert \omega \big\rVert_{Z_{1,0} \cup Z_{2,0}}
= \mathscr{O}\big(1\big) \big\lVert\omega^\mathrm{zm}\big\rVert_{Y,\mathrm{max}} \;.
\end{equation}
\end{lemme}
\begin{proof}
We will follow \cite[Lemma 3.10]{pzz}.
Suppose,
on the contrary,
that there exist
$R_i\rightarrow\infty$,
$T_i = R_i^\kappa$,
$\lambda_i\in[-\sqrt{R_i},\sqrt{R_i}]\backslash\{0\}$ and
$\omega_i\in\Omega^\bullet(Z_{R_i},F)$ such that
\begin{equation}
\label{eq1-pf-prop-zm-minoration}
R_iD^{Z_{R_i}}_{T_i}\omega_i = \lambda_i\omega_i \;,\hspace{5mm}
\big\lVert \omega_i \big\rVert_{Z_{1,0} \cup Z_{2,0}} = 1 \;,\hspace{5mm}
\lim_{i\rightarrow\infty} \big\lVert\omega^\mathrm{zm}_i\big\rVert_{Y,\mathrm{max}} = 0 \;.
\end{equation}
Without loss of generality,
we may assume that
\begin{equation}
\label{eq2a-pf-prop-zm-minoration}
\liminf_{i\rightarrow\infty} \big\lVert \omega_i \big\rVert_{Z_{1,0}} > 0 \;.
\end{equation}

\noindent \textit{Step 1}.
We extract a convergent subsequence of $\big(\omega_i\big)_i$.

By the Trace theorem for Sobolev spaces and \eqref{eq1-pf-prop-zm-minoration},
we have
\begin{equation}
\label{eq2b-pf-prop-zm-minoration}
\big\lVert \omega_i \big\rVert_{\partial Z_{1,0}} = \mathscr{O}\big(1\big) \;,\hspace{5mm}
\big\lVert \omega_i \big\rVert_{\partial Z_{2,0}} = \mathscr{O}\big(1\big) \;.
\end{equation}
By Lemma \ref{lem-cylinder-nz-estimation}, \eqref{eq-def-zmnz} and \eqref{eq2b-pf-prop-zm-minoration},
we have
\begin{equation}
\label{eqb0-pf-prop-zm-minoration}
\big\lVert \omega_i^\mathrm{nz} \big\rVert_{Y,\mathrm{max}}
= \mathscr{O}\big(1\big) \;.
\end{equation}
By \eqref{eq1-pf-prop-zm-minoration} and \eqref{eqb0-pf-prop-zm-minoration},
we have
\begin{equation}
\label{eqb1-pf-prop-zm-minoration}
\big\lVert \omega_i \big\rVert_{Y,\mathrm{max}}
\leqslant \big\lVert \omega_i^\mathrm{zm} \big\rVert_{Y,\mathrm{max}} + \big\lVert \omega_i^\mathrm{nz} \big\rVert_{Y,\mathrm{max}}
= \mathscr{O}\big(1\big) \;.
\end{equation}

For $r\in\N$ and $R\geqslant r$,
let $IY_r \subseteq Z_{1,r}\subseteq Z_{1,\infty}$ and $Z_{1,r}\subseteq Z_R$ be the canonical isometric embeddings.
By \eqref{eq-def-fTder} and \eqref{eq-def-DRT-cylinder},
we have
\begin{equation}
\label{eq3-pf-prop-zm-minoration}
R\DRT\big|_{Z_{1,0}} = R \Doneinf\big|_{Z_{1,0}} \;,\hspace{5mm}
R\DRT\big|_{IY_r} = R \Doneinf\big|_{IY_r} + Tf'_T \hat{c}\big|_{IY_r} \;.
\end{equation}
By \eqref{eq1-pf-prop-zm-minoration}
and \eqref{eq3-pf-prop-zm-minoration},
we have
\begin{equation}
\label{eq4-pf-prop-zm-minoration}
\Doneinf \omega_i\big|_{Z_{1,0}}
= \lambda_iR_i^{-1}\omega_i\big|_{Z_{1,0}} \;,\hspace{5mm}
\Doneinf \omega_i\big|_{IY_r}
= \Big(\lambda_iR_i^{-1} -
R_i^{-1}T_if_{T_i}'\hat{c}\Big)\omega_i\big|_{IY_r} \;.
\end{equation}
Since $\lambda_iR_i^{-1}\rightarrow 0$
and $R_i^{-1}T_i\rightarrow 0$,
by the second identity in \eqref{eq25-pf-prop-FRT},
\eqref{eq1-pf-prop-zm-minoration},
\eqref{eqb1-pf-prop-zm-minoration}
and \eqref{eq4-pf-prop-zm-minoration},
the series $\big(\omega_i\big|_{Z_{1,r}}\big)_i$
is $H^1$-bounded.
Using Rellich's lemma,
by extracting a subsequence,
we may suppose that
$\omega_i\big|_{Z_{1,r}}$ is $L^2$-convergent.
Applying \eqref{eq4-pf-prop-zm-minoration} once again,
we see that $\big(\omega_i\big|_{Z_{1,r}}\big)_i$ is $H^1$-Cauchy.
Let $\omega_{\infty,r}$ be the limit of $\big(\omega_i\big|_{Z_{1,r}}\big)_i$,
which is at least a $H^1$-current on $Z_{1,r}$ with values in $F$.
Taking the limit of \eqref{eq4-pf-prop-zm-minoration},
we get
\begin{equation}
\label{eq6-pf-prop-zm-minoration}
\Doneinf\omega_{\infty,r}\big|_{Z_{1,r}} = 0 \;.
\end{equation}
Since $\Doneinf$ is elliptic,
equation \eqref{eq6-pf-prop-zm-minoration} implies
$\omega_{\infty,r}\in\Omega^\bullet(Z_{1,r},F)$.

The standard diagonal argument allows us to extract a subsequence $\big(\omega_{i_j}\big)_j$ of $\big(\omega_i\big)_i$
such that for any $r\in\N$,
$\big(\omega_{i_j}\big|_{Z_{1,r}}\big)_j$ converges to $\omega_{r,\infty}$ in $H^1$-norm.
Now we replace $\big(\omega_i\big)_i$ by $\big(\omega_{i_j}\big)_j$.
There exists $\omega_{\infty}\in\Omega^\bullet(Z_{1,\infty},F)$ such that
for any $r\in\N$,
\begin{equation}
\label{eq7-pf-prop-zm-minoration}
\omega_i\big|_{Z_{1,r}} \rightarrow \omega_\infty\big|_{Z_{1,r}}
\hspace{5mm}\text{in } H^1\text{-norm} \;.
\end{equation}
Since \eqref{eq6-pf-prop-zm-minoration} holds for all $r\in\N$,
we have
\begin{equation}
\label{eq8-pf-prop-zm-minoration}
\Doneinf\omega_\infty = 0 \;.
\end{equation}

\noindent \textit{Step 2}.
We show that $\omega_\infty$ is $L^2$-integrable.

By the Trace theorem for Sobolev spaces,
\eqref{eqb1-pf-prop-zm-minoration}
and \eqref{eq7-pf-prop-zm-minoration},
we have
\begin{equation}
\label{eq9-pf-prop-zm-minoration}
\big\lVert \omega_\infty \big\rVert_{Y,\mathrm{max}}
< +\infty \;.
\end{equation}
By \eqref{eq1-pf-prop-zm-minoration} and \eqref{eq7-pf-prop-zm-minoration},
we have
\begin{equation}
\label{eq9a-pf-prop-zm-minoration}
\omega_\infty^\mathrm{zm} = 0 \;.
\end{equation}
By \cite[(2.10)]{pzz} and \eqref{eq8-pf-prop-zm-minoration}-\eqref{eq9a-pf-prop-zm-minoration},
there exists $a>0$ such that for $r>0$,
we have
\begin{equation}
\label{eq16-pf-prop-zm-minoration}
\big\lVert \omega_\infty \big\rVert_{\partial Z_{1,r}}
= \mathscr{O}\big(e^{-ar}\big) \;.
\end{equation}
In particular,
$\omega_\infty$ is $L^2$-integrable.

\noindent \textit{Step 3}.
We look for a contradiction.

By \eqref{eq-hh}, \eqref{eq-def-hha12}, \eqref{eq8-pf-prop-zm-minoration} and \eqref{eq16-pf-prop-zm-minoration},
we have
\begin{equation}
\label{eq10-pf-prop-zm-minoration}
(\omega_{\infty},0,0)\in\hha(Z_{12,\infty},F) \;.
\end{equation}
Recall that
$P_{R,T}F_{R,T}: \hha(Z_{12,\infty},F) \rightarrow \Ker\big(\DRT\big)$
was constructed in \textsection \ref{subsec-approx-kernel}.
Set
\begin{equation}
\label{eq11-pf-prop-zm-minoration}
\mu_i = P_{R_i,T_i}F_{R_i,T_i}(\omega_\infty,0,0)
\in \Ker\big(D^{Z_{R_i}}_{T_i}\big) \subseteq \Omega^\bullet(Z_{R_i},F) \;.
\end{equation}

We decompose $Z_{R_i}$ into two pieces $Z_{R_i} = Z_{1,R_i}\cup Z_{2,0}$.
We have
\begin{equation}
\label{eq12-pf-prop-zm-minoration}
\big\langle\mu_i,\omega_i\big\rangle_{Z_{R_i}} -
\big\langle\omega_\infty,\omega_i\big\rangle_{Z_{1,R_i}}
= \big\langle\mu_i-\omega_\infty,\omega_i\big\rangle_{Z_{1,R_i}}
+ \big\langle\mu_i,\omega_i\big\rangle_{Z_{2,0}} \;.
\end{equation}
By \eqref{eq1-pf-prop-zm-minoration}
and \eqref{eqb1-pf-prop-zm-minoration},
we have
\begin{equation}
\label{eq13-pf-prop-zm-minoration}
\big\lVert\omega_i\big\rVert^2_{Z_{1,R_i}} +
\big\lVert\omega_i\big\rVert^2_{Z_{2,0}} =
\big\lVert\omega_i\big\rVert^2_{Z_{R_i}} = \mathscr{O}\big(R_i\big) \;.
\end{equation}
By Proposition \ref{prop-FRT}
and \eqref{eq11-pf-prop-zm-minoration},
we have
\begin{equation}
\label{eq14a-pf-prop-zm-minoration}
\big\lVert \mu_i - F_{R_i,T_i}(\omega_\infty,0,0) \big\rVert^2_{Z_{R_i}}
= \mathscr{O}\big(R_i^{-2+\kappa}\big) \;.
\end{equation}
By \eqref{eq-resol-RdFRdFstar}, \eqref{eq-def-chi12}, \eqref{eq-chi-IYR}, \eqref{eq1-def-FRT}, \eqref{eq24-pf-prop-FRT} and \eqref{eq25-pf-prop-FRT},
we have
\begin{equation}
\label{eq14b-pf-prop-zm-minoration}
\big\lVert F_{R_i,T_i}(\omega_\infty,0,0) - \omega_\infty \big\rVert^2_{Z_{1,R_i}}
= \mathscr{O}\big(R_i^{-2+\kappa}\big) \;,\hspace{5mm}
\big\lVert F_{R_i,T_i}(\omega_\infty,0,0) \big\rVert^2_{Z_{2,0}} = 0 \;.
\end{equation}
By \eqref{eq14a-pf-prop-zm-minoration} and \eqref{eq14b-pf-prop-zm-minoration},
we have
\begin{equation}
\label{eq14-pf-prop-zm-minoration}
\big\lVert\mu_i-\omega_\infty\big\rVert^2_{Z_{1,R_i}}
= \mathscr{O}\big(R_i^{-2+\kappa}\big) \;,\hspace{5mm}
\big\lVert\mu_i\big\rVert^2_{Z_{2,0}}
= \mathscr{O}\big(R_i^{-2+\kappa}\big) \;.
\end{equation}
By \eqref{eq12-pf-prop-zm-minoration}, \eqref{eq13-pf-prop-zm-minoration} and \eqref{eq14-pf-prop-zm-minoration},
we have
\begin{equation}
\label{eq15-pf-prop-zm-minoration}
\big\langle\mu_i,\omega_i\big\rangle_{Z_{R_i}} -
\big\langle\omega_\infty,\omega_i\big\rangle_{Z_{1,R_i}}
= \mathscr{O}\big(R_i^{-1/2+\kappa/2}\big) \;.
\end{equation}

By the dominated convergence theorem,
\eqref{eqb1-pf-prop-zm-minoration},
\eqref{eq7-pf-prop-zm-minoration} and
\eqref{eq16-pf-prop-zm-minoration},
we have
\begin{equation}
\label{eq17-pf-prop-zm-minoration}
\lim_{i\rightarrow +\infty}\big\langle\omega_\infty,\omega_i\big\rangle_{Z_{1,R_i}} =
\big\lVert\omega_\infty\big\rVert_{Z_{1,\infty}}^2 \;.
\end{equation}
Since $\kappa\in]0,1/3[$,
by \eqref{eq15-pf-prop-zm-minoration}
and \eqref{eq17-pf-prop-zm-minoration},
we have
\begin{equation}
\lim_{i\rightarrow +\infty} \big\langle\mu_i,\omega_i\big\rangle_{Z_{R_i}} =
\big\lVert\omega_\infty\big\rVert_{Z_{1,\infty}}^2 \;.
\end{equation}
By \eqref{eq2a-pf-prop-zm-minoration}
and \eqref{eq7-pf-prop-zm-minoration},
we have $\big\lVert\omega_\infty\big\rVert_{Z_{1,\infty}}^2>0$.
Thus $\big\langle\mu_i,\omega_i\big\rangle_{Z_{R_i}}\neq 0$ for $i$ large enough.
But,
by \eqref{eq1-pf-prop-zm-minoration},
\eqref{eq11-pf-prop-zm-minoration}
and the assumption $\lambda_i\neq 0$,
we have $\big\langle\mu_i,\omega_i\big\rangle_{Z_{R_i}}=0$.
Contradiction.
This completes the proof of Lemma \ref{prop-zm-minoration}.
\end{proof}

Recall that the operator $c$ was defined in \eqref{eq-def-cchat}.
For $\sigma\in\hh(Y,F)[du]$,
we denote $\sigma = \sigma^+ + \sigma^-$
such that $c \sigma^\pm = \mp i \sigma^{\pm}$.

For $j=1,2$,
let $C_j(\lambda)\in\mathrm{End}\big(\hh(Y,F)[du]\big)$
be the scattering matrix as in \cite[(2.32)]{pzz} with $X_\infty$ replaced by $Z_{j,\infty}$.
By \cite[(2.35)]{pzz},
we have
\begin{equation}
\label{eq-anticomm-cc}
c \hspace{0.5mm} C_j(\lambda) = - C_j(\lambda) c \;.
\end{equation}

\begin{lemme}
\label{prop-zm-sca-estimate}
For $T = R^\kappa \gg 1$ and $\omega\in\Omega^\bullet(Z_R,F)$
an eigensection of $R\DRT$ associated with eigenvalue $\lambda\in\big[-\sqrt{R},\sqrt{R}\big]\backslash\{0\}$,
we have
\begin{equation}
\label{eq-prop-zm-sca-estimate}
\Big\lVert
\omega^{\mathrm{zm},+} - C_j\big(\lambda/R\big)\omega^{\mathrm{zm},-}
\Big\rVert^2_{\partial Z_{j,0}}
= \mathscr{O}\big(R^{-2+\kappa}\big) \big\lVert \omega \big\rVert^2_{Z_{1,0}\cup Z_{2,0}}
\;,\; \text{ for } j=1,2 \;.
\end{equation}
\end{lemme}
\begin{proof}
We will follow \cite[Lemma 3.12]{pzz}.
We only prove \eqref{eq-prop-zm-sca-estimate} for $j=1$.

Let ${\omega}'\in\Omega^\bullet(Z_{1,\infty},F)$
be the unique generalized eigensection of $\Doneinf$ associated with eigenvalue $\lambda/R$
satisfying
\begin{equation}
\label{eq1-pf-prop-zm-sca-estimate}
{{\omega}'}^{\mathrm{zm},-}\big|_{\partial Z_{1,0}} = \omega^{\mathrm{zm},-} \big|_{\partial Z_{1,0}} \in \hh(Y,F)[du] \;.
\end{equation}
The existence and uniqueness of ${\omega}'$ are guaranteed by \cite[Prop. 2.4]{pzz}.
Moreover,
by \cite[Prop. 2.4]{pzz} and \eqref{eq1-pf-prop-zm-sca-estimate},
we have
\begin{equation}
\label{eq2-pf-prop-zm-sca-estimate}
{{\omega}'}^{\mathrm{zm},+}\big|_{\partial Z_{1,0}}
= C_1 \big( \lambda/R \big) {{\omega}'}^{\mathrm{zm},-}\big|_{\partial Z_{1,0}}
= C_1 \big( \lambda/R \big) \omega^{\mathrm{zm},-}\big|_{\partial Z_{1,0}} \;.
\end{equation}
By the theory of ordinary differential equation,
there exists ${\omega}''\in\Omega^\bullet(Z_{1,R},F)$ satisfying
\begin{equation}
\label{eq3-pf-prop-zm-sca-estimate}
{\omega}''\big|_{Z_{1,0}} = {\omega}'\big|_{Z_{1,0}} \;,\hspace{5mm}
{{\omega}''}^\mathrm{nz} = {{\omega}'}^\mathrm{nz} \;,\hspace{5mm}
R\DRT\big|_{IY_R} {{\omega}''}^\mathrm{zm} = \lambda{{\omega}''}^\mathrm{zm} \;.
\end{equation}
Set
\begin{equation}
\label{eq4-pf-prop-zm-sca-estimate}
\mu = \omega\big|_{Z_{1,R}} - {\omega}''\big|_{Z_{1,R}}  \;.
\end{equation}
By \eqref{eq1-pf-prop-zm-sca-estimate}-\eqref{eq4-pf-prop-zm-sca-estimate}, we have
\begin{equation}
\label{eq5-pf-prop-zm-sca-estimate}
\mu^\mathrm{zm} \big|_{\partial Z_{1,0}} =
\left( \omega^{\mathrm{zm},+ } - C_1\big(\lambda/R\big) \omega^{\mathrm{zm},-} \right) \big|_{\partial Z_{1,0}} \;.
\end{equation}
By the construction of $\mu$,
we have $R\DRT\big|_{IY_R} \mu^\mathrm{zm} = \lambda\mu^\mathrm{zm}$.
By Lemma \ref{lem-cylinder-zm-estimation},
\eqref{eq-anticomm-cc} and \eqref{eq5-pf-prop-zm-sca-estimate},
we have
\begin{align}
\label{eq6-pf-prop-zm-sca-estimate}
\begin{split}
\big\langle c\mu^\mathrm{zm},\mu^\mathrm{zm} \big\rangle_{\partial Z_{1,R/2}} & =
\big\langle c\mu^\mathrm{zm},\mu^\mathrm{zm} \big\rangle_{\partial Z_{1,0}} \\ & =
- i \big\lVert \mu^\mathrm{zm} \big\rVert^2_{\partial Z_{1,0}} =
- i \Big\lVert \omega^{\mathrm{zm},+ } - C_1\big(\lambda/R\big) \omega^{\mathrm{zm},-} \Big\rVert^2_{\partial Z_{1,0}} \;.
\end{split}
\end{align}

By the Trace theorem for Sobolev spaces and Proposition \ref{prop-sobolev},
we have
\begin{equation}
\label{eq7a-pf-prop-zm-sca-estimate}
\big\lVert \omega \big\rVert_{\partial Z_{1,0}}
= \mathscr{O}\big(1\big)
\big\lVert \omega \big\rVert_{Z_{1,0}} \;,\hspace{5mm}
\big\lVert \omega \big\rVert_{\partial Z_{2,0}}
= \mathscr{O}\big(1\big)
\big\lVert \omega \big\rVert_{Z_{2,0}} \;.
\end{equation}
Applying Lemma \ref{lem-cylinder-nz-estimation} to $\omega^\mathrm{nz}$,
there exists a universal constant $a>0$ such that
\begin{equation}
\label{eq7b-pf-prop-zm-sca-estimate}
\big\lVert\omega^\mathrm{nz}\big\rVert_{\partial Z_{1,r}}
= \mathscr{O}\big(e^{-ar}\big)
\Big( \big\lVert \omega^\mathrm{nz} \big\rVert_{\partial Z_{1,0}} +
      \big\lVert \omega^\mathrm{nz} \big\rVert_{\partial Z_{2,0}} \Big) \;,
\hspace{5mm} \text{for } 0\leqslant r \leqslant R/2 \;.
\end{equation}
By \cite[Prop. 2.4]{pzz},
we have $\big\lVert {\omega'}^\mathrm{nz}\big\rVert_{Z_{1,\infty}\backslash Z_{1,0}} < +\infty$.
Moreover,
by \cite[(2.10),(2.38)]{pzz},
there exists a universal constant $a>0$ such that
\begin{equation}
\label{eq7c-pf-prop-zm-sca-estimate}
\big\lVert {\omega'}^\mathrm{nz} \big\rVert_{\partial Z_{1,r}}
= \mathscr{O}\big(e^{-ar}\big)
\big\lVert {\omega'}^\mathrm{nz} \big\rVert_{Z_{1,\infty}\backslash Z_{1,0}}
= \mathscr{O}\big(e^{-ar}\big)
\big\lVert {\omega'}^\mathrm{zm} \big\rVert_{\partial Z_{1,0}} \;,
\hspace{5mm} \text{for } r \geqslant 0 \;.
\end{equation}
Since $C_1\big(\lambda/R\big)$ is unitary,
\eqref{eq1-pf-prop-zm-sca-estimate}
and \eqref{eq2-pf-prop-zm-sca-estimate}
imply
\begin{equation}
\label{eq7d-pf-prop-zm-sca-estimate}
\big\lVert {\omega'}^\mathrm{zm} \big\rVert^2_{\partial Z_{1,0}} =
\big\lVert {\omega'}^{\mathrm{zm},-} \big\rVert^2_{\partial Z_{1,0}} +
\big\lVert {\omega'}^{\mathrm{zm},+} \big\rVert^2_{\partial Z_{1,0}} =
2 \big\lVert \omega^{\mathrm{zm},-} \big\rVert^2_{\partial Z_{1,0}} \;.
\end{equation}
By \eqref{eq7a-pf-prop-zm-sca-estimate}-\eqref{eq7d-pf-prop-zm-sca-estimate},
we have
\begin{equation}
\label{eq8a-pf-prop-zm-sca-estimate}
\big\lVert\omega^\mathrm{nz}\big\rVert_{\partial Z_{1,r}}
= \mathscr{O}\big(e^{-ar}\big)
\big\lVert \omega \big\rVert_{Z_{1,0}\cup Z_{2,0}} \;,\hspace{5mm}
\big\lVert{\omega'}^\mathrm{nz}\big\rVert_{\partial Z_{1,r}}
= \mathscr{O}\big(e^{-ar}\big)
\big\lVert \omega \big\rVert_{Z_{1,0}\cup Z_{2,0}} \;.
\end{equation}
By the second identity in \eqref{eq3-pf-prop-zm-sca-estimate},  \eqref{eq4-pf-prop-zm-sca-estimate}
and \eqref{eq8a-pf-prop-zm-sca-estimate},
we have
\begin{equation}
\label{eq8b-pf-prop-zm-sca-estimate}
\big\lVert\mu^\mathrm{nz}\big\rVert_{\partial Z_{1,r}}
= \mathscr{O}\big(e^{-ar}\big)
\big\lVert \omega \big\rVert_{Z_{1,0}\cup Z_{2,0}} \;.
\end{equation}

We identify $IY_R\subseteq Z_{1,R}$ with $[0,2R]\times Y$.
By the construction of $\mu$ and \eqref{eq3-pf-prop-zm-minoration},
we have
\begin{equation}
\label{eq9p-pf-prop-zm-sca-estimate}
R\DRT\mu\big|_{Z_{1,0}} = \lambda\mu \;,\hspace{5mm}
R\DRT\mu\big|_{[0,2R]\times Y} = \lambda\mu - Tf_T'\hat{c}{\omega'}^\mathrm{nz} \;.
\end{equation}
By \eqref{eq-def-zmnz} and \eqref{eq9p-pf-prop-zm-sca-estimate},
we have
\begin{equation}
\label{eq9-pf-prop-zm-sca-estimate}
\big\langle R\DRT\mu,\mu \big\rangle_{Z_{1,R/2}}
- \big\langle \mu,R\DRT\mu \big\rangle_{Z_{1,R/2}}
= 2i T  \mathrm{Im} \left\langle \mu^\mathrm{nz},
f_T'\hat{c}{\omega'}^\mathrm{nz}\right\rangle_{[0,R]\times Y} \;.
\end{equation}
On the other hand,
by Green's formula (cf. \cite[(2.8)]{pzz}),
we have
\begin{equation}
\label{eq10-pf-prop-zm-sca-estimate}
\big\langle R\DRT\mu,\mu \big\rangle_{Z_{1,R/2}}
- \big\langle \mu,R\DRT\mu \big\rangle_{Z_{1,R/2}}
= R\big\langle c\mu,\mu \big\rangle_{\partial Z_{1,R/2}} \;.
\end{equation}
By \eqref{eq9-pf-prop-zm-sca-estimate}, \eqref{eq10-pf-prop-zm-sca-estimate}
and the assumption $T = R^\kappa$,
we have
\begin{align}
\label{eq11-pf-prop-zm-sca-estimate}
\begin{split}
\Big| \big\langle c\mu^\mathrm{zm},\mu^\mathrm{zm} \big\rangle_{\partial Z_{1,R/2}}
+ \big\langle c\mu^\mathrm{nz},\mu^\mathrm{nz} \big\rangle_{\partial Z_{1,R/2}}\Big|
& = \Big|\big\langle c\mu,\mu \big\rangle_{\partial Z_{1,R/2}}\Big| \\
& \leqslant  2R^{-1+\kappa} \left| \left\langle \mu^\mathrm{nz},
f_T'\hat{c}{\omega'}^\mathrm{nz}\right\rangle_{[0,R]\times Y}\right| \;.
\end{split}
\end{align}

By \eqref{eq-def-finf}
and \eqref{eq-compare-f-fT},
we have
\begin{equation}
\label{eq12a-pf-prop-zm-sca-estimate}
\big|f'_T\big|_{\partial Z_{1,r}} = \mathscr{O}\big(R^{-1}\big) r \;.
\end{equation}
By \eqref{eq8a-pf-prop-zm-sca-estimate}
\eqref{eq8b-pf-prop-zm-sca-estimate} and
\eqref{eq12a-pf-prop-zm-sca-estimate},
we have
\begin{align}
\label{eq12-pf-prop-zm-sca-estimate}
\begin{split}
& \left| \left\langle \mu^\mathrm{nz},
f_T'\hat{c}{\omega'}^\mathrm{nz}\right\rangle_{[0,R]\times Y}\right| \\
& \leqslant \int_0^R
\big|f_T'\big|_{\partial Z_{1,r}}
\big\lVert\mu^\mathrm{nz}\big\rVert_{\partial Z_{1,r}}
\big\lVert{\omega'}^\mathrm{nz}\big\rVert_{\partial Z_{1,r}}
dr
= \mathscr{O}\big(R^{-1}\big)
\big\lVert \omega \big\rVert^2_{Z_{1,0}\cup Z_{2,0}} \;.
\end{split}
\end{align}
By \eqref{eq8b-pf-prop-zm-sca-estimate},
\eqref{eq11-pf-prop-zm-sca-estimate} and
\eqref{eq12-pf-prop-zm-sca-estimate},
we have
\begin{equation}
\label{eq13-pf-prop-zm-sca-estimate}
\Big| \big\langle c\mu^\mathrm{zm},\mu^\mathrm{zm} \big\rangle_{\partial Z_{1,R/2}} \Big|
= \mathscr{O}\big(R^{-2+\kappa}\big)
\big\lVert \omega \big\rVert^2_{Z_{1,0}\cup Z_{2,0}} \;.
\end{equation}
From \eqref{eq6-pf-prop-zm-sca-estimate} and
\eqref{eq13-pf-prop-zm-sca-estimate},
we obtain \eqref{eq-prop-zm-sca-estimate} with $j=1$.
This completes the proof of Lemma \ref{prop-zm-sca-estimate}.
\end{proof}

\begin{proof}[Proof of Theorem \ref{thm-central-spgap}]
First we consider the case $j=0$.

Let $\omega\in\Omega^\bullet(Z_R,F)$ be an eigensection of $R\DRT$
associated with eigenvalue $\lambda\in[-\sqrt{T},\sqrt{T}]\backslash\{0\}$.
By Lemmas \ref{prop-zm-minoration},
\ref{prop-zm-sca-estimate},
we have $\omega^\mathrm{zm}\neq 0$ and
\begin{equation}
\label{eq1-pf-thm-central-spgap}
\big\lVert \omega^{\mathrm{zm},+ }
- C_j\big(\lambda/R\big) \omega^{\mathrm{zm},-} \big\rVert_{\partial Z_{j,0}}
= \mathscr{O}\big(R^{-1+\kappa/2}\big)
\big\lVert \omega^\mathrm{zm} \big\rVert_{Y,\mathrm{max}}
\;,\;\text{ for } j=1,2 \;.
\end{equation}
Since $\lambda\mapsto C_j(\lambda)$ is analytic (cf. \cite[\textsection 4]{m} \cite[Prop. 2.3]{pzz}),
by \eqref{eq1-pf-thm-central-spgap}
and the assumption $|\lambda|\leqslant T^{1/2} = R^{\kappa/2}$,
we have
\begin{equation}
\big\lVert \omega^{\mathrm{zm},+ }
- C_j\big(0\big) \omega^{\mathrm{zm},-} \big\rVert_{\partial Z_{j,0}}
= \mathscr{O}\big(R^{-1+\kappa/2}\big)
\big\lVert \omega^\mathrm{zm} \big\rVert_{Y,\mathrm{max}}
\;,\;\text{ for } j=1,2 \;.
\end{equation}
Moreover,
as $C_j(0)$ is unitary and $\big(C_j(0)\big)^2=\mathrm{Id}$ (cf. \cite[Prop. 2.3]{pzz}),
we have
\begin{equation}
\label{eq2-pf-thm-central-spgap}
\big\lVert \omega^\mathrm{zm}
- C_j\big(0\big) \omega^\mathrm{zm} \big\rVert_{\partial Z_{j,0}}
= \mathscr{O}\big(R^{-1+\kappa/2}\big)
\big\lVert \omega^\mathrm{zm} \big\rVert_{Y,\mathrm{max}}
\;,\;\text{ for } j=1,2 \;.
\end{equation}

By \cite[(2.48)]{pzz} and \eqref{eq-def-LL},
we have
\begin{equation}
\label{eq-LL-Ker}
\LL^\bullet_j = \Ker\big(\mathrm{Id}-C_j(0)\big) \;.
\end{equation}
For $j=1,2$,
let
\begin{equation}
\label{eq-def-hhY-P1P2}
P_j: \hh(Y,F)[du] \rightarrow \LL^\bullet_j
\end{equation}
be orthogonal projections with respect to $\big\lVert\cdot\big\rVert_Y$.
We denote
\begin{equation}
\label{eq-def-hhY-P1P2perp}
P_j^\perp = \mathrm{Id} - P_j \;.
\end{equation}
By \eqref{eq-LL-Ker},
the estimate \eqref{eq2-pf-thm-central-spgap} is equivalent to the follows,
\begin{equation}
\label{eq3-pf-thm-central-spgap}
\left\lVert P_j^\perp\omega^\mathrm{zm} \right\rVert_{\partial Z_{j,0}}
= \mathscr{O}\big(R^{-1+\kappa/2}\big)
\big\lVert \omega^\mathrm{zm} \big\rVert_{Y,\mathrm{max}}
\;,\;\text{ for } j=1,2 \;.
\end{equation}

Let $D^{\hh(Y,F)}_{T,\bd}$ be the operator $\DVTbd$ in \eqref{eq-def-dVT} with
\begin{equation}
V = \hh(Y,F) \;,\hspace{5mm}
V_j = \LL^\bullet_{j,\mathrm{abs}} \;,\hspace{5mm} \text{for } j=1,2 \;.
\end{equation}
Applying Proposition \ref{prop-Df-eigen-approx}
to \eqref{eq3-pf-thm-central-spgap}
with $\epsilon = R^{-1+3\kappa}$
and using the assumption $T = R^\kappa$,
we get
\begin{equation}
\label{eq4-pf-thm-central-spgap}
\big\lVert
\omega^\mathrm{zm}
- P^{[\lambda-\epsilon,\lambda+\epsilon]}_T\omega^\mathrm{zm}
\big\rVert_{IY_R}
= \mathscr{O}\big(R^{-\kappa/2}\big) \big\lVert \omega^\mathrm{zm} \big\rVert_{IY_R} \;.
\end{equation}
By \eqref{eq4-pf-thm-central-spgap},
for $T = R^\kappa \gg 1$,
we have $P^{[\lambda-\epsilon,\lambda+\epsilon]}_T\omega^\mathrm{zm}\neq 0$.
As a consequence,
\begin{equation}
\label{eq5-pf-thm-central-spgap}
[\lambda-\epsilon,\lambda+\epsilon] \cap \Sp\big(D^{\hh(Y,F)}_{T,\bd}\big) \neq \emptyset \;.
\end{equation}
From Theorem \ref{thm-witten-spec}
and \eqref{eq5-pf-thm-central-spgap},
we obtain \eqref{eq-thm-central-spgap} with $j=0$.

We turn to the cases $j=1,2,3$.
Proceeding in the same way as in
\cite[\textsection 3.5]{pzz},
we may replace $Z_{j,R}$ by its 'double',
which is a compact manifold without boundary.
Then we apply \eqref{eq-thm-central-spgap} with $j=0$.
This completes the proof of Theorem \ref{thm-central-spgap}.
\end{proof}

For convenience, we denote
\begin{equation}
\label{eq-def-hhsqcup}
\hha(Z_{1,\infty} \sqcup Z_{2,\infty},F) =
\hha(Z_{1,\infty},F) \oplus \hha(Z_{2,\infty},F) \;.
\end{equation}
Similarly to the constructions of $F_{R,T}$ and $G_{R,T}$
in \textsection \ref{subsec-approx-kernel},
we define
\begin{equation}
\label{eq0-def-FRTplus}
F^+_{R,T}, G^+_{R,T}: \hha(Z_{1,\infty} \sqcup Z_{2,\infty},F) \rightarrow \Omega^\bullet(Z_R,F)
\end{equation}
as follows:
for $(\omega_1,\hat{\omega}_1,\omega_2,\hat{\omega}_2)\in\hha(Z_{1,\infty} \sqcup Z_{2,\infty},F)$,
\begin{align}
\label{eq1-def-FRTplus}
\begin{split}
F^+_{R,T}(\omega_1,\hat{\omega}_1,\omega_2,\hat{\omega}_2) \big|_{Z_{j,0}} & =
G^+_{R,T}(\omega_1,\hat{\omega}_1,\omega_2,\hat{\omega}_2) \big|_{Z_{j,0}} = \omega_j \;,\hspace{5mm} \text{for } j=1,2 \;,\\
F^+_{R,T}(\omega_1,\hat{\omega}_1,\omega_2,\hat{\omega}_2) \big|_{IY_R}
& = e^{-Tf_T} \Big( \chi_1\hat{\omega}_1 + \chi_2\hat{\omega}_2 \Big) \\
& \hspace{5mm} + e^{-Tf_T}d^{Z_R}\Big( \chi_1\Res(\omega_1,\hat{\omega}_1)
+ \chi_2\Res(\omega_2,\hat{\omega}_2) \Big) \;,\\
G^+_{R,T}(\omega_1,\hat{\omega}_1,\omega_2,\hat{\omega}_2) \big|_{IY_R}
& = e^{-Tf_T} \Big( \chi_1\hat{\omega}_1 + \chi_2\hat{\omega}_2 \Big) \\
& \hspace{5mm}  + e^{Tf_T}d^{Z_R,*} \Big( \chi_1\stRes(\omega_1,\hat{\omega}_1)
+ \chi_2\stRes(\omega_2,\hat{\omega}_2) \Big) \;.
\end{split}
\end{align}
By \cite[(2.52)]{pzz}, \eqref{eq-hh} and \eqref{eq-def-hhsqcup},
we have $d^{Y,*} \hat{\omega}_j = i_{\frac{\partial}{\partial u}} \hat{\omega}_j = 0$ for $j=1,2$.
Then,
by \eqref{eq-def-DRT-conj},
we have
\begin{equation}
\label{eq3a-def-FRTplus}
\stdiffRT e^{-Tf_T} \Big( \chi_1\hat{\omega}_1 + \chi_2\hat{\omega}_2 \Big)
=  e^{Tf_T} \Big( d^{Y,*} - i_{\frac{\partial}{\partial u}} \frac{\partial}{\partial u} \Big)
e^{-2Tf_T} \Big( \chi_1\hat{\omega}_1 + \chi_2\hat{\omega}_2 \Big) = 0 \;.
\end{equation}
By \eqref{eq1-def-FRTplus} and \eqref{eq3a-def-FRTplus},
we have
\begin{equation}
\label{eq3-def-FRTplus}
\stdiffRT G^+_{R,T}(\omega_1,\hat{\omega}_1,\omega_2,\hat{\omega}_2) = 0 \;.
\end{equation}

Let $P^{[-1,1]}_{R,T}:\Omega^\bullet(Z_R,F)\rightarrow\mathscr{E}^{[-1,1]}_{0,R,T}$
be the orthogonal projection with respect to $\big\lVert\cdot\big\rVert_{Z_R}$,
where $\mathscr{E}^{[-1,1]}_{0,R,T}\subseteq\Omega^\bullet(Z_R,F)$ was defined in \eqref{eq-def-ERT}.

\begin{prop}
\label{prop-FRTplus}
For $T = R^\kappa \gg 1$ and $(\omega_1,\hat{\omega}_1,\omega_2,\hat{\omega}_2)\in\hha(Z_{1,\infty} \sqcup Z_{2,\infty},F)$,
we have
\begin{equation}
\label{eq-prop-FRTplus}
\Big\lVert \big( \Id - P^{[-1,1]}_{R,T} \big)G^+_{R,T}(\omega_1,\hat{\omega}_1,\omega_2,\hat{\omega}_2) \Big\rVert^2_{H^1,Z_R}
= \mathscr{O}\big(R^{-2+\kappa}\big) \Big(\big\lVert\omega_1\big\rVert^2_{Z_{1,0}} + \big\lVert\omega_2\big\rVert^2_{Z_{2,0}} \Big) \;.
\end{equation}
\end{prop}
\begin{proof}
Though the constructions of $F^+_{R,T}$ and $G^+_{R,T}$
are different from the constructions of $F_{R,T}$ and $G_{R,T}$ in \eqref{eq1-def-FRT},
we can directly verify that $(F^+_{R,T}-G^+_{R,T})(\omega_1,\hat{\omega}_1,\omega_2,\hat{\omega}_2)$ satisfies \eqref{eq12-pf-prop-FRT}.
Then,
similarly to \eqref{eq28-pf-prop-FRT},
we have
\begin{align}
\label{eq1-pf-prop-FRTplus}
\begin{split}
\Big\lVert (F^+_{R,T}-G^+_{R,T})(\omega_1,\hat{\omega}_1,\omega_2,\hat{\omega}_2) \Big\rVert^2_{Z_R}
& = \mathscr{O}\big(R^{-2+\kappa}\big) \Big(\big\lVert\omega_1\big\rVert^2_{Z_{1,0}} + \big\lVert\omega_2\big\rVert^2_{Z_{2,0}} \Big) \;,\\
\Big\lVert \DRT (F^+_{R,T}-G^+_{R,T})(\omega_1,\hat{\omega}_1,\omega_2,\hat{\omega}_2) \Big\rVert^2_{Z_R}
& = \mathscr{O}\big(R^{-2+\kappa}\big) \Big(\big\lVert\omega_1\big\rVert^2_{Z_{1,0}} + \big\lVert\omega_2\big\rVert^2_{Z_{2,0}} \Big) \;.
\end{split}
\end{align}

By \eqref{eq-def-finf}, \eqref{eq-compare-f-fT}, \eqref{eq-def-DRT-conj},
\eqref{eq-def-chi12}, \eqref{eq-chi-IYR}, \eqref{eq1-def-FRTplus} and the identities $D^Y \hat{\omega}_j = 0$ for $j=1,2$,
we have
\begin{align}
\label{eq2-pf-prop-FRTplus}
\begin{split}
\Big\lVert \diffRT F^+_{R,T}(\omega_1,\hat{\omega}_1,\omega_2,\hat{\omega}_2) \Big\rVert^2_{Z_R}
& = \mathscr{O}\big(e^{-aT}\big) \Big(\big\lVert\omega_1\big\rVert^2_{Z_{1,0}} + \big\lVert\omega_2\big\rVert^2_{Z_{2,0}} \Big) \;,\\
\Big\lVert \DRT \diffRT F^+_{R,T}(\omega_1,\hat{\omega}_1,\omega_2,\hat{\omega}_2) \Big\rVert^2_{Z_R}
& = \mathscr{O}\big(e^{-aT}\big) \Big(\big\lVert\omega_1\big\rVert^2_{Z_{1,0}} + \big\lVert\omega_2\big\rVert^2_{Z_{2,0}} \Big) \;,
\end{split}
\end{align}
where $a>0$ is a universal constant.

By
Corollary \ref{cor-proj-estimate-better},
\eqref{eq3-def-FRTplus},
\eqref{eq1-pf-prop-FRTplus} and \eqref{eq2-pf-prop-FRTplus},
we have
\begin{align}
\label{eq3-pf-prop-FRTplus}
\begin{split}
& \Big\lVert \big( \Id - P^{[-1,1]}_{R,T} \big)G^+_{R,T}(\omega_1,\hat{\omega}_1,\omega_2,\hat{\omega}_2) \Big\rVert^2_{Z_R} \\
& \hspace{20mm} + \Big\lVert \DRT \big( \Id - P^{[-1,1]}_{R,T} \big)G^+_{R,T}(\omega_1,\hat{\omega}_1,\omega_2,\hat{\omega}_2) \Big\rVert^2_{Z_R} \\
& \hspace{60mm} = \mathscr{O}\big(R^{-2+\kappa}\big) \Big(\big\lVert\omega_1\big\rVert^2_{Z_{1,0}} + \big\lVert\omega_2\big\rVert^2_{Z_{2,0}} \Big) \;.
\end{split}
\end{align}
From Proposition \ref{prop-sobolev}
and \eqref{eq3-pf-prop-FRTplus},
we obtain \eqref{eq-prop-FRTplus}.
This completes the proof of Proposition \ref{prop-FRTplus}.
\end{proof}

We define
\begin{equation}
\label{eq0-def-IRTplus}
I^+_{R,T}: \hh(Y,F) \rightarrow \Omega^{\bullet+1}(Z_R,F)
\end{equation}
as follows:
for $\hat{\omega}\in \hh(Y,F)$,
\begin{align}
\label{eq1-def-IRTplus}
\begin{split}
I^+_{R,T}(\hat{\omega}) \big|_{Z_{j,0}} & = 0  \;,\hspace{5mm} \text{for } j=1,2 \;,\\
I^+_{R,T}(\hat{\omega}) \big|_{IY_R} & = \chi_3 e^{Tf_T-T}du\wedge\hat{\omega} \;.
\end{split}
\end{align}
We have
\begin{equation}
\label{eq3-def-IRTplus}
\diffRT I^+_{R,T}(\hat{\omega}) = 0 \;.
\end{equation}

\begin{prop}
\label{prop-IRTplus}
For $T = R^\kappa \gg 1$ and $\hat{\omega}\in \hh(Y,F)$,
we have
\begin{equation}
\label{eq-prop-IRTplus}
\Big\lVert \big( \Id - P^{[-1,1]}_{R,T} \big)I^+_{R,T}(\hat{\omega}) \Big\rVert^2_{H^1,Z_R}
= \mathscr{O}\big(e^{-aT}\big) \big\lVert\hat{\omega}\big\rVert^2_Y \;.
\end{equation}
\end{prop}
\begin{proof}
By \eqref{eq-def-finf}, \eqref{eq-compare-f-fT}, \eqref{eq-def-DRT-conj}
and the construction of $I^+_{R,T}$ (see \eqref{eq0-def-IRTplus}),
we have
\begin{equation}
\label{eq1-pf-prop-IRTplus}
\Big\lVert \DRT I^+_{R,T}(\hat{\omega}) \Big\rVert^2_{Z_R}
= \mathscr{O}\big(e^{-aT}\big) \big\lVert\hat{\omega}\big\rVert^2_Y \;,\hspace{5mm}
\Big\lVert \DsRT I^+_{R,T}(\hat{\omega}) \Big\rVert^2_{Z_R}
= \mathscr{O}\big(e^{-aT}\big) \big\lVert\hat{\omega}\big\rVert^2_Y \;.
\end{equation}
By Corollary \ref{cor-proj-estimate-better}
and \eqref{eq1-pf-prop-IRTplus},
we have
\begin{equation}
\label{eq2-pf-prop-IRTplus}
\Big\lVert \big( \Id - P^{[-1,1]}_{R,T} \big)I^+_{R,T}(\hat{\omega}) \Big\rVert^2_{Z_R}
+ \Big\lVert \DRT\big( \Id - P^{[-1,1]}_{R,T} \big)I^+_{R,T}(\hat{\omega}) \Big\rVert^2_{Z_R}
= \mathscr{O}\big(e^{-aT}\big) \big\lVert\hat{\omega}\big\rVert^2_Y \;,
\end{equation}
where $a>0$ is a universal constant.
From Proposition \ref{prop-sobolev}
and \eqref{eq2-pf-prop-IRTplus},
we obtain \eqref{eq-prop-IRTplus}.
This completes the proof of Proposition \ref{prop-IRTplus}.
\end{proof}

We identify $C^{0,\bullet}_0 = W^\bullet_1 \oplus W^\bullet_2 = H^\bullet(Z_1,F) \oplus H^\bullet(Z_2,F)$
with $\hha(Z_{1,\infty} \sqcup Z_{2,\infty},F)$
via the map \eqref{eq-hha2coh}.
We identify $C^{1,\bullet}_0 = V^\bullet = H^\bullet(Y,F)$ with $\hh(Y,F)$
via the isomorphism $H^\bullet(Y,F) \simeq \hh(Y,F)$ given by the Hodge theorem.
We define a map
\begin{align}
\label{eq-def-SRT}
\begin{split}
& \mathscr{S}_{R,T}: C^{\bullet,\bullet}_0 \rightarrow \mathscr{E}^{[-1,1]}_{0,R,T} \;, \\
& \mathscr{S}_{R,T}\Big|_{C^{0,\bullet}_0} = P^{[-1,1]}_{R,T}G^+_{R,T} \;,\hspace{5mm}
\mathscr{S}_{R,T}\Big|_{C^{1,\bullet}_0} = P^{[-1,1]}_{R,T}I^+_{R,T} \;.
\end{split}
\end{align}

\begin{prop}
\label{prop-perp-GI}
The vector subspaces
$\mathscr{S}_{R,T}(C^{0,\bullet}_0), \mathscr{S}_{R,T}(C^{1,\bullet}_0)
\subseteq \Omega^\bullet(Z_R,F)$
are orthogonal with respect to $\big\langle\cdot,\cdot\big\rangle_{Z_R}$.
\end{prop}
\begin{proof}
We consider $\sigma_0\in C^{0,\bullet}_0$
and $\sigma_1\in C^{1,\bullet}_0$.
Since the supports of $G_{R,T}^+(\sigma_0)$
and $I_{R,T}^+(\sigma_1)$ are mutually disjoint,
we have
\begin{equation}
\label{eq1-pf-prop-perp-GI}
\big\langle G_{R,T}^+(\sigma_0),I_{R,T}^+(\sigma_1) \big\rangle_{Z_R} = 0 \;.
\end{equation}
On the other hand,
by \eqref{eq3-def-FRTplus}
and \eqref{eq3-def-IRTplus},
we have
\begin{align}
\label{eq2-pf-prop-perp-GI}
\begin{split}
& G_{R,T}^+(\sigma_0) \in \Ker\big(\stdiffRT\big)
= \Ker\big(\DRT\big) \oplus \mathrm{Im}\big(\stdiffRT\big) \;,\\
& I_{R,T}^+(\sigma_1) \in \Ker\big(\diffRT\big)
= \Ker\big(\DRT\big) \oplus \mathrm{Im}\big(\diffRT\big) \;.
\end{split}
\end{align}
Since
$P_{R,T}^{\R\backslash[-1,1]}:=\mathrm{Id}-P_{R,T}^{[-1,1]}$
commutes with $\diffRT$ and $\stdiffRT$,
we have
\begin{equation}
\label{eq3-pf-prop-perp-GI}
P_{R,T}^{\R\backslash[-1,1]}  G_{R,T}^+(\sigma_0) \in \mathrm{Im}\big(\stdiffRT\big) \;,\hspace{5mm}
P_{R,T}^{\R\backslash[-1,1]}  I_{R,T}^+(\sigma_1) \in \mathrm{Im}\big(\diffRT\big) \;,
\end{equation}
which implies
\begin{equation}
\label{eq4-pf-prop-perp-GI}
\Big\langle P_{R,T}^{\R\backslash[-1,1]} G_{R,T}^+(\sigma_0),
P_{R,T}^{\R\backslash[-1,1]} I_{R,T}^+(\sigma_1) \Big\rangle_{Z_R} = 0 \;.
\end{equation}
From \eqref{eq1-pf-prop-perp-GI},
\eqref{eq4-pf-prop-perp-GI}
and the obvious identity
\begin{align}
\begin{split}
& \Big\langle G_{R,T}^+(\sigma_0),I_{R,T}^+(\sigma_1) \Big\rangle_{Z_R} \\
& = \Big\langle \mathscr{S}_{R,T}(\sigma_0),
\mathscr{S}_{R,T}(\sigma_1) \Big\rangle_{Z_R} +
\Big\langle P_{R,T}^{\R\backslash[-1,1]} G_{R,T}^+(\sigma_0),
P_{R,T}^{\R\backslash[-1,1]} I_{R,T}^+(\sigma_1) \Big\rangle_{Z_R} \;,
\end{split}
\end{align}
we obtain
$\Big\langle \mathscr{S}_{R,T}(\sigma_0),
\mathscr{S}_{R,T}(\sigma_1) \Big\rangle_{Z_R} = 0$.
This completes the proof of Proposition \ref{prop-perp-GI}.
\end{proof}

\begin{thm}
\label{thm-SRT-bij}
For $T = R^\kappa \gg 1$,
the map $\mathscr{S}_{R,T}$ is bijective.
\end{thm}
\begin{proof}
We will use the identifications \eqref{eq-id-HCHmathscr}.
We construct a vector subspace
$U^{\bullet,\bullet}\subseteq H^\bullet(C^{\bullet,\bullet}_0,\partial)$
as follows,
\begin{align}
\label{eq0-pf-thm-SRT-bij}
\begin{split}
U^{0,\bullet} & = \Big\{(\omega_1,\omega_2,\hat{\omega})\in\hha(Z_{12,\infty},F)\;:\; \omega_j \text{ is a generalized eigensection} \\
& \hspace{60mm} \text{of } \Djinf \text{ associated with } 0 \;,\; j=1,2 \Big\} \;,\\
U^{1,\bullet} & = H^1(C^{\bullet,\bullet}_0,\partial)
= \LL^{\bullet,\perp}_{1,\mathrm{abs}} \cap \LL^{\bullet,\perp}_{2,\mathrm{abs}} \;.
\end{split}
\end{align}

\noindent \textit{Step 1}. We show that
for $\sigma\in \mathscr{S}^H_{R,T}\big(U^{0,\bullet}\big)$
or $\sigma\in \mathscr{S}^H_{R,T}\big(U^{1,\bullet}\big)$,
\begin{align}
\label{eq1-pf-thm-SRT-bij}
\begin{split}
& \big\lVert\sigma^\mathrm{zm}\big\rVert^2_{Y,\mathrm{max}}
= \mathscr{O}\big(R^{-1+\kappa}\big) \big\lVert\sigma^\mathrm{zm}\big\rVert^2_{IY_R} \;,\\
& \big\lVert\sigma\big\rVert^2_{Z_{1,0}\cup Z_{2,0}} +
\big\lVert\sigma^\mathrm{nz}\big\rVert^2_{IY_R}
= \mathscr{O}\big(1\big) \big\lVert\sigma^\mathrm{zm}\big\rVert^2_{Y,\mathrm{max}} \;.
\end{split}
\end{align}

By the construction of $\mathscr{S}^H$ (see \eqref{eq2-def-SHRT}),
for $\sigma\in \mathscr{S}^H_{R,T}\big(U^{0,\bullet}\big)$,
there exists $(\omega_1,\omega_2,\hat{\omega})\in U^{0,\bullet}$ such that
$\sigma = P_{R,T}F_{R,T}(\omega_1,\omega_2,\hat{\omega})$,
where $F_{R,T}$ was defined in \eqref{eq1-def-FRT}.
We denote $\widetilde{\sigma} = F_{R,T}(\omega_1,\omega_2,\hat{\omega})$.
By \eqref{eq1-def-FRT},
we have
\begin{equation}
\label{eq11-pf-thm-SRT-bij}
\widetilde{\sigma}^\mathrm{zm} = e^{-Tf_T}\hat{\omega} \;.
\end{equation}
By \eqref{eq-def-finf}, \eqref{eq-compare-f-fT}, \eqref{eq11-pf-thm-SRT-bij} and the assumption $T = R^\kappa$,
we have
\begin{equation}
\label{eq12-pf-thm-SRT-bij}
\big\lVert\widetilde{\sigma}^\mathrm{zm}\big\rVert^2_{Y,\mathrm{max}}
= \mathscr{O}\big(R^{-1+\kappa}\big)
\big\lVert\widetilde{\sigma}^\mathrm{zm}\big\rVert^2_{IY_R} \;.
\end{equation}
By \eqref{eq-resol-RdFRdFstar}, \eqref{eq1-def-FRT}, \eqref{eq24-pf-prop-FRT} and \cite[(2.36)-(2.38)]{pzz},
we have
\begin{align}
\label{eq13-pf-thm-SRT-bij}
\begin{split}
& \big\lVert\omega_1\big\rVert^2_{Z_{1,0}}+\big\lVert\omega_2\big\rVert^2_{Z_{2,0}}
= \big\lVert\widetilde{\sigma}\big\rVert^2_{Z_{1,0}\cup Z_{2,0}}
= \mathscr{O}\big(1\big) \big\lVert\widetilde{\sigma}^\mathrm{zm}\big\rVert^2_{Y,\mathrm{max}} \;,\\
& \big\lVert\widetilde{\sigma}^\mathrm{nz}\big\rVert^2_{IY_R}
= \mathscr{O}\big(1\big) \big\lVert\widetilde{\sigma}\big\rVert^2_{Z_{1,0}\cup Z_{2,0}} \;.
\end{split}
\end{align}
By \eqref{eq13-pf-thm-SRT-bij},
we have
\begin{equation}
\label{eq15-pf-thm-SRT-bij}
\big\lVert\widetilde{\sigma}\big\rVert^2_{Z_{1,0}\cup Z_{2,0}} +
\big\lVert\widetilde{\sigma}^\mathrm{nz}\big\rVert^2_{IY_R}
= \mathscr{O}\big(1\big) \big\lVert\widetilde{\sigma}^\mathrm{zm}\big\rVert^2_{Y,\mathrm{max}} \;.
\end{equation}
From the Trace theorem for Sobolev spaces,
Proposition \ref{prop-FRT}
and \eqref{eq12-pf-thm-SRT-bij}-\eqref{eq15-pf-thm-SRT-bij},
we obtain \eqref{eq1-pf-thm-SRT-bij}
with $\sigma\in \mathscr{S}^H_{R,T}\big(U^{0,\bullet}\big)$.

By the construction of $\mathscr{S}^H$ (see \eqref{eq2-def-SHRT}),
for $\sigma\in \mathscr{S}^H_{R,T}\big(U^{1,\bullet}\big)$,
there exists $\hat{\omega}\in U^{1,\bullet}$ such that
$\sigma = P_{R,T}I_{R,T}(\hat{\omega})$,
where $I_{R,T}$ was defined in \eqref{eq1-def-IRT}.
We denote $\widetilde{\sigma} = I_{R,T}(\hat{\omega})$.
By \eqref{eq1-def-IRT},
we have
\begin{align}
\label{eq17-pf-thm-SRT-bij}
\begin{split}
& \Big(1+\mathscr{O}\big(e^{-aT}\big)\Big) \big\lVert\hat{\omega}\big\rVert^2_Y
= \big\lVert\widetilde{\sigma}^\mathrm{zm}\big\rVert^2_{Y,\mathrm{max}}
= \mathscr{O}\big(R^{-1+\kappa}\big) \big\lVert\widetilde{\sigma}^\mathrm{zm}\big\rVert^2_{IY_R} \;,\\
& \big\lVert\widetilde{\sigma}^\mathrm{nz}\big\rVert^2_{IY_R}
= \big\lVert\widetilde{\sigma}\big\rVert^2_{Z_{1,0}\cup Z_{2,0}} = 0 \;,
\end{split}
\end{align}
where $a>0$ is a universal constant.
From the Trace theorem for Sobolev spaces,
Proposition \ref{prop-IRT}
and \eqref{eq17-pf-thm-SRT-bij},
we obtain \eqref{eq1-pf-thm-SRT-bij}
with $\sigma\in \mathscr{S}^H_{R,T}\big(U^{1,\bullet}\big)$.

\noindent \textit{Step 2}. We show that
for $\sigma\in\mathscr{E}^{\{\lambda\}}_{R,T}$ with $\lambda\in [-1,1]\backslash\{0\}$,
\begin{align}
\label{eq2-pf-thm-SRT-bij}
\begin{split}
& \big\lVert\sigma^\mathrm{zm}\big\rVert^2_{Y,\mathrm{max}}
= \mathscr{O}\big(R^{-1+2\kappa}\big) \big\lVert\sigma^\mathrm{zm}\big\rVert^2_{IY_R} \;,\\
& \big\lVert\sigma\big\rVert^2_{Z_{1,0}\cup Z_{2,0}} +
\big\lVert\sigma^\mathrm{nz}\big\rVert^2_{IY_R}
= \mathscr{O}\big(1\big) \big\lVert\sigma^\mathrm{zm}\big\rVert^2_{Y,\mathrm{max}} \;.
\end{split}
\end{align}

By the Trace theorem for Sobolev spaces,
Lemma \ref{lem-cylinder-nz-estimation} and Proposition \ref{prop-sobolev},
we have
\begin{equation}
\label{eq2a-pf-thm-SRT-bij}
\big\lVert\sigma^\mathrm{nz}\big\rVert^2_{IY_R} =
\mathscr{O}\big(1\big) \big\lVert\sigma^\mathrm{nz}\big\rVert^2_{\partial Z_{1,0}\cup\partial Z_{2,0}}
= \mathscr{O}\big(1\big) \big\lVert\sigma\big\rVert^2_{\partial Z_{1,0}\cup\partial Z_{2,0}}
= \mathscr{O}\big(1\big) \big\lVert\sigma\big\rVert^2_{Z_{1,0}\cup Z_{2,0}} \;.
\end{equation}
By \eqref{eq-dirac-cylinder-spec-expansion-hh},
$\sigma^\mathrm{zm}$ is an eigensection of $D^{\hh(Y,F)}_T$
associated with $\lambda$, i.e.,
\begin{equation}
\label{eq21-pf-thm-SRT-bij}
\Big(c\frac{\partial}{\partial s}+Tf'_T\hat{c}\Big) \sigma^\mathrm{zm} =
D^{\hh(Y,F)}_T\sigma^\mathrm{zm} =
\lambda \sigma^\mathrm{zm} \;.
\end{equation}
The first inequality in \eqref{eq2-pf-thm-SRT-bij}
follows from the Sobolev inequality,
\eqref{eq21-pf-thm-SRT-bij}
and the assumption $\lambda\in[-1,1]$.
The second inequality in \eqref{eq2-pf-thm-SRT-bij}
follows from Lemma \ref{prop-zm-minoration}
and \eqref{eq2a-pf-thm-SRT-bij}.

Let $\mathscr{E}^{[-1,1]}_T \subseteq \Omega^\bullet\big([-1,1],\hh(Y,F)\big)$
be the eigenspace of $D^{\hh(Y,F)}_T$ associated with eigenvalues in $[-1,1]$.
Let $P^{[-1,1]}_T: \Omega^\bullet\big([-1,1],\hh(Y,F)\big) \rightarrow \mathscr{E}^{[-1,1]}_T$
be the orthogonal projection.

\noindent \textit{Step 3}. We show that the map
\begin{align}
\label{eq3-pf-thm-SRT-bij}
\begin{split}
\pi_{R,T}:
\mathscr{S}^H_{R,T}\big(U^{\bullet,\bullet}\big) \oplus
\mathscr{E}^{[-1,1]\backslash\{0\}}_{0,R,T}
&
\rightarrow \mathscr{E}^{[-1,1]}_T \\
\sigma & \mapsto P^{[-1,1]}_T\sigma^\mathrm{zm}
\end{split}
\end{align}
is injective.

Let
\begin{equation}
\label{eq3a-pf-thm-SRT-bij}
\sigma_1,\cdots,\sigma_m \in
\mathscr{S}^H_{R,T}\big(U^{\bullet,\bullet}\big) \oplus
\mathscr{E}^{[-1,1]\backslash\{0\}}_{0,R,T}
\end{equation}
be a basis such that
each $\sigma_i$ belongs to one of the following vector spaces
\begin{equation}
\label{eq3b-pf-thm-SRT-bij}
\mathscr{S}^H_{R,T}\big(U^{0,\bullet}\big),\hspace{5mm}
\mathscr{S}^H_{R,T}\big(U^{1,\bullet}\big),\hspace{5mm}
\mathscr{E}^{\{\lambda\}}_{R,T} \hspace{5mm}
\text{with } \lambda\in[-1,1]\backslash\{0\} \;.
\end{equation}
We suppose that
for $i\neq j$
with $\sigma_i,\sigma_j$ belonging to
the same vector space in \eqref{eq3b-pf-thm-SRT-bij},
\begin{equation}
\label{eq3d-pf-thm-SRT-bij}
\big\langle \sigma_i,\sigma_j \big\rangle_{Z_R} = 0 \;.
\end{equation}

By the constructions of $F_{R,T}$ and $I_{R,T}$,
we have
\begin{equation}
\label{eq35-pf-thm-SRT-bij}
F_{R,T}\big(U^{0,\bullet}\big) \perp
I_{R,T}\big(U^{1,\bullet}\big) \;.
\end{equation}
By Propositions \ref{prop-FRT}, \ref{prop-IRT}
and \eqref{eq35-pf-thm-SRT-bij},
for $\sigma_i\in\mathscr{S}^H_{R,T}\big(U^{0,\bullet}\big)$
and $\sigma_j\in\mathscr{S}^H_{R,T}\big(U^{1,\bullet}\big)$,
we have
\begin{equation}
\label{eq36-pf-thm-SRT-bij}
\big\langle\sigma_i,\sigma_j\big\rangle_{Z_R}
=  \mathscr{O}\big(R^{-1+\kappa/2}\big)
\big\lVert\sigma_i\big\rVert_{Z_R}
\big\lVert\sigma_j\big\rVert_{Z_R} \;.
\end{equation}
Since $\mathscr{S}^H_{R,T}\big(U^{\bullet,\bullet}\big) \subseteq \Ker\big(\DRT\big)$,
for $\sigma_i\in\mathscr{S}^H_{R,T}\big(U^{\bullet,\bullet}\big)$
and $\sigma_j\in\mathscr{E}^{[-1,1]\backslash\{0\}}_{R,T}$,
we have
\begin{equation}
\label{eq37-pf-thm-SRT-bij}
\big\langle\sigma_i,\sigma_j\big\rangle_{Z_R} = 0 \;.
\end{equation}
By \eqref{eq3d-pf-thm-SRT-bij},
\eqref{eq36-pf-thm-SRT-bij} and
\eqref{eq37-pf-thm-SRT-bij},
we have
\begin{equation}
\label{eq38-pf-thm-SRT-bij}
\big\langle\sigma_i,\sigma_j\big\rangle_{Z_R}
= \Big( \delta_{ij} + \mathscr{O}\big(R^{-1+\kappa/2}\big) \Big)
\big\lVert\sigma_i\big\rVert_{Z_R}
\big\lVert\sigma_j\big\rVert_{Z_R} \;,
\end{equation}
where $\delta_{ij}$ is the Kronecker delta.

By Steps 1, 2 and the obvious identity
\begin{equation}
\big\langle\sigma_i,\sigma_j\big\rangle_{Z_R} =
\big\langle\sigma_i^\mathrm{zm},\sigma_j^\mathrm{zm}\big\rangle_{IY_R}
+ \big\langle\sigma_i^\mathrm{nz},\sigma_j^\mathrm{nz}\big\rangle_{IY_R}
+ \big\langle\sigma_i,\sigma_j\big\rangle_{Z_{1,0}\cup Z_{2,0}} \;,
\end{equation}
we have
\begin{equation}
\label{eq39-pf-thm-SRT-bij}
\big\langle\sigma_i^\mathrm{zm},\sigma_j^\mathrm{zm}\big\rangle_{IY_R}
= \big\langle\sigma_i,\sigma_j\big\rangle_{Z_R}
+ \mathscr{O}\big(R^{-1+2\kappa}\big)
\big\lVert\sigma_i\big\rVert_{Z_R}
\big\lVert\sigma_j\big\rVert_{Z_R} \;.
\end{equation}

Recall that the maps $P^\perp_j$ with $j=1,2$
were defined by \eqref{eq-def-hhY-P1P2}-\eqref{eq-def-hhY-P1P2perp}.
By \eqref{eq3-pf-thm-central-spgap},
for $\sigma\in \mathscr{E}^{\{\lambda\}}_{R,T}$
with $\lambda\in[-1,1]\backslash\{0\}$,
we have
\begin{equation}
\label{eq31-pf-thm-SRT-bij}
\Big\lVert P^\perp_j \sigma^\mathrm{zm} \Big\rVert^2_{\partial Z_{j,0}} =
\mathscr{O}\big(R^{-2+\kappa}\big) \big\lVert \sigma^\mathrm{zm} \big\rVert^2_{Y,\mathrm{max}}
\;,\hspace{5mm} \text{ for } j=1,2 \;.
\end{equation}
By the constructions of $F_{R,T}$ and  $I_{R,T}$,
for $\widetilde{\sigma}\in F_{R,T}\big(U^{0,\bullet}\big)$ or $\widetilde{\sigma}\in I_{R,T}\big(U^{1,\bullet}\big)$,
we have
\begin{equation}
\label{eq32-pf-thm-SRT-bij}
P^\perp_j \widetilde{\sigma}^\mathrm{zm} \big|_{\partial Z_{j,0}} = 0 \;,\hspace{5mm} \text{ for } j=1,2 \;.
\end{equation}
By the Trace theorem for Sobolev spaces,
Propositions \ref{prop-FRT}, \ref{prop-IRT},
\eqref{eq13-pf-thm-SRT-bij},
\eqref{eq17-pf-thm-SRT-bij} and
\eqref{eq32-pf-thm-SRT-bij},
for $\sigma\in \mathscr{S}^H_{R,T}\big(U^{0,\bullet}\big)$
or $\sigma\in \mathscr{S}^H_{R,T}\big(U^{1,\bullet}\big)$,
we have
\begin{equation}
\label{eq33-pf-thm-SRT-bij}
\Big\lVert P^\perp_j \sigma^\mathrm{zm} \Big\rVert^2_{\partial Z_{j,0}} =
\mathscr{O}\big(R^{-2+\kappa}\big) \big\lVert \sigma^\mathrm{zm} \big\rVert^2_{Y,\mathrm{max}} \;,
\hspace{5mm} \text{ for } j=1,2 \;.
\end{equation}
Applying Proposition \ref{prop-Df-eigen-approx}
to \eqref{eq31-pf-thm-SRT-bij} and \eqref{eq33-pf-thm-SRT-bij} with $\epsilon=1$,
we get
\begin{equation}
\label{eq34a-pf-thm-SRT-bij}
\Big\lVert \sigma^\mathrm{zm}_i - P^{[-2,2]}_T\sigma^\mathrm{zm}_i \Big\rVert_{IY_R} =
\mathscr{O}\big(R^{-1+3\kappa}\big)
\big\lVert \sigma^\mathrm{zm}_i \big\rVert_{IY_R} \;.
\end{equation}
By Theorem \ref{thm-witten-spec},
we have $P^{[-1,1]}_T = P^{[-2,2]}_T$.
Then equation \eqref{eq34a-pf-thm-SRT-bij} yields
\begin{equation}
\label{eq34-pf-thm-SRT-bij}
\big\langle P^{[-1,1]}_T\sigma_i^\mathrm{zm},P^{[-1,1]}_T\sigma_j^\mathrm{zm} \big\rangle_{IY_R}
= \big\langle\sigma_i^\mathrm{zm},\sigma_j^\mathrm{zm}\big\rangle_{IY_R}
+ \mathscr{O}\big(R^{-1+3\kappa}\big)
\big\lVert\sigma_i^\mathrm{zm}\big\rVert_{IY_R}
\big\lVert\sigma_j^\mathrm{zm}\big\rVert_{IY_R} \;.
\end{equation}

By \eqref{eq38-pf-thm-SRT-bij},
\eqref{eq39-pf-thm-SRT-bij}
and \eqref{eq34-pf-thm-SRT-bij},
we have
\begin{equation}
\label{eq3x-pf-thm-SRT-bij}
\big\langle P^{[-1,1]}_T\sigma_i^\mathrm{zm},P^{[-1,1]}_T\sigma_j^\mathrm{zm} \big\rangle_{IY_R}
= \Big( \delta_{ij} + \mathscr{O}\big(R^{-1+3\kappa}\big) \Big)
\big\lVert\sigma_i\big\rVert_{Z_R}
\big\lVert\sigma_j\big\rVert_{Z_R} \;.
\end{equation}
By \eqref{eq3-pf-thm-SRT-bij} and \eqref{eq3x-pf-thm-SRT-bij},
the Gram matrix
$\Big(\big\langle\pi_{R,T}(\sigma_i),\pi_{R,T}(\sigma_j)\big\rangle_{IY_R}\Big)_{1\leqslant i,j \leqslant m}$
is positive-definite.
Hence the map $\pi_{R,T}$ is injective.

\noindent \textit{Step 4}. We show that the map $\mathscr{S}_{R,T}$ is bijective.

By Theorems \ref{thm-witten-estimates}, \ref{thm-SRTH-bij} and Step 3,
we have
\begin{align}
\label{eq41-pf-thm-SRT-bij}
\begin{split}
\dim \Big( \mathscr{E}^{[-1,1]\backslash\{0\}}_{0,R,T} \Big) + \dim U^{\bullet,\bullet}
& = \dim \Big( \mathscr{E}^{[-1,1]\backslash\{0\}}_{0,R,T} \Big) + \dim \mathscr{S}^H_{R,T}\big(U^{\bullet,\bullet}\big) \\
& \leqslant \dim \mathscr{E}^{[-1,1]}_T
= \dim C^{\bullet,\bullet}_\mathrm{r} \;.
\end{split}
\end{align}
By Theorem \ref{thm-SRTH-bij},
we have
\begin{equation}
\label{eq43-pf-thm-SRT-bij}
\dim \Ker\big(\DRT\big) = \dim H^\bullet(C^{\bullet,\bullet}_0,\partial) \;.
\end{equation}
By \eqref{eq-def-K}, \eqref{eq-ses-CCr} and \eqref{eq-complex-evaluation},
we have
\begin{equation}
\dim C^{\bullet,\bullet}_0 - \dim C^{\bullet,\bullet}_\mathrm{r}
= \dim K^\bullet_1 + \dim K^\bullet_2 \;.
\end{equation}
By the construction of $U^{\bullet,\bullet}$
and \eqref{eq-commute-hha12}-\eqref{eq-def-WKK},
we have
\begin{equation}
\label{eq45-pf-thm-SRT-bij}
\dim H^\bullet(C^{\bullet,\bullet}_0,\partial) - \dim U^{\bullet,\bullet} =
\dim K^\bullet_1 + \dim K^\bullet_2 \;.
\end{equation}
From \eqref{eq41-pf-thm-SRT-bij}-\eqref{eq45-pf-thm-SRT-bij},
we obtain
\begin{equation}
\label{eq4-pf-thm-SRT-bij}
\dim \mathscr{E}^{[-1,1]}_{0,R,T} =
\dim \mathscr{E}^{[-1,1]\backslash\{0\}}_{0,R,T}
+ \dim \Ker\big(\DRT\big)
\leqslant \dim C^{\bullet,\bullet}_0  \;.
\end{equation}

By Propositions \ref{prop-FRTplus}-\ref{prop-perp-GI} and \eqref{eq-def-SRT},
the map $\mathscr{S}_{R,T}: \dim C^{\bullet,\bullet}_0 \rightarrow \mathscr{E}^{[-1,1]}_{R,T}$ is injective.
Then,
by \eqref{eq4-pf-thm-SRT-bij},
it is bijective.
This completes the proof of Theorem \ref{thm-SRT-bij}.
\end{proof}

\subsection{De Rham operator on $\mathscr{E}^{[-1,1]}_{0,R,T}$}
\label{subsec-approx-de-rham}

\begin{prop}
\label{prop-SRT-grading}
For $T = R^\kappa \gg 1$,
we have
\begin{equation}
\label{eq-prop-SRT-grading}
\diffRT \mathscr{S}_{R,T} \big(C^{1,\bullet}_0\big) = 0 \;,\hspace{5mm}
\diffRT \mathscr{S}_{R,T} \big(C^{0,\bullet}_0\big)
\subseteq \mathscr{S}_{R,T} \big(C^{1,\bullet}_0\big) \;.
\end{equation}
\end{prop}
\begin{proof}
Since $\diffRT$ commutes with $P^{[-1,1]}_{R,T}$,
\eqref{eq3-def-IRTplus} and \eqref{eq-def-SRT} yield
\begin{equation}
\label{eq1-pf-prop-SRT-grading}
\diffRT \mathscr{S}_{R,T} \big(C^{1,\bullet}_0\big) =
\diffRT P^{[-1,1]}_{R,T} I^+_{R,T} \big(C^{1,\bullet}_0\big) =
P^{[-1,1]}_{R,T} \diffRT  I^+_{R,T} \big(C^{1,\bullet}_0\big) = 0 \;.
\end{equation}
Since $\stdiffRT$ commutes with $P^{[-1,1]}_{R,T}$,
\eqref{eq3-def-FRTplus} and \eqref{eq-def-SRT} yield
\begin{equation}
\label{eq2-pf-prop-SRT-grading}
\stdiffRT \mathscr{S}_{R,T} \big(C^{0,\bullet}_0\big) =
\stdiffRT P^{[-1,1]}_{R,T} G^+_{R,T} \big(C^{0,\bullet}_0\big) =
P^{[-1,1]}_{R,T} \stdiffRT  G^+_{R,T} \big(C^{0,\bullet}_0\big) = 0 \;.
\end{equation}
Thus $\mathscr{S}_{R,T} \big(C^{0,\bullet}_0\big)$ is perpendicular to the image of $\diffRT$.
On the other hand,
by Proposition \ref{prop-perp-GI} and Theorem \ref{thm-SRT-bij},
we have an orthogonal decomposition
\begin{equation}
\label{eq3-pf-prop-SRT-grading}
\mathscr{E}^{[-1,1]}_{0,R,T} =
\mathscr{S}_{R,T}\big(C^{0,\bullet}_0\big) \oplus \mathscr{S}_{R,T}\big(C^{1,\bullet}_0\big) \;.
\end{equation}
Hence $\diffRT \mathscr{S}_{R,T} \big(C^{0,\bullet}_0\big)$ must lie in $\mathscr{S}_{R,T} \big(C^{1,\bullet}_0\big)$.
This completes the proof of Proposition \ref{prop-SRT-grading}.
\end{proof}

For $\omega\in\Omega^\bullet(Z_R,F)$,
we will view $\omega^\mathrm{zm}$ as an element in
$\Omega^\bullet\big([-R,R],\hh(Y,F)\big)$.
Set
\begin{equation}
\label{eq-def-tauRT}
\tau_{R,T}(\omega) = \int_{-R}^R e^{Tf_T}\omega^\mathrm{zm} \in \hh(Y,F) \;.
\end{equation}

\begin{lemme}
\label{lem-approx-derham}
There exists $a>0$ such that
for $T = R^\kappa \gg 1$ and $\hat{\omega}\in\hh(Y,F)$,
we have
\begin{equation}
\label{eq-lem-approx-derham}
\Big\lVert e^{-T} \tau_{R,T}\Big(\mathscr{S}_{R,T}\big(\hat{\omega}\big)\Big)
- \sqrt{\pi} R^{1-\kappa/2} \hat{\omega} \Big\rVert_Y
= \mathscr{O}\big(e^{-aT}\big) \big\lVert\hat{\omega}\big\rVert_Y \;.
\end{equation}
\end{lemme}
\begin{proof}
By \eqref{eq-def-finf}, \eqref{eq-compare-f-fT}, \eqref{eq-def-chi3} and \eqref{eq1-def-IRTplus}
and the assumption $T = R^\kappa$,
we have
\begin{align}
\label{eq1-pf-lem-approx-derham}
\begin{split}
& \int_{-R}^R e^{Tf_T} \Big(I^+_{R,T}\big(\hat{\omega}\big)\Big)^\mathrm{zm}
= \int_{-R}^R \chi_3 e^{2Tf_T-2T}du \, e^T \hat{\omega} \\
& = \Big( 1 + \mathscr{O}\big(e^{-aT}\big)\Big) \int_{-\infty}^{+\infty} e^{-Tu^2/R^2}du \, e^T \hat{\omega}
= \Big( 1 + \mathscr{O}\big(e^{-aT}\big)\Big) \sqrt{\pi} R^{1-\kappa/2} e^T \hat{\omega}\;.
\end{split}
\end{align}
From Proposition \ref{prop-IRTplus},
\eqref{eq-def-SRT} and \eqref{eq1-pf-lem-approx-derham},
we obtain \eqref{eq-lem-approx-derham}.
This completes the proof of Lemma \ref{lem-approx-derham}.
\end{proof}

We will use the notation in \eqref{eq-OE}.

\begin{thm}
\label{prop-approx-derham}
For $T = R^\kappa \gg 1$,
we have
\begin{equation}
\label{eq-prop-approx-derham}
\mathscr{S}_{R,T}^{-1} \circ \diffRT \circ \mathscr{S}_{R,T}
= \pi^{-1/2}R^{-1+\kappa/2}e^{-T}
\Big(\partial + \mathscr{O}_{\mathrm{End}(C^{\bullet,\bullet}_0)}\big(R^{-1+\kappa/2}\big)\Big) \;.
\end{equation}
\end{thm}
\begin{proof}
For $\omega\in\mathscr{S}_{R,T} \big(C^{0,\bullet}_0\big)$,
by \eqref{eq-def-fT}, \eqref{eq-def-DRT-conj} and \eqref{eq-def-tauRT},
we have
\begin{equation}
\label{eq1-pf-prop-approx-derham}
\tau_{R,T}\big(\diffRT\omega\big)
= \int_{-R}^R d \big( e^{Tf_T} \omega^\mathrm{zm} \big)
=  i_{\frac{\partial}{\partial u}} du \wedge \Big( \omega^\mathrm{zm}\big|_{\partial Z_{2,0}} - \omega^\mathrm{zm}\big|_{\partial Z_{1,0}} \Big) \;.
\end{equation}

By the Trace theorem for Sobolev spaces, Proposition \ref{prop-FRTplus} and \eqref{eq-def-SRT},
for $j=1,2$
and $(\omega_1,\hat{\omega}_1,\omega_2,\hat{\omega}_2)\in\hha(Z_{1,\infty} \sqcup Z_{2,\infty},F)$,
we have
\begin{align}
\label{eq2-pf-prop-approx-derham}
\begin{split}
\Big\lVert \Big(\mathscr{S}_{R,T}(\omega_1,\hat{\omega}_1,\omega_2,\hat{\omega}_2)
- G^+_{R,T}(\omega_1,\hat{\omega}_1,\omega_2,\hat{\omega}_2)\Big)^\mathrm{zm} \Big\rVert_{\partial Z_{j,0}} \hspace{20mm} & \\
= \mathscr{O}\big(R^{-1+\kappa/2}\big)
\Big( \big\lVert\omega_1\big\rVert_{Z_{1,0}} + \big\lVert\omega_2\big\rVert_{Z_{2,0}} \Big) & \;.
\end{split}
\end{align}
By \eqref{eq1-def-FRTplus},
we have
\begin{equation}
\label{eq3-pf-prop-approx-derham}
\Big(G^+_{R,T}(\omega_1,\hat{\omega}_1,\omega_2,\hat{\omega}_2)\Big)^\mathrm{zm}\Big|_{\partial Z_{j,0}}
= \hat{\omega}_j \in \hh(Y,F) \;.
\end{equation}
By \eqref{eq1-pf-prop-approx-derham}-\eqref{eq3-pf-prop-approx-derham},
we have
\begin{align}
\label{eq4-pf-prop-approx-derham}
\begin{split}
\Big\lVert \hat{\omega}_2 - \hat{\omega}_1
- \tau_{R,T}\Big(\diffRT\mathscr{S}_{R,T}(\omega_1,\hat{\omega}_1,\omega_2,\hat{\omega}_2)\Big) \Big\rVert_Y \hspace{20mm} & \\
= \mathscr{O}\big(R^{-1+\kappa/2}\big)
\Big( \big\lVert\omega_1\big\rVert_{Z_{1,0}} + \big\lVert\omega_2\big\rVert_{Z_{2,0}} \Big) & \;.
\end{split}
\end{align}
From Proposition \ref{prop-SRT-grading},
Lemma \ref{lem-approx-derham},
\eqref{eq1-def-complex}, \eqref{eq-complex-evaluation}
and \eqref{eq4-pf-prop-approx-derham},
we obtain \eqref{eq-prop-approx-derham}.
This completes the proof of Theorem \ref{prop-approx-derham}.
\end{proof}

\subsection{$L^2$-metric on $\mathscr{E}^{[-1,1]}_{0,R,T}$}
\label{subsec-approx-metrics}

Recall that the Hermitian metric $h^{C^{\bullet,\bullet}_0}_{R,T}$ on $C^{\bullet,\bullet}_0$
was constructed in \eqref{eq-def-hC}.
We denote by $\big\lVert\cdot\big\rVert_{R,T}$ the norm on $C^{\bullet,\bullet}_0$
associated with $h^{C^{\bullet,\bullet}_0}_{R,T}$.

\begin{prop}
\label{prop-SRT-metric}
For $T = R^\kappa \gg 1$ and $\sigma\in C^{\bullet,\bullet}_0$,
we have
\begin{equation}
\label{eq-prop-SRT-metric}
\Big\lVert \mathscr{S}_{R,T}(\sigma) \Big\rVert^2_{Z_R} =
\big\lVert \sigma \big\rVert^2_{R,T}
\Big( 1 + \mathscr{O}\big(R^{-1/2+\kappa/4}\big) \Big)\;.
\end{equation}
\end{prop}
\begin{proof}
Let $\big\lVert\cdot\big\lVert'_{R,T}$ be the norm on $C^{\bullet,\bullet}_0$ defined as follows:
for $\sigma_0\in C^{0,\bullet}_0$ and $\sigma_1\in C^{1,\bullet}_0$,
\begin{equation}
\label{eq0-pf-prop-SRT-metric}
{\big\lVert \sigma_0 + \sigma_1 \big\lVert'}_{R,T}^2 =
\Big\lVert G^+_{R,T}(\sigma_0) \Big\rVert_{Z_R}^2 +
\Big\lVert I^+_{R,T}(\sigma_1) \Big\rVert_{Z_R}^2 \;,
\end{equation}
where $G^+_{R,T}$ and $I^+_{R,T}$ were constructed in \eqref{eq1-def-FRTplus} and \eqref{eq1-def-IRTplus}.
By \eqref{eq1-def-FRTplus} and \eqref{eq1-def-IRTplus},
we have
\begin{equation}
\label{eq1-pf-prop-SRT-metric}
\Big\langle G^+_{R,T}(\sigma_0),I^+_{R,T}(\sigma_1) \Big\rangle_{Z_R} = 0 \;.
\end{equation}
By Propositions \ref{prop-FRTplus}, \ref{prop-IRTplus},
\eqref{eq-def-SRT}, \eqref{eq0-pf-prop-SRT-metric} and \eqref{eq1-pf-prop-SRT-metric},
we have
\begin{equation}
\label{eq2-pf-prop-SRT-metric}
\Big\lVert \mathscr{S}_{R,T}(\cdot) \Big\rVert^2_{Z_R} =
{\big\lVert \cdot \big\rVert'}^2_{R,T}
\Big( 1 + \mathscr{O}\big(R^{-1+\kappa/2}\big) \Big)\;.
\end{equation}

By \eqref{eq1-def-FRTplus} and \eqref{eq0-pf-prop-SRT-metric},
the decomposition $C^{\bullet,\bullet}_0 = W^\bullet_1 \oplus W^\bullet_2 \oplus V^\bullet$
is orthogonal with respect to $\big\lVert\cdot\big\lVert'_{R,T}$.
Thus it remains to show that
\begin{equation}
\label{eq4-pf-prop-SRT-metric}
{\big\lVert \sigma \big\rVert'}^2_{R,T} =
\big\lVert \sigma \big\rVert^2_{R,T}
\Big( 1 + \mathscr{O}\big(R^{-1/2+\kappa/4}\big) \Big)
\end{equation}
for $\sigma$ belonging to $W^\bullet_1$ or $W^\bullet_2$ or $V^\bullet$.

First we consider $\sigma \in V^\bullet = \hh(Y,F)$.
By \eqref{eq-def-hC},
we have
\begin{equation}
\label{eq11-pf-prop-SRT-metric}
\big\lVert \sigma \big\rVert_{R,T}^2
= \big\lVert \sigma \big\rVert_Y^2
R^{1-\kappa/2} \sqrt{\pi} \;.
\end{equation}
On the other hand,
by \eqref{eq1-def-IRTplus},
we have
\begin{equation}
\label{eq12-pf-prop-SRT-metric}
{\big\lVert\sigma\big\rVert'}^2_{R,T}
= \Big\lVert I^+_{R,T}(\sigma) \Big\rVert_{Z_R}^2
= \big\lVert \sigma \big\rVert_Y^2 R^{1-\kappa/2} \sqrt{\pi}
\Big( 1 + \mathscr{O}\big(e^{-aT}\big) \Big) \;,
\end{equation}
where $a>0$ is a universal constant.
From \eqref{eq11-pf-prop-SRT-metric}
and \eqref{eq12-pf-prop-SRT-metric},
we obtain \eqref{eq4-pf-prop-SRT-metric}
for $\sigma\in V^\bullet$.

Now we consider $\sigma \in W^\bullet_1$.
For $(\omega,\hat{\omega})\in K^{\bullet,\perp}_1\subseteq W^\bullet_1 = \hha(Z_{1,\infty},F)$,
where $K^{\bullet,\perp}_1$ was constructed in \eqref{eq-def-WKK},
by \eqref{eq1-def-FRTplus},
we have
\begin{align}
\label{eq21-pf-prop-SRT-metric}
\begin{split}
\Big\lVert G^+_{R,T}(\omega,\hat{\omega},0,0) \Big\rVert^2_{Z_R}
& = \big\lVert \omega \big\rVert^2_{Z_{1,0}} + \Big\lVert \chi_1 e^{-Tf_T} \hat{\omega} \Big\rVert^2_{IY_R} \\
& \hspace{5mm} + \Big\lVert e^{Tf_T}d^{Z_R,*}
\Big( \chi_1\stRes(\omega,\hat{\omega}) \Big)
\Big\rVert^2_{IY_R} \\
& \hspace{5mm} + 2 \mathrm{Re}
\Big\langle \chi_1 e^{-Tf_T} \hat{\omega},e^{Tf_T}d^{Z_R,*}
\Big( \chi_1\stRes(\omega,\hat{\omega}) \Big)
\Big\rangle_{IY_R} \;.
\end{split}
\end{align}
Similarly to \eqref{eq1-pf-lem-approx-derham},
we have
\begin{equation}
\label{eq22-pf-prop-SRT-metric}
\Big\lVert \chi_1 e^{-Tf_T} \hat{\omega} \Big\rVert^2_{IY_R}
= \big\lVert \hat{\omega} \big\rVert_Y^2 R^{1-\kappa/2} \frac{\sqrt{\pi}}{2}
\Big( 1 + \mathscr{O}\big(e^{-aT}\big) \Big) \;,
\end{equation}
where $a>0$ is a universal constant.
By the Trace theorem for Sobolev spaces and Proposition \ref{prop-sobolev},
we have
\begin{equation}
\label{eq23a-pf-prop-SRT-metric}
\big\lVert \hat{\omega} \big\rVert_Y \leqslant \big\lVert \omega \big\rVert_{\partial Z_{1,0}}
= \mathscr{O}\big(1\big) \big\lVert\omega\big\rVert_{Z_{1,0}} \;.
\end{equation}
By \eqref{eq-resol-RdFRdFstar}, \eqref{eq24-pf-prop-FRT}, \eqref{eq25-pf-prop-FRT} and \eqref {eq23a-pf-prop-SRT-metric},
we have
\begin{align}
\label{eq23-pf-prop-SRT-metric}
\begin{split}
& \Big\lVert e^{Tf_T}d^{Z_R,*}
\Big( \chi_1\stRes(\omega,\hat{\omega}) \Big)
\Big\rVert^2_{IY_R}
= \mathscr{O}\big(1\big)
\big\lVert\omega\big\rVert^2_{Z_{1,0}} \;,\\
& \Big\langle \chi_1 e^{-Tf_T} \hat{\omega},e^{Tf_T}d^{Z_R,*}
\Big( \chi_1\stRes(\omega,\hat{\omega}) \Big)
\Big\rangle_{IY_R}
= \mathscr{O}\big(1\big)
\big\lVert\omega\big\rVert^2_{Z_{1,0}} \;.
\end{split}
\end{align}
By \eqref{eq-def-WKK},
$\omega$ is a generalized eigensection of $\Doneinf$.
Then,
by \cite[(2.38)]{pzz},
we have
\begin{equation}
\label{eq24-pf-prop-SRT-metric}
\big\lVert \omega \big\rVert^2_{Z_{1,0}} =
\mathscr{O}\big(1\big) \big\lVert \hat{\omega} \big\rVert^2_Y \;.
\end{equation}
By \eqref{eq21-pf-prop-SRT-metric}-\eqref{eq24-pf-prop-SRT-metric},
we have
\begin{equation}
\label{eq25-pf-prop-SRT-metric}
\Big\lVert G^+_{R,T}(\omega,\hat{\omega},0,0) \Big\rVert^2_{IY_R} =
\big\lVert\hat{\omega}\big\rVert_Y^2
\Big( R^{1-\kappa/2} \sqrt{\pi}/2
+ \mathscr{O}\big(1\big) \Big)  \;.
\end{equation}

For $(\tau,0)\in K^\bullet_1\subseteq W^\bullet_1 = \hha(Z_{1,\infty},F)$,
where $K^\bullet_1$ was constructed in \eqref{eq-def-WKK},
by \eqref{eq1-def-FRTplus},
we have
\begin{equation}
\label{eq31-pf-prop-SRT-metric}
\Big\lVert G^+_{R,T}(\tau,0,0,0) \Big\rVert^2_{Z_R}
= \big\lVert \tau \big\rVert^2_{Z_{1,0}}
+ \Big\lVert e^{Tf_T}d^{Z_R,*}
\Big( \chi_1\stRes(\tau,0) \Big)
\Big\rVert^2_{IY_R} \;.
\end{equation}
We will use the canonical embedding $Z_{1,R}\subseteq Z_{1,\infty}$.
Since $e^{Tf_T}d^{Z_R,*} \Big( \chi_1\stRes(\tau,0) \Big)$
vanishes near $\partial Z_{1,R}$,
it may be extended to a section on $[0,+\infty)\times Y \subseteq Z_{1,\infty}$.
We use the identification $IY_R = [0,2R] \times Y \subseteq [0,+\infty) \times Y$.
By \eqref{eq-resol-RdFRdFstar}, \eqref{eq-def-chi12}, \eqref{eq-chi-IYR}, \eqref{eq24-pf-prop-FRT} and \eqref{eq25-pf-prop-FRT},
we have
\begin{align}
\label{eq32a-pf-prop-SRT-metric}
\begin{split}
& \big\lVert \tau^\mathrm{nz} \big\rVert_{[0,+\infty)\times Y}
= \big\lVert \tau^\mathrm{nz} \big\rVert_{IY_R} + \mathscr{O}\big(e^{-aR}\big) \big\lVert \tau \big\rVert_{Z_{1,0}}
= \mathscr{O}\big(1\big) \big\lVert \tau \big\rVert_{Z_{1,0}} \;,\\
& \Big\lVert e^{Tf_T}d^{Z_R,*} \Big( \chi_1\stRes(\tau,0) \Big) - \tau^\mathrm{nz} \Big\rVert_{IY_R}
= \mathscr{O}\big(R^{-2+\kappa}\big) \big\lVert \tau \big\rVert_{Z_{1,0}}  \;,
\end{split}
\end{align}
where $a>0$ is a universal constant.
By \eqref{eq32a-pf-prop-SRT-metric},
we have
\begin{equation}
\label{eq32-pf-prop-SRT-metric}
\Big\lVert e^{Tf_T}d^{Z_R,*}
\Big( \chi_1\stRes(\tau,0) \Big)
\Big\rVert^2_{IY_R}
= \big\lVert \tau^\mathrm{nz} \big\rVert^2_{[0,+\infty)\times Y}
+ \mathscr{O}\big(R^{-2+\kappa}\big) \big\lVert \tau \big\rVert^2_{Z_{1,0}}  \;.
\end{equation}
By \eqref{eq-def-WKK},
the zero-mode $\tau^\mathrm{zm}$ vanishes.
As a consequence, we have
\begin{equation}
\label{eq33-pf-prop-SRT-metric}
\big\lVert \tau \big\rVert^2_{Z_{1,\infty}} =
\big\lVert \tau \big\rVert^2_{Z_{1,0}} +
\big\lVert \tau^\mathrm{nz} \big\rVert^2_{[0,+\infty)\times Y} \;.
\end{equation}
By \eqref{eq31-pf-prop-SRT-metric}-\eqref{eq33-pf-prop-SRT-metric},
we have
\begin{equation}
\label{eq34-pf-prop-SRT-metric}
\Big\lVert G^+_{R,T}(\tau,0,0,0) \Big\rVert^2_{Z_R}
= \big\lVert \tau \big\rVert^2_{Z_{1,\infty}}
\Big( 1 + \mathscr{O}\big(R^{-2+\kappa}\big) \Big) \;.
\end{equation}

For $(\omega,\hat{\omega})\in K^{\bullet,\perp}_1$
and $(\tau,0)\in K^\bullet_1$,
we have
\begin{align}
\label{eq41-pf-prop-SRT-metric}
\begin{split}
\Big\langle G^+_{R,T}(\omega,\hat{\omega},0,0),G^+_{R,T}(\tau,0,0,0) \Big\rangle_{Z_R}
= \big\langle\omega,\tau\big\rangle_{Z_{1,0}} \hspace{40mm} & \\
+ \Big\langle
e^{Tf_T}d^{Z_R,*}\Big( \chi_1\stRes(\omega,\hat{\omega}) \Big),
e^{Tf_T}d^{Z_R,*}\Big( \chi_1\stRes(\tau,0) \Big)
\Big\rangle_{IY_R} & \\
+ \Big\langle \chi_1 e^{-Tf_T} \hat{\omega},
e^{Tf_T}d^{Z_R,*} \Big( \chi_1\stRes(\tau,0) \Big)
\Big\rangle_{IY_R} & \;.
\end{split}
\end{align}
Similarly to \eqref{eq23-pf-prop-SRT-metric},
by \eqref{eq-resol-RdFRdFstar}, \eqref{eq24-pf-prop-FRT} and \eqref{eq25-pf-prop-FRT},
we have
\begin{align}
\label{eq42-pf-prop-SRT-metric}
\begin{split}
& \Big\langle
e^{Tf_T}d^{Z_R,*}\Big( \chi_1\stRes(\omega,\hat{\omega}) \Big),
e^{Tf_T}d^{Z_R,*}\Big( \chi_1\stRes(\tau,0) \Big)
\Big\rangle_{IY_R} \\
& \hspace{73mm} = \mathscr{O}\big(1\big)
\big\lVert\omega\big\rVert_{Z_{1,0}} \big\lVert\tau\big\rVert_{Z_{1,0}} \;,\\
& \Big\langle \chi_1 e^{-Tf_T} \hat{\omega},
e^{Tf_T}d^{Z_R,*} \Big( \chi_1\stRes(\tau,0) \Big)
\Big\rangle_{IY_R}
= \mathscr{O}\big(1\big)
\big\lVert\omega\big\rVert_{Z_{1,0}} \big\lVert\tau\big\rVert_{Z_{1,0}} \;.
\end{split}
\end{align}
By \eqref{eq24-pf-prop-SRT-metric},
\eqref{eq41-pf-prop-SRT-metric}
and \eqref{eq42-pf-prop-SRT-metric},
we have
\begin{equation}
\label{eq43-pf-prop-SRT-metric}
\Big\langle G^+_{R,T}(\omega,\hat{\omega},0,0),G^+_{R,T}(\tau,0,0,0) \Big\rangle_{Z_R} =
\mathscr{O}\big(1\big) \big\lVert\hat{\omega}\big\rVert_Y \big\lVert\tau\big\rVert_{Z_{1,0}} \;.
\end{equation}

From \eqref{eq25-pf-prop-SRT-metric},
\eqref{eq34-pf-prop-SRT-metric}
and \eqref{eq43-pf-prop-SRT-metric},
we obtain \eqref{eq4-pf-prop-SRT-metric} with $\sigma\in W^\bullet_1$.
We can prove
\eqref{eq4-pf-prop-SRT-metric} with $\sigma\in W^\bullet_2$
in the same way.
This completes the proof of Proposition \ref{prop-SRT-metric}.
\end{proof}

We will use the following identifications,
\begin{align}
\label{eq-id-HC}
\begin{split}
& H^0(C^{\bullet,\bullet}_0,\partial)
= \Ker\big(\partial: C^{0,\bullet}_0 \rightarrow  C^{1,\bullet}_0 \big)
\subseteq C^{0,\bullet}_0 \;,\\
& H^1(C^{\bullet,\bullet}_0,\partial)
= \Big( \mathrm{Im} \big(\partial: C^{0,\bullet}_0 \rightarrow  C^{1,\bullet}_0 \big) \Big)^\perp
\subseteq C^{1,\bullet}_0 \;,
\end{split}
\end{align}
where the orthogonal is taken with respect to the metric $h^{V^\bullet}_{R,T}$ in \eqref{eq-def-hVRT}.
Since all the $h^{V^\bullet}_{R,T}$ are mutually proportional, this is independent of $R,T$.

\begin{cor}
\label{cor-SRT-metric}
For $T = R^\kappa \gg 1$ and $\sigma\in H^\bullet(C^{\bullet,\bullet}_0,\partial)$,
we have
\begin{equation}
\label{eq-cor-SRT-metric}
\Big\lVert \mathscr{S}^H_{R,T}(\sigma) \Big\rVert^2_{Z_R} =
\big\lVert \sigma \big\rVert^2_{R,T}
\Big( 1 + \mathscr{O}\big(R^{-1/2+\kappa/4}\big) \Big)\;.
\end{equation}
\end{cor}
\begin{proof}
By \eqref{eq1-def-FRT} and \eqref{eq1-def-FRTplus},
for
$\sigma \in H^0(C^{\bullet,\bullet}_0,\partial) = \hha(Z_{12,\infty},F)
\subseteq \hha(Z_{1,\infty} \sqcup Z_{2,\infty},F) = C^{0,\bullet}_0$,
we have
\begin{equation}
\label{eq1-pf-cor-SRT-metric}
\Big\lVert F_{R,T}(\sigma) - F^+_{R,T}(\sigma) \Big\rVert_{Z_R}
= \mathscr{O}\big(e^{-aT}\big) \big\lVert \sigma \big\rVert_{R,T} \;,
\end{equation}
where $a>0$ is a universal constant.
By the first identity in \eqref{eq1-pf-prop-FRTplus}
and \eqref{eq1-pf-cor-SRT-metric},
we have
\begin{equation}
\label{eqa-pf-cor-SRT-metric}
\Big\lVert F_{R,T}(\sigma) - G^+_{R,T}(\sigma) \Big\rVert_{Z_R}
= \mathscr{O}\big(R^{-1+\kappa/2}\big) \big\lVert \sigma \big\rVert_{R,T} \;.
\end{equation}
By \eqref{eq1-def-IRT} and \eqref{eq1-def-IRTplus},
for
$\sigma\in H^1(C^{\bullet,\bullet}_0,\partial)
= \LL^{\bullet,\perp}_{1,\mathrm{abs}}\cap\LL^{\bullet,\perp}_{1,\mathrm{abs}}
\subseteq \hh(Y,F) = C^{1,\bullet}_0$,
we have
\begin{equation}
\label{eqb-pf-cor-SRT-metric}
I_{R,T}(\sigma) = I^+_{R,T}(\sigma) \;.
\end{equation}
By Propositions \ref{prop-FRT}, \ref{prop-IRT}, \ref{prop-FRTplus}, \ref{prop-IRTplus},
\eqref{eq2-def-SHRT}, \eqref{eq-def-SRT}, \eqref{eqa-pf-cor-SRT-metric} and \eqref{eqb-pf-cor-SRT-metric},
we have
\begin{equation}
\label{eq-pf-cor-SRT-metric}
\Big\lVert \mathscr{S}^H_{R,T}(\sigma) - \mathscr{S}_{R,T}(\sigma) \Big\rVert_{Z_R}
= \mathscr{O}\big(R^{-1/2+\kappa/4}\big) \big\lVert \sigma \big\rVert_{R,T} \;.
\end{equation}
From Proposition \ref{prop-SRT-metric}
and \eqref{eq-pf-cor-SRT-metric},
we obtain \eqref{eq-cor-SRT-metric}.
This completes the proof of Corollary \ref{cor-SRT-metric}.
\end{proof}

\begin{proof}[Proof of Theorem \ref{thm-central-SRT}]
First we consider the case $j=0$. We will show that
the map $\mathscr{S}_{R,T}$ constructed in this section satisfies the desired properties.
The first property follows from \eqref{eq1-def-FRTplus}, \eqref{eq1-def-IRTplus} and \eqref{eq-def-SRT}.
The second property follows from Theorem \ref{thm-SRT-bij}.
The third property follows from Proposition \ref{prop-SRT-metric}.
The fourth property follows from Theorem \ref{prop-approx-derham}
and Proposition \ref{prop-SRT-metric}.

For the cases $j=1,2,3$,
we will only give the constructions of $\mathscr{S}^H_{j,R,T}$ and $\mathscr{S}_{j,R,T}$,
the proof of these properties is then essentially the same as in the case $j=0$.

Concerning the cases $j=1,2$,
we construct $F_{j,R,T}: \hha(Z_{j,\infty},F) \rightarrow \Omega^\bullet(Z_{j,R},F)$ as follows:
for $(\omega,\hat{\omega})\in\hha(Z_{j,\infty},F)$,
\begin{align}
\label{eq-def-FRTj}
\begin{split}
F_{j,R,T}(\omega,\hat{\omega}) \big|_{Z_{j,0}}
& = \omega \;,\\
F_{j,R,T}(\omega,\hat{\omega}) \big|_{IY_R}
& = e^{-Tf_T}\hat{\omega}
+ e^{-Tf_T}d^{Z_R} \Big( \chi_j\Res(\omega,\hat{\omega}) \Big) \;.
\end{split}
\end{align}
Let
$P_{j,R,T}: \Omega^\bullet(Z_{j,R},F) \rightarrow \Ker\big(\DjRT\big)$
be the orthogonal projection with respect to $\big\lVert\cdot\big\rVert_{Z_{j,R}}$.
We identify
$H^\bullet(C^{\bullet,\bullet}_j,\partial) = H^0(C^{\bullet,\bullet}_j,\partial)$
with $\hha(Z_{j,\infty},F)$.
We define
\begin{equation}
\mathscr{S}^H_{j,R,T} = P_{j,R,T}F_{j,R,T}:\; H^\bullet(C^{\bullet,\bullet}_j,\partial) \rightarrow \Ker\big(\DjRT\big) \;.
\end{equation}
We construct $G^+_{j,R,T}: \hha(Z_{j,\infty},F)\oplus\hh(Y,F) \rightarrow \Omega^\bullet(Z_{j,R},F)$ as follows:
for $(\omega,\hat{\omega})\in\hha(Z_{j,\infty},F)$ and $\hat{\mu}\in\hh(Y,F)$,
\begin{align}
\label{eq1-def-FRTjplus}
\begin{split}
G^+_{j,R,T}(\omega,\hat{\omega},\hat{\mu}) \big|_{Z_{j,0}}
& = \omega \;,\\
G^+_{j,R,T}(\omega,\hat{\omega},\hat{\mu}) \big|_{IY_R}
& = e^{-Tf_T} \Big( \chi_j\hat{\omega} + \chi_{3-j} \hat{\mu} \Big)
+ e^{Tf_T}d^{Z_R,*} \Big( \chi_j\stRes(\omega,\hat{\omega}) \Big) \;.
\end{split}
\end{align}
We also construct $I^+_{j,R,T}: \hh(Y,F) \rightarrow \Omega^{\bullet+1}(Z_{j,R},F)$ as follows:
for $\hat{\omega}\in \hh(Y,F)$,
\begin{equation}
\label{eq-def-IRTjplus}
I^+_{j,R,T}(\hat{\omega}) \big|_{Z_{j,0}} = 0  \;,\hspace{5mm}
I^+_{j,R,T}(\hat{\omega}) \big|_{IY_R}
= \chi_3 e^{Tf_T-T}du\wedge\hat{\omega} \;.
\end{equation}
Let
$P^{[-1,1]}_{j,R,T}: \Omega^\bullet(Z_{j,R},F) \rightarrow \mathscr{E}^{[-1,1]}_{j,R,T}$
be the orthogonal projection with respect to $\big\lVert\cdot\big\rVert_{Z_{j,R}}$.
We identify $C^{0,\bullet}_j$ with $\hha(Z_{j,\infty},F)\oplus\hh(Y,F)$.
We identify $C^{1,\bullet}_j$ with $\hh(Y,F)$.
We define
\begin{equation}
\label{eq-def-SRTj}
\mathscr{S}_{j,R,T}\Big|_{C^{0,\bullet}_j} = P^{[-1,1]}_{j,R,T}G^+_{j,R,T} \;,\hspace{5mm}
\mathscr{S}_{j,R,T}\Big|_{C^{1,\bullet}_j} = P^{[-1,1]}_{j,R,T}I^+_{j,R,T} \;.
\end{equation}

Now,
concerning the case $j=3$,
we construct $F_{3,R,T}: \hh(Y,F) \rightarrow \Omega^\bullet(IY_R,F)$ as follows:
for $\hat{\omega}\in\hh(Y,F)$,
\begin{equation}
\label{eq2-def-FRT3}
F_{3,R,T}(\hat{\omega}) \big|_{IY_R}
= e^{-Tf_T}\hat{\omega} \;.
\end{equation}
Let
$P_{3,R,T}: \Omega^\bullet(IY_R,F) \rightarrow \Ker\big(\DthreeRT\big)$
be the orthogonal projection with respect to $\big\lVert\cdot\big\rVert_{IY_R}$.
We identify
$H^\bullet(C^{\bullet,\bullet}_3,\partial) = H^0(C^{\bullet,\bullet}_3,\partial)$
with $\hh(Y,F)$.
We define
\begin{equation}
\mathscr{S}^H_{3,R,T} = P_{3,R,T}F_{3,R,T}:
H^\bullet(C^{\bullet,\bullet}_3,\partial) \rightarrow \Ker\big(\DthreeRT\big) \;.
\end{equation}
We construct $G^+_{3,R,T}: \hh(Y,F) \oplus \hh(Y,F) \rightarrow \Omega^\bullet(IY_R,F)$ as follows:
for $(\hat{\mu}_1,\hat{\mu}_2)\in\hh(Y,F)\oplus\hh(Y,F)$,
\begin{equation}
\label{eq2-def-FRT3plus}
G^+_{3,R,T}(\hat{\mu}_1,\hat{\mu}_2)
= e^{-Tf_T} \Big( \chi_1\hat{\mu}_1 + \chi_2\hat{\mu}_2 \Big) \;.
\end{equation}
We also construct $I^+_{3,R,T}: \hh(Y,F) \rightarrow \Omega^{\bullet+1}(IY_R,F)$ as follows:
for $\hat{\omega}\in \hh(Y,F)$,
\begin{equation}
\label{eq2-def-IRT3plus}
I^+_{3,R,T}(\hat{\omega})
= \chi_3 e^{Tf_T-T}du\wedge\hat{\omega} \;.
\end{equation}
Let
$P^{[-1,1]}_{3,R,T}: \Omega^\bullet(IY_R,F) \rightarrow \mathscr{E}^{[-1,1]}_{3,R,T}$
be the orthogonal projection with respect to $\big\lVert\cdot\big\rVert_{IY_R}$.
We identify $C^{0,\bullet}_3$ with $\hh(Y,F)\oplus\hh(Y,F)$,
and identify $C^{1,\bullet}_3$ with $\hh(Y,F)$.
We define
\begin{equation}
\label{eq-def-SRT3}
\mathscr{S}_{3,R,T}\Big|_{C^{0,\bullet}_3} = P^{[-1,1]}_{3,R,T}F^+_{3,R,T} \;,\hspace{5mm}
\mathscr{S}_{3,R,T}\Big|_{C^{1,\bullet}_3} = P^{[-1,1]}_{3,R,T}I^+_{3,R,T} \;.
\end{equation}

This completes the proof of Theorem \ref{thm-central-SRT}.
\end{proof}

\begin{rem}
\label{rem-j}
The proof of Theorem \ref{thm-central-SRT} may be summarized as follows:
all the results hold with
$\mathscr{S}^H_{R,T}$ replaced by $\mathscr{S}^H_{j,R,T}$ and
$\mathscr{S}_{R,T}$ replaced by $\mathscr{S}_{j,R,T}$.
\end{rem}

\section{Analytic torsion forms associated with a fibration}
\label{sect-tf}

The purpose of this section is to prove Theorem \ref{thm-central-tf}.
Many arguments in this section follow \cite{ble}.
This section is organized as follows.
In \textsection \ref{subsect-tf},
we decompose the analytic torsion form in question into two terms:
small time contribution and large time contribution.
In \textsection \ref{subsect-s-time},
we estimate the small time contribution.
In \textsection \ref{subsect-l-time},
we estimate the large time contribution.
Theorem \ref{thm-central-tf} will be proved in this subsection.

In the whole section,
we take $T=R^\kappa$,
where $\kappa\in]0,1/3[$ is a fixed constant.
For ease of notations,
we will systematically omit a parameter ($R$ or $T$)
as long as there is no confusion.

\subsection{Decomposition of analytic torsion forms}
\label{subsect-tf}

For $j=0,1,2,3$ and $R>0$,
we denote $\mathscr{F}_{j,R} = \Omega^\bullet(Z_{j,R},F)$,
which is a complex vector bundle of infinite dimension over $S$.
Let $\nabla^{\mathscr{F}_{j,R}}$ be the connection on $\mathscr{F}_{j,R}$ defined in \eqref{eq-def-n-F}.
Let $h^{\mathscr{F}_{j,R}}$ be the $L^2$-metric on $\mathscr{F}_{j,R}$
with respect to $g^{TZ_{j,R}}$ and $h^F$.
Let $\omega^{\mathscr{F}_{j,R}}\in\Omega^1\big(S,\mathrm{End}(\mathscr{F}_{j,R})\big)$
be as in \eqref{eq-def-omegaF} with $(\n^\mathscr{F},h^\mathscr{F})$ replaced by $(\n^{\mathscr{F}_{j,R}},h^{\mathscr{F}_{j,R}})$.
We may extend the construction above to $R=+\infty$ as follows,
\begin{equation}
\label{eq-def-Fjinf}
\mathscr{F}_{j,\infty} =
\big\{\omega\in\Omega^\bullet(Z_{j,\infty},F)\;:\;\omega \; \text{is }L^2\text{-integrable} \big\} \;.
\end{equation}
By \cite[Prop. 4.15]{bz},
$\omega^{\mathscr{F}_{j,\infty}}\in\Omega^1(S,\mathrm{End}(\mathscr{F}_{j,\infty}))$ is well-defined.
Let $\mathscr{D}_{j,R,t}$ be the operator in \eqref{eq-Dt} with $(\n^\mathscr{F},h^\mathscr{F})$ replaced by $(\n^{\mathscr{F}_{j,R}},h^{\mathscr{F}_{j,R}})$.
We have
\begin{equation}
\label{eq-fDRTt}
\mathscr{D}_{j,R,t} =
\sqrt{t} \big( \stdiffjRT - \diffjRT \big)
+ \omega^{\mathscr{F}_{j,R}}
- \frac{1}{\sqrt{t}} \hat{c}(\mathcal{T})
\in \Omega^\bullet\big(S,\mathrm{End}\big(\mathscr{F}_{j,R}\big)\big) \;.
\end{equation}
We remark that
\begin{equation}
\label{eq-square-tDRT}
\Big( \stdiffjRT - \diffjRT \Big)^2 =
- \Big( \stdiffjRT + \diffjRT \Big)^2 = - \DsjRT \;.
\end{equation}

We denote $\mathscr{T}_{j,R} = \mathscr{T}_{j,R,T}\big|_{T=R^\kappa}$,
where $\mathscr{T}_{j,R,T}$ was defined in \eqref{eq-def-TjRT}.
By \eqref{eq-def-atf},
we have
\begin{align}
\label{eq-def-analytic-tf}
\begin{split}
\mathscr{T}_{j,R} = - \int_0^{+\infty}
\bigg\{ & \varphi \tr\Big[(-1)^{N^{TZ}}\frac{N^{TZ}}{2}f'\big(\mathscr{D}_{j,R,t}\big)\Big]
- \frac{\chi'(Z_j,F)}{2} \\
& \hspace{5mm} - \Big(\frac{\dim Z \mathrm{rk}(F)\chi(Z_j)}{4} - \frac{\chi'(Z_j,F)}{2}\Big)f'\Big(\frac{i\sqrt{t}}{2}\Big) \bigg\}
\frac{dt}{t} \;.
\end{split}
\end{align}
Let $\mathscr{T}_{j,R}^\mathrm{S}$ (resp. $\mathscr{T}_{j,R}^\mathrm{L}$)
be as in \eqref{eq-def-analytic-tf}
with $\int_0^{+\infty}$ replaced by $\int_0^{R^{2-\kappa/2}}$ (resp. $\int_{R^{2-\kappa/2}}^{+\infty}$).
The following identity is obvious,
\begin{equation}
\label{eq-TSL}
\mathscr{T}_{j,R} =
\mathscr{T}_{j,R}^\mathrm{S} + \mathscr{T}_{j,R}^\mathrm{L} \;.
\end{equation}

\subsection{Small time contributions}
\label{subsect-s-time}

For $r\geqslant 1$ and an operator $A$ on a Hilbert space,
the Schauder $r$-norm of $A$ is defined as follows,
\begin{equation}
\big\lVert A \big\rVert_r = \Big( \tr \big[(A^*A)^{r/2}\big] \Big)^{1/r} \;.
\end{equation}
If $A$ is orthogonally diagonalizable,
we have
\begin{equation}
\label{eq-schauder-sp}
\big\lVert A \big\rVert_r = \bigg( \sum_{\lambda\in\Sp(A)} |\lambda|^r \bigg)^{1/r} \;.
\end{equation}
Let $\big\lVert A \big\rVert_\infty$ be the operator norm of $A$.
These norms satisfy the H{\"o}lder's inequality:
for $r_1,r_2,r_3\in[1,+\infty]$ with $1/r_1+1/r_2=1/r_3$,
we have
\begin{equation}
\label{eq-holder}
\big\lVert AB \big\rVert_{r_3} \leqslant
\big\lVert A \big\rVert_{r_1}\big\lVert B \big\rVert_{r_2} \;.
\end{equation}
Moreover,
if $A$ is of finite rank,
we have
\begin{equation}
\label{eq-schauder-frank}
\big\lVert A \big\rVert_r \leqslant
\big(\mathrm{rk}(A)\big)^{1/r} \big\lVert A \big\rVert_\infty \;.
\end{equation}
The proofs in this subsection involve sophisticated estimate of Schauder norms,
which follows \cite[\textsection 9]{ble}.

We remark that $\Sp\big(\mathscr{D}_{j,R,t}\big) \subseteq i\R$.

\begin{lemme}
\label{lem-fps-G}
There exist $\alpha,\beta>0$ such that
for $r\geqslant \dim Z+1$, $R \gg 1$, $0<t\leqslant R^{2-\kappa/2}$
and $\lambda\in\C$ with $\mathrm{Re}(\lambda)=\pm 1$,
we have
\begin{equation}
\label{eq-lem-fps-G}
\Big\lVert \big( \lambda - \mathscr{D}_{j,R,t} \big)^{-1} \Big\rVert_r
= \mathscr{O}\big(R^\beta\big) |\lambda| t^{-\alpha}  \;,\hspace{5mm}\text{for } j=0,1,2,3 \;.
\end{equation}
\end{lemme}
\begin{proof}
We only consider the case $j=0$.

We denote
\begin{equation}
\label{eq0-pf-lem-fps-G}
a = \Big(\lambda - \sqrt{t} \big( \stdiffRT - \diffRT \big)\Big)^{-1}  \;,\hspace{5mm}
b = \omega^{\mathscr{F}_R} - \frac{1}{\sqrt{t}} \hat{c}(\mathcal{T}) \;.
\end{equation}
Since $b\in\Omega^{>0}\big(S,\mathrm{End}(\mathscr{F}_R)\big)$,
we have
\begin{equation}
\label{eq1-pf-lem-fps-G}
\big( \lambda - \mathscr{D}_{R,t} \big)^{-1}
=  a + aba + \cdots + a(ba)^{\dim S} \;.
\end{equation}
The same technique as above was used in \cite[(2.45)]{bl}.
Note that $\Sp\big(\stdiffRT - \diffRT\big) \subseteq i\R$ and $\mathrm{Re}(\lambda)=\pm 1$,
by \eqref{eq0-pf-lem-fps-G},
we have
\begin{equation}
\label{eq01-pf-lem-fps-G}
\big\lVert a \big\rVert_\infty \leqslant 1 \;,\hspace{15mm}
\big\lVert b \big\rVert_\infty = \mathscr{O}\big(1\big) (1+t^{-1/2}) \;.
\end{equation}
We will temporarily treat $R$ and $T$ as independent parameters.
Note that $\Sp\big(\stdiffRT - \diffRT\big)$ depends continuously on $R,T$,
there exist $\big(\mu_{k,R,T}\in i\R\big)_{k\in\N}$ such that
\begin{itemize}
\item[-] for $R\geqslant 1$ and $T\geqslant 0$,
we have $\big\{\mu_{k,R,T}\;:\;k\in\N\big\} = \Sp\big(\stdiffRT - \diffRT\big)$;
\item[-] the function $(R,T) \mapsto \mu_{R,T}$ is continuous.
\end{itemize}
We will estimate $\mu_{k,R,T}$ with $k\in\N$ fixed.
For ease of notations,
we will omit the index $k$.
By \cite[(3.95)]{pzz}
and \eqref{eq-square-tDRT},
we have
\begin{equation}
\label{eq2-pf-lem-fps-G}
\big|\mu_{R,0}\big| \geqslant R^{-1}\big|\mu_{1,0}\big| \;.
\end{equation}
Similarly to the first identity in \eqref{eq-def-DRT-cylinder},
we have
\begin{equation}
\label{eq3a-pf-lem-fps-G}
\big( \stdiffRT - \diffRT \big)\big|_{IY_R}
= \hat{c}c \big( d^{Y,*} - d^Y \big)  - \hat{c} \; \frac{\partial}{\partial u} - R^{-1}Tf'_T c \;.
\end{equation}
By \eqref{eq-compare-f-fT}, \eqref{eq-square-tDRT}, \eqref{eq3a-pf-lem-fps-G} and the identity $f_T\big|_{Z_{1,0} \cup Z_{2,0}} = 0$,
there exists $\delta>0$ independent of $R,T,\mu_{R,T}$ such that
\begin{equation}
\label{eq3-pf-lem-fps-G}
\big|\mu_{R,T}-\mu_{R,0}\big| \leqslant \delta R^{-1}T \;.
\end{equation}
We consider the triangle spanned by
$\lambda,\sqrt{t}\mu_{R,T} \in \C$.
Let $A$ be its area.
As $\mathrm{Re}(\lambda)=\pm 1$,
we have
\begin{equation}
|\lambda| \big|\lambda-\sqrt{t}\mu_{R,T}\big| \geqslant
2A = \big|\sqrt{t}\mu_{R,T}\big| \;.
\end{equation}
Equivalently, we have
\begin{equation}
\label{eq4a-pf-lem-fps-G}
\big|\lambda-\sqrt{t}\mu_{R,T}\big|^{-1}
\leqslant |\lambda| \big|\sqrt{t}\mu_{R,T}\big|^{-1} \;.
\end{equation}
If $|\mu_{R,T}|\geqslant 1/R$,
by \eqref{eq2-pf-lem-fps-G}-\eqref{eq4a-pf-lem-fps-G},
we have
\begin{align}
\label{eq4-pf-lem-fps-G}
\begin{split}
\big|\lambda-\sqrt{t}\mu_{R,T}\big|^{-1}
& \leqslant |\lambda| t^{-1/2}\big|\mu_{R,0}\big|^{-1} \left| 1+\frac{\mu_{R,0}-\mu_{R,T}}{\mu_{R,T}} \right| \\
& \leqslant  |\lambda| t^{-1/2} \big|\mu_{1,0}\big|^{-1} R(1 + \delta T )
= \mathscr{O}\big(R^{1+\kappa}\big) |\lambda| t^{-1/2}  \big|\mu_{1,0}\big|^{-1} \;,
\end{split}
\end{align}
where $\mathscr{O}\big(R^{1+\kappa}\big)$ is uniform,
i.e.,  it is bounded by $CR^{1+\kappa}$ with $C>0$ independent of $R,T,\mu_{R,T}$.
On the other hand,
as $\mathrm{Re}(\lambda)=\pm 1$ and $\mu_{R,T} \in i\R$,
we have the obvious estimate
\begin{equation}
\label{eq5-pf-lem-fps-G}
\big|\lambda-\sqrt{t}\mu_{R,T}\big|^{-1} \leqslant 1 \;.
\end{equation}
By Theorem \ref{thm-central-spgap} and \eqref{eq-square-tDRT},
there exists $\alpha>0$ such that
\begin{equation}
\label{eq6a-pf-lem-fps-G}
\Sp \Big(i\big( \stdiffRT - \diffRT \big)\Big) \subseteq
\big]-\infty,-\alpha\sqrt{T}/R\big]
\cup \big[-1/R,1/R\big] \cup
\big[\alpha\sqrt{T}/R,+\infty\big[ \;.
\end{equation}
Moreover,
by Theorem \ref{thm-central-SRT},
the number of eigenvalues lying in $[-1/R,1/R]$ is constant for $R^\kappa = T \gg 1$.
Let $P_{R,T}^{\R\backslash\{0\}}: \mathscr{F}_{R}
\rightarrow \Big(\Ker\big(\DRT\big)\Big)^\perp$
be the orthogonal projection.
By \eqref{eq-schauder-sp},
the first identity in \eqref{eq0-pf-lem-fps-G}
and \eqref{eq4-pf-lem-fps-G}-\eqref{eq6a-pf-lem-fps-G},
we have
\begin{equation}
\label{eq6-pf-lem-fps-G}
\big\lVert a \big\rVert_r =
\mathscr{O}\big(1\big) +
\mathscr{O}\big(R^{1+\kappa}\big)
|\lambda| t^{-1/2} \Big\lVert \big(\DRT\big)^{-1} P_{R,T}^{\R\backslash\{0\}} \big|_{R=1,T=0}\Big\rVert_r \;.
\end{equation}
Since $r\geqslant \dim Z+1$,
by Weyl's law, we have
$\Big\lVert \big(\DRT\big)^{-1} P_{R,T}^{\R\backslash\{0\}} \big|_{R=1,T=0}\Big\rVert_r < +\infty$.
Then \eqref{eq6-pf-lem-fps-G} becomes
\begin{equation}
\label{eq7-pf-lem-fps-G}
\big\lVert a \big\rVert_r =
\mathscr{O}\big(1\big) +
\mathscr{O}\big(R^{1+\kappa}\big) |\lambda| t^{-1/2} \;.
\end{equation}

By \eqref{eq-holder} and
\eqref{eq1-pf-lem-fps-G},
we have
\begin{equation}
\label{eq8-pf-lem-fps-G}
\Big\lVert \big( \lambda - \mathscr{D}_{R,t} \big)^{-1} \Big\rVert_r
\leqslant \big\lVert a \big\rVert_r \sum_{k=0}^{\dim S} \big\lVert b \big\rVert_\infty^k \big\lVert a \big\rVert_\infty^k \;.
\end{equation}
From \eqref{eq01-pf-lem-fps-G}, \eqref{eq7-pf-lem-fps-G}, \eqref{eq8-pf-lem-fps-G}
and the assumption $0<t\leqslant R^{2-\kappa/2}$,
we obtain \eqref{eq-lem-fps-G}.
This completes the proof of Lemma \ref{lem-fps-G}.
\end{proof}

Let $\rho : \R \rightarrow [0,1]$ be a smooth even function such that
\begin{equation}
\rho(x) = 1 \hspace{5mm}\text{for}\hspace{2mm} |x|\leqslant 1/2 \;,\hspace{5mm}
\rho(x) = 0 \hspace{5mm}\text{for}\hspace{2mm} |x|\geqslant 1 \;.
\end{equation}
For $\varsigma>0$ and $z\in\mathbb{C}$,
set
\begin{align}
\begin{split}
F_\varsigma(z) \; & = (1+2z^2) \int_{-\infty}^{+\infty} \exp\left(\sqrt{2}x z\right)
\exp\left(-\frac{x^2}{2}\right)\rho(\sqrt{2\varsigma} x)\frac{d x}{\sqrt{2\pi}} \;,\\
G_\varsigma(z) \; & = (1+2z^2) \int_{-\infty}^{+\infty} \exp\left(\sqrt{2}x z\right)
\exp\left(-\frac{x^2}{2}\right)\big(1-\rho(\sqrt{2\varsigma} x)\big)\frac{d x}{\sqrt{2\pi}} \;.
\end{split}
\end{align}
The construction above follows \cite[Def. 13.3]{ble}.
We have
\begin{equation}
\label{eq-FplusG}
F_\varsigma(z)+ G_\varsigma(z) = f'(z) \;.
\end{equation}
Moreover,
$F_\varsigma\big|_{i\R}$ and $G_\varsigma\big|_{i\R}$ take real values,
and lie in the Schwartz space $\mathcal{S}(i\R)$.

\begin{prop}
\label{prop-fps-G}
There exists $\alpha>0$ such that
for $R \gg 1$ and $0<t\leqslant R^{2-\kappa/2}$,
we have
\begin{equation}
\label{eq-prop-fps-G}
\Big\lVert G_{tR^{-2+\kappa/4}} \big(\mathscr{D}_{j,R,t}\big) \Big\rVert_1
\leqslant \exp\big(-\alpha R^{2-\kappa/4}/t\big) \;,\hspace{5mm} j=0,1,2,3 \;.
\end{equation}
\end{prop}
\begin{proof}
We only consider the case $j=0$.

Due to the relation
$\frac{\partial^m}{\partial x^m} \exp\left(\sqrt{2}x z\right)
= 2^{m/2} z^m \exp\left(\sqrt{2}x z\right)$,
we can integrate by parts in the expression of $z^m G_\varsigma(z)$
and obtain that for $m\in\mathbb{N}$,
there exists $C_m>0$ such that
for $z\in\mathbb{C}$ with $|\mathrm{Re}(z)|\leqslant 1$, we have
\begin{equation}
\label{eq2-pf-prop-fps-G}
\big|z\big|^m\big|G_\varsigma(z)\big| \leqslant C_m \exp\left(-\frac{1}{32\varsigma}\right) \;.
\end{equation}

The function $G_\varsigma (z)$ is an even holomorphic function.
Therefore there exists a holomorphic function $\widetilde G_\varsigma(z)$
such that
\begin{equation}
\label{eq1-pf-prop-fps-G}
G_\varsigma (z) = \widetilde G_\varsigma (z^2)\;.
\end{equation}
Set
\begin{equation}
\label{eq3-pf-prop-fps-G}
U = \Big\{z\in\mathbb{C} \;:\;
4\mathrm{Re}(z)+ |\mathrm{Im}(z)|^2<4 \Big\}\;.
\end{equation}
We have
\begin{equation}
\label{eq4-pf-prop-fps-G}
\sqrt{U} := \Big\{z\in\mathbb{C}\;:\: z^2\in U \Big\} = \Big\{z\in\mathbb{C}\;:\;|\mathrm{Re}(z)|<1\Big\} \;.
\end{equation}
By \eqref{eq2-pf-prop-fps-G},
\eqref{eq1-pf-prop-fps-G} and
\eqref{eq4-pf-prop-fps-G},
for $z\in U$, we have
\begin{equation}
\label{eq5-pf-prop-fps-G}
\big|z\big|^{m/2}\big|\widetilde G_\varsigma(z)\big|
\leqslant C_m \exp\left(-\frac{1}{32\varsigma}\right) \;.
\end{equation}
The technique above follows \cite[\textsection 13 c)]{ble}.

For $r\in\mathbb{N}$,
let $\widetilde G_{r,\varsigma}(z)$ be the unique holomorphic function satisfying
\begin{equation}
\label{eq6-pf-prop-fps-G}
\frac{1}{r!}\frac{d^r}{d z^r} \widetilde G_{r,\varsigma}(z)
= \widetilde G_\varsigma(z) \;,\hspace{5mm}
\lim_{z\rightarrow-\infty} \widetilde G_{r,\varsigma}(z)
= 0 \;.
\end{equation}
By \eqref{eq5-pf-prop-fps-G}
and \eqref{eq6-pf-prop-fps-G},
for $m>2r$, there exists $C_{m,r}>0$ such that
for $z\in U$,
\begin{equation}
\label{eq7-pf-prop-fps-G}
\big|\widetilde G_{r,\varsigma}(z)\big|
\leqslant C_{m,r} \big|z\big|^{r-m/2} \exp\left(-\frac{1}{32\varsigma}\right) \;.
\end{equation}

In the rest of the proof,
we fix $\varsigma = tR^{-2+\kappa/4}$ and $\N \ni r \geqslant 1+ \dim Z /2$.
We have
\begin{equation}
\label{eq8-pf-prop-fps-G}
G_\varsigma \big(\mathscr{D}_{R,t}\big)
= \widetilde{G}_\varsigma \big(\mathscr{D}_{R,t}^2\big)
= \frac{1}{2\pi i} \int_{\partial U}
\widetilde G_{r,\varsigma}(\lambda)
\big( \lambda-\mathscr{D}_{R,t}^2 \big)^{-r-1} d \lambda \;.
\end{equation}
By \eqref{eq-holder}, we have
\begin{equation}
\label{eq9-pf-prop-fps-G}
\left\lVert \big( \lambda-\mathscr{D}_{R,t}^2 \big)^{-r-1}  \right\rVert_1
\leqslant \left\lVert \big( \sqrt{\lambda} - \mathscr{D}_{R,t} \big)^{-1} \right\rVert_{2r+2}^{r+1}
\left\lVert \big( - \sqrt{\lambda} - \mathscr{D}_{R,t} \big)^{-1} \right\rVert_{2r+2}^{r+1} \;.
\end{equation}
From Lemma \ref{lem-fps-G}
and \eqref{eq7-pf-prop-fps-G}-\eqref{eq9-pf-prop-fps-G},
we obtain \eqref{eq-prop-fps-G}.
This completes the proof of Proposition \ref{prop-fps-G}.
\end{proof}

Let $\chi'(Z_j,F)$ be as in \eqref{eq-def-chiprim}
with $Z$ replaced by $Z_j$.
Set
\begin{equation}
\label{eq-def-chiprim-sum}
\chi' = \sum_{j=0}^3 (-1)^{j(j-3)/2} \chi'(Z_j,F) \;.
\end{equation}

\begin{prop}
\label{prop-fps}
There exists $\alpha>0$ such that
for $R \gg 1$,
we have
\begin{align}
\label{eq-prop-fps}
\begin{split}
& \sum_{j=0}^3 (-1)^{j(j-3)/2} \mathscr{T}_{j,R}^\mathrm{S} \\
& = - \frac{\chi'}{2} \int_0^{R^{2-\kappa/2}}
\bigg\{ f'\Big(\frac{i\sqrt{t}}{2}\Big) - 1 \bigg\}
\frac{dt}{t}
+ \mathscr{O}\Big(\exp\big(-\alpha R^{\kappa/4}\big)\Big) \;.
\end{split}
\end{align}
\end{prop}
\begin{proof}
Let
\begin{align}
\label{eq0-pf-prop-fps}
\begin{split}
& F_{tR^{-2+\kappa/4}} \big(\mathscr{D}_{j,R,t}\big) (x,y) \\
& \hspace{15mm} \in \Big(\Lambda^\bullet\big(T^*Z_{j,R}\big)\otimes F\Big)_x
\otimes \Big(\Lambda^\bullet\big(T^*Z_{j,R}\big)\otimes F\Big)_y^*
\otimes \pi_{j,R}^*\big(\Lambda^\bullet(T^*S)\big)
\end{split}
\end{align}
be the integration kernel of the operator $F_{tR^{-2+\kappa/4}} \big(\mathscr{D}_{j,R,t}\big)$
with respect to the Riemannian volume form associated with $g^{TZ_{j,R}}$.
Let $d(\cdot,\cdot)$ be the distance function on $Z_{j,R}$.
By the finite propagation speed of the wave equation for $\mathscr{D}_{j,R,t}$
(cf. \cite[\textsection 2.6, Thm. 6.1]{tay},
\cite[Appendix D.2]{mm}),
we have
\begin{equation}
\label{eq1-pf-prop-fps}
F_\varsigma \big(\mathscr{D}_{j,R,t}\big) (x,y)=0 \hspace{5mm}
\text{for}\hspace{3mm} d\big(x,y\big) \geqslant \sqrt{t/\varsigma} \;.
\end{equation}
In particular,
we have
\begin{equation}
\label{eq2-pf-prop-fps}
F_{tR^{-2+\kappa/4}} \big(\mathscr{D}_{j,R,t}\big) (x,y)=0 \hspace{5mm}
\text{for}\hspace{3mm}
d\big(x,y\big) \geqslant R^{1-\kappa/8} \;,\;
t\leqslant R^{2-\kappa/2} \;.
\end{equation}
Using \eqref{eq2-pf-prop-fps} in the same way as in \cite[Thm. 4.5]{pzz},
we get
\begin{equation}
\label{eq3-pf-prop-fps}
\sum_{j=0}^3 (-1)^{j(j-3)/2}
\tr\Big[(-1)^{N^{TZ}}\frac{N^{TZ}}{2}F_{tR^{-2+\kappa/4}}\big(\mathscr{D}_{j,R,t}\big)\Big] = 0 \;,\hspace{5mm}
\text{for}\hspace{3mm}
t\leqslant R^{2-\kappa/2} \;.
\end{equation}

By Proposition \ref{prop-fps-G},
we have
\begin{equation}
\label{eq4-pf-prop-fps}
\int_0^{R^{2-\kappa/2}}
\tr\Big[(-1)^{N^{TZ}}\frac{N^{TZ}}{2}G_{tR^{-2+\kappa/4}}\big(\mathscr{D}_{j,R,t}\big)\Big] \frac{dt}{t}
= \mathscr{O}\Big(\exp\big(-\alpha R^{\kappa/4}\big)\Big) \;.
\end{equation}
From \eqref{eq-TSL},
\eqref{eq-FplusG},
\eqref{eq3-pf-prop-fps},
\eqref{eq4-pf-prop-fps}
and the identity
\begin{equation}
\chi(Z) - \chi(Z_1) - \chi(Z_2) + \chi(IY) = 0 \;,
\end{equation}
we get \eqref{eq-prop-fps}.
This completes the proof of Proposition \ref{prop-fps}.
\end{proof}

\subsection{Large time contributions}
\label{subsect-l-time}

Set
\begin{align}
\begin{split}
& U_{j,R,t} = \Big\{ \lambda\in\C \;:\; \big|\mathrm{Re}(\lambda)\big| < 1 \;,\;
\big|\mathrm{Im}(\lambda)\big| > t^{1/2}R^{-1+\kappa/4} \Big\} \\
& \cup \bigcup_{\mu\in i[-R^{-1},R^{-1}] \cap \Sp\big(\stdiffjRT - \diffjRT\big)}
\Big\{ \lambda\in\C \;:\; \big|\mathrm{Re}(\lambda)\big| < 1 \;,\; \big|\mathrm{Im}(\lambda-\sqrt{t}\mu)\big| < 1 \Big\} \;.
\end{split}
\end{align}
By Theorem \ref{thm-central-spgap} and \eqref{eq-square-tDRT},
for $R \gg 1$,
we have
\begin{equation}
\label{eq-Ua-cover-Sp}
\Sp \Big( \sqrt{t} \big( \stdiffjRT - \diffjRT \big) \Big) \subseteq
U_{j,R,t} \;.
\end{equation}

Set
\begin{equation}
\label{eq-def-tilde-DRT}
\widetilde{\mathscr{D}}_{j,R,t} =
\sqrt{t} P_{j,R,T}^{[-1,1]} \big( \stdiffjRT - \diffjRT \big)P_{j,R,T}^{[-1,1]} + P_{j,R,T}^{[-1,1]} \omega^{\mathscr{F}_{j,R}} P_{j,R,T}^{[-1,1]} \;.
\end{equation}

We fix $p,q\in\N$ such that
\begin{equation}
\label{eq-choose-p-q}
q > \dim Z \;,\hspace{5mm}
1-\frac{\kappa}{4} + \frac{3\kappa q}{2} - \frac{\kappa p}{4} \leqslant 0 \;.
\end{equation}

\begin{lemme}
\label{prop-ltime-proj}
There exists $\alpha>0$ such that
for $R \gg 1$, $t\geqslant R^{2-\kappa/2}$
and $\lambda\in\partial U_{j,R,t}$,
we have
\begin{equation}
\label{eq-prop-ltime-proj}
\Big\lVert \big( \lambda -\mathscr{D}_{j,R,t} \big)^{-p}
- P_{j,R,T}^{[-1,1]}\big( \lambda - \widetilde{\mathscr{D}}_{j,R,t} \big)^{-p}P_{j,R,T}^{[-1,1]} \Big\rVert_1
= \mathscr{O}\big(R^{1-\kappa/2}\big) |\lambda|^\alpha t^{-1/2} \;.
\end{equation}
\end{lemme}
\begin{proof}
The technique that we will apply is similar to \cite[Theorems 9.30]{b97}.
We only consider the case $j=0$.

Set
\begin{equation}
\label{eq0-pf-prop-ltime-proj}
\mathscr{D}^\oplus_{R,t} =
P_{R,T}^{[-1,1]} \mathscr{D}_{R,t} P_{R,T}^{[-1,1]} +
P_{R,T}^{\R\backslash[-1,1]} \mathscr{D}_{R,t} P_{R,T}^{\R\backslash[-1,1]} \;.
\end{equation}

\noindent \textit{Step 1}.
We show that
\begin{equation}
\label{eq1-pf-prop-ltime-proj}
\Big\lVert P_{R,T}^{\R\backslash[-1,1]} \big( \lambda - \mathscr{D}_{R,t}^\oplus \big)^{-p} P_{R,T}^{\R\backslash[-1,1]} \Big\rVert_1
= \mathscr{O}\big(1\big) |\lambda|^\alpha t^{-1/2} \;.
\end{equation}

We denote
\begin{align}
\label{eq11-pf-prop-ltime-proj}
\begin{split}
a & = P_{R,T}^{\R\backslash[-1,1]} \Big(\lambda - \sqrt{t} \big( \stdiffRT - \diffRT \big)\Big)^{-1} P_{R,T}^{\R\backslash[-1,1]} \;,\\
b & = P_{R,T}^{\R\backslash[-1,1]} \Big(\omega^{\mathscr{F}_R} - \frac{1}{\sqrt{t}} \hat{c}(\mathcal{T})\Big) P_{R,T}^{\R\backslash[-1,1]} \;.
\end{split}
\end{align}
By \eqref{eq-fDRTt} and \eqref{eq11-pf-prop-ltime-proj},
we have
\begin{equation}
\label{eq12-pf-prop-ltime-proj}
P_{R,T}^{\R\backslash[-1,1]} \big( \lambda - \mathscr{D}_{R,t}^\oplus \big)^{-1} P_{R,T}^{\R\backslash[-1,1]}
=  a + aba + \cdots + a(ba)^{\dim S} \;.
\end{equation}
By Theorem \ref{thm-central-spgap},
\eqref{eq4a-pf-lem-fps-G}
and the first identity in \eqref{eq11-pf-prop-ltime-proj},
we have
\begin{equation}
\label{eq13a-pf-prop-ltime-proj}
\big\lVert a \big\rVert_\infty =
\mathscr{O}\big(R^{1-\kappa/2}\big) |\lambda| t^{-1/2} \;.
\end{equation}
By the second identity in \eqref{eq11-pf-prop-ltime-proj},
we have
\begin{equation}
\label{eq13b-pf-prop-ltime-proj}
\big\lVert b \big\rVert_\infty = \mathscr{O}\big(1\big) \;.
\end{equation}
By \eqref{eq12-pf-prop-ltime-proj}-\eqref{eq13b-pf-prop-ltime-proj}
and the assumption $t\geqslant R^{2-\kappa/2}$,
we have
\begin{align}
\label{eq14-pf-prop-ltime-proj}
\begin{split}
& \Big\lVert P_{R,T}^{\R\backslash[-1,1]} \big(\lambda-\mathscr{D}_{R,t}^\oplus\big)^{-1} P_{R,T}^{\R\backslash[-1,1]} \Big\rVert_\infty \\
& = \mathscr{O}\big(R^{1-\kappa/2}\big) |\lambda|^{\dim S+1} t^{-1/2}
= \mathscr{O}\big(R^{-\kappa/4}\big) |\lambda|^{\dim S+1}  \;.
\end{split}
\end{align}
Similarly to \eqref{eq6-pf-lem-fps-G},
by \eqref{eq-schauder-sp}, \eqref{eq4-pf-lem-fps-G}
and the first identity in \eqref{eq11-pf-prop-ltime-proj},
we have
\begin{equation}
\label{eq15-pf-prop-ltime-proj}
\big\lVert a \big\rVert_q =
\mathscr{O}\big(R^{1+\kappa}\big) |\lambda| t^{-1/2} \;.
\end{equation}
By \eqref{eq-holder},
\eqref{eq12-pf-prop-ltime-proj}-\eqref{eq13b-pf-prop-ltime-proj}
and \eqref{eq15-pf-prop-ltime-proj},
we have
\begin{equation}
\label{eq16-pf-prop-ltime-proj}
\Big\lVert P_{R,T}^{\R\backslash[-1,1]} \big(\lambda-\mathscr{D}_{R,t}^\oplus\big)^{-1} P_{R,T}^{\R\backslash[-1,1]} \Big\rVert_q =
\mathscr{O}\big(R^{1+\kappa}\big) |\lambda|^{\dim S+1} t^{-1/2}  \;.
\end{equation}
By \eqref{eq-holder}, we have
\begin{align}
\label{eq17-pf-prop-ltime-proj}
\begin{split}
& \Big\lVert P_{R,T}^{\R\backslash[-1,1]} \big(\lambda-\mathscr{D}_{R,t}^\oplus\big)^{-p} P_{R,T}^{\R\backslash[-1,1]} \Big\rVert_1 \\
& \leqslant \Big\lVert P_{R,T}^{\R\backslash[-1,1]} \big(\lambda-\mathscr{D}_{R,t}^\oplus\big)^{-1} P_{R,T}^{\R\backslash[-1,1]} \Big\rVert_q^q \,
\Big\lVert P_{R,T}^{\R\backslash[-1,1]} \big(\lambda-\mathscr{D}_{R,t}^\oplus\big)^{-1} P_{R,T}^{\R\backslash[-1,1]} \Big\rVert_\infty^{p-q} \;.
\end{split}
\end{align}
From
\eqref{eq-choose-p-q},
\eqref{eq14-pf-prop-ltime-proj},
\eqref{eq16-pf-prop-ltime-proj}
and \eqref{eq17-pf-prop-ltime-proj},
we obtain \eqref{eq1-pf-prop-ltime-proj}.

\noindent \textit{Step 2}.
We show that
\begin{equation}
\label{eq2-pf-prop-ltime-proj}
\Big\lVert \big( \lambda -\mathscr{D}_{R,t} \big)^{-p}
- \big( \lambda - \mathscr{D}_{R,t}^\oplus \big)^{-p} \Big\rVert_1
= \mathscr{O}\big(R^{1-\kappa/2}\big) |\lambda|^\alpha t^{-1/2} \;.
\end{equation}

Since $\stdiffRT - \diffRT$ commutes with $P^{[-1,1]}_{R,T}$,
we have
\begin{align}
\label{eq21-pf-prop-ltime-proj}
\begin{split}
& P_{R,T}^{[-1,1]} \mathscr{D}_{R,t} P_{R,T}^{\R\backslash[-1,1]}
= P_{R,T}^{[-1,1]} \omega^{\mathscr{F}_R} P_{R,T}^{\R\backslash[-1,1]} -
\frac{1}{\sqrt{t}} P^{[-1,1]}_{R,T} \hat{c}(\mathcal{T}) P_{R,T}^{\R\backslash[-1,1]} \;,\\
& P_{R,T}^{\R\backslash[-1,1]} \mathscr{D}_{R,t} P_{R,T}^{[-1,1]}
= P_{R,T}^{\R\backslash[-1,1]} \omega^{\mathscr{F}_R} P^{[-1,1]}_{R,T} -
\frac{1}{\sqrt{t}} P_{R,T}^{\R\backslash[-1,1]} \hat{c}(\mathcal{T}) P_{R,T}^{[-1,1]} \;.
\end{split}
\end{align}
As a consequence, we have
\begin{equation}
\label{eq22-pf-prop-ltime-proj}
\Big\lVert P_{R,T}^{[-1,1]} \mathscr{D}_{R,t} P_{R,T}^{\R\backslash[-1,1]} \Big\rVert_\infty
= \mathscr{O}\big(1\big) \;,\hspace{10mm}
\Big\lVert P_{R,T}^{\R\backslash[-1,1]} \mathscr{D}_{R,t} P_{R,T}^{[-1,1]} \Big\rVert_\infty
= \mathscr{O}\big(1\big) \;.
\end{equation}
The same argument as in \eqref{eq12-pf-prop-ltime-proj}-\eqref{eq14-pf-prop-ltime-proj} yields
\begin{equation}
\label{eq23-pf-prop-ltime-proj}
\Big\lVert P_{R,T}^{[-1,1]} \big(\lambda-\mathscr{D}_{R,t}^\oplus \big)^{-1} P_{R,T}^{[-1,1]} \Big\rVert_\infty
= \mathscr{O}\big(1\big) \;.
\end{equation}
We denote
\begin{align}
\label{eq24-pf-prop-ltime-proj}
\begin{split}
& \mathcal{A} = \Big\{ P_{R,T}^{[-1,1]}\big(\lambda-\mathscr{D}_{R,t}^\oplus \big)^{-1}P_{R,T}^{[-1,1]} \;,\;
P_{R,T}^{\R\backslash[-1,1]}\big(\lambda-\mathscr{D}_{R,t}^\oplus\big)^{-1}P_{R,T}^{\R\backslash[-1,1]} \Big\} \;,\\
& \mathcal{B} = \Big\{ P_{R,T}^{[-1,1]} \mathscr{D}_{R,t} P_{R,T}^{\R\backslash[-1,1]} \;,\; P_{R,T}^{\R\backslash[-1,1]} \mathscr{D}_{R,t} P_{R,T}^{[-1,1]}  \Big\} \;.
\end{split}
\end{align}
We have
\begin{align}
\label{eq25-pf-prop-ltime-proj}
\begin{split}
\big(\lambda -\mathscr{D}_{R,t} \big)^{-1}
- \big( \lambda - \mathscr{D}_{R,t}^\oplus \big)^{-1}
& = \sum_{k=1}^{\dim S}\sum_{a_i\in\mathcal{A},b_i\in\mathcal{B}}
a_0b_1a_1b_2a_2 \cdots b_ka_k \\
& = \sum_{k=1}^{\dim S}\sum_{a_i\in\mathcal{A},b_i\in\mathcal{B},a_0\neq a_1}
a_0b_1a_1b_2a_2 \cdots b_ka_k \;.
\end{split}
\end{align}
By \eqref{eq14-pf-prop-ltime-proj}
and \eqref{eq22-pf-prop-ltime-proj}-\eqref{eq25-pf-prop-ltime-proj},
we have
\begin{equation}
\label{eq26-pf-prop-ltime-proj}
\Big\lVert \big( \lambda -\mathscr{D}_{R,t} \big)^{-1}
- \big( \lambda - \mathscr{D}_{R,t}^\oplus \big)^{-1} \Big\rVert_\infty
= \mathscr{O}\big(R^{1-\kappa/2}\big) |\lambda|^{(\dim S+1)^2} t^{-1/2}  \;.
\end{equation}
By \eqref{eq14-pf-prop-ltime-proj},
\eqref{eq23-pf-prop-ltime-proj}
and \eqref{eq26-pf-prop-ltime-proj},
we have
\begin{equation}
\label{eq27-pf-prop-ltime-proj}
\Big\lVert \big( \lambda -\mathscr{D}_{R,t} \big)^{-p}
- \big( \lambda - \mathscr{D}_{R,t}^\oplus \big)^{-p} \Big\rVert_\infty
= \mathscr{O}\big(R^{1-\kappa/2}\big) |\lambda|^{p(\dim S+1)^2} t^{-1/2} \;.
\end{equation}

By \eqref{eq25-pf-prop-ltime-proj},
we have
\begin{equation}
\label{eq28-pf-prop-ltime-proj}
\mathrm{Im}\Big(\big( \lambda -\mathscr{D}_{R,t} \big)^{-p}
- \big( \lambda - \mathscr{D}_{R,t}^\oplus \big)^{-p}\Big) \subseteq
\sum_{k=1}^p \sum_{a\in\mathcal{A},b\in\mathcal{B}}
\mathrm{Im} \big(a^k b\big) \;,
\end{equation}
whose dimension is bounded by
$4p \dim\big(\mathscr{E}^{[-1,1]}_{0,R,T}\big)\dim\big(\Lambda^\bullet(T^*S)\big)$.
Hence
\begin{equation}
\label{eq29-pf-prop-ltime-proj}
\mathrm{rk}\Big(\big( \lambda -\mathscr{D}_{R,t} \big)^{-p}
- \big( \lambda - \mathscr{D}_{R,t}^\oplus \big)^{-p}\Big)
\leqslant 4p \dim\big(\mathscr{E}^{[-1,1]}_{0,R,T}\big) \dim\big(\Lambda^\bullet(T^*S)\big) \;.
\end{equation}
From \eqref{eq-schauder-frank},
\eqref{eq27-pf-prop-ltime-proj}
and \eqref{eq29-pf-prop-ltime-proj},
we obtain \eqref{eq2-pf-prop-ltime-proj}.

\noindent \textit{Step 3}.
We show that
\begin{equation}
\label{eq3-pf-prop-ltime-proj}
\Big\lVert P_{R,T}^{[-1,1]}\big( \lambda - \mathscr{D}_{R,t}^\oplus \big)^{-p}P_{R,T}^{[-1,1]}
- P_{R,T}^{[-1,1]}\big( \lambda - \widetilde{\mathscr{D}}_{R,t} \big)^{-p}P_{R,T}^{[-1,1]} \Big\rVert_1
= \mathscr{O}\big(1\big) t^{-1/2} \;.
\end{equation}

Using the identity
\begin{equation}
P^{[-1,1]}_{R,T} \Big(
\mathscr{D}_{R,t}^\oplus - \widetilde{\mathscr{D}}_{R,t}
\Big) P^{[-1,1]}_{R,T} =
-\frac{1}{\sqrt{t}} P^{[-1,1]}_{R,T} \hat{c}(\mathcal{T}) P^{[-1,1]}_{R,T}
\end{equation}
and proceeding in the same way as
\eqref{eq12-pf-prop-ltime-proj}-\eqref{eq14-pf-prop-ltime-proj},
we can show that
\begin{equation}
\label{eq31-pf-prop-ltime-proj}
\Big\lVert P_{R,T}^{[-1,1]}\big( \lambda - \mathscr{D}_{R,t}^\oplus \big)^{-p}P_{R,T}^{[-1,1]}
- P_{R,T}^{[-1,1]}\big( \lambda - \widetilde{\mathscr{D}}_{R,t} \big)^{-p}P_{R,T}^{[-1,1]} \Big\rVert_\infty
= \mathscr{O}\big(1\big) t^{-1/2} \;.
\end{equation}
Since the rank of the operator in \eqref{eq31-pf-prop-ltime-proj}
is bounded by $\dim\big(\mathscr{E}^{[-1,1]}_{0,R,T}\big)\dim\big(\Lambda^\bullet(T^*S)\big)$,
from \eqref{eq-schauder-frank}
and \eqref{eq31-pf-prop-ltime-proj},
we obtain \eqref{eq3-pf-prop-ltime-proj}.

By Steps 1-3,
we have \eqref{eq-prop-ltime-proj}.
This completes the proof of Lemma \ref{prop-ltime-proj}.
\end{proof}

Recall that the complexes $(C^{\bullet,\bullet}_j,\partial)$ with $j=0,1,2,3$ were defined by
\eqref{eq1-def-complex},
\eqref{eq-def-Cj} and
\eqref{eq-complex-evaluation}.
We denote
\begin{equation}
\label{eq-def-chiprimC}
\chi'(C^{\bullet,\bullet}_j) =
\sum_{p=0}^1 \sum_{q=0}^{\dim Z} (-1)^{p+q}(p+q)\dim C^{p,q}_j \;.
\end{equation}
Set
\begin{align}
\label{eq-def-analytic-tf-tilde}
\begin{split}
\widetilde{\mathscr{T}}_{j,R} =
- \int_0^{+\infty} \bigg\{ \varphi \tr \Big[ & (-1)^{N^{TZ}}\frac{N^{TZ}}{2}f'\big(\widetilde{\mathscr{D}}_{j,R,t}\big)\Big]
- \frac{1}{2}\chi'(Z_j,F) \\
& - \frac{1}{2}\big(\chi'(C^{\bullet,\bullet}_j) - \chi'(Z_j,F)\big) f'\Big(\frac{i\sqrt{t}}{2}\Big) \bigg\} \frac{dt}{t} \;.
\end{split}
\end{align}
By \cite[Remark 2.21]{bl}, Theorem \ref{thm-central-SRT} and \eqref{eq-def-tilde-DRT},
$\widetilde{\mathscr{T}}_{j,R}$ is well-defined.

\begin{prop}
\label{prop-fwedge-proj}
For $R \gg 1$,
we have
\begin{equation}
\label{eq-prop-fwedge-proj}
\sum_{j=0}^3 (-1)^{j(j-3)/2}\mathscr{T}_{j,R} =
\sum_{j=0}^3 (-1)^{j(j-3)/2}\widetilde{\mathscr{T}}_{j,R} + \mathscr{O}\big(R^{-\kappa/4}\big) \;.
\end{equation}
\end{prop}
\begin{proof}
Let $\widetilde{\mathscr{T}}_{j,R}^\mathrm{S}$ (resp. $\widetilde{\mathscr{T}}_{j,R}^\mathrm{L}$)
be as in \eqref{eq-def-analytic-tf-tilde}
with $\int_0^{+\infty}$ replaced by $\int_0^{R^{2-\kappa/2}}$ (resp. $\int_{R^{2-\kappa/2}}^{+\infty}$).
The following identity is obvious,
\begin{equation}
\label{eq1-pf-prop-fwedge-proj}
\widetilde{\mathscr{T}}_{j,R} =
\widetilde{\mathscr{T}}_{j,R}^\mathrm{S} + \widetilde{\mathscr{T}}_{j,R}^\mathrm{L} \;.
\end{equation}

We fix $\N \ni r > \dim Z$.
Let
$f_r: \overline{U_{j,R,t}} \rightarrow \C$
be a holomorphic function satisfying
\begin{equation}
\label{eq2-pf-prop-fwedge-proj}
\frac{1}{r!}\frac{d^r}{d z^r} f_r(z) = f'(z) \;,\hspace{5mm}
\lim_{z\rightarrow \pm i\infty} f_r(z) = 0 \;.
\end{equation}
We further assume that
for each bounded connected component $V \subseteq U_{j,R,t}$,
there exists $z\in V$ such that $f_r(z)= 0$.
Note that the bounded (resp. unbounded) connected components of $U_{j,R,t}$ cover the small (resp. large) eigenvalues of $\sqrt{t} \big( \stdiffjRT - \diffjRT \big)$,
by Theorem \ref{thm-central-SRT} and \eqref{eq-Ua-cover-Sp},
the total area of the bounded connected components of $U_{j,R,t}$ is bounded by a universal constant.
Then,
there exists $C>0$ such that for any $z\in \overline{U_{j,R,t}}$,
we have
\begin{equation}
\label{eq30-pf-prop-fwedge-proj}
\big|f_r(z)\big| \leqslant C e^{-|z|} \;.
\end{equation}
By \eqref{eq-Ua-cover-Sp} and \eqref{eq2-pf-prop-fwedge-proj},
for $R \gg 1$,
we have
\begin{align}
\label{eq3-pf-prop-fwedge-proj}
\begin{split}
f'\big(\mathscr{D}_{j,R,t}\big) = & \frac{1}{2\pi i} \int_{\partial U_{j,R,t}}
f_r(\lambda) \big( \lambda-\mathscr{D}_{j,R,t} \big)^{-r-1} d \lambda \;,\\
f'\big(\widetilde{\mathscr{D}}_{j,R,t}\big) = & \frac{1}{2\pi i} \int_{\partial U_{j,R,t}}
f_r(\lambda) \big( \lambda-\widetilde{\mathscr{D}}_{j,R,t} \big)^{-r-1} d \lambda \;.
\end{split}
\end{align}
By Lemma \ref{prop-ltime-proj}, \eqref{eq30-pf-prop-fwedge-proj} and \eqref{eq3-pf-prop-fwedge-proj},
for $R \gg 1$ and $t\geqslant R^{2-\kappa/2}$,
we have
\begin{equation}
\label{eq4a-pf-prop-fwedge-proj}
\Big\lVert f'\big(\mathscr{D}_{j,R,t}\big) - f'\big(\widetilde{\mathscr{D}}_{j,R,t}\big) \Big\rVert_1
= \mathscr{O}\big(R^{1-\kappa/2}\big) t^{-1/2} \;.
\end{equation}
On the other hand,
by \eqref{eq-mv-top}, \eqref{eq-complex-evaluation}, \eqref{eq-def-analytic-tf}, \eqref{eq-TSL}, \eqref{eq-def-analytic-tf-tilde} and \eqref{eq1-pf-prop-fwedge-proj},
we have
\begin{align}
\label{eq4b-pf-prop-fwedge-proj}
\begin{split}
& \sum_{j=0}^3 (-1)^{j(j-3)/2} \Big( \mathscr{T}_{j,R}^\mathrm{L} - \widetilde{\mathscr{T}}_{j,R}^\mathrm{L} \Big) \\
& = - \sum_{j=0}^3 (-1)^{j(j-3)/2} \int_{R^{2-\kappa/2}}^{+\infty} \varphi \tr\Big[(-1)^{N^{TZ}}\frac{N^{TZ}}{2}
\Big( f'\big(\mathscr{D}_{j,R,t}\big)-f'\big(\widetilde{\mathscr{D}}_{j,R,t}\big) \Big) \Big] \frac{dt}{t} \;.
\end{split}
\end{align}
By \eqref{eq4a-pf-prop-fwedge-proj} and \eqref{eq4b-pf-prop-fwedge-proj},
we have
\begin{equation}
\label{eq4-pf-prop-fwedge-proj}
\sum_{j=0}^3 (-1)^{j(j-3)/2} \mathscr{T}_{j,R}^\mathrm{L} =
\sum_{j=0}^3 (-1)^{j(j-3)/2} \widetilde{\mathscr{T}}_{j,R}^\mathrm{L} + \mathscr{O}\big(R^{-\kappa/4}\big) \;.
\end{equation}

By Theorem \ref{thm-central-SRT} and \eqref{eq-square-tDRT},
for $R \gg 1$,
we have
\begin{equation}
\Sp \Big( \big( \stdiffjRT - \diffjRT \big)\big|_{\mathscr{E}^{[-1,1]}_{j,R,T}} \Big)
\subseteq i[-\exp(-R^\kappa),\exp(-R^\kappa)] \;,
\end{equation}
where the exponential term comes from $e^{-T} = e^{-R^\kappa}$ in \eqref{eq4-thm-central-SRT}.
As a consequence,
we have
\begin{equation}
\label{eq5-pf-prop-fwedge-proj}
\Big\lVert \sqrt{t} P_{j,R,T}^{[-1,1]} \big( \stdiffjRT - \diffjRT \big) P_{j,R,T}^{[-1,1]} \Big\rVert_\infty
\leqslant  \exp(-R^\kappa) t^{1/2} \;.
\end{equation}
By \eqref{eq-def-tilde-DRT}
and \eqref{eq5-pf-prop-fwedge-proj},
we have
\begin{align}
\label{eq6-pf-prop-fwedge-proj}
\begin{split}
\tr\Big[(-1)^{N^{TZ}}\frac{N^{TZ}}{2}f'\big(\widetilde{\mathscr{D}}_{j,R,t}\big)\Big]
- \tr\Big[(-1)^{N^{TZ}}\frac{N^{TZ}}{2}f'\big( P_{j,R,T}^{[-1,1]}\omega^{\mathscr{F}_{j,R}}P_{j,R,T}^{[-1,1]}\big)\Big] & \\
= \mathscr{O}\big(\exp(-R^\kappa)\big) t^{1/2}  & \;.
\end{split}
\end{align}
Note that $f'$ is an even function,
by the proof of \cite[Prop. 1.3]{bl},
the function
\begin{equation}
\label{eq7a-pf-prop-fwedge-proj}
\R \ni s \mapsto
\tr\Big[(-1)^{N^{TZ}}\frac{N^{TZ}}{2}f'\big(sP_{j,R,T}^{[-1,1]}\omega^{\mathscr{F}_{j,R}}P_{j,R,T}^{[-1,1]}\big)\Big] \in Q^S
\end{equation}
is constant.
Taking $s=0,1$ in \eqref{eq7a-pf-prop-fwedge-proj} and using the identity $f'(0) = 1$,
we get
\begin{align}
\label{eq7b-pf-prop-fwedge-proj}
\begin{split}
\tr\Big[(-1)^{N^{TZ}}\frac{N^{TZ}}{2}f'\big(P_{j,R,T}^{[-1,1]}\omega^{\mathscr{F}_{j,R}}P_{j,R,T}^{[-1,1]}\big)\Big]
= \tr\Big[(-1)^{N^{TZ}}\frac{N^{TZ}}{2}P_{j,R,T}^{[-1,1]}\Big] \;.
\end{split}
\end{align}
By Theorem \ref{thm-central-SRT}, \eqref{eq-def-chiprimC} and \eqref{eq7b-pf-prop-fwedge-proj},
we have
\begin{equation}
\label{eq7-pf-prop-fwedge-proj}
\tr\Big[(-1)^{N^{TZ}}\frac{N^{TZ}}{2}f'\big(P_{j,R,T}^{[-1,1]}\omega^{\mathscr{F}_{j,R}}P_{j,R,T}^{[-1,1]}\big)\Big]
= \frac{1}{2}\chi'(C^{\bullet,\bullet}_j) \;.
\end{equation}
By \eqref{eq-def-chiprim-sum}, \eqref{eq-def-analytic-tf-tilde}, \eqref{eq1-pf-prop-fwedge-proj}, \eqref{eq6-pf-prop-fwedge-proj} and \eqref{eq7-pf-prop-fwedge-proj},
we have
\begin{equation}
\label{eq8-pf-prop-fwedge-proj}
\sum_{j=0}^3 (-1)^{j(j-3)/2} \widetilde{\mathscr{T}}_{j,R}^\mathrm{S}
= - \frac{\chi'}{2} \int_0^{R^{2-\kappa/2}}
\bigg\{ f'\Big(\frac{i\sqrt{t}}{2}\Big) - 1 \bigg\}
\frac{dt}{t}
+ \mathscr{O}\big(R^{-\kappa/4}\big) \;.
\end{equation}
From Proposition \ref{prop-fps},
\eqref{eq-TSL},
\eqref{eq1-pf-prop-fwedge-proj},
\eqref{eq4-pf-prop-fwedge-proj}
and \eqref{eq8-pf-prop-fwedge-proj},
we obtain \eqref{eq-prop-fwedge-proj}.
This completes the proof of Proposition \ref{prop-fwedge-proj}.
\end{proof}

We denote $\mathscr{G} = \Omega^\bullet(Y,F)$,
which is a complex vector bundle of infinite dimension over $S$.
We define the connection $\nabla^{\mathscr{G}}$
on $\mathscr{G}$
in the same way as in \eqref{eq-def-n-F}.
Let $h^\mathscr{G}$
be the $L^2$-metric on $\mathscr{G}$
with respect to $h^F$ and $g^{TY}$.
Let $\omega^\mathscr{G} \in \Omega^1\big(S,\mathrm{End}(\mathscr{G})\big)$
be as in \eqref{eq-def-omegaF}
with $(\n^\mathscr{F},h^\mathscr{F})$ replaced by $(\n^\mathscr{G},h^\mathscr{G})$.

Let $\nabla^{V^\bullet}$ be the canonical flat connection on $V^\bullet = H^\bullet(Y,F)$ (see \cite[Def. 2.4]{bl}).
We identify $V^\bullet$ with $\hh(Y,F)\subseteq\mathscr{G}$ via the Hodge theorem.
Recall that $h^{V^\bullet}$ is the $L^2$-metric on $V^\bullet$ defined after \eqref{eq-def-hW}.
Let $\omega^{V^\bullet} \in \Omega^1\big(S,\mathrm{End}(V^\bullet)\big)$
be as in \eqref{intro-def-omegaF} with $(\n^F,h^F)$ replaced by $(\n^V,h^V)$.
Let $P^{V^\bullet}: \mathscr{G}\rightarrow V^\bullet$
be the orthogonal projection with respect to $h^\mathscr{G}$.
By \cite[Prop. 3.14]{bl},
we have
\begin{equation}
\label{eq-omega-V}
\omega^{V^\bullet} = P^{V^\bullet} \omega^\mathscr{G} P^{V^\bullet} \;.
\end{equation}
Recall that the flat sub vector bundles
$V^\bullet_j = \mathrm{Im}\big(\tau_j : W^\bullet_j \rightarrow V^\bullet\big) \subseteq V^\bullet$ with $j=1,2$
were defined by \eqref{eq-def-V1-V2} and \eqref{eq-complex-evaluation}.
The identity \eqref{eq-omega-V} also holds
with $V^\bullet$ replaced by $V^\bullet_j$.
Let
\begin{equation}
\tau^\perp_j : K^{\bullet,\perp}_j \rightarrow V^\bullet_j
\end{equation}
be the restriction of $\tau_j: W^\bullet_j \rightarrow V^\bullet_j \subseteq V^\bullet$ defined in \eqref{eq-complex-evaluation}
to $K^{\bullet,\perp}_j\subseteq W^\bullet_j$ defined in \eqref{eq-def-WKK},
which is bijective.
Set
\begin{equation}
\omega^{K^{\bullet,\perp}_j} =
\big(\tau^\perp_j\big)^{-1} \circ \omega^{V^\bullet_j} \circ \tau^\perp_j
\in\Omega^1\big(S,\mathrm{End}(K^{\bullet,\perp}_j)\big) \;.
\end{equation}

For $j=1,2$,
let $\nabla^{W^\bullet_j}$ be the canonical flat connection on $W^\bullet_j = H^\bullet(Z_j,F)$ (see \cite[Def. 2.4]{bl}).
The sub vector bundle $K^\bullet_j \subseteq W^\bullet_j$ is preserved by $\nabla^{W^\bullet_j}$.
Then $\nabla^{K^\bullet_j} := \nabla^{W^\bullet_j}\big|_{K^\bullet_j}$ is a flat connection on $K^\bullet_j$.
The Hermitian metric $h^{K^\bullet_j}$ was constructed in \eqref{eq-def-hK}.
Let $\omega^{K^\bullet_j} \in \Omega^1\big(S,\mathrm{End}(K^\bullet_j)\big)$
be as in \eqref{intro-def-omegaF} with $(\n^F,h^F)$ replaced by $(\n^{K^\bullet_j},h^{K^\bullet_j})$.
We identify $K^\bullet_j$ with $\mathscr{F}_{j,\infty}\cap\Ker\big(\Djinf\big)$ via the map \eqref{eq-hha2coh}.
Let $P^{K^\bullet_j}: \mathscr{F}_{j,\infty}\rightarrow K^\bullet_j$
be the orthogonal projection.
By \cite[Prop. 3.14]{bl},
we have
\begin{equation}
\label{eq-omega-K}
\omega^{K^\bullet_j} = P^{K^\bullet_j} \omega^{\mathscr{F}_{j,\infty}} P^{K^\bullet_j} \;.
\end{equation}

Recall that
$C^{\bullet,\bullet}_0 =
W^\bullet_1 \oplus W^\bullet_2 \oplus V^\bullet =
K^\bullet_1 \oplus K^{\bullet,\perp}_1 \oplus
K^\bullet_2 \oplus K^{\bullet,\perp}_2 \oplus
V^\bullet$.
Set
\begin{equation}
\label{eq-def-omegaC}
\omega^{C^{\bullet,\bullet}_0} =
\omega^{K^\bullet_1} \oplus \omega^{K^{\bullet,\perp}_1} \oplus
\omega^{K^\bullet_2} \oplus \omega^{K^{\bullet,\perp}_2} \oplus
\omega^{V^\bullet}
\in \Omega^1\big(S,\mathrm{End}(C^{\bullet,\bullet}_0)\big) \;.
\end{equation}
For $j=1,2,3$,
we may construct $\omega^{C^{\bullet,\bullet}_j}$ in the same way.

Recall that the bijection
$\mathscr{S}_{j,R,T}: C^{\bullet,\bullet}_j
\rightarrow \mathscr{E}^{[-1,1]}_{j,R,T} \subseteq \Omega^\bullet(Z_{j,R},F)$
was defined in \eqref{eq-def-SRT}, \eqref{eq-def-SRTj} and \eqref{eq-def-SRT3}.

\begin{lemme}
\label{lem-SRT-omega}
For $j=0,1,2,3$
and $R \gg 1$,
we have
\begin{equation}
\label{eq-lem-SRT-omega}
\mathscr{S}_{j,R,T}^{-1} \circ
\Big(P^{[-1,1]}_{j,R,T} \omega^{\mathscr{F}_{j,R}} P^{[-1,1]}_{j,R,T}\Big)
\circ \mathscr{S}_{j,R,T} =
\omega^{C^{\bullet,\bullet}_j}
+ \mathscr{O}_{R,T} \big(R^{-1/2+\kappa/4}\big) \;,
\end{equation}
where $\mathscr{O}_{R,T}\big(\cdot\big)$ was defined in the paragraph above Theorem \ref{thm-central-SRT}.
\end{lemme}
\begin{proof}
We only prove the case $j=0$.

We consider $\sigma \in C^{\bullet,\bullet}$.
All the estimates in the proof of Proposition \ref{prop-SRT-metric} hold
with $\big\lVert\sigma\big\rVert^2_{Z_R}$ replaced by $\big\langle\sigma,\omega^{\mathscr{F}_R}\sigma\big\rangle_{Z_R}$.
Hence \eqref{eq-prop-SRT-metric} holds
with $\big\lVert\sigma\big\rVert^2_{Z_R}$ replaced by $\big\langle\sigma,\omega^{\mathscr{F}_R}\sigma\big\rangle_{Z_R}$, i.e.,
\begin{equation}
\label{eq11-pf-lem-SRT-omega}
\Big\langle \mathscr{S}_{R,T}(\sigma),
\omega^{\mathscr{F}_R}\mathscr{S}_{R,T}(\sigma) \Big\rangle_{Z_R}
= \big\langle \sigma,\omega^{C^{\bullet,\bullet}}\sigma \big\rangle_{R,T}
+ \mathscr{O}\big(R^{-1/2+\kappa/4}\big) \big\lVert\sigma\big\rVert^2_{R,T} \;.
\end{equation}
From the polarization identity, Proposition \ref{prop-SRT-metric} and \eqref{eq11-pf-lem-SRT-omega},
we obtain \eqref{eq-lem-SRT-omega} with $j=0$.
This completes the proof of Lemma \ref{lem-SRT-omega}.
\end{proof}

For $j=0,1,2,3$,
let $\n^{C^{\bullet,\bullet}_j}$ be the flat connection on  $C^{\bullet,\bullet}_j$
induced by $\n^{W^\bullet_1}$, $\n^{W^\bullet_2}$ and $\n^{V^\bullet}$.
Let $\omega^{C^{\bullet,\bullet}_j}_{R,T} \in \Omega^\bullet\big(S,\mathrm{End}(C^{\bullet,\bullet}_j)\big)$
be as in \eqref{intro-def-omegaF} with $(\n^F,h^F)$ replaced by $(\n^{C^{\bullet,\bullet}_j},h^{C^{\bullet,\bullet}_j}_{R,T})$,
where the metric $h^{C^{\bullet,\bullet}_j}_{R,T}$ was defined in \eqref{eq-def-hC}.

\begin{lemme}
\label{lem-SRT-omegaRT}
For $j=0,1,2,3$ and $R \gg 1$,
we have
\begin{equation}
\label{eq-lem-SRT-omegaRT}
\omega^{C^{\bullet,\bullet}_j}_{R,T} =
\omega^{C^{\bullet,\bullet}_j}
+ \mathscr{O}_{R,T} \big(R^{-1/2+\kappa/4}\big) \;.
\end{equation}
\end{lemme}
\begin{proof}
We only prove the case $j=0$.

Let $\omega^{W^\bullet_j}_{R,T} \in \Omega^1\big(S,\mathrm{End}(W^\bullet_j)\big)$
be as in \eqref{intro-def-omegaF} with $(\n^F,h^F)$ replaced by $(\nabla^{W^\bullet_j},h^{W^\bullet_j}_{R,T})$,
where the metric $h^{W^\bullet_j}_{R,T}$ was defined in \eqref{eq-def-hW}.
More precisely,
we have
\begin{equation}
\label{eq1-pf-lem-SRT-omegaRT}
\omega^{W^\bullet_j}_{R,T} = \big(h^{W^\bullet_j}_{R,T}\big)^{-1} \nabla^{W^\bullet_j} h^{W^\bullet_j}_{R,T} \;.
\end{equation}
By \eqref{eq1-pf-lem-SRT-omegaRT} and the paragraph above Lemma \ref{lem-SRT-omegaRT},
we have
\begin{equation}
\label{eq2-pf-lem-SRT-omegaRT}
\omega^{C^{\bullet,\bullet}}_{R,T} =
\omega^{W^\bullet_1}_{R,T} \oplus
\omega^{W^\bullet_2}_{R,T} \oplus
\omega^{V^\bullet} \;.
\end{equation}
Comparing \eqref{eq-def-omegaC}
with \eqref{eq2-pf-lem-SRT-omegaRT},
it remains to show that
\begin{equation}
\label{eq3-pf-lem-SRT-omegaRT}
\omega^{W^\bullet_j}_{R,T} =
\omega^{K^\bullet_j} \oplus \omega^{K^{\bullet,\perp}_j}
+ \mathscr{O}_{R,T} \big(R^{-1/2+\kappa/4}\big) \;,\hspace{5mm} \text{for } j=1,2.
\end{equation}

Let $P_j : \; W^\bullet_j \rightarrow K^\bullet_j$ and $P^\perp_j : \; W^\bullet_j \rightarrow K^{\bullet,\perp}_j$
be the projections with respect to the decomposition $W^\bullet_j = K^\bullet_j \oplus K^{\bullet,\perp}_j$.
Set
\begin{equation}
\label{eq5-pf-lem-SRT-omegaRT}
\nabla^{W^\bullet_j,\oplus} =
P_j\nabla^{W^\bullet_j}P_j + P^\perp_j\nabla^{W^\bullet_j}P^\perp_j \;.
\end{equation}
Since $K^\bullet_j \subseteq W^\bullet_j$ is a flat sub vector bundle,
we have
\begin{equation}
\label{eq6-pf-lem-SRT-omegaRT}
\nabla^{W^\bullet_j} - \nabla^{W^\bullet_j,\oplus} =
P_j \nabla^{W^\bullet_j} P^\perp_j
\in \Omega^1\Big(S,\mathrm{Hom}\big(K^{\bullet,\perp}_j,K^\bullet_j\big)\Big) \;.
\end{equation}
We denote by $\big\lVert\cdot\big\rVert_{R,T}$
the operator norm on $\mathrm{Hom}\big(K^{\bullet,\perp}_j,K^\bullet_j\big)$ with respect to $h^{W^\bullet_j}_{R,T}$.
By \eqref{eq-def-hW},
we have
\begin{equation}
\label{eq7a-pf-lem-SRT-omegaRT}
\big\lVert\cdot\big\rVert_{R,T}
= R^{-1/2}T^{1/4} \big\lVert\cdot\big\rVert_{1,1} = R^{-1/2+\kappa/4} \big\lVert\cdot\big\rVert_{1,1} \;.
\end{equation}
By \eqref{eq6-pf-lem-SRT-omegaRT} and \eqref{eq7a-pf-lem-SRT-omegaRT},
we have
\begin{equation}
\label{eq7-pf-lem-SRT-omegaRT}
\nabla^{W^\bullet_j} - \nabla^{W^\bullet_j,\oplus} =
\mathscr{O}_{R,T}\big(R^{-1/2+\kappa/4}\big) \;.
\end{equation}
By \eqref{eq-def-hW} and \eqref{eq5-pf-lem-SRT-omegaRT},
we have
\begin{equation}
\label{eq8-pf-lem-SRT-omegaRT}
\big(h^{W^\bullet_j}_{R,T}\big)^{-1}
\nabla^{W^\bullet_j,\oplus}
h^{W^\bullet_j}_{R,T} =
\omega^{K^\bullet_j} \oplus \omega^{K^{\bullet,\perp}_j} \;.
\end{equation}
From \eqref{eq1-pf-lem-SRT-omegaRT}, \eqref{eq7-pf-lem-SRT-omegaRT}
and \eqref{eq8-pf-lem-SRT-omegaRT},
we obtain \eqref{eq3-pf-lem-SRT-omegaRT}.
This completes the proof of Lemma \ref{lem-SRT-omegaRT}.
\end{proof}

\begin{proof}[Proof of Theorem \ref{thm-central-tf}]
Applying Remark \ref{lem-notlfat} to the map
$\mathscr{S}_{j,R,T}: C^{\bullet,\bullet}_j \rightarrow \mathscr{E}^{[-1,1]}_{j,R,T}$
and using Theorem \ref{thm-central-SRT},
Lemmas \ref{lem-SRT-omega}, \ref{lem-SRT-omegaRT},
\eqref{eq-def-TR} and \eqref{eq-def-analytic-tf-tilde},
we get
\begin{equation}
\label{eq3-pf-thm-central-tf}
\widetilde{\mathscr{T}}_{j,R} - \widehat{\mathscr{T}}_{j,R} =
\mathscr{O}\big(R^{-1/4+\kappa/8}\big) \;.
\end{equation}
From Proposition \ref{prop-fwedge-proj}
and \eqref{eq3-pf-thm-central-tf},
we obtain \eqref{eq-thm-central-tf}.
This completes the proof of Theorem \ref{thm-central-tf}.
\end{proof}

\section{Torsion forms associated with the Mayer-Vietoris exact sequence}
\label{sect-mv}

The purpose of this section is to prove Theorem \ref{thm-central-mv}.
This section is organized as follows.
In \textsection \ref{subsect-fil-mv},
we introduce a filtration of the Mayer-Vietoris exact sequence in question.
In \textsection \ref{subsect-mv-estimate},
we estimate the torsion form associated with the Mayer-Vietoris exact sequence.
Theorem \ref{thm-central-mv} will be proved in this subsection.
In the whole section,
we take $T=R^\kappa$,
where $\kappa\in]0,1/3[$ is a fixed constant.
For ease of notations,
we will systematically omit a parameter ($R$ or $T$)
as long as there is no confusion.

\subsection{A filtration of the Mayer-Vietoris exact sequence}
\label{subsect-fil-mv}

Recall that $W^\bullet_1$, $W^\bullet_2$, $V^\bullet$, $V^\bullet_1$ and $V^\bullet_2$
were defined by \eqref{eq-def-V1-V2} and \eqref{eq-complex-evaluation}.
Recall that $W^\bullet_{12}\subseteq W^\bullet_1 \oplus W^\bullet_2$
was defined by \eqref{eq-def-W12}.
For convenience,
we denote $V^\bullet_\mathrm{quot} = V^\bullet/(V^\bullet_1+V^\bullet_2)$.

For $k\in\N$, we construct a truncation of the exact sequence \eqref{eq-mv-top} as follows,
\begin{equation}
\label{eq-truncation-mv}
\xymatrix{
\cdots \ar[r] &
H^k(Z,F) \ar[r] &
W^k_1 \oplus W^k_2 \ar[r] &
V^k \ar[r] &
V^k_\mathrm{quot} \ar[r] &
0 \;. }
\end{equation}
The truncations of \eqref{eq-mv-top} at degree $k-1$ and $k$ fit into
the following commutative diagram with exact rows and columns,
\begin{align}
\label{eq-truncation-mv-quot}
\begin{split}
\xymatrix{
\cdots \ar[r] &
V^{k-1}_\mathrm{quot} \ar[r] \ar[d]  &
0 &
&
&
\\
\cdots \ar[r] &
H^k(Z,F) \ar[r] \ar[d] &
W^k_1 \oplus W^k_2 \ar[r] \ar[d] &
V^k \ar[r] \ar[d] &
V^k_\mathrm{quot} \ar[r] \ar[d] &
0 \\
0 \ar[r] &
W^k_{12} \ar[r] &
W^k_1 \oplus W^k_2 \ar[r] &
V^k \ar[r] &
V^k_\mathrm{quot} \ar[r] &
0 \;.}
\end{split}
\end{align}

We equip
$H^\bullet(Z,F)$, $W^\bullet_1$, $W^\bullet_2$ and $V^\bullet$ in \eqref{eq-truncation-mv-quot}
with Hermitian metrics induced by $\big\lVert\cdot\big\rVert_{Z_{j,R}}$ ($j=0,1,2,3$)
via the identification \eqref{eq-hodge-DRT}.
We equip $W^\bullet_{12}$ in \eqref{eq-truncation-mv-quot}
with the Hermitian metric induced by $h^{W^\bullet_1}_{R,T} \oplus h^{W^\bullet_2}_{R,T}$,
which were defined in \eqref{eq-def-hW},
via the embedding $W^\bullet_{12}\hookrightarrow W^\bullet_1 \oplus W^\bullet_2$.
Set
\begin{equation}
\label{eq-def-aRT}
a_{R,T} = R \int_{-1}^1 \chi_3(s) e^{2Tf_T(s) - T} ds \;,
\end{equation}
where $\chi_3: \R \rightarrow \R$ was defined in \eqref{eq-def-chi3}.
We equip
$V^\bullet_\mathrm{quot}$
in \eqref{eq-truncation-mv-quot}
with the quotient metric of $ a_{R,T}^{-2} h^{V^\bullet}_{R,T}$,
where $h^{V^\bullet}_{R,T}$ was defined in \eqref{eq-def-hVRT}.
Let $\mathscr{T}_{\mathrm{hor},R,T}^k$ be
the torsion form associated with the third row in \eqref{eq-truncation-mv-quot}.
Let $\mathscr{T}_{\mathrm{vert},R,T}^k$ be
the torsion form associated with the (unique) non trivial column in \eqref{eq-truncation-mv-quot}.

\begin{prop}
\label{prop-decomp-mv}
The following identity holds in $Q^S/Q^{S,0}$,
\begin{equation}
\label{eq-prop-decomp-mv}
\mathscr{T}_{\mathscr{H},R,T} =
\sum_{k=0}^n (-1)^k \mathscr{T}_{\mathrm{hor},R,T}^k
- \sum_{k=1}^n (-1)^k \mathscr{T}_{\mathrm{vert},R,T}^k \;.
\end{equation}
\end{prop}
\begin{proof}
Let $\mathscr{T}_{\mathscr{H},R,T}(k)$ be the torsion form associated with \eqref{eq-truncation-mv}.
In particular,
$\mathscr{T}_{\mathscr{H},R,T}(-1) = 0$ and
$\mathscr{T}_{\mathscr{H},R,T}(\dim Z) = \mathscr{T}_{\mathscr{H},R,T}$.
Applying \cite[Thm. A1.4]{bl} to \eqref{eq-truncation-mv-quot},
we get
\begin{equation}
\label{eq1-pf-prop-decomp-mv}
\mathscr{T}_{\mathscr{H},R,T}(k-1)
- \mathscr{T}_{\mathscr{H},R,T}(k)
+ (-1)^k \mathscr{T}_{\mathrm{hor},R,T}^k
- (-1)^k \mathscr{T}_{\mathrm{vert},R,T}^k \in Q^{S,0} \;.
\end{equation}
Taking the sum of \eqref{eq1-pf-prop-decomp-mv} for $k=0,1,\cdots,\dim Z$,
we obtain \eqref{eq-prop-decomp-mv}.
This completes the proof of Proposition \ref{prop-decomp-mv}.
\end{proof}

\subsection{Estimating $\mathscr{T}_{\mathrm{vert},R,T}^k$ and $\mathscr{T}_{\mathrm{hor},R,T}^k$}
\label{subsect-mv-estimate}

For $j=0,1,2,3$,
under the identification \eqref{eq-id-HC},
we view $H^\bullet(C^{\bullet,\bullet}_j,\partial)$ as a vector subspace of $C^{\bullet,\bullet}_j$.
Let
\begin{equation}
P^H_j:
C^{\bullet,\bullet}_j \rightarrow H^\bullet(C^{\bullet,\bullet}_j,\partial)
\end{equation}
be the orthogonal projection with respect to $h^{C^{\bullet,\bullet}_j}_{R,T}$ (see \eqref{eq-def-hC}).
Note that $K_j^\bullet \subseteq H^0(C^{\bullet,\bullet}_j,\partial)$,
by \eqref{eq-def-hW} and \eqref{eq-def-hC},
$P^H_j$ is independent of $R,T$.
Let $\omega^{H^\bullet(C^{\bullet,\bullet}_j,\partial)}$
(resp. $\omega^{H^\bullet(C^{\bullet,\bullet}_j,\partial)}_{R,T}$)
be the $1$-form on $S$ with values in $\mathrm{End}(C^{\bullet,\bullet}_j)$
induced by  $\omega^{C^{\bullet,\bullet}_j}$ in \eqref{eq-def-omegaC}
(resp. $\omega^{C^{\bullet,\bullet}_j}_{R,T}$ in \eqref{eq-lem-SRT-omegaRT})
via the projection $P^H_j$.
More precisely,
we have
\begin{align}
\label{eq-PH-omega}
\begin{split}
\omega^{H^\bullet(C^{\bullet,\bullet}_j,\partial)} & =
P^H_j \omega^{C^{\bullet,\bullet}_j} P^H_j
\in \Omega^1\big(S,\mathrm{End}\big(H^\bullet(C^{\bullet,\bullet}_j,\partial)\big)\big) \;,\\
\omega^{H^\bullet(C^{\bullet,\bullet}_j,\partial)}_{R,T} & =
P^H_j \omega^{C^{\bullet,\bullet}_j}_{R,T} P^H_j
\in \Omega^1\big(S,\mathrm{End}\big(H^\bullet(C^{\bullet,\bullet}_j,\partial)\big)\big) \;.
\end{split}
\end{align}
By \eqref{eq-HC},
we have $H^\bullet(C^{\bullet,\bullet}_j,\partial) = H^0(C^{\bullet,\bullet}_j,\partial) = W^\bullet_j$ for $j=1,2$.
Then we have
\begin{equation}
\omega^{H^\bullet(C^{\bullet,\bullet}_j,\partial)}_{R,T} = \omega^{W^\bullet_j}_{R,T}
\end{equation}
for $j=1,2$,
where $\omega^{W^\bullet_j}_{R,T}$ was defined in \eqref{eq1-pf-lem-SRT-omegaRT}.

For $j=0,1,2,3$,
let $h^{H^\bullet(Z_j,F)}_{R,T}$ be
the Hermitian metric on $H^\bullet(Z_j,F) = H^\bullet(Z_{j,R},F)$
induced by $\big\lVert\cdot\big\rVert_{Z_{j,R}}$
via the identification \eqref{eq-hodge-DRT}.
Let $\nabla^{H^\bullet(Z_j,F)}$ be the canonical flat connection on $H^\bullet(Z_j,F)$
(see \cite[Def. 2.4]{bl}).
Let $\omega^{H^\bullet(Z_j,F)}_{R,T} \in \Omega^1\big(S,\mathrm{End}\big(H^\bullet(Z_j,F)\big)\big)$
be as in \eqref{intro-def-omegaF} with $(\n^F,h^F)$ replaced by $(\nabla^{H^\bullet(Z_j,F)},h^{H^\bullet(Z_j,F)}_{R,T})$.

Let $\big[\mathscr{S}^H_{j,R,T}\big]_T: H^\bullet(C^{\bullet,\bullet},\partial) \rightarrow H^\bullet(Z_j,F)$
be the map defined by \eqref{eq2-def-SHRT}, \eqref{eq-def-T-class}, \eqref{eq-def-SRTj} and \eqref{eq-def-SRT3}.
Recall that the notation $\mathscr{O}_{R,T}\big(\cdot\big)$
was defined in the paragraph above Theorem \ref{thm-central-SRT}.

\begin{lemme}
\label{lem-omegaH}
For $R \gg 1$,
we have
\begin{align}
\label{eq-lem-omegaH}
\begin{split}
\Big(\big[\mathscr{S}^H_{j,R,T}\big]_T\Big)^{-1} \circ
\omega^{H^\bullet(Z_j,F)}_{R,T} \circ
\big[\mathscr{S}^H_{j,R,T}\big]_T & =
\omega^{H^\bullet(C^{\bullet,\bullet}_j,\partial)}
+ \mathscr{O}_{R,T}\big(R^{-1/2+\kappa/4}\big) \;,\\
\omega^{H^\bullet(C^{\bullet,\bullet}_j,\partial)}_{R,T} & =
\omega^{H^\bullet(C^{\bullet,\bullet}_j,\partial)}
+ \mathscr{O}_{R,T}\big(R^{-1/2+\kappa/4}\big) \;.
\end{split}
\end{align}
\end{lemme}
\begin{proof}
Recall that $P_{j,R,T}: \mathscr{F}_{j,R} \rightarrow \Ker\big(\DjRT\big)$ was defined
above Proposition \ref{prop-FRT}, above \eqref{eq-def-SRTj} and above \eqref{eq-def-SRT3}.
By \cite[Prop. 3.14]{bl},
the following identity holds under the identification \eqref{eq-hodge-DRT},
\begin{equation}
\label{eq1-pf-lem-omegaH}
\omega^{H^\bullet(Z_j,F)}_{R,T} =
P_{j,R,T} \omega^{\mathscr{F}_{j,R}} P_{j,R,T} \;,
\end{equation}
where $\omega^{\mathscr{F}_{j,R}}$ was defined in \textsection \ref{subsect-tf}.
By Lemma \ref{lem-SRT-omega} and \eqref{eq-pf-cor-SRT-metric},
we have
\begin{equation}
\label{eq2-pf-lem-omegaH}
\Big(\mathscr{S}^H_{j,R,T}\Big)^{-1} \circ
\Big(P_{j,R,T} \omega^{\mathscr{F}_{j,R}} P_{j,R,T}\Big)
\circ \mathscr{S}^H_{j,R,T} =
P^H_j \omega^{C^{\bullet,\bullet}_j} P^H_j
+ \mathscr{O}_{R,T} \big(R^{-1/2+\kappa/4}\big) \;.
\end{equation}
From \eqref{eq-PH-omega}, \eqref{eq1-pf-lem-omegaH} and \eqref{eq2-pf-lem-omegaH},
we obtain the first identity in \eqref{eq-lem-omegaH}.
The second identity in \eqref{eq-lem-omegaH}
is a direct consequence of Lemma \ref{lem-SRT-omegaRT} and \eqref{eq-PH-omega}.
This completes the proof of Lemma \ref{lem-omegaH}.
\end{proof}

\begin{prop}
\label{prop-vert-mv}
For $R \gg 1$,
we have
\begin{equation}
\label{eq-prop-vert-mv}
\mathscr{T}_{\mathrm{vert},R,T}^k
= \mathscr{O}\big(R^{-1/4+\kappa/8}\big) \;.
\end{equation}
\end{prop}
\begin{proof}
Recall that
$H^0(C^{\bullet,k}_0,\partial) = W^k_{12}$ and
$H^1(C^{\bullet,k-1}_0,\partial) = V^{k-1}_\mathrm{quot}$.
First we show that the following diagram commutes,
\begin{equation}
\label{eq1-pf-prop-vert-mv}
\xymatrix{
V^{k-1}_\mathrm{quot}
\ar[d]^{a_{R,T} \mathrm{Id}} \ar[r]  &
V^{k-1}_\mathrm{quot} \oplus W^k_{12}
\ar[d]^{\big[\mathscr{S}^H_{R,T}\big]_T} \ar[r] &
W^k_{12}
\ar[d]^{\mathrm{Id}} \\
V^{k-1}_\mathrm{quot}
\ar[r]^{\hspace{-5mm}\delta} &
H^k(Z,F)
\ar[r]^{\hspace{2.5mm}\alpha} &
W^k_{12} \;,
}
\end{equation}
where $a_{R,T}$ was defined by \eqref{eq-def-aRT},
the first row consists of canonical injection and projection,
the second row is the (unique) non trivial column in \eqref{eq-truncation-mv-quot}.
We remark that
\eqref{eq1-pf-prop-vert-mv}
is not a commutative diagram of flat complex vector bundles over $S$.

Let $\eta: [-R,R] \rightarrow \R$ be a smooth function such that
\begin{equation}
\label{eq1e-pf-prop-vert-mv}
\eta\big|_{[-R,-R/2]}=0 \;,\hspace{5mm}
\eta\big|_{[R/2,R]}=1 \;.
\end{equation}
We will view $\eta$ as a function on $IY_R$.
Let $\sigma \in \mathscr{H}^{k-1}(Y,F) = V^{k-1}$.
Let $\overline{\sigma}\in V^{k-1}_\mathrm{quot}$ be the image of $\sigma$.
Let $\omega\in\Omega^\bullet(Z_R,F)$ such that
\begin{equation}
\label{eq11-pf-prop-vert-mv}
\omega\big|_{Z_{1,0}} = 0 \;,\hspace{5mm}
\omega\big|_{Z_{2,0}} = 0 \;,\hspace{5mm}
\omega\big|_{IY_R} = d\eta \wedge \sigma \;.
\end{equation}
Then we have
\begin{equation}
\label{eq10-pf-prop-vert-mv}
\delta(\overline{\sigma}) = [\omega] \in H^k(Z_R,F) = H^k(Z,F) \;.
\end{equation}
Let $\sigma'\in \big(V^{k-1}_1+V^{k-1}_2\big)^\perp \subseteq V^{k-1}$.
By \eqref{eq0-def-IRT}-\eqref{eq3-def-IRT},
we have
\begin{equation}
\label{eq13-pf-prop-vert-mv}
I_{R,T}(\sigma')\big|_{Z_{1,0}} = 0 \;,\hspace{5mm}
I_{R,T}(\sigma')\big|_{Z_{2,0}} = 0 \;,\hspace{5mm}
I_{R,T}(\sigma')\big|_{IY_R} = \chi_3 e^{Tf_T-T} du \wedge \sigma' \;.
\end{equation}
Let $\overline{\sigma}'\in V^{k-1}_\mathrm{quot}$ be the image of $\sigma'$.
By \eqref{eq2-def-SHRT}, \eqref{eq-def-T-class}, \eqref{eq-def-SRTj} and \eqref{eq-def-SRT3},
we have
\begin{equation}
\label{eq12-pf-prop-vert-mv}
\big[\mathscr{S}^H_{R,T}\big]_T(\overline{\sigma}') =
\big[e^{Tf_T}I_{R,T}(\sigma')\big] \in H^k(Z_R,F) = H^k(Z,F) \;.
\end{equation}
By \eqref{eq-def-aRT}
and \eqref{eq1e-pf-prop-vert-mv}-\eqref{eq12-pf-prop-vert-mv},
we have
\begin{equation}
\label{eq14-pf-prop-vert-mv}
\big[\mathscr{S}^H_{R,T}\big]_T(\overline{\sigma}) =
\Big( \int_{-R}^R \chi_3(u)e^{2Tf_T(u)-T}du \Big) \delta(\overline{\sigma}) =
a_{R,T} \delta(\overline{\sigma}) \;.
\end{equation}
Hence the left square in \eqref{eq1-pf-prop-vert-mv} commutes.

Let $(\omega_1,\omega_2,\hat{\omega})\in\mathscr{H}^k_\mathrm{abs}(Z_{12,\infty},F)$.
Its image in $W^k_{12}$ via the identification \eqref{eq-iso-hha12}
is given by $\Big( \big[\omega_1\big|_{Z_{1,0}}\big],\big[\omega_2\big|_{Z_{2,0}}\big] \Big)$.
By \eqref{eq2-def-SHRT}, \eqref{eq-def-T-class}, \eqref{eq-def-SRTj} and \eqref{eq-def-SRT3},
we have
\begin{align}
\label{eq16-pf-prop-vert-mv}
\begin{split}
& \big[\mathscr{S}^H_{R,T}\big]_T \Big( \big[\omega_1\big|_{Z_{1,0}}\big],\big[\omega_2\big|_{Z_{2,0}}\big] \Big) \\
& = \big[e^{Tf_T} F_{R,T}(\omega_1,\omega_2,\hat{\omega})\big]
\in H^k(Z_R,F) = H^k(Z,F) \;.
\end{split}
\end{align}
By \eqref{eq0-def-FRT}-\eqref{eq3-def-FRT},
we have
\begin{equation}
\label{eq17-pf-prop-vert-mv}
F_{R,T}(\omega_1,\omega_2,\hat{\omega})\big|_{Z_{1,0}}
= \omega_1\big|_{Z_{1,0}} \;,\hspace{5mm}
F_{R,T}(\omega_1,\omega_2,\hat{\omega})\big|_{Z_{2,0}}
= \omega_2\big|_{Z_{2,0}} \;.
\end{equation}
On the other hand,
for $[\omega]\in H^k(Z_R,F) = H^k(Z,F)$,
we have
\begin{equation}
\label{eq15-pf-prop-vert-mv}
\alpha([\omega]) = \Big( \big[\omega\big|_{Z_{1,0}}\big],\big[\omega\big|_{Z_{2,0}}\big] \Big)
\in W^k_{12} \subseteq W^k_1 \oplus W^k_2 \;.
\end{equation}
By \eqref{eq16-pf-prop-vert-mv}-\eqref{eq15-pf-prop-vert-mv},
we have
\begin{equation}
\label{eq18-pf-prop-vert-mv}
\alpha \circ \big[\mathscr{S}^H_{R,T}\big]_T
\Big( \big[\omega_1\big|_{Z_{1,0}}\big],\big[\omega_2\big|_{Z_{2,0}}\big] \Big) =
\Big( \big[\omega_1\big|_{Z_{1,0}}\big],\big[\omega_2\big|_{Z_{2,0}}\big] \Big) \;.
\end{equation}
Hence the right square in \eqref{eq1-pf-prop-vert-mv} commutes.

We equip
$H^k(Z,F) = H^k(Z_R,F)$ in \eqref{eq1-pf-prop-vert-mv}
with the Hermitian metric induced by $\big\lVert\cdot\big\rVert_{Z_R}$
via the identification \eqref{eq-hodge-DRT}.
We equip $W^k_{12}$ in \eqref{eq1-pf-prop-vert-mv}
with the Hermitian metric induced by $h^{W^\bullet_1}_{R,T} \oplus h^{W^\bullet_2}_{R,T}$
via the embedding $W^k_{12}\hookrightarrow W^k_1 \oplus W^k_2$.
We equip
$V^\bullet_\mathrm{quot}$
in the first row of \eqref{eq1-pf-prop-vert-mv}
with the quotient metric of $h^{V^\bullet}_{R,T}$.
We equip
$V^\bullet_\mathrm{quot}$
in the second row of \eqref{eq1-pf-prop-vert-mv}
with the quotient metric of $ a_{R,T}^{-2} h^{V^\bullet}_{R,T}$.
Then the torsion form of the first row in \eqref{eq1-pf-prop-vert-mv}
vanishes,
and the torsion form of the second row in \eqref{eq1-pf-prop-vert-mv}
equals $\mathscr{T}_{\mathrm{vert},R,T}^k$.
Applying Proposition \ref{prop-comparison-tf} to \eqref{eq1-pf-prop-vert-mv}
and using Corollary \ref{cor-SRT-metric} and Lemma \ref{lem-omegaH},
we obtain \eqref{eq-prop-vert-mv}.
This completes the proof of Proposition \ref{prop-vert-mv}.
\end{proof}

\begin{prop}
\label{prop-hor-mv}
For $R \gg 1$,
the following identity holds in $Q^S/Q^{S,0}$,
\begin{equation}
\label{eq-prop-hor-mv}
\mathscr{T}_{\mathrm{hor},R,T}^k
= \widehat{\mathscr{T}}_{\mathscr{H},R,T}^k + \mathscr{O}\big(R^{-1/4+\kappa/8}\big)
\end{equation}
with $\mathscr{T}_{\mathrm{hor},R,T}^k$ being as in \eqref{eq-prop-decomp-mv}
and $\widehat{\mathscr{T}}_{\mathscr{H},R,T}^k$ being as in \eqref{eq-def-THR}.
\end{prop}
\begin{proof}
We denote $b_{R,T} = \pi^{1/2}RT^{-1/2}e^T$.
By \eqref{eq-def-aRT},
there exists $a>0$ such that
\begin{equation}
\label{eq-abRT}
a_{R,T} = b_{R,T} \Big( 1 + \mathscr{O}\big(e^{-aT}\big) \Big)  \;.
\end{equation}

Let $p: V^k \rightarrow V^k_\mathrm{quot}$ be the canonical projection.
The following commutative diagram is obvious,
\begin{equation}
\label{eq2-pf-prop-hor-mv}
\xymatrix{
W^k_{12} \ar[r] \ar[d]^{\mathrm{Id}} &
W^k_1 \oplus W^k_2   \ar[r] \ar[d]^{\mathrm{Id}} &
V^k \ar[r]^{b_{R,T}^{-1}p} \ar[d]^{\mathrm{Id}} &
V^k_\mathrm{quot}  \ar[d]^{b_{R,T}\mathrm{Id}} \\
W^k_{12} \ar[r] &
W^k_1 \oplus W^k_2 \ar[r] &
V^k \ar[r]^{\hspace{-2.5mm}p} &
V^k_\mathrm{quot}  \;.}
\end{equation}

We equip $W^k_{12}$ in \eqref{eq2-pf-prop-hor-mv}
with the Hermitian metric induced by $h^{W^k_1}_{R,T} \oplus W^{W^k_2}_{R,T}$
via the inclusion $W^k_{12} \subseteq W^k_1 \oplus W^k_2$.
We equip $W^k_1 \oplus W^k_2$ in the first row of \eqref{eq2-pf-prop-hor-mv}
with the Hermitian metric $h^{W^k_1}_{2R,T} \oplus h^{W^k_2}_{2R,T}$.
We equip $W^k_1 \oplus W^k_2 = H^k(Z_{1,R},F) \oplus H^k(Z_{2,R},F)$
in the second row of \eqref{eq2-pf-prop-hor-mv}
with the Hermitian metric induced by $\big\lVert\cdot\big\rVert_{Z_{j,R}}$ ($j=1,2$)
via the identification \eqref{eq-hodge-DRT}.
We equip $V^k$ in the first row of \eqref{eq2-pf-prop-hor-mv}
with the Hermitian metric $h^{V^k}_{R,T}$.
We equip $V^k$ in the second row of \eqref{eq2-pf-prop-hor-mv}
with the Hermitian metric induced by $\big\lVert\cdot\big\rVert_{IY_R}$
via the identification \eqref{eq-hodge-DRT}.
We equip
$V^\bullet_\mathrm{quot}$
in the first (resp. second) row of \eqref{eq2-pf-prop-hor-mv}
with the quotient metric of $h^{V^\bullet}_{R,T}$
(resp. $a_{R,T}^{-2} h^{V^\bullet}_{R,T}$).
The torsion form of the first row of \eqref{eq2-pf-prop-hor-mv}
is given by $\widehat{\mathscr{T}}_{\mathscr{H},R,T}^k$,
and the torsion form of the second row of \eqref{eq2-pf-prop-hor-mv}
is given by $\mathscr{T}_{\mathrm{hor},R,T}^k$.
For $j=1,2,3,4$,
let $\mathscr{T}_j\in Q^S$ be the torsion form of the $j$-th column in \eqref{eq2-pf-prop-hor-mv}.
Applying \cite[Thm. A1.4]{bl} to \eqref{eq2-pf-prop-hor-mv},
we get
\begin{equation}
\label{eq3-pf-prop-hor-mv}
\widehat{\mathscr{T}}_{\mathscr{H},R,T}^k - \mathscr{T}_{\mathrm{hor},R,T}^k
- \mathscr{T}_1 + \mathscr{T}_2 - \mathscr{T}_3 + \mathscr{T}_4
\in Q^{S,0} \;.
\end{equation}
Since the first vertical map is isometric,
we have
\begin{equation}
\label{eq4-pf-prop-hor-mv}
\mathscr{T}_1 = 0 \;.
\end{equation}
By Corollary \ref{cor-comparison-tf},
Corollary \ref{cor-SRT-metric}, Remark \ref{rem-j} and Lemma \ref{lem-omegaH},
we have
\begin{equation}
\label{eq5-pf-prop-hor-mv}
\mathscr{T}_2 = \mathscr{O}\big(R^{-1/4+\kappa/8}\big) \;,\hspace{5mm}
\mathscr{T}_3 = \mathscr{O}\big(R^{-1/4+\kappa/8}\big) \;.
\end{equation}
By Corollary \ref{cor-comparison-tf} and \eqref{eq-abRT},
we have
\begin{equation}
\label{eq6-pf-prop-hor-mv}
\mathscr{T}_4 =
\mathscr{O}\big(e^{-aT/2}\big) \;.
\end{equation}
From \eqref{eq3-pf-prop-hor-mv}-\eqref{eq6-pf-prop-hor-mv},
we obtain \eqref{eq-prop-hor-mv}.
This completes the proof of Proposition \ref{prop-hor-mv}.
\end{proof}

\begin{proof}[Proof of Theorem \ref{thm-central-mv}]
We combine Propositions \ref{prop-decomp-mv}, \ref{prop-vert-mv}, \ref{prop-hor-mv} and \eqref{eq-def-THR}.
\end{proof}


\def\cprime{$'$}
\providecommand{\bysame}{\leavevmode\hbox to3em{\hrulefill}\thinspace}
\providecommand{\MR}{\relax\ifhmode\unskip\space\fi MR }
\providecommand{\MRhref}[2]{%
  \href{http://www.ams.org/mathscinet-getitem?mr=#1}{#2}
}
\providecommand{\href}[2]{#2}


\end{document}